\documentclass[11pt,reqno]{amsart}
\usepackage{amsmath,amsfonts,amssymb}
\usepackage{mathrsfs}

\def\C{\mathbb C}

\def\O{\mathcal O}

\def\x{\mathbf x}

\def\D{\mathbf D}

\def\1{\mathbf 1}

\def\O{\mathcal O}

\def\1{\bold 1}

\def\div{\mathrm{div}\,}

\def\eps{\varepsilon}

\def\le{\leqslant}

\def\ge{\geqslant}

\usepackage{amsthm}
\theoremstyle{theorem}
\newtheorem{theorem}{Theorem}[section]
\newtheorem{proposition}[theorem]{Proposition}
\newtheorem{lemma}[theorem]{Lemma}
\newtheorem{condition}[theorem]{Condition}
\newtheorem{remark}[theorem]{Remark}
\newtheorem{corollary}[theorem]{Corollary}

\numberwithin{equation}{section}

\theoremstyle{plain}
\newtoks\thehProclaim
\newtheorem*{Proclaim}{\the\thehProclaim}

\usepackage[paperwidth=20.5cm,paperheight=29.7cm,width=16cm,height=25.5cm,top=24mm,left=19mm,headsep=11mm]{geometry}

\begin{document}

\title[Homogenization of the Dirichlet problem]{Homogenization of the Dirichlet problem for~elliptic systems: Two-parametric error estimates}

\author{Yu.~M.~Meshkova and T.~A.~Suslina}

\thanks{Supported by RFBR (grant no.~16-01-00087). The first author was supported by
``{Native Towns}'', a social investment program of PJSC ``{Gazprom Neft}'', by the ``{Dynasty}'' \ foundation, and by the Rokhlin grant.}


\address{Chebyshev Laboratory, St. Petersburg State University, 14th Line V.O., 29b, St.~Petersburg, 199178, Russia}
\email{y.meshkova@spbu.ru,\quad juliavmeshke@yandex.ru}

 \address{St.~Petersburg State University, 7/9 Universitetskaya nab., St.~Petersburg, 199034, Russia}

\email{t.suslina@spbu.ru,\quad suslina@list.ru}


\begin{abstract}
Let $\mathcal{O}\subset\mathbb{R}^d$ be a bounded domain of class $C^{1,1}$.
In $L_2(\mathcal{O};\mathbb{C}^n)$, we study a selfadjoint matrix elliptic second order differential operator
$B_{D,\varepsilon}$, $0<\varepsilon\leqslant 1$, with the Dirichlet boundary condition. The principal part of the operator is given in a factorized form. The operator involves lower order terms with unbounded coefficients. The coefficients of $B_{D,\varepsilon}$ are periodic and depend on $\mathbf{x}/\varepsilon$.
We study the generalized resolvent $\left(B_{D,\varepsilon}-\zeta Q_0(\cdot/\varepsilon)\right)^{-1}$, where $Q_0$ is a periodic bounded and positive definite matrix-valued function, and $\zeta$ is a complex-valued parameter. We obtain approximations for the generalized resolvent in the $L_2(\mathcal{O};\mathbb{C}^n)$-operator norm and in the norm of operators acting from $L_2(\mathcal{O};\mathbb{C}^n)$ to the Sobolev space $H^1(\mathcal{O};\mathbb{C}^n)$, with two-parametric error estimates (depending on $\varepsilon$ and $\zeta$).
\end{abstract}
\maketitle

\section*{Introduction}
\setcounter{section}{0}
\setcounter{equation}{0}

The paper concerns homogenization theory of periodic differential operators (DO's). A broad literature is devoted to homogenization problems.
First of all, we mention the books \cite{BeLP,BaPa,OShaY,ZhKO}.

\subsection{Statement of the problem}
Let $\Gamma$ be a lattice in $\mathbb{R}^d$, and let $\Omega$ be the cell of  $\Gamma$.
 For $\Gamma$-periodic functions in $\mathbb{R}^d$, we use the notation  $\psi ^\varepsilon (\mathbf{x}):=\psi (\mathbf{x}/\varepsilon)$ and $\overline{\psi}:=\vert \Omega\vert ^{-1}\int _\Omega \psi (\mathbf{x})\,d\mathbf{x}$.

Let $\mathcal{O}\subset\mathbb{R}^d$ be a bounded domain of class  $C^{1,1}$. In $L_2(\mathcal{O};\mathbb{C}^n)$, we consider a
selfadjoint matrix strongly elliptic second order DO  $\mathcal{B}_{D,\varepsilon}$, $0<\varepsilon\leqslant 1$,
with the Dirichlet boundary condition.
The principal part $A_{D,\varepsilon}$ of the operator $\mathcal{B}_{D,\varepsilon}$ is given in a factorized form
$A_{D,\varepsilon}=b(\mathbf{D})^*g^\varepsilon (\mathbf{x})b(\mathbf{D})$,
where $b(\mathbf{D})$ is a matrix homogeneous first order DO and
$g(\mathbf{x})$ is a bounded and positive definite $\Gamma$-periodic
matrix-valued function in $\mathbb{R}^d$.  (The precise assumptions on $b(\mathbf{D})$ and $g(\mathbf{x})$ are given below in Subsection~\ref{Subsection Operator A_eps}.) The homogenization problem for the operator $A_{D,\varepsilon}$
was studied in \cite{PSu,Su13,Su15}. In the present paper, we consider a more general class
of selfadjoint DO's $\mathcal{B}_{D,\varepsilon}$ involving lower order terms:
\begin{equation}
\label{B_D,eps in introduction}
\mathcal{B}_{D,\varepsilon}=b(\mathbf{D})^* g^\varepsilon (\mathbf{x}) b(\mathbf{D})
+\sum_{j=1}^d\bigl(a_j^\varepsilon (\mathbf{x})D_j+D_ja_j^\varepsilon(\mathbf{x})^*\bigr)
+Q^\varepsilon (\mathbf{x}).
\end{equation}
Here $a_j(\mathbf{x})$, $j=1,\dots,d$, and $Q(\mathbf{x})$ are $\Gamma$-periodic matrix-valued functions; in general, they are unbounded.
(The precise assumptions on the coefficients are given below in Subsection~\ref{Subsection lower order terms}). The precise definition
of the operator $\mathcal{B}_{D,\varepsilon}$ is given in terms of the corresponding quadratic form defined on the Sobolev space
$H^1_0(\mathcal{O};\mathbb{C}^n)$.

The coefficients of the operator~\eqref{B_D,eps in introduction} oscillate rapidly
for small $\varepsilon$. A typical homogenization problem for the operator $\mathcal{B}_{D,\varepsilon}$ is to approximate
the resolvent $(\mathcal{B}_{D,\varepsilon}- z I)^{-1}$ or the generalized resolvent $(\mathcal{B}_{D,\varepsilon}- z Q_0^\varepsilon)^{-1}$
for small  $\varepsilon$. Here $Q_0(\mathbf{x})$ is a positive definite and bounded $\Gamma$-periodic matrix-valued function.

\subsection{A survey of the results on the operator error estimates}
 In a series of papers \cite{BSu, BSu05, BSu06}, M.~Sh.~Birman and T.~A.~Suslina developed an operator-theoretic (spectral) approach
to homogenization problems. They studied the operator
\begin{equation}
\label{A_eps = in introduction}
A_\varepsilon =b(\mathbf{D})^*g^\varepsilon (\mathbf{x})b(\mathbf{D})
\end{equation}
acting in $L_2(\mathbb{R}^d;\mathbb{C}^n)$. In \cite{BSu}, it was shown that
the resolvent $(A_\varepsilon +I)^{-1}$ converges in the $L_2(\mathbb{R}^d;\mathbb{C}^n)$-operator norm to
the resolvent of the effective operator $A^0=b(\mathbf{D})^*g^0b(\mathbf{D})$, as  $\varepsilon\rightarrow 0$.
Here $g^0$ is a constant positive effective matrix. It was proved that
\begin{equation}
\label{A_e res ots in introduction}
\Vert (A_\varepsilon +I)^{-1}-(A^0 +I)^{-1}\Vert _{L_2(\mathbb{R}^d)\rightarrow L_2(\mathbb{R}^d)}
\leqslant C\varepsilon.
\end{equation}
In \cite{BSu06}, approximation for the resolvent $(A_\varepsilon +I)^{-1}$ in the norm of operators acting from
$L_2(\mathbb{R}^d;\mathbb{C}^n)$ to the Sobolev space $H^1(\mathbb{R}^d;\mathbb{C}^n)$ was obtained:
\begin{equation}
\label{appr res A_eps with corrector inintroduction}
\Vert (A_\varepsilon +I)^{-1}-(A^0+I)^{-1}-\varepsilon K(\varepsilon)\Vert _{L_2(\mathbb{R}^d)\rightarrow H^1(\mathbb{R}^d)}
\leqslant C\varepsilon.
\end{equation}
Here $K(\varepsilon)$ is a corrector. The operator $K(\varepsilon )$ involves rapidly oscillating factors and so depends on $\varepsilon$.
Herewith, $\Vert \varepsilon K(\varepsilon )\Vert _{L_2\rightarrow H^1}=O(1)$.
Estimates \eqref{A_e res ots in introduction} and \eqref{appr res A_eps with corrector inintroduction} are order-sharp.
The constants in estimates are controlled explicitly in terms of the problem data.
Such inequalities are called \textit{operator error estimates} in homogenization theory.
The method of \cite{BSu, BSu05, BSu06} is based on the scaling transformation, the Floquet-Bloch theory, and the analytic perturbation theory.

Later the spectral method was adapted by T.~A.~Suslina \cite{SuAA,SuAA14} to the case of the operator
\begin{equation}
\label{B_eps in introduction}
\mathcal{B}_\varepsilon =A_\varepsilon +\sum _{j=1}^d \left( a_j^\varepsilon (\mathbf{x})D_j +D_j (a_j^\varepsilon (\mathbf{x}))^*\right) +Q^\varepsilon(\mathbf{x})
\end{equation}
acting in $L_2(\mathbb{R}^d;\mathbb{C}^n)$.
It is convenient to fix a real-valued parameter $\lambda$ so that the operator
${B}_\varepsilon:= \mathcal{B}_\varepsilon + \lambda Q_0^\varepsilon$
is positive definite. In \cite{SuAA}, the following analogs of estimates \eqref{A_e res ots in introduction} and \eqref{appr res A_eps with corrector inintroduction} were obtained:
\begin{align}
\label{B__eps appr res in introduction}
\Vert & {B}_\varepsilon^{-1}-({B}^0)^{-1}\Vert _{L_2(\mathbb{R}^d)\rightarrow L_2(\mathbb{R}^d)}
\leqslant C\varepsilon,\\
\label{B__eps appr res with corrector in introduction}
\Vert & {B}_\varepsilon^{-1}-({B}^0)^{-1} -\varepsilon \mathcal{K}(\varepsilon)\Vert _{L_2(\mathbb{R}^d)\rightarrow H^1(\mathbb{R}^d)}
\leqslant C\varepsilon.
\end{align}
Here $B^0$ is the corresponding effective operator
and $\mathcal{K}(\varepsilon)$ is the corresponding corrector.

A different approach to operator error estimates was suggested by V.~V.~Zhikov.
In \cite{Zh1,Zh2,ZhPas}, estimates of the form \eqref{A_e res ots in introduction} and \eqref{appr res A_eps with corrector inintroduction}
were obtained for the acoustics operator and the operator of elasticity theory.
The method (``modified method of first order approximation''   or ``shift method'')
was based on analysis of the first order approximation to the solution and introduction of an additional parameter.
In \cite{Zh1,Zh2,ZhPas}, in addition to problems in $\mathbb{R}^d$,  homogenization problems in a bounded domain
$\mathcal{O}\subset \mathbb{R}^d$ with the Dirichlet or Neumann boundary conditions were studied.
Further results of V.~V.~Zhikov, S.~E.~Pastukhova, and their collaborators are discussed in the recent survey \cite{ZhPasUMN}.

In the presence of lower order terms, homogenization problem for the operator
\eqref{B_eps in introduction} in $\mathbb{R}^d$ was studied by D.~I.~Borisov~\cite{Bo} (this work precedes \cite{SuAA}). The effective operator was
constructed and error estimates \eqref{B__eps appr res in introduction}, \eqref{B__eps appr res with corrector in introduction} were obtained.
Moreover, it was assumed that the coefficients depend on both fast and slow variables. However, in \cite{Bo}
the coefficients of the operator $\mathcal{B}_\varepsilon$ were assumed to be sufficiently smooth.
We also mention the very recent paper \cite{Se} by N.~N.~Senik, where the non-selfadjoint second order elliptic operator (involving lower order terms)
on an infinite cylinder was studied. The coefficients oscillate along the cylinder and belong to some classes of multipliers;
estimates of the form \eqref{B__eps appr res in introduction} and \eqref{B__eps appr res with corrector in introduction} were obtained.

Operator error estimates for the Dirichlet and Neumann problems for second order elliptic equations (without lower order terms)
in a bounded domain with sufficiently smooth boundary were studied by many authors. Apparently, the first result is due to Sh.~Moskow and M.~Vogelius who proved an estimate
\begin{equation}
\label{A_D,eps L2 ots in introduction}
\Vert A_{D,\varepsilon}^{-1}-(A_D^0)^{-1}\Vert _{L_2(\mathcal{O})\rightarrow L_2(\mathcal{O})}\leqslant C\varepsilon,
\end{equation}
see \cite[Corollary~2.2]{MoV}. Here the operator $A_{D,\varepsilon}$ acts in $L_2(\mathcal{O})$, where $\mathcal{O}\subset \mathbb{R}^2$,
and is given by $-\div g^\varepsilon (\mathbf{x})\nabla$ with the Dirichlet condition on  $\partial\mathcal{O}$. The matrix-valued
function  $g(\mathbf{x})$ is assumed to be infinitely smooth. In the case of the Neumann boundary condition,
a similar estimate was obtained in \cite[Corollary~1]{MoV2}. Also, in that paper the authors found approximation with corrector for the inverse operator in the norm of operators acting from $L_2(\mathcal{O})$ to the Sobolev space $H^1(\mathcal{O})$, with error estimate of order $O(\sqrt{\varepsilon})$.
The order of this estimate is worse than in $\mathbb{R}^d$ because of the boundary influence.

For arbitrary dimension, homogenization problems in a bounded domain with sufficiently smooth boundary
were studied in  \cite{Zh1,Zh2}, and \cite{ZhPas}.
The acoustics and elasticity operators with the Dirichlet or Neumann boundary conditions and without any smoothness assumptions on coefficients
were considered. The authors obtained approximation with corrector for the inverse operator
 in the $(L_2\rightarrow H^1)$-norm with error estimate of order $O(\sqrt{\varepsilon})$.
The analog of estimate \eqref{A_D,eps L2 ots in introduction}, but of order $O(\sqrt{\varepsilon})$, was deduced.
(In the case of the Dirichlet problem for the acoustics equation, the  $(L_2\rightarrow L_2)$-estimate was improved in \cite{ZhPas},
but the order was not sharp.) Similar results for the operator $-\div g^\varepsilon (\mathbf{x})\nabla$
in a smooth bounded domain $\mathcal{O}\subset\mathbb{R}^d$ with the Dirichlet or Neumann boundary conditions were
obtained by G.~Griso \cite{Gr1,Gr2} with the help of the  ``unfolding''  method. In \cite{Gr2},
sharp-order estimate \eqref{A_D,eps L2 ots in introduction} (for the same operator) was proved.
For elliptic systems similar results were independently obtained in \cite{KeLiS} and in \cite{PSu,Su13}.
Further results and a detailed survey can be found in \cite{Su_SIAM,Su15}.
Let us only mention the forthcoming paper \cite{SZ}, where estimate of the form \eqref{A_D,eps L2 ots in introduction}
and the $(L_2 \to H^1)$-approximation were obtained for the elasticity operator with mixed (Dirichlet and Neumann) boundary conditions.

Operator error estimates for the second order matrix elliptic operator (with lower order terms) in a bounded domain $\mathcal{O}$ with the Dirichlet or
Neumann conditions were recently established by Q.~Xu \cite{Xu, Xu2, Xu3}.
Some results were even obtained for problems in Lipschitz domains.
However, in those papers a rather restrictive condition of uniform ellipticity was imposed.
We compare our results with the results of \cite{Xu3} below in Subsection~\ref{Main results}.

Up to now, we have discussed the results about approximation of the resolvent at a fixed regular point. Approximations for the resolvent
 $(A_\varepsilon -\zeta I)^{-1}$ of the operator \eqref{A_eps = in introduction}
with error estimates depending on $\varepsilon$ and $\zeta\in\mathbb{C}\setminus\mathbb{R}_+$ were
recently obtained by T.~A.~Suslina  \cite{Su15}. In \cite{Su15}, the operators $A_{D,\varepsilon}$ and $A_{N,\varepsilon}$
given by the expression \eqref{A_eps = in introduction} in a bounded domain with the Dirichlet or Neumann boundary conditions
were also studied. Approximations for the resolvents of these operators with two-parametric error estimates
(with respect to $\varepsilon$ and $\zeta$) were obtained.
Note that investigation of the two-parametric estimates was stimulated by the study of homogenization for parabolic systems,
based on the following representation of the operator exponential
\begin{equation}\label{exp}
e^{-A_{D,\varepsilon}t}=-\frac{1}{2\pi i}\int _\gamma e^{-\zeta t}(A_{D,\varepsilon}-\zeta I)^{-1}\,d\zeta,
\end{equation}
where $\gamma\subset\mathbb{C}$ is a positively oriented contour enclosing the spectrum of $A_{D,\varepsilon}$.
(A similar representation holds for $e^{-A_{N,\varepsilon}t}$.) Details can be found in \cite{MSu16}.

The present paper relies on the following two-parametric estimates for the operator ${B}_\varepsilon$ obtained in \cite{MSu15}:
\begin{align}
\label{B_eps with zeta in introduction}
&\Vert ({B}_\varepsilon -\zeta Q_0^\varepsilon )^{-1}-({B}^0-\zeta\overline{Q_0})^{-1}\Vert _{L_2(\mathbb{R}^d)\rightarrow L_2(\mathbb{R}^d)}\leqslant C(\phi)\varepsilon\vert \zeta\vert ^{-1/2},\\
\label{B_eps with zeta and corrector in introduction}
&\Vert ( {B}_\varepsilon -\zeta Q_0^\varepsilon )^{-1}-( {B}^0-\zeta\overline{Q_0})^{-1}-\varepsilon K(\varepsilon;\zeta)\Vert _{L_2(\mathbb{R}^d)\rightarrow H^1(\mathbb{R}^d)}
\leqslant C(\phi)\varepsilon .
\end{align}
Here $\phi =\mathrm{arg}\,\zeta\in (0,2\pi)$ and $\vert\zeta\vert\geqslant 1$. The dependence of constants in estimates on $\phi$ is traced.
Estimates \eqref{B_eps with zeta in introduction} and \eqref{B_eps with zeta and corrector in introduction} are uniform with respect to
$\phi$ in any domain of the form
\begin{equation}
\label{sector in introduction}
\lbrace\zeta=\vert\zeta\vert e^{i\phi}\in\mathbb{C}: \; \vert\zeta\vert\geqslant 1, \; \phi _0\leqslant \phi\leqslant 2\pi-\phi _0\rbrace
\end{equation}
with arbitrarily small $\phi _0 >0$. (In \cite{MSu15}, error estimates in the case where $\phi \in (0,2\pi)$ and $\vert\zeta\vert < 1$
were also obtained.) For details, see Section \ref{Section problem in R^d} below.

\subsection{Main results\label{Main results}} Before we formulate the results, it is convenient to turn to the positive definite operator $B_{D,\varepsilon}=\mathcal{B}_{D,\varepsilon}+\lambda Q_0^\varepsilon $, choosing an appropriate constant $\lambda$.
Let $B_D^0$ be the corresponding effective operator.
Main results are the following estimates:
\begin{align}
\label{main result 1}
\Vert &(B_{D,\varepsilon}-\zeta Q_0^\varepsilon )^{-1}-(B_D^0-\zeta \overline{Q_0})^{-1}\Vert _{L_2(\mathcal{O})\rightarrow L_2(\mathcal{O})}
\leqslant C(\phi)\varepsilon\vert\zeta\vert ^{-1/2},
\\
\label{main result 2}
\begin{split}
\Vert &(B_{D,\varepsilon}-\zeta Q_0^\varepsilon )^{-1}-(B_D^0-\zeta \overline{Q_0})^{-1}-\varepsilon K_D(\varepsilon;\zeta)\Vert _{L_2(\mathcal{O})\rightarrow H^1(\mathcal{O})}
\leqslant C(\phi)\bigl(\varepsilon ^{1/2}\vert \zeta\vert ^{-1/4}+\varepsilon\bigr),
\end{split}
\end{align}
for $\zeta\in\mathbb{C}\setminus\mathbb{R}_+$, $\vert\zeta\vert\geqslant 1$, and sufficiently small $\varepsilon$.
The constants $C(\phi)$ are controlled explicitly in terms of the problem data and the angle $\phi =\mathrm{arg}\,\zeta$.
Estimates \eqref{main result 1} and \eqref{main result 2} are uniform with respect to $\phi$ in any domain
\eqref{sector in introduction} with arbitrarily small $\phi _0 >0$.

For fixed $\zeta$, estimate \eqref{main result 1} has sharp order $O(\varepsilon)$.
The order of estimate \eqref{main result 2} is worse than in $\mathbb{R}^d$ (see \eqref{B_eps with zeta and corrector in introduction})
because of the boundary influence. The order of the $(L_2\rightarrow H^1)$-estimate can be improved up to the sharp order $O(\varepsilon )$ by passing to a strictly interior subdomain or by taking into account the boundary layer correction term. (See Theorems \ref{Theorem with Diriclet corrector} and \ref{Theorem O'} below.)

In the general case, the corrector in \eqref{main result 2} contains a smoothing operator.
We distinguish the cases where a simpler corrector can be used.
Besides estimates for the generalized resolvent, we also find approximation in the $(L_2\rightarrow L_2)$-norm for the operator
$g^\varepsilon b(\mathbf{D})(B_{D,\varepsilon}-\zeta Q_0^\varepsilon )^{-1}$ corresponding to the flux.
For completeness, we find approximations for the
generalized resolvent in a larger domain of the parameter $\zeta$; the corresponding error estimates have
a different behavior with respect to $\zeta$. (See Section~\ref{Section Another approximation} below.)

When this work was finished, the authors learned about very recent paper \cite{Xu3}, where close results were obtained.
Let us compare the results. On the one hand, there are some advantages of our research.
First, we study the operator  \eqref{B_D,eps in introduction} which is strongly elliptic, while in \cite{Xu3} (as well as in \cite{KeLiS, Xu, Xu2})
a rather restrictive condition of uniform ellipticity is imposed.
Second, we admit lower order terms with unbounded coefficients (from appropriate $L_p(\Omega)$-classes), while in
\cite{Xu3} these coefficients are assumed to be bounded. Third, we obtain two-parametric error estimates (with respect to $\varepsilon$
and $\zeta$), while in \cite{Xu3} estimates are one-parametric (with respect to $\varepsilon$).
On the other hand, there are several advantages of \cite{Xu3}: some results are obtained in the case of Lipschitz domains;
the operator may be non-selfadjoint (only the principal part of the operator is assumed to be selfadjoint).

\subsection{Method} The proofs rely on the method developed in~ \cite{PSu,Su13,Su15}. It is based on consideration of the
associated problem in $\mathbb{R}^d$, application of the results
\eqref{B_eps with zeta in introduction}, \eqref{B_eps with zeta and corrector in introduction}
(obtained in \cite{MSu15}), introduction of the boundary layer correction term, and a carefull analysis of this term.
 We base our argument upon the employment of the Steklov smoothing operator (borrowed from \cite{ZhPas})
 and estimates in the $\varepsilon$-neighborhood of the boundary.
 We trace the dependence of estimates on the spectral parameter carefully.
 Additional technical difficulties (as compared with \cite{Su15}) are related to taking
 lower order terms with unbounded coefficients into account.
 First we prove estimate \eqref{main result 2}, and next we prove \eqref{main result 1}, using \eqref{main result 2} and the
 duality arguments.

Approximations in a larger domain of the parameter $\zeta$ are deduced from the already proved estimates at the point $\zeta =-1$ and appropriate resolvent identities.

The results of the present paper will be applied to study homogenization
 for the solution $\mathbf{u}_\varepsilon(\mathbf{x},t)$, $\mathbf{x}\in \mathcal{O}$, $t\ge 0$, of the first initial boundary value problem:
\begin{equation*}
\begin{cases}
Q_0^\varepsilon(\mathbf{x}) \partial _t\mathbf{u}_\varepsilon(\mathbf{x},t)
= - B_{D,\varepsilon}\mathbf{u}_\varepsilon(\mathbf{x},t),\quad \mathbf{x}\in \mathcal{O},\; t>0;
\\
\mathbf{u}_\varepsilon(\mathbf{x},t)=0,\quad \mathbf{x} \in \partial\mathcal{O},\; t>0;
\quad Q_0^\varepsilon(\mathbf{x})\mathbf{u}_\varepsilon(\mathbf{x},0)=\boldsymbol{\varphi}(\mathbf{x}),
\end{cases}
\end{equation*}
where $\boldsymbol{\varphi}\in L_2(\mathcal{O};\mathbb{C}^n)$. A separate paper \cite{MSu3} will be devoted to this subject.
The method will be based on using the representation
\begin{equation*}
\mathbf{u}_\varepsilon(\cdot,t)=-\frac{1}{2\pi i}\int _\gamma e^{-\zeta t}(B_{D,\varepsilon}-\zeta Q_0 ^\varepsilon )^{-1}\boldsymbol{\varphi}\,d\zeta ,
\end{equation*}
where $\gamma\subset\mathbb{C}$ is a suitable contour; cf. \eqref{exp}.

\subsection{Plan of the paper} The paper consists of eleven sections. In Section~\ref{Section problem in R^d}, we introduce the class of operators $B_\varepsilon$ acting in $L_2(\mathbb{R}^d;\mathbb{C}^n)$ and formulate the results about approximations for the generalized resolvent $(B_\varepsilon -\zeta Q_0^\varepsilon )^{-1}$ obtained in \cite{MSu15}. In Section~\ref{Chapter 2 Dirichlet}, the class of operators  $B_{D,\varepsilon}$ is described, the effective operator $B_D^0$ is defined, and the main results are formulated. Section~\ref{Section Lemmas} contains auxiliary material. In Section~\ref{Section Proof Dirischlet corrector}, we prove the $(L_2\rightarrow H^1)$-approximation with
the boundary layer correction term. In Section~\ref{Section proof of L_2->H^1 theorem}, approximation~\eqref{main result 2} and
approximation for the flux are obtained. In Section~\ref{Section proof of L_2 rightarrow L_2 theorem}, the
$(L_2\rightarrow L_2)$-estimate \eqref{main result 1} is proved. In Section~\ref{Chapter Special cases}, we distinguish the cases where
the smoothing operator can be removed. In Section~\ref{Chapter Interior subdomain}, we find approximations for the generalized resolvent  in a strictly interior subdomain. Estimates in a larger domain of the spectral parameter are obtained in Section~\ref{Section Another approximation}.
 Section~\ref{more} contains more results: mainly, they concern some improvements of the behavior of the right-hand sides in estimates
 with respect to $\phi=\mathrm{arg}\,\zeta$.
Applications of the general results can be found in Section~\ref{Section Examples}.

\subsection{Notation} Let $\mathfrak{H}$ and $\mathfrak{H}_*$ be complex separable Hilbert spaces. The symbols $(\cdot ,\cdot)_\mathfrak{H}$ and $\Vert \cdot\Vert _\mathfrak{H}$ stand for the inner product and the norm in $\mathfrak{H}$, respectively. The symbol $\Vert \cdot\Vert _{\mathfrak{H}\rightarrow\mathfrak{H}_*}$ denotes the norm of a linear continuous operator acting from  $\mathfrak{H}$ to $\mathfrak{H}_*$.

The symbols $\langle \cdot ,\cdot\rangle$ and $\vert \cdot\vert$ stand for the inner product and
the norm in $\mathbb{C}^n$, $\mathbf{1}_n$ is the unit $(n\times n)$-matrix.
For an $(m\times n)$-matrix $a$, the symbol
$\vert a\vert$ denotes the norm of $a$ viewed as a linear operator from $\mathbb{C}^n$ to $\mathbb{C}^m$.
For $z\in\mathbb{C}$, we denote by $z^*$ the complex conjugate number; this nonstandard notation is employed because
we write $\overline{g}$ for the mean value of a periodic function $g$.
Next, we use the notation $\mathbf{x}=(x_1,\dots , x_d)\in\mathbb{R}^d$, $iD_j=\partial _j =\partial /\partial x_j$, $j=1,\dots,d$, $\mathbf{D}=-i\nabla=(D_1,\dots ,D_d)$. The $L_p$-classes of $\mathbb{C}^n$-valued functions in a domain $\mathcal{O}\subset\mathbb{R}^d$ are denoted by $L_p(\mathcal{O};\mathbb{C}^n)$, $1\leqslant p\leqslant \infty$.
The Sobolev spaces of $\mathbb{C}^n$-valued functions in a domain $\mathcal{O}\subset\mathbb{R}^d$ are denoted by $H^s(\mathcal{O};\mathbb{C}^n)$. Next, $H^1_0(\mathcal{O};\mathbb{C}^n)$ is the closure of $C_0^\infty (\mathcal{O};\mathbb{C}^n)$ in $H^1(\mathcal{O};\mathbb{C}^n)$. If $n=1$, we write simply $L_p(\mathcal{O})$, $H^s(\mathcal{O})$, etc.,
but sometimes, if this does not lead to confusion, we use this short notation  also for
spaces of vector-valued or matrix-valued functions.

We use the notation $\mathbb{R}_+=[0,\infty)$. Different constants in estimates are denoted by
 $c$, $\mathfrak{c}$, $C$, $\mathcal{C}$, $\mathfrak{C}$, $\beta$, $\gamma$, $k$, $\kappa$ (possibly, with indices or marks).

\section{Homogenization problem for elliptic operator acting in~$L_2(\mathbb{R}^d;\mathbb{C}^n)$}
\label{Section problem in R^d}

\subsection{Lattices in $\mathbb{R}^d$} Let $\Gamma \subset \mathbb{R}^d$ be a lattice generated by a basis  $\mathbf{a}_1,\dots ,\mathbf{a}_d \in \mathbb{R}^d$, i.~e.,
$\Gamma =\bigl\{\mathbf{a}\in \mathbb{R}^d : \mathbf{a}=\sum _{j=1}^d \nu _j \mathbf{a}_j,\; \nu _j\in \mathbb{Z}\bigr\}$,
and let $\Omega$ be the elementary cell of $\Gamma$:
$\Omega =
\bigl\{\mathbf{x}\in \mathbb{R}^d :\mathbf{x}=\sum _{j=1}^d \tau _j \mathbf{a}_j , \; -\frac{1}{2}<\tau _j<\frac{1}{2}\bigr\}$.
By $\vert \Omega \vert $ we denote the Lebesgue measure of $\Omega$: $\vert \Omega \vert =\mathrm{meas}\,\Omega$.
The basis $\mathbf{b}_1,\dots ,\mathbf{b}_d$ in $\mathbb{R}^d$ dual to the basis $\mathbf{a}_1,\dots ,\mathbf{a}_d$
is defined by the relations $\langle \mathbf{b}_i ,\mathbf{a}_j \rangle =2\pi \delta _{ij}$. This basis generates
the lattice $\widetilde{\Gamma}$ dual to $\Gamma$.
Denote $2r_0:=\min_{0\ne {\mathbf b} \in \widetilde{\Gamma}} |{\mathbf b}|$ and $2r_1:=\mathrm{diam}\,\Omega$.

By $\widetilde{H}^1(\Omega)$ we denote the subspace of all functions in
 $H^1(\Omega)$ whose $\Gamma$-periodic extension to $\mathbb{R}^d$ belongs to $H^1_{\mathrm{loc}}(\mathbb{R}^d)$. If $h (\mathbf{x})$~is a~$\Gamma$-periodic measurable matrix-valued function in $\mathbb{R}^d$, we put $h ^\varepsilon (\mathbf{x}):= h (\mathbf{x}/\varepsilon)$, $\varepsilon >0$; $\overline{h}:=\vert \Omega\vert ^{-1}\int _\Omega h(\mathbf{x})\,d\mathbf{x}$, and $\underline{h}:=\left(\vert \Omega\vert ^{-1}\int _\Omega h(\mathbf{x})^{-1}\,d\mathbf{x}\right)^{-1}$. Here, in the definition of  $\overline{h}$ it is assumed that $h\in L_{1,\mathrm{loc}}(\mathbb{R}^d)$, and in the definition of $\underline{h}$ it is assumed that the matrix $h(\mathbf{x})$ is square and nondegenerate, and  $h^{-1}\in L_{1,\mathrm{loc}}(\mathbb{R}^d)$. By $[h^\varepsilon ]$ we denote the operator of multiplication by the matrix-valued function $h^\varepsilon (\mathbf{x})$.

\subsection{The Steklov smoothing} The Steklov smoothing operator $S_\varepsilon^{(k)}$, $\varepsilon >0$, acts in $L_2(\mathbb{R}^d;\mathbb{C}^k)$ (where $k\in\mathbb{N}$) and is given by
\begin{equation}
\label{S_eps}
\begin{split}
(S_\varepsilon^{(k)} \mathbf{u})(\mathbf{x}):=\vert \Omega \vert ^{-1}\int _\Omega \mathbf{u}(\mathbf{x}-\varepsilon \mathbf{z})\,d\mathbf{z},\quad \mathbf{u}\in L_2(\mathbb{R}^d;\mathbb{C}^k).
\end{split}
\end{equation}
We shall omit the index $k$ in the notation and write simply $S_\varepsilon$. Obviously,
$S_\varepsilon \mathbf{D}^\alpha \mathbf{u}=\mathbf{D}^\alpha S_\varepsilon \mathbf{u}$ for $\mathbf{u}\in H^s(\mathbb{R}^d;\mathbb{C}^k)$  and any multiindex $\alpha$ such that $\vert \alpha\vert \leqslant s$.
Note that
\begin{equation}
\label{S_eps <= 1}
\Vert S_\varepsilon \Vert _{H^l(\mathbb{R}^d)\rightarrow H^l(\mathbb{R}^d)}\leqslant 1,\quad l\in \mathbb{Z}_+.
\end{equation}
We need the following properties of the operator $S_\varepsilon$
(see  \cite[Lemmas 1.1 and 1.2]{ZhPas} or \cite[Propositions 3.1 and 3.2]{PSu}).

\begin{proposition}
\label{Proposition S__eps - I}
For any function  $\mathbf{u}\in H^1(\mathbb{R}^d;\mathbb{C}^k)$ we have
\begin{equation*}
\Vert S_\varepsilon \mathbf{u}-\mathbf{u}\Vert _{L_2(\mathbb{R}^d)}\leqslant \varepsilon r_1\Vert \mathbf{D}\mathbf{u}\Vert _{L_2(\mathbb{R}^d)},
\quad \varepsilon >0.
\end{equation*}
\end{proposition}

\begin{proposition}
\label{Proposition f^eps S_eps}
Let $h$ be a $\Gamma$-periodic function in $\mathbb{R}^d$ such that $h\in L_2(\Omega)$.
Then the operator  $[h^\varepsilon ]S_\varepsilon $ is continuous in $L_2(\mathbb{R}^d;\mathbb{C}^k)$, and
\begin{equation*}
\Vert [h^\varepsilon]S_\varepsilon \Vert _{L_2(\mathbb{R}^d)\rightarrow L_2(\mathbb{R}^d)}\leqslant \vert \Omega \vert ^{-1/2}\Vert h\Vert _{L_2(\Omega)},\quad \varepsilon >0.
\end{equation*}
\end{proposition}

\subsection{The class of operators $A_\varepsilon$}
\label{Subsection Operator A_eps}
In $L_2(\mathbb{R}^d;\mathbb{C}^n)$, we consider the operator $A_\varepsilon$ given by the differential expression  $A_\varepsilon =b(\mathbf{D})^*g^\varepsilon (\mathbf{x})b(\mathbf{D})$. Here $g(\mathbf{x})$ is a $\Gamma$-periodic $(m\times m)$-matrix-valued function (in general, with complex entries). We assume that $g(\mathbf{x})>0$ and $g,g^{-1}\in L_\infty (\mathbb{R}^d)$. Next, $b(\mathbf{D})$ is the DO given  by
$b(\mathbf{D})=\sum _{j=1}^d b_jD_j$, where $b_j$, $j=1,\dots ,d$, are constant $(m\times n)$-matrices (in general, with complex entries). It is assumed that  $m\geqslant n$ and that the symbol $b(\boldsymbol{\xi})=\sum _{j=1}^d b_j\xi_j$ of the operator $b(\mathbf{D})$
has maximal rank:
$\mathrm{rank}\,b(\boldsymbol{\xi})=n$ for $0\neq \boldsymbol{\xi}\in\mathbb{R}^d$.
This condition is equivalent to the estimates
\begin{equation}
\label{<b^*b<}
\alpha _0\mathbf{1}_n \leqslant b(\boldsymbol{\theta})^*b(\boldsymbol{\theta})
\leqslant \alpha _1\mathbf{1}_n,\quad
\boldsymbol{\theta}\in \mathbb{S}^{d-1},\quad
0<\alpha _0\leqslant \alpha _1<\infty,
\end{equation}
with some positive constants $\alpha _0$ and $\alpha _1$.
From \eqref{<b^*b<} it follows that
\begin{equation}
\label{b_l <=}
\vert b_j\vert \leqslant \alpha _1^{1/2},\quad j=1,\dots ,d.
\end{equation}

The precise definition of the operator $A_\varepsilon$ is given in terms of the quadratic form
\begin{equation*}
\mathfrak{a}_\varepsilon [\mathbf{u},\mathbf{u}]=\int _{\mathbb{R}^d}\langle g^\varepsilon (\mathbf{x})b(\mathbf{D})\mathbf{u},b(\mathbf{D})\mathbf{u}\rangle \,d\mathbf{x},\quad \mathbf{u}\in H^1(\mathbb{R}^d;\mathbb{C}^n).
\end{equation*}
Under the above assumptions, this form is closed and nonnegative.
Using the Fourier transformation and condition \eqref{<b^*b<}, it is easy to check that
\begin{equation}
\label{<a_eps<}
\alpha _0\Vert g^{-1}\Vert ^{-1}_{L_\infty}\Vert \mathbf{D}\mathbf{u}\Vert ^2_{L_2(\mathbb{R}^d)}\leqslant \mathfrak{a}_\varepsilon [\mathbf{u},\mathbf{u}]\leqslant \alpha _1\Vert g\Vert _{L_\infty}\Vert \mathbf{D}\mathbf{u}\Vert^2_{L_2(\mathbb{R}^d)},\quad
 \mathbf{u}\in H^1(\mathbb{R}^d;\mathbb{C}^n).
\end{equation}
Let $c_1:=\alpha _0 ^{-1/2}\Vert g^{-1}\Vert ^{1/2}_{L_\infty}$. Then the lower estimate
 \eqref{<a_eps<} can be written as
 \begin{equation}
\label{Du <= c_1^2a}
\Vert\mathbf{D}\mathbf{u}\Vert ^2_{L_2(\mathbb{R}^d)}\leqslant c_1^2 \mathfrak{a}_\varepsilon [\mathbf{u},\mathbf{u}],\quad \mathbf{u}\in H^1(\mathbb{R}^d;\mathbb{C}^n).
\end{equation}

\subsection{The operator $B_\varepsilon$}
\label{Subsection lower order terms}
We study a selfadjoint operator $B_\varepsilon$ whose principal part coincides with $A_\varepsilon$.
To define the lower order terms, we introduce $\Gamma$-periodic $(n\times n)$-matrix-valued functions $a_j$, $j=1,\dots ,d$, (in general, with complex entries) such that
\begin{align}
\label{a_j cond}
a_j \in L_\rho (\Omega ), \quad\rho =2 \ \mbox{for}\ d=1,\quad\rho >d\ \mbox{for}\ d\geqslant 2,\quad  j=1,\dots ,d.
\end{align}
Next, let $Q$ and $Q_0$ be $\Gamma$-periodic Hermitian $(n\times n)$-matrix-valued functions (with complex entries) such that
\begin{align}
\label{Q condition}
&Q\in L_s(\Omega ),\quad s=1 \ \mbox{for}\ d=1,\quad s >d/2\ \mbox{for}\ d\geqslant 2;\\
& Q_0(\mathbf{x})>0;\quad Q_0,Q_0^{-1}\in L_\infty (\mathbb{R}^d).\nonumber
\end{align}
(Our assumptions on $Q$ correspond to Example 2.4 from \cite{MSu15}.)
For convenience of further references, the following set of parameters is called the  {\it ``initial data''}:
\begin{equation}
\label{problem data}
\begin{split}
&d,\,m,\,n,\,\rho ,\,s ;\,\alpha _0,\, \alpha _1 ,\,\Vert g\Vert _{L_\infty},\, \Vert g^{-1}\Vert _{L_\infty},\, \Vert a_j\Vert _{L_\rho (\Omega)},\, j=1,\dots ,d;\\
&\Vert Q\Vert _{L_s(\Omega)};\, \Vert Q_0\Vert _{L_\infty},\,\Vert Q_0^{-1}\Vert _{L_\infty};\;\text{the parameters of the lattice}\;\Gamma .
\end{split}
\end{equation}

Consider the following quadratic form
\begin{equation}
\label{b_eps =}
\begin{aligned}
\mathfrak{b}_\varepsilon [\mathbf{u},\mathbf{u}]=\mathfrak{a}_\varepsilon [\mathbf{u},\mathbf{u}]+2\mathrm{Re}\,\sum _{j=1}^d
\left(a_j^\varepsilon D_j \mathbf{u},\mathbf{u}\right)_{L_2(\mathbb{R}^d)}
+\left((Q^\varepsilon +\lambda Q_0^\varepsilon)\mathbf{u},\mathbf{u}\right)_{L_2(\mathbb{R}^d)},
 \quad \mathbf{u}\in H^1(\mathbb{R}^d;\mathbb{C}^n).
\end{aligned}
\end{equation}
We fix a constant $\lambda$ so that the form $\mathfrak{b}_\varepsilon$ is nonnegative (see \eqref{lambda =} below).
Let us check that the form $\mathfrak{b}_\varepsilon$ is closed.
By the H\"older inequality and the Sobolev embedding theorem, it is easily seen  (see  \cite[(5.11)--(5.14)]{SuAA}) that
for any $\nu>0$ there exist constants $C_j(\nu)>0$ such that
\begin{equation*}
\Vert a_j^*\mathbf{u}\Vert ^2 _{L_2(\mathbb{R}^d)}\leqslant \nu \Vert \mathbf{D}\mathbf{u}\Vert ^2_{L_2(\mathbb{R}^d)}+C_j(\nu )\Vert \mathbf{u}\Vert ^2 _{L_2(\mathbb{R}^d)},
\quad \mathbf{u}\in H^1(\mathbb{R}^d;\mathbb{C}^n),\quad  j=1,\dots ,d.
\end{equation*}
Using the change of variable $\mathbf{y}:=\varepsilon ^{-1}\mathbf{x}$ and denoting $\mathbf{u}(\mathbf{x})=:\mathbf{v}(\mathbf{y})$, we deduce
\begin{equation*}
\begin{split}
\Vert   (a_j^\varepsilon )^*\mathbf{u}\Vert ^2_{L_2(\mathbb{R}^d)}&=\int _{\mathbb{R}^d}\vert a_j(\varepsilon ^{-1}\mathbf{x})^*\mathbf{u}(\mathbf{x})\vert ^2\,d\mathbf{x}
=\varepsilon ^d\int _{\mathbb{R}^d}\vert a_j(\mathbf{y})^*\mathbf{v}(\mathbf{y})\vert ^2\,d\mathbf{y}
\\
&\leqslant \varepsilon ^d\nu \int _{\mathbb{R}^d}\vert \mathbf{D}_{\mathbf{y}}\mathbf{v}(\mathbf{y})\vert ^2\,d\mathbf{y}
+\varepsilon ^d C_j(\nu)\int _{\mathbb{R}^d}\vert \mathbf{v}(\mathbf{y})\vert ^2\,d\mathbf{y}
\\
&\leqslant \nu \Vert \mathbf{D}\mathbf{u}\Vert ^2_{L_2(\mathbb{R}^d)}+C_j(\nu)\Vert \mathbf{u}\Vert ^2_{L_2(\mathbb{R}^d)},\quad \mathbf{u}\in H^1(\mathbb{R}^d;\mathbb{C}^n),\quad 0<\varepsilon\leqslant 1.
\end{split}
\end{equation*}
Hence, by \eqref{<a_eps<}, for any $\nu >0$ there exists a constant $C(\nu)>0$ such that
\begin{equation}
\label{sum a-j u}
\sum _{j=1}^d \Vert (a_j^\varepsilon)^*\mathbf{u}\Vert ^2 _{L_2(\mathbb{R}^d)}
\leqslant \nu
\mathfrak{a}_\varepsilon [\mathbf{u},\mathbf{u}]
+C(\nu)\Vert \mathbf{u}\Vert ^2_{L_2(\mathbb{R}^d)},\quad\mathbf{u}\in H^1(\mathbb{R}^d;\mathbb{C}^n),
\quad 0<\varepsilon\leqslant 1.
\end{equation}
If $\nu$ is fixed, then $C(\nu)$ depends only on $d$, $\rho$, $\alpha _0$,  the norms $\Vert g^{-1}\Vert _{L_\infty}$, $\Vert a_j\Vert _{L_\rho (\Omega)}$, $j=1,\dots ,d$, and the parameters of the lattice $\Gamma$.
From \eqref{Du <= c_1^2a} and \eqref{sum a-j u} it follows that
\begin{equation}
\label{2Re sum j <=}
2\biggl\vert \mathrm{Re}\,\sum _{j=1}^d (D_j\mathbf{u},(a_j^\varepsilon)^*\mathbf{u})_{L_2(\mathbb{R}^d)}\biggr\vert
\leqslant\frac{1}{4}\mathfrak{a}_\varepsilon [\mathbf{u},\mathbf{u}]+c_2\Vert \mathbf{u}\Vert ^2_{L_2(\mathbb{R}^d)},
\quad \mathbf{u}\in H^1(\mathbb{R}^d;\mathbb{C}^n),\quad  0<\varepsilon\leqslant 1,
\end{equation}
where $c_2 :=8c_1^2C(\nu _0)$ with $\nu _0:=2^{-6}\alpha _0\Vert g^{-1}\Vert ^{-1}_{L_\infty}$.

Next, by condition \eqref{Q condition} on $Q$, for any $\nu >0$ there exists a constant $C_Q(\nu)>0$ such that
\begin{equation}
\label{(Qu,u)<=}
\vert (Q^\varepsilon \mathbf{u},\mathbf{u})_{L_2(\mathbb{R}^d)}\vert \leqslant\nu \Vert \mathbf{D}\mathbf{u}\Vert ^2_{L_2(\mathbb{R}^d)}+C_Q(\nu)\Vert \mathbf{u}\Vert ^2_{L_2(\mathbb{R}^d)},\quad
\mathbf{u}\in H^1(\mathbb{R}^d;\mathbb{C}^n),\quad 0<\varepsilon\leqslant 1 .
\end{equation}
For $\nu$ fixed,  $C_Q(\nu)$ is controlled in terms of $d$, $s$, $\Vert Q\Vert _{L_s(\Omega)}$, and the parameters of the lattice $\Gamma$.

As in \cite[Subsection~2.8]{MSu15}, we fix $\lambda$ in \eqref{b_eps =} as follows:
\begin{equation}
\label{lambda =}
\lambda := (C_Q(\nu _*)+c_2)\Vert Q_0^{-1}\Vert _{L_\infty}\quad\text{for}\;\nu _* :=2^{-1}\alpha _0\Vert g^{-1}\Vert ^{-1}_{L_\infty}.
\end{equation}

Combining \eqref{Du <= c_1^2a}, \eqref{2Re sum j <=},  \eqref{(Qu,u)<=} with $\nu=\nu_*$,  and \eqref{lambda =},
we deduce the following lower estimate for the form \eqref{b_eps =}:
\begin{align}
\label{b_eps >=}
\mathfrak{b}_\varepsilon[\mathbf{u},\mathbf{u}]\geqslant \frac{1}{4}\mathfrak{a}_\varepsilon [\mathbf{u},\mathbf{u}]\geqslant c_*\Vert \mathbf{D}\mathbf{u}\Vert ^2_{L_2(\mathbb{R}^d)},\quad\mathbf{u}\in H^1(\mathbb{R}^d;\mathbb{C}^n);
\quad
c_* :=\frac{1}{4}\alpha _0\Vert g^{-1}\Vert ^{-1}_{L_\infty}.
\end{align}
From \eqref{<a_eps<}, \eqref{2Re sum j <=}, and \eqref{(Qu,u)<=} with $\nu=1$ it follows that
\begin{equation}
\label{b_eps <=}
\mathfrak{b}_\varepsilon [\mathbf{u},\mathbf{u}]\leqslant c_3\Vert \mathbf{u}\Vert ^2 _{H^1(\mathbb{R}^d)},
 \quad
\mathbf{u}\in H^1(\mathbb{R}^d;\mathbb{C}^n),
\end{equation}
where $c_3:= \max\{\frac{5}{4}\alpha _1\Vert g\Vert _{L_\infty}+1; C_Q(1)+\lambda\Vert Q_0\Vert _{L_\infty}+c_2\}$.

Thus, the form $\mathfrak{b}_\varepsilon$ is closed and nonnegative.
The selfadjoint operator in $L_2(\mathbb{R}^d;\mathbb{C}^n)$ generated by this form is denoted by $B_\varepsilon$. Formally, we have
\begin{equation}
\label{B_eps}
B_\varepsilon = b(\mathbf{D})^* g^\varepsilon (\mathbf{x})b(\mathbf{D})+\sum _{j=1}^d \left(
a_j^\varepsilon (\mathbf{x})D_j +D_j a_j^\varepsilon (\mathbf{x})^*\right)
+Q^\varepsilon (\mathbf{x}) +\lambda Q_0^\varepsilon (\mathbf{x}).
\end{equation}

\subsection{The effective matrix}
The effective operator for $A_\varepsilon =b(\mathbf{D})^*g^\varepsilon (\mathbf{x})b(\mathbf{D})$ is given by
 $A^0=b(\mathbf{D})^*g^0b(\mathbf{D})$, where $g^0$ is a constant $(m\times m)$-matrix called the effective matrix.
Suppose that a $\Gamma$-periodic $(n\times m)$-matrix-valued function $\Lambda (\mathbf{x})$ is the (weak) solution of the problem
\begin{equation}
\label{Lambda problem}
b(\mathbf{D})^*g(\mathbf{x})(b(\mathbf{D})\Lambda (\mathbf{x})+\mathbf{1}_m)=0,\quad \int _{\Omega }\Lambda (\mathbf{x})\,d\mathbf{x}=0.
\end{equation}
Then the effective matrix is given by
\begin{align}
\label{g^0}
&g^0=\vert \Omega \vert ^{-1}\int _{\Omega} \widetilde{g}(\mathbf{x})\,d\mathbf{x},
\\
\label{tilde g}
&\widetilde{g}(\mathbf{x}):=g(\mathbf{x})(b(\mathbf{D})\Lambda (\mathbf{x})+\mathbf{1}_m).
\end{align}

From \eqref{Lambda problem} it easily follows that
\begin{equation}
\label{b(D)Lambda<=}
\Vert b(\mathbf{D})\Lambda\Vert _{L_2(\Omega)}
\leqslant \vert \Omega\vert ^{1/2} m^{1/2}\Vert g\Vert _{L_\infty}^{1/2}\Vert g^{-1}\Vert _{L_\infty}^{1/2}.
\end{equation}
We also need the following estimates
proved in \cite[(6.28) and Subsection~7.3]{BSu05}:
\begin{align}
\label{Lambda <=}
&\Vert \Lambda \Vert _{L_2(\Omega)}\leqslant \vert \Omega \vert ^{1/2}M_1; \quad M_1 :=m^{1/2}(2r_0)^{-1}\alpha _0^{-1/2}\Vert g\Vert ^{1/2}_{L_\infty}\Vert g^{-1}\Vert ^{1/2}_{L_\infty},\\
\label{DLambda<=}
&\Vert \mathbf{D}\Lambda \Vert _{L_2(\Omega)}\leqslant \vert \Omega \vert ^{1/2}M_2;\quad M_2 :=m^{1/2}\alpha _0^{-1/2}\Vert g\Vert ^{1/2}_{L_\infty}\Vert g^{-1}\Vert ^{1/2}_{L_\infty}.
\end{align}

It can be checked that $g^0$ is positive definite.
 The effective matrix satisfies the estimates known as the Voigt--Reuss bracketing
(see, e.~g., \cite[Chapter~3, Theorem~1.5]{BSu}).

\begin{proposition}
Let $g^0$ be the effective matrix \eqref{g^0}. Then
\begin{equation}
\label{Foigt-Reiss}
\underline{g}\leqslant g^0\leqslant \overline{g}.
\end{equation}
If $m=n$, then $g^0=\underline{g}$.
\end{proposition}

Inequalities \eqref{Foigt-Reiss} imply that
\begin{equation}
\label{|g^0|<=}
\vert g^0\vert \leqslant \Vert g\Vert _{L_\infty},\quad \vert (g^0)^{-1}\vert \leqslant \Vert g^{-1}\Vert _{L_\infty}.
\end{equation}

Now we distinguish the cases where one of the inequalities in \eqref{Foigt-Reiss}
becomes an identity, see \cite[Chapter~3, Propositions~1.6 and~1.7]{BSu}.

\begin{proposition}
The identity $g^0=\overline{g}$ is equivalent to the relations
\begin{equation}
\label{overline-g}
b(\D)^* {\mathbf g}_j(\x) =0,\ \ j=1,\dots,m,
\end{equation}
where ${\mathbf g}_j(\x)$, $j=1,\dots,m,$ are the columns of the matrix $g(\x)$.
\end{proposition}

\begin{proposition} The identity $g^0 =\underline{g}$ is equivalent to the relations
\begin{equation}
\label{underline-g}
{\mathbf l}_j(\x) = {\mathbf l}_j^0 + b(\D) {\mathbf f}_j(\x),\ \ {\mathbf l}_j^0\in \C^m,\ \
{\mathbf f}_j \in \widetilde{H}^1(\Omega;\C^m),\ \ j =1,\dots,m,
\end{equation}
where ${\mathbf l}_j(\x)$, $j=1,\dots,m,$ are the columns of the matrix $g(\x)^{-1}$.
\end{proposition}

\subsection{The effective operator}
\label{Subsection Effective operator}
In order to define the effective operator for $B_\varepsilon$, consider \hbox{a~$\Gamma$-periodic} $(n\times n)$-matrix-valued function $\widetilde{\Lambda}(\mathbf{x})$ which is the (weak) solution of the problem
\begin{equation}
\label{tildeLambda_problem}
b(\mathbf{D})^*g(\mathbf{x})b(\mathbf{D})\widetilde{\Lambda }(\mathbf{x})+\sum \limits _{j=1}^dD_ja_j(\mathbf{x})^*=0,\quad \int _{\Omega }\widetilde{\Lambda }(\mathbf{x})\,d\mathbf{x}=0.
\end{equation}
The following estimates for $\widetilde{\Lambda}$
were proved in \cite[(7.49)--(7.52)]{SuAA}:
\begin{align}
\label{tilde Lambda<=}
&\Vert \widetilde{\Lambda}\Vert _{L_2(\Omega)}\leqslant (2r_0)^{-1}C_an^{1/2}\alpha _0^{-1}\Vert g^{-1}\Vert _{L_\infty},
\\
\label{D tilde Lambda}
&\Vert \mathbf{D}\widetilde{\Lambda}\Vert _{L_2(\Omega)}\leqslant C_a n^{1/2}\alpha _0^{-1}\Vert g^{-1}\Vert _{L_\infty},
\\
\label{b(D) tilde Lambda <=}
&\Vert b(\mathbf{D})\widetilde{\Lambda}\Vert _{L_2(\Omega)}
\leqslant C_a n^{1/2}\alpha_0 ^{-1/2}\Vert g^{-1}\Vert _{L_\infty},
\end{align}
where $C_a^2 :=\sum _{j=1}^d \int _\Omega \vert a_j(\mathbf{x})\vert ^2\,d\mathbf{x}$.
Next, we define constant matrices $V$ and $W$ as follows:
\begin{align}
\label{V}
&V :=\vert \Omega \vert ^{-1}\int _{\Omega}(b(\mathbf{D})\Lambda (\mathbf{x}))^*g(\mathbf{x})(b(\mathbf{D})\widetilde{\Lambda}(\mathbf{x}))\,d\mathbf{x},\\
\label{W}
&W :=\vert \Omega \vert ^{-1}\int _{\Omega} (b(\mathbf{D})\widetilde{\Lambda}(\mathbf{x}))^*g(\mathbf{x})(b(\mathbf{D})\widetilde{\Lambda}(\mathbf{x}))\,d\mathbf{x}.
\end{align}
The effective operator for the operator \eqref{B_eps} is given by
\begin{equation}
\label{B^0}
B^0 =b(\mathbf{D})^*g^0b(\mathbf{D})-b(\mathbf{D})^*V-V^*b(\mathbf{D})
+\sum \limits _{j=1}^d (\overline{a_j+a_j^*})D_j-W+\overline{Q}+\lambda \overline{Q_0}.
\end{equation}
The operator $B^0$ is the elliptic second order operator with constant coefficients with the symbol
\begin{equation}
\label{L(xi)}
L(\boldsymbol{\xi})=b(\boldsymbol{\xi})^*g^0 b(\boldsymbol{\xi})-b(\boldsymbol{\xi})^*V-V^*b(\boldsymbol{\xi})+\sum\limits _{j=1}^d (\overline{a_j+a_j^*})\xi _j+\overline{Q}-W+\lambda \overline{Q_0}.
\end{equation}

\begin{lemma}
The symbol \eqref{L(xi)} of the operator \eqref{B^0} satisfies
\begin{equation}
\label{L(xi)<=}
c_*\vert \boldsymbol{\xi}\vert ^2 \mathbf{1}_n\leqslant
L(\boldsymbol{\xi})\leqslant C_L(\vert \boldsymbol{\xi}\vert ^2+1)\mathbf{1}_n ,\quad\boldsymbol{\xi}\in \mathbb{R}^d,
\end{equation}
where $c_*$ is defined in \eqref{b_eps >=}.
The constant $C_L$ depends only on the initial data  \eqref{problem data}.
\end{lemma}

\begin{proof} The lower estimate \eqref{L(xi)<=} is proved in \cite[(2.30)]{MSu15}. Let us check the upper estimate.
By \eqref{<b^*b<},  \eqref{|g^0|<=}, and \eqref{L(xi)},
\begin{equation}
\label{lemma L(xi) proof 1}
L(\boldsymbol{\xi})
\leqslant \alpha _1 \Vert g\Vert _{L_\infty}\vert \boldsymbol{\xi}\vert ^2 \mathbf{1}_n
+2\vert V\vert \alpha _1^{1/2}\vert \boldsymbol{\xi}\vert \mathbf{1}_n
+2\Bigl(\sum _{j=1}^d \vert \overline{a_j}\vert ^2\Bigr)^{1/2}\vert \boldsymbol{\xi}\vert \mathbf{1}_n
+\left(\vert \overline{Q}\vert +\lambda \vert \overline{Q_0}\vert \right)\mathbf{1}_n.
\end{equation}
We have taken into account that, obviously, the matrix \eqref{W} is nonnegative.
According to \eqref{b(D)Lambda<=}, \eqref{b(D) tilde Lambda <=}, and \eqref{V},
\begin{equation}
\label{lemma L(xi) proof 2}
\vert V\vert \leqslant
\vert\Omega\vert ^{-1}\Vert g\Vert _{L_\infty}\Vert b(\mathbf{D})\Lambda\Vert _{L_2(\Omega)}\Vert b(\mathbf{D})\widetilde{\Lambda}\Vert _{L_2(\Omega)}\leqslant C_V,
\end{equation}
where $C_V :=\vert \Omega \vert ^{-1/2}\alpha _0^{-1/2} C_a m^{1/2} n^{1/2}\Vert g\Vert _{L_\infty}^{3/2}\Vert g^{-1}\Vert ^{3/2}_{L_\infty}$.
Clearly,
\begin{equation}
\label{lemma L(xi) proof 3}
\sum _{j=1}^d\vert \overline{a_j}\vert ^2\leqslant \vert \Omega \vert ^{-1}C_a^2,\quad \vert \overline{Q}\vert \leqslant\vert \Omega \vert ^{-1/s}\Vert Q\Vert _{L_s(\Omega)},\quad \vert \overline{Q_0}\vert \leqslant \Vert Q_0\Vert _{L_\infty}.
\end{equation}
Now relations \eqref{lemma L(xi) proof 1}--\eqref{lemma L(xi) proof 3} imply the upper estimate \eqref{L(xi)<=} with the constant
$C_L :=\max \lbrace \alpha _1 \Vert g\Vert _{L_\infty};\vert \Omega \vert ^{-1/s}\Vert Q\Vert _{L_s(\Omega)}+\lambda \Vert Q_0\Vert _{L_\infty}\rbrace
+\alpha _1^{1/2}C_V+\vert \Omega \vert ^{-1/2}C_a$.
\end{proof}

\begin{corollary}
The quadratic form $\mathfrak{b}^0$ of the operator \eqref{B^0} satisfies
\begin{equation}
\label{b0 estimates from corollary}
c_*\Vert \mathbf{D}\mathbf{u}\Vert ^2_{L_2(\mathbb{R}^d)}\leqslant \mathfrak{b}^0[\mathbf{u},\mathbf{u}]\leqslant C_L\Vert \mathbf{u}\Vert ^2_{H^1(\mathbb{R}^d)},\quad\mathbf{u}\in H^1(\mathbb{R}^d;\mathbb{C}^n).
\end{equation}
\end{corollary}

\subsection{The results about approximation of the generalized resolvent}

In this subsection, we formulate the results
proved in \cite[Theorems 5.1 and 5.2]{MSu15}.

\begin{theorem}[\cite{MSu15}]
\label{Th L2 Rd}
Suppose that the assumptions of Subsections~\textnormal{\ref{Subsection Operator A_eps}--\ref{Subsection Effective operator}} are satisfied.
Let $\zeta \in\mathbb{C}\setminus \mathbb{R}_+$, $\zeta =\vert \zeta \vert e^{i\phi}$, $\phi \in (0,2\pi)$, and $\vert \zeta \vert \geqslant 1$.
Denote
\begin{equation}
\label{c(phi)}
c(\phi):=\begin{cases}
\vert \sin \phi \vert ^{-1}, &\phi\in (0,\pi /2)\cup (3\pi /2 ,2\pi),\\
1, &\phi\in [\pi /2,3\pi /2].
\end{cases}
\end{equation}
Then for $0<\varepsilon \leqslant 1$ we have
\begin{equation}
\label{1.44a}
\Vert(B_\varepsilon -\zeta Q_0^\varepsilon )^{-1}-(B^0-\zeta\overline{Q_0})^{-1}\Vert _{L_2(\mathbb{R}^d)\rightarrow L_2(\mathbb{R}^d)}\leqslant C_1 c(\phi)^2 \varepsilon \vert \zeta \vert ^{-1/2}.
\end{equation}
The constant $C_1$ depends only on the initial data \eqref{problem data}.
\end{theorem}

Next, we introduce a corrector
\begin{equation}
\label{K(eps;zeta)}
K(\varepsilon ;\zeta):=\left([\Lambda ^\varepsilon] b(\mathbf{D})+[\widetilde{\Lambda}^\varepsilon] \right)S _\varepsilon (B^0-\zeta \overline{Q_0})^{-1}.
\end{equation}
The corrector \eqref{K(eps;zeta)} is a bounded operator acting from $L_2(\mathbb{R}^d;\mathbb{C}^n)$ to $H^1(\mathbb{R}^d;\mathbb{C}^n)$.
This can be easily checked by using Proposition~\ref{Proposition f^eps S_eps} and relations $\Lambda ,\widetilde{\Lambda}\in \widetilde{H}^1(\Omega)$. Note that $\Vert \varepsilon K(\varepsilon ;\zeta)\Vert _{L_2(\mathbb{R}^d)\rightarrow H^1(\mathbb{R}^d)}=O(1)$ for small $\varepsilon$ and fixed  $\zeta$.

\begin{theorem}[\cite{MSu15}]
\label{Theorem H1 Rd}
Suppose that the assumptions of Theorem~\textnormal{\ref{Th L2 Rd}} are satisfied. Let $K(\varepsilon ;\zeta)$ be the operator~\eqref{K(eps;zeta)}. Then for $0<\varepsilon \leqslant 1$, $\zeta \in\mathbb{C}\setminus \mathbb{R}_+$, and $\vert \zeta \vert \geqslant 1$ we have
\begin{align*}
&\Vert (B_\varepsilon -\zeta Q_0^\varepsilon )^{-1}-(B^0-\zeta \overline{Q_0})^{-1}-\varepsilon K(\varepsilon ;\zeta )\Vert _{L_2(\mathbb{R}^d)\rightarrow L_2(\mathbb{R}^d)}
\leqslant C_2 c(\phi )^2\varepsilon \vert \zeta \vert ^{-1/2},
\\
&\Vert \mathbf{D}\left(
(B_\varepsilon -\zeta Q_0^\varepsilon )^{-1}-(B^0-\zeta \overline{Q_0})^{-1}-\varepsilon K(\varepsilon ;\zeta )
\right)
\Vert _{L_2(\mathbb{R}^d)\rightarrow L_2(\mathbb{R}^d)}
\leqslant C_3 c(\phi )^2\varepsilon .
\end{align*}
The constants $C_2$ and $C_3$ are controlled explicitly in terms of the initial data~\eqref{problem data}.
\end{theorem}

We also need estimates in the case where $\zeta\in\mathbb{C}\setminus \mathbb{R}_+$ and $|\zeta|< 1$. The following result is a particular case
of Theorem~9.1 from~\cite{MSu15}.

\begin{theorem}[\cite{MSu15}]
\label{Theorem Dr appr Rd}
Suppose that the assumptions of Subsections~\textnormal{\ref{Subsection Operator A_eps}--\ref{Subsection Effective operator}} are satisfied.
 Let  $\zeta= \vert \zeta \vert e^{i\phi} \in\mathbb{C}\setminus \mathbb{R}_+$, $|\zeta|<1$, and
$0<\varepsilon\leqslant 1$. Let $K(\varepsilon;\zeta)$ be the operator \eqref{K(eps;zeta)}. Then
\begin{align}
\label{Dr appr in Rd}
\Vert & (B_\varepsilon -\zeta Q_0^\varepsilon )^{-1}-(B^0-\zeta \overline{Q_0})^{-1}\Vert _{L_2(\mathbb{R}^d)\rightarrow L_2(\mathbb{R}^d)}
\leqslant \widehat{C}_1  c(\phi)^2 \varepsilon \vert \zeta  \vert ^{-2},
\\
\nonumber
\Vert &(B_\varepsilon -\zeta Q_0^\varepsilon )^{-1}-(B^0-\zeta \overline{Q_0})^{-1}-\varepsilon K(\varepsilon ;\zeta)\Vert _{L_2(\mathbb{R}^d)\rightarrow H^1(\mathbb{R}^d)}
\leqslant \widehat{C}_2   c(\phi)^2 \varepsilon \vert \zeta  \vert ^{-2}.
\end{align}
The constants $\widehat{C}_1$ and $\widehat{C}_2$ depend only on the initial data \eqref{problem data}.
\end{theorem}

\section{Statement of the problem. Main results}
\label{Chapter 2 Dirichlet}

\subsection{Statement of the problem}

Let $\mathcal{O}\subset \mathbb{R}^d$ be a bounded domain of class $C^{1,1}$. In $L_2(\mathcal{O};\mathbb{C}^n)$, we consider
the operator $B_{D,\varepsilon}$, $0<\varepsilon\leqslant 1$, formally given by the differential expression \eqref{B_eps}
with the Dirichlet boundary condition. The precise definition of the operator $B_{D,\varepsilon}$ is given in terms of the quadratic form
\begin{equation}
\label{b_D,eps}
\begin{split}
\mathfrak{b}_{D,\varepsilon }[\mathbf{u},\mathbf{u}]
&=(g^\varepsilon b(\mathbf{D})\mathbf{u},b(\mathbf{D})\mathbf{u})_{L_2(\mathcal{O})}+2\mathrm{Re}\,\sum _{j=1}^d ( D_j \mathbf{u},(a_j^\varepsilon)^*\mathbf{u})_{L_2(\mathcal{O})}
\\
&+(Q^\varepsilon \mathbf{u},\mathbf{u})_{L_2(\mathcal{O})}
+\lambda (Q_0^\varepsilon \mathbf{u},\mathbf{u})_{L_2(\mathcal{O})},\quad \mathbf{u}\in H^1_0(\mathcal{O};\mathbb{C}^n).
\end{split}
\end{equation}
We extend $\mathbf{u}\in H^1_0(\mathcal{O};\mathbb{C}^n)$ by zero to $\mathbb{R}^d\setminus \mathcal{O}$.
Then $\mathbf{u}\in H^1(\mathbb{R}^d;\mathbb{C}^n)$. By  \eqref{b_eps >=} and \eqref{b_eps <=},
\begin{equation}
\label{b_D,eps ots}
c_*\Vert \mathbf{D}\mathbf{u}\Vert ^2_{L_2(\mathcal{O})}\leqslant \mathfrak{b}_{D,\varepsilon}[\mathbf{u},\mathbf{u}]
\leqslant c_3 \Vert \mathbf{u}\Vert ^2_{H^1(\mathcal{O})},\quad \mathbf{u}\in H^1 _0(\mathcal{O};\mathbb{C}^n).
\end{equation}
Combining this with the Friedrichs inequality, we deduce
\begin{equation}
\label{b D,eps >= H1-norm}
\mathfrak{b}_{D,\varepsilon}[\mathbf{u},\mathbf{u}]\geqslant c_*(\mathrm{diam}\,\mathcal{O})^{-2}\Vert \mathbf{u}\Vert ^2_{L_2(\mathcal{O})},\quad \mathbf{u}\in H^1_0(\mathcal{O};\mathbb{C}^n).
\end{equation}
Thus, the form $\mathfrak{b}_{D,\varepsilon}$ is closed and positive definite.
It generates a selfadjoint operator in $L_2(\mathcal{O};\mathbb{C}^n)$, which is denoted by $B_{D,\varepsilon}$.
By \eqref{b_D,eps ots} and \eqref{b D,eps >= H1-norm},
\begin{equation}
\label{H^1-norm <= BDeps^1/2}
\Vert \mathbf{u}\Vert _{H^1(\mathcal{O})}\leqslant c_4 \Vert B_{D,\varepsilon}^{1/2}\mathbf{u}\Vert _{L_2(\mathcal{O})},\quad \mathbf{u}\in H^1_0(\mathcal{O};\mathbb{C}^n);\quad
c_4 :=c_*^{-1/2}\left(1+(\mathrm{diam}\,\mathcal{O})^{2}\right)^{1/2}.
\end{equation}

We also need to introduce an auxiliary operator $\widetilde{B}_{D,\varepsilon}$.
 We factorize the matrix $Q_0(\mathbf{x})^{-1}$: there exists a $\Gamma$-periodic matrix-valued function $f(\mathbf{x})$ such that
$f, f^{-1} \in L_\infty({\mathbb R}^d)$ and
\begin{equation}
\label{Q_0=}
Q_0(\mathbf{x})^{-1}=f(\mathbf{x})f(\mathbf{x})^*.
\end{equation}
(For instance, one can take $f(\mathbf{x}) = Q_0(\mathbf{x})^{-1/2}$.)
Let $\widetilde{B}_{D,\varepsilon}$ be the selfadjoint operator in $L_2(\mathcal{O};\mathbb{C}^n)$ generated by the quadratic form
\begin{equation}
\label{tilde b D,eps=}
\widetilde{\mathfrak{b}}_{D,\varepsilon }[\mathbf{u},\mathbf{u}]:=\mathfrak{b}_{D,\varepsilon}[f^\varepsilon \mathbf{u},f^\varepsilon \mathbf{u}],
\quad \mathrm{Dom}\,\widetilde{\mathfrak{b}}_{D,\varepsilon}=\lbrace \mathbf{u}\in L_2(\mathcal{O};\mathbb{C}^n) : f^\varepsilon \mathbf{u}\in H^1_0(\mathcal{O};\mathbb{C}^n)\rbrace .
\end{equation}
Formally, $\widetilde{B}_{D,\varepsilon}=(f^\varepsilon )^*B_{D,\varepsilon}f^\varepsilon $.  Note that
\begin{align}
\label{B deps and tilde B D,eps tozd resolvent}
(B_{D,\varepsilon}-\zeta Q_0^\varepsilon)^{-1}=f^\varepsilon (\widetilde{B}_{D,\varepsilon}-\zeta I)^{-1}(f^\varepsilon)^* .
\end{align}

{\it Our goal} is to approximate the generalized resolvent $(B_{D,\varepsilon}-\zeta Q_0^\varepsilon )^{-1}$ and to prove two-parametric error estimates  (with respect to $\varepsilon$ and $\zeta$). We assume that $\zeta \in \mathbb{C}\setminus\mathbb{R}_+$.
In other words, we are interested in the behavior of the generalized solution $\mathbf{u}_\varepsilon \in H^1_0(\mathcal{O};\mathbb{C}^n)$ of the problem
\begin{equation}
\label{Dirichlet problem}
B_\eps \mathbf{u}_\varepsilon - \zeta Q_0^\varepsilon  \mathbf{u}_\varepsilon = \mathbf{F} \ \textrm{in}\  {\mathcal O};
\quad \mathbf{u}_\varepsilon \vert _{\partial \mathcal{O}}=0,
\end{equation}
where $\mathbf{F}\in L_2(\mathcal{O};\mathbb{C}^n)$, for small $\varepsilon$. We have
$\mathbf{u}_\varepsilon =(B_{D,\varepsilon }-\zeta Q_0^\varepsilon )^{-1}\mathbf{F}$.

\begin{lemma}
\label{Lemma estimates u_D,eps}
Let $\zeta \in \mathbb{C}\setminus \mathbb{R}_+$. Suppose that $\mathbf{u}_\varepsilon$ is the generalized solution of  problem \eqref{Dirichlet problem}. Then for $0<\varepsilon \leqslant 1$ we have
\begin{align}
\label{lemma u_eps 1}
&\Vert \mathbf{u}_\varepsilon \Vert _{L_2(\mathcal{O})}\leqslant c(\phi)\vert \zeta \vert ^{-1}\Vert Q_0^{-1}\Vert _{L_\infty}\Vert \mathbf{F}\Vert _{L_2(\mathcal{O})},\\
\label{lemma u_eps 2}
&\Vert \mathbf{D}\mathbf{u}_\varepsilon \Vert _{L_2(\mathcal{O})}\leqslant \mathcal{C}_1 c(\phi)\vert \zeta \vert ^{-1/2}\Vert \mathbf{F}\Vert _{L_2(\mathcal{O})}.
\end{align}
In operator terms,
\begin{align}
\label{2.10a}
&\Vert (B_{D,\varepsilon}-\zeta Q_0^\varepsilon )^{-1} \Vert _{L_2(\mathcal{O})\rightarrow L_2(\mathcal{O})}\leqslant c(\phi)\vert \zeta \vert ^{-1}\Vert Q_0^{-1}\Vert _{L_\infty},\\
\label{2.10b}
&\Vert \mathbf{D}(B_{D,\varepsilon}-\zeta Q_0^\varepsilon )^{-1} \Vert _{L_2(\mathcal{O})\rightarrow L_2(\mathcal{O})}\leqslant \mathcal{C}_1 c(\phi)\vert \zeta \vert ^{-1/2}.
\end{align}
Here $c(\phi)$ is given by \eqref{c(phi)} and
$ \mathcal{C}_1 := 2\alpha _0 ^{-1/2}\Vert g^{-1}\Vert _{L_\infty}^{1/2}
\Vert Q_0^{-1}\Vert ^{1/2}_{L_\infty}\left(1+\Vert Q_0\Vert _{L_\infty}\Vert Q_0^{-1}\Vert _{L_\infty}\right)^{1/2}$.
\end{lemma}

\begin{proof}
From \eqref{Q_0=}, \eqref{B deps and tilde B D,eps tozd resolvent}, and the inequality $\widetilde{B}_{D,\varepsilon}> 0$
it follows that
\begin{equation*}
\begin{split}
\Vert \mathbf{u}_\varepsilon \Vert _{L_2(\mathcal{O})} &\leqslant \Vert f\Vert^2_{L_\infty}\Vert (\widetilde{B}_{D,\varepsilon}-\zeta I)^{-1}\Vert _{L_2(\mathcal{O})\rightarrow L_2(\mathcal{O})}\Vert \mathbf{F}\Vert _{L_2(\mathcal{O})}\\
&\leqslant \mathrm{dist}\,\lbrace \zeta ;\mathbb{R}_+\rbrace \Vert Q_0^{-1}\Vert _{L_\infty}\Vert \mathbf{F}\Vert _{L_2(\mathcal{O})}
=c(\phi)\vert \zeta \vert ^{-1}\Vert Q_0^{-1}\Vert _{L_\infty}\Vert \mathbf{F}\Vert _{L_2(\mathcal{O})},
\end{split}
\end{equation*}
which implies \eqref{lemma u_eps 1}.
To check \eqref{lemma u_eps 2}, we write down the integral identity for $\mathbf{u}_\varepsilon$:
\begin{equation*}
\mathfrak{b}_{D,\varepsilon}[\mathbf{u}_\varepsilon ,\boldsymbol{\eta}]-\zeta (Q_0^\varepsilon \mathbf{u}_\varepsilon ,\boldsymbol{\eta})_{L_2(\mathcal{O})}=(\mathbf{F},\boldsymbol{\eta})_{L_2(\mathcal{O})},\quad \boldsymbol{\eta}\in H^1_0(\mathcal{O};\mathbb{C}^n).
\end{equation*}
Substituting $\boldsymbol{\eta}= \mathbf{u}_\varepsilon$ and using the lower estimate \eqref{b_D,eps ots}, \eqref{lemma u_eps 1}, and~\eqref{b_eps >=},
we arrive at \eqref{lemma u_eps 2}.
\end{proof}

\subsection{The form $\mathfrak{b}_{N,\varepsilon}$}
Apart from the form \eqref{b_D,eps}, we need the quadratic form $\mathfrak{b}_{N,\varepsilon}$ defined by the same
expression, but on the class $H^1(\mathcal{O};\mathbb{C}^n)$:
\begin{equation}
\label{b mathfrak eps =}
\begin{split}
\mathfrak{b}_{N,\varepsilon} [\mathbf{u},\mathbf{u}]&=(g^\varepsilon b(\mathbf{D})\mathbf{u},b(\mathbf{D})\mathbf{u})_{L_2(\mathcal{O})}+2\mathrm{Re}\,\sum _{j=1}^d ( D_j \mathbf{u},(a_j^\varepsilon)^*\mathbf{u})_{L_2(\mathcal{O})}\\
&+(Q^\varepsilon \mathbf{u},\mathbf{u})_{L_2(\mathcal{O})}+\lambda (Q_0^\varepsilon \mathbf{u},\mathbf{u})_{L_2(\mathcal{O})},\quad \mathbf{u}\in H^1(\mathcal{O};\mathbb{C}^n).
\end{split}
\end{equation}
This form corresponds to the Neumann problem.
Let us estimate the form~\eqref{b mathfrak eps =} from above.
By \eqref{b_l <=}, we have
\begin{equation}
\label{b mathfrak eps <= start}
\begin{split}
\mathfrak{b}_{N,\varepsilon} [\mathbf{u},\mathbf{u}]
&\leqslant
d\alpha_1\Vert g\Vert _{L_\infty}\Vert \mathbf{D}\mathbf{u}\Vert ^2_{L_2(\mathcal{O})}
+\sum _{j=1}^d \int _\mathcal{O}\vert a_j^\varepsilon (\mathbf{x})\vert ^2\vert \mathbf{u}(\mathbf{x})\vert ^2\,d\mathbf{x}
+\Vert \mathbf{D}\mathbf{u}\Vert ^2_{L_2(\mathcal{O})}\\
&+\int _{\mathcal{O}}\vert Q^\varepsilon (\mathbf{x})\vert \vert \mathbf{u}(\mathbf{x})\vert ^2\,d\mathbf{x}
+\lambda\Vert Q_0\Vert _{L_\infty}\Vert \mathbf{u}\Vert ^2_{L_2(\mathcal{O})}.
\end{split}
\end{equation}
From the H\"older inequality it follows that
\begin{equation}
\label{int o aj Holder}
\int _\mathcal{O}\vert a_j^\varepsilon (\mathbf{x})\vert ^2\vert \mathbf{u}(\mathbf{x})\vert ^2\,d\mathbf{x}
\leqslant
\left(\int _\mathcal{O}\vert a_j^\varepsilon(\mathbf{x})\vert ^\rho\,d\mathbf{x}\right)^{2/\rho}
\Vert \mathbf{u}\Vert _{L_q(\mathcal{O})}^2,
\end{equation}
where $\rho$ is as in \eqref{a_j cond}, $q=\infty$ for $d=1$, and $q=2\rho/(\rho -2 )$ for $d\geqslant 2$.
Next, we cover the domain $\mathcal{O}$ by the union of cells of the lattice $\varepsilon\Gamma$
intersecting $\mathcal{O}$ (here $0<\varepsilon\leqslant 1$). Let $N_\varepsilon$ be the number of cells in this covering. Clearly, this union of cells
is contained in the domain $\widetilde{\mathcal{O}}$ which is the $2r_1$-neighborhood of $\mathcal{O}$, where $2r_1=\mathrm{diam}\,\Omega$. Therefore, $N_\varepsilon\leqslant \mathfrak{c}_1\varepsilon ^{-d}$, where $\mathfrak{c}_1$ depends only on the domain $\mathcal{O}$
and the parameters of the lattice $\Gamma$. We have
\begin{equation}
\label{int O aj <= c1}
\int _\mathcal{O}\vert a_j^\varepsilon (\mathbf{x})\vert ^\rho \,d\mathbf{x}\leqslant \mathfrak{c}_1\varepsilon ^{-d}\int _{\varepsilon\Omega}\vert a_j^\varepsilon (\mathbf{x})\vert ^\rho \,d\mathbf{x}
=\mathfrak{c}_1\Vert a_j\Vert ^\rho _{L_\rho (\Omega)}.
\end{equation}
Relations \eqref{int o aj Holder} and \eqref{int O aj <= c1} imply that
\begin{equation}
\label{int O a_j <= ... p=}
\int _\mathcal{O}\vert a_j^\varepsilon (\mathbf{x})\vert ^2\vert \mathbf{u}(\mathbf{x})\vert ^2\,d\mathbf{x}
\leqslant \mathfrak{c}_1 ^{2/\rho}\Vert a_j\Vert ^2 _{L_\rho (\Omega)}\Vert \mathbf{u}\Vert ^2 _{L_q(\mathcal{O})}.
\end{equation}
By the continuous embedding $H^1(\mathcal{O};\mathbb{C}^n)\hookrightarrow L_q(\mathcal{O};\mathbb{C}^n)$, we have
\begin{equation}
\label{2.17 new}
\Vert \mathbf{u}\Vert _{L_q(\mathcal{O})}
\leqslant C(q,\mathcal{O})\Vert \mathbf{u}\Vert _{H^1(\mathcal{O})},
\end{equation}
where $C(q,\mathcal{O})$ is the corresponding embedding constant.
From \eqref{int O a_j <= ... p=} and \eqref{2.17 new} it follows that
\begin{equation}
\label{2.18 new}
\sum _{j=1}^d \int _\mathcal{O}\vert a_j^\varepsilon (\mathbf{x})\vert ^2\vert \mathbf{u}(\mathbf{x})\vert ^2\,d\mathbf{x}
\leqslant \mathfrak{c}_1^{2/\rho}C(q,\mathcal{O})^2\widehat{C}_a^2\Vert \mathbf{u}\Vert ^2 _{H^1(\mathcal{O})},\quad \mathbf{u}\in H^1(\mathcal{O};\mathbb{C}^n).
\end{equation}
Here
$\widehat{C}_a^2:=\sum _{j=1}^d\Vert a_j\Vert ^2 _{L_\rho (\Omega)}$.
Similarly to \eqref{int o aj Holder}--\eqref{2.18 new}, by \eqref{Q condition}, we obtain
\begin{equation}
\label{2.19 new}
\int _\mathcal{O}\vert Q^\varepsilon (\mathbf{x})\vert \vert \mathbf{u}(\mathbf{x})\vert ^2\,d\mathbf{x}
\leqslant \mathfrak{c}_1^{1/s}\Vert Q\Vert _{L_s(\Omega)}C(\check{q},\mathcal{O})^2\Vert \mathbf{u}\Vert ^2 _{H^1(\mathcal{O})},
\end{equation}
where $\check{q}=\infty$ for $d=1$ and $\check{q}=2s/(s-1)$ for $d\geqslant 2$.

Relations \eqref{b mathfrak eps <= start}, \eqref{2.18 new}, and \eqref{2.19 new} imply that
\begin{equation}
\label{b mathfrak eps <= c2}
\begin{split}
\mathfrak{b}_{N,\varepsilon} [\mathbf{u},\mathbf{u}]
\leqslant \mathfrak{c}_2\Vert \mathbf{u}\Vert ^2 _{H^1(\mathcal{O})},\quad \mathbf{u}\in H^1(\mathcal{O};\mathbb{C}^n),
\end{split}
\end{equation}
where $\mathfrak{c}_2 :=1+d\alpha _1 \Vert g\Vert _{L_\infty}
+\mathfrak{c}_1^{2/\rho}C(q,\mathcal{O})^2\widehat{C}_a^2
+\mathfrak{c}_1^{1/s}\Vert Q\Vert _{L_s(\Omega)}C(\check{q},\mathcal{O})^2
+\lambda\Vert Q_0\Vert _{L_\infty}$.

\subsection{The homogenized problem} In $L_2(\mathcal{O};\mathbb{C}^n)$, we consider the quadratic form
\begin{equation*}
\begin{split}
\mathfrak{b}_D^0[\mathbf{u},\mathbf{u}]&=(g^0b(\mathbf{D})\mathbf{u},b(\mathbf{D})\mathbf{u})_{L_2(\mathcal{O})}
+2\mathrm{Re}\,\sum _{j=1}^d (\overline{a_j}D_j\mathbf{u},\mathbf{u})_{L_2(\mathcal{O})}
-2\mathrm{Re}\,(V\mathbf{u},b(\mathbf{D})\mathbf{u})_{L_2(\mathcal{O})}
\\
&-(W\mathbf{u},\mathbf{u})_{L_2(\mathcal{O})}
+(\overline{Q}\mathbf{u},\mathbf{u})_{L_2(\mathcal{O})}
+\lambda (\overline{Q_0}\mathbf{u},\mathbf{u})_{L_2(\mathcal{O})},
\quad \mathbf{u}\in H^1_0(\mathcal{O};\mathbb{C}^n).
\end{split}
\end{equation*}
Extending $\mathbf{u}\in H^1_0(\mathcal{O};\mathbb{C}^n)$ by zero to $\mathbb{R}^d\setminus\mathcal{O}$,
using \eqref{b0 estimates from corollary}  and the Friedrichs inequality, we obtain
\begin{align}
\label{b_D^0 ots }
&c_*\Vert \mathbf{D}\mathbf{u}\Vert ^2_{L_2(\mathcal{O})}\leqslant \mathfrak{b}_D^0[\mathbf{u},\mathbf{u}]\leqslant C_L\Vert \mathbf{u}\Vert ^2_{H^1(\mathcal{O})},\quad \mathbf{u}\in H^1_0(\mathcal{O};\mathbb{C}^n),
\\
\label{b_D^0 ots 2}
&\mathfrak{b}_D^0[\mathbf{u},\mathbf{u}]\geqslant c_*(\mathrm{diam}\,\mathcal{O})^{-2}\Vert \mathbf{u}\Vert ^2_{L_2(\mathcal{O})},\quad \mathbf{u}\in H^1_0(\mathcal{O};\mathbb{C}^n).
\end{align}
The selfadjoint operator in $L_2(\mathcal{O};\mathbb{C}^n)$ corresponding to the form $\mathfrak{b}_D^0$
is denoted by $B_D^0$. From \eqref{b_D^0 ots } and \eqref{b_D^0 ots 2} it follows that
\begin{equation}
\label{H^1-norm <= BD0^1/2}
\Vert \mathbf{u}\Vert _{H^1(\mathcal{O})}\leqslant c_4 \Vert (B_D^0)^{1/2}\mathbf{u}\Vert _{L_2(\mathcal{O})},\quad \mathbf{u}\in H^1_0(\mathcal{O};\mathbb{C}^n),
\end{equation}
where $c_4$ is as in \eqref{H^1-norm <= BDeps^1/2}.
Since $\partial\mathcal{O}\in C^{1,1}$, the operator $B_D^0$ is given by the differential expression
\eqref{B^0}
on the domain $H^2(\mathcal{O};\mathbb{C}^n)\cap H^1_0(\mathcal{O};\mathbb{C}^n)$. We have
\begin{equation}
\label{BD0 ^-1 L2->H2}
\Vert (B_D^0)^{-1}\Vert _{L_2(\mathcal{O})\rightarrow H^2(\mathcal{O})}\leqslant \widehat{c}.
\end{equation}
Here $\widehat{c}$ depends only on the initial data \eqref{problem data} and the domain $\mathcal{O}$.
This fact follows from the theorems about regularity of solutions of the strongly elliptic systems (see \cite[Chapter~4]{McL}).

\begin{remark}
Instead of the condition $\partial\mathcal{O}\in C^{1,1}$, one could impose the following implicit condition:
a bounded domain $\mathcal{O}\subset \mathbb{R}^d$ with Lipschitz boundary is such that
estimate \eqref{BD0 ^-1 L2->H2} holds. The results of the paper remain valid for such domain.
In the case of the scalar elliptic operators, wide sufficient conditions on $\partial \mathcal{O}$ ensuring \eqref{BD0 ^-1 L2->H2}
can be found in \textnormal{\cite{KoE}} and \textnormal{\cite[Chapter~7]{MaSh} (}in particular, it suffices that $\partial\mathcal{O}\in C^\alpha$, $\alpha >3/2${\rm)}.
\end{remark}

Let $f_0:= (\overline{Q_0})^{-1/2}$. By \eqref{Q_0=},
\begin{equation}
\label{f_0<=}
\vert f_0\vert \leqslant \Vert f\Vert _{L_\infty}=\Vert Q_0^{-1}\Vert^{1/2}_{L_\infty},\quad \vert f_0^{-1}\vert \leqslant \Vert f^{-1}\Vert _{L_\infty}=\Vert Q_0\Vert ^{1/2}_{L_\infty}.
\end{equation}
In what follows, we need the operator $\widetilde{B}_D^0 :=f_0B_D^0 f_0$. Note that
 \begin{equation}
 \label{B_D0 and tilde B_D0 resolvents}
 (B_D^0-\zeta\overline{Q_0})^{-1}=f_0(\widetilde{B}_D^0-\zeta I)^{-1}f_0.
 \end{equation}

The function $\mathbf{u}_0 = (B_D^0-\zeta \overline{Q_0})^{-1}\mathbf{F}$ is the solution of the ``homogenized problem''
 \begin{equation}
 \label{Dirichlet homogenized problem}
 B^0 \mathbf{u}_0 - \zeta\overline{Q_0}\mathbf{u}_0 =  \mathbf{F}\ \textrm{in}\
 \mathcal{O};\quad \mathbf{u}_0\vert _{\partial \mathcal{O}}=0.
 \end{equation}

 \begin{lemma}
 \label{Lemma hom problem estimates}
For $\zeta \in \mathbb{C}\setminus \mathbb{R}_+$ the solution $\mathbf{u}_0$ of problem \eqref{Dirichlet homogenized problem}
satisfies
 \begin{align}
 &\Vert \mathbf{u}_0\Vert _{L_2(\mathcal{O})}\leqslant c(\phi)\vert \zeta \vert ^{-1}\Vert Q_0 ^{-1}\Vert _{L_\infty}\Vert \mathbf{F}\Vert _{L_2(\mathcal{O})},
\nonumber
 \\
 &\Vert \mathbf{D}\mathbf{u}_0\Vert _{L_2(\mathcal{O})}\leqslant \mathcal{C}_1c(\phi)\vert \zeta\vert ^{-1/2}\Vert \mathbf{F}\Vert _{L_2(\mathcal{O})},
 \nonumber
 \\
 &\Vert \mathbf{u}_0\Vert _{H^2(\mathcal{O})}\leqslant\mathcal{C}_2c(\phi)\Vert \mathbf{F}\Vert _{L_2(\mathcal{O})}.
 \nonumber
 \end{align}
Here the constant $\mathcal{C}_1$ is as in Lemma~\textnormal{\ref{Lemma estimates u_D,eps}} and $\mathcal{C}_2 :=\widehat{c}\Vert Q_0\Vert _{L_\infty}^{1/2}\Vert Q_0^{-1}\Vert _{L_\infty}^{1/2}$. In operator terms,
 \begin{align}
  \label{lemma hom probl 1}
 &\Vert (B_D^0-\zeta \overline{Q_0})^{-1}\Vert _{L_2(\mathcal{O})\rightarrow L_2(\mathcal{O})}\leqslant c(\phi)\vert \zeta \vert ^{-1}\Vert Q_0^{-1}\Vert _{L_\infty},\\
 \label{lemma hom probl 2}
 &\Vert \mathbf{D}(B_D^0-\zeta\overline{Q_0})^{-1}\Vert _{L_2(\mathcal{O})\rightarrow L_2(\mathcal{O})}\leqslant \mathcal{C}_1 c(\phi)\vert \zeta \vert ^{-1/2},\\
   \label{lemma hom probl 3}
 &\Vert (B_D^0-\zeta \overline{Q_0})^{-1}\Vert _{L_2(\mathcal{O})\rightarrow H^2(\mathcal{O})}\leqslant \mathcal{C}_2 c(\phi).
 \end{align}
 \end{lemma}

 \begin{proof}
Estimates  \eqref{lemma hom probl 1} and  \eqref{lemma hom probl 2} can be checked by the same way
as estimates of Lemma~\ref{Lemma estimates u_D,eps}. Let us prove~\eqref{lemma hom probl 3}.
Obviously,
 \begin{equation}
 \label{d-vo lm 2.2 1}
 \Vert  (B_D^0-\zeta \overline{Q_0})^{-1}\Vert _{L_2(\mathcal{O})\rightarrow H^2(\mathcal{O})}
\leqslant \Vert (B_D^0)^{-1}\Vert _{L_2(\mathcal{O})\rightarrow H^2(\mathcal{O})}\Vert B_D^0(B_D^0-\zeta \overline{Q_0})^{-1}\Vert _{L_2(\mathcal{O})\rightarrow L_2(\mathcal{O})}.
 \end{equation}
By \eqref{B_D0 and tilde B_D0 resolvents}, we have $B_D^0(B_D^0-\zeta\overline{Q_0})^{-1}=B_D^0f_0(\widetilde{B}_D^0-\zeta I)^{-1}f_0=f_0^{-1}\widetilde{B}_D^0 (\widetilde{B}_{D}^0-\zeta I)^{-1}f_0$. Hence,
 \begin{equation}
 \label{d-vo lm 2.2 2}
 \Vert B_D^0(B_D^0-\zeta \overline{Q_0})^{-1}\Vert _{L_2(\mathcal{O})\rightarrow L_2(\mathcal{O})}
\leqslant \vert f_0^{-1}\vert \vert f_0\vert \sup\limits _{x\geqslant 0}\frac{x}{\vert x-\zeta\vert}
\leqslant \Vert Q_0\Vert _{L_\infty}^{1/2}\Vert Q_0^{-1}\Vert _{L_\infty}^{1/2}c(\phi).
 \end{equation}
We have taken \eqref{f_0<=} into account. Now, relations \eqref{BD0 ^-1 L2->H2}, \eqref{d-vo lm 2.2 1}, and \eqref{d-vo lm 2.2 2}
imply \eqref{lemma hom probl 3}.
 \end{proof}

\subsection{Formulation of the results}

We choose the numbers $\varepsilon _0, \varepsilon _1\in (0,1]$ according to the following condition.

\begin{condition}
\label{condition varepsilon}
Let $\varepsilon _0\in (0,1]$ be such that the strip $(\partial\mathcal{O})_{\varepsilon} :=\left\lbrace \mathbf{x}\in \mathbb{R}^d : \mathrm{dist}\,\lbrace \mathbf{x};\partial\mathcal{O}\rbrace <\varepsilon \right\rbrace $ can be covered by a finite number of open sets
admitting diffeomorphisms of class $C^{0,1}$ rectifying the boundary $\partial\mathcal{O}$.
Denote $\varepsilon _1 :=\varepsilon _0 (1+r_1)^{-1}$, where $2r_1=\mathrm{diam}\,\Omega$.
\end{condition}

Clearly, $\varepsilon _1$ depends only on the domain $\mathcal{O}$ and the parameters of the lattice $\Gamma$.
Note that Condition~\ref{condition varepsilon} would be provided only by the assumption that $\partial\mathcal{O}$ is Lipschitz.
We have imposed a more restrictive condition $\partial\mathcal{O}\in C^{1,1}$ in order to ensure estimate \eqref{BD0 ^-1 L2->H2}.

Now, we formulate the main results.

\begin{theorem}
\label{Theorem Dirichlet L2}
Suppose that $\mathcal{O}\subset\mathbb{R}^d$ is a bounded domain of class $C^{1,1}$.
Let $\zeta =\vert \zeta \vert e^{i\phi}\in \mathbb{C}\setminus \mathbb{R}_+$ and $\vert \zeta \vert \geqslant 1$.
Let $\mathbf{u}_\varepsilon$ be the solution of problem \eqref{Dirichlet problem} with $\mathbf{F}\in L_2(\mathcal{O};\mathbb{C}^n)$.
Let $\mathbf{u}_0$ be the solution of  problem \eqref{Dirichlet homogenized problem}. Suppose that $\varepsilon _1$ is subject to Condition~\textnormal{\ref{condition varepsilon}}.
Then for $0<\varepsilon\leqslant \varepsilon _1$ we have
\begin{equation}
\label{Th L2 solutions}
\Vert \mathbf{u}_\varepsilon -\mathbf{u}_0\Vert _{L_2(\mathcal{O})}\leqslant C_4 c(\phi)^5 \varepsilon \vert \zeta \vert ^{-1/2}\Vert \mathbf{F}\Vert _{L_2(\mathcal{O})}.
\end{equation}
Here $c(\phi)$ is given by \eqref{c(phi)}{\rm ;} the constant $C_4$ depends only on the initial data \eqref{problem data} and the domain $\mathcal{O}$.
In operator terms,
\begin{equation}
\label{Th L2}
\Vert (B_{D,\varepsilon}-\zeta Q_0^\varepsilon )^{-1}-(B_D^0-\zeta \overline{Q_0})^{-1}\Vert _{L_2(\mathcal{O})\rightarrow L_2(\mathcal{O})}\leqslant C_4 c(\phi)^5 \varepsilon \vert \zeta \vert ^{-1/2}.
\end{equation}
\end{theorem}

In order to approximate the solution in the Sobolev space  $H^1(\mathcal{O};\mathbb{C}^n)$, we introduce a corrector.
For this, we fix a linear continuous extension operator
\begin{equation}
\label{P_O H^1, H^2}
\begin{split}
P_\mathcal{O}: H^l(\mathcal{O};\mathbb{C}^n)\rightarrow H^l(\mathbb{R}^d;\mathbb{C}^n),\quad l =0,1,2.
\end{split}
\end{equation}
Such a ``universal'' extension operator exists for any bounded Lipschitz domain  (see \cite{St} or \cite{R}).
Herewith,
\begin{equation}
\label{PO}
\| P_{\mathcal O} \|_{H^l({\mathcal O}) \to H^l({\mathbb R}^d)} \leqslant C_{\mathcal O}^{(l)},\quad l =0,1,2,
\end{equation}
where the constant $C_{\mathcal O}^{(l)}$ depends only on $l$ and the domain ${\mathcal O}$.
By $R_\mathcal{O}$ we denote the operator of restriction of functions in $\mathbb{R}^d$ to the domain $\mathcal{O}$.
We put
\begin{equation}
\label{K_D(eps,zeta)}
K_D(\varepsilon ;\zeta ):=
R_\mathcal{O}\bigl([\Lambda ^\varepsilon]b(\mathbf{D})+[\widetilde{\Lambda}^\varepsilon ]\bigr)S_\varepsilon P_\mathcal{O}(B_D^0-\zeta \overline{Q_0})^{-1}.
\end{equation}
The continuity of the operator $K_D(\varepsilon ;\zeta ) : L_2(\mathcal{O};\mathbb{C}^n)\rightarrow H^1(\mathcal{O};\mathbb{C}^n)$
can be checked by analogy with the continuity of the operator \eqref{K(eps;zeta)}.

Let $\widetilde{\mathbf{u}}_0=P_\mathcal{O}\mathbf{u}_0$. By $\mathbf{v}_\varepsilon$ we denote the first order approximation
of the solution $\mathbf{u}_\varepsilon$:
\begin{align}
\label{v_eps}
&\widetilde{\mathbf{v}}_\varepsilon : =\widetilde{\mathbf{u}}_0+\varepsilon \Lambda ^\varepsilon S_\varepsilon b(\mathbf{D})\widetilde{\mathbf{u}}_0+\varepsilon\widetilde{\Lambda}^\varepsilon S_\varepsilon \widetilde{\mathbf{u}}_0,\\
\label{v_eps=}
&\mathbf{v}_\varepsilon :=\widetilde{\mathbf{v}}_\varepsilon \vert _{\mathcal{O}},
\end{align}
i.~e.,  $\mathbf{v}_\varepsilon = (B_D^0-\zeta \overline{Q_0})^{-1}\mathbf{F}+\varepsilon K_D(\varepsilon ;\zeta )\mathbf{F}$,
where  $K_D(\varepsilon ;\zeta)$ is given by  \eqref{K_D(eps,zeta)}.

\begin{theorem}
\label{Theorem Dirichlet H1}
Suppose that the assumptions of Theorem~\textnormal{\ref{Theorem Dirichlet L2}} are satisfied.
Suppose that $\Lambda (\mathbf{x})$ and $\widetilde{\Lambda}(\mathbf{x})$ are $\Gamma$-periodic solutions of problems
\eqref{Lambda problem} and \eqref{tildeLambda_problem}, respectively. Let $S_\varepsilon $ be the smoothing operator~\eqref{S_eps},
and let $P_\mathcal{O}$ be the extension operator \eqref{P_O H^1, H^2}. Denote $\widetilde{\mathbf{u}}_0=P_\mathcal{O}\mathbf{u}_0$.
Let $\mathbf{v}_\varepsilon$ be defined by \eqref{v_eps} and \eqref{v_eps=}.
Then for $\zeta\in\mathbb{C}\setminus\mathbb{R}_+$, $\vert \zeta\vert\geqslant 1$, and $0<\varepsilon \leqslant \varepsilon _1$ we have
\begin{equation}
\label{Th L2->H1 solutions}
\Vert \mathbf{u}_\varepsilon -\mathbf{v}_\varepsilon \Vert _{H^1(\mathcal{O})}\leqslant  \bigl( C_5 c(\phi)^2\varepsilon ^{1/2}\vert \zeta \vert ^{-1/4}+C_6 c(\phi)^4\varepsilon\bigr)\Vert \mathbf{F}\Vert _{L_2(\mathcal{O})}.
\end{equation}
In operator terms,
\begin{align}
\label{Th L2->H1}
\Vert (B_{D,\varepsilon}-\zeta Q_0^\varepsilon )^{-1}-(B_D^0-\zeta \overline{Q_0})^{-1}-\varepsilon K_D(\varepsilon ;\zeta )\Vert _{L_2(\mathcal{O})\rightarrow H^1(\mathcal{O})}
\leqslant C_5 c(\phi)^2\varepsilon ^{1/2}\vert \zeta \vert ^{-1/4}+C_6 c(\phi)^4\varepsilon ,
\end{align}
where the operator $K_D(\varepsilon;\zeta)$ is given by \eqref{K_D(eps,zeta)}. Let $\widetilde{g}(\mathbf{x})$ be
the matrix-valued function defined by~\eqref{tilde g}.
Let $\mathbf{p}_\varepsilon:=g^\varepsilon b(\mathbf{D})\mathbf{u}_\varepsilon$.
Then for $\zeta\in\mathbb{C}\setminus\mathbb{R}_+$, $\vert \zeta\vert\geqslant 1$, and $0<\varepsilon \leqslant \varepsilon _1$ we have
\begin{equation}
\label{Th fluxes}
\Vert \mathbf{p}_\varepsilon-\widetilde{g}^\varepsilon S_\varepsilon b(\mathbf{D})\widetilde{\mathbf{u}}_0-g^\varepsilon \bigl(b(\mathbf{D})\widetilde{\Lambda}\bigr)^\varepsilon S_\varepsilon \widetilde{\mathbf{u}}_0\Vert _{L_2(\mathcal{O})}
\leqslant \bigl(\widetilde{C}_5c(\phi)^2\varepsilon ^{1/2}\vert \zeta \vert ^{-1/4}+\widetilde{C}_6c(\phi)^4\varepsilon \bigr)\Vert \mathbf{F}\Vert _{L_2(\mathcal{O})}.
\end{equation}
In operator terms,
\begin{equation}
\label{2.41a}
\Vert  g^\varepsilon b(\mathbf{D})(B_{D,\varepsilon}-\zeta Q_0^\varepsilon )^{-1}-G_D(\varepsilon ;\zeta)\Vert _{L_2(\mathcal{O})\rightarrow L_2(\mathcal{O})}
\leqslant\widetilde{C}_5c(\phi)^2\varepsilon ^{1/2}\vert \zeta \vert ^{-1/4}+\widetilde{C}_6c(\phi)^4\varepsilon .
\end{equation}
Here
$G_D(\varepsilon ;\zeta) := \widetilde{g}^\varepsilon S_\varepsilon b(\mathbf{D})P_\mathcal{O}(B_D^0-\zeta\overline{Q_0})^{-1}+g^\varepsilon\bigl( b(\mathbf{D})\widetilde{\Lambda}\bigr)^\varepsilon S_\varepsilon P_\mathcal{O}(B_D^0-\zeta \overline{Q_0})^{-1}$.
The constants $C_5$, $C_6$, $\widetilde{C}_5$, and $\widetilde{C}_6$ depend only on the initial data~\eqref{problem data}
and the domain $\mathcal{O}$.
\end{theorem}

The first order approximation $\mathbf{v}_\varepsilon$ of the solution $\mathbf{u}_\varepsilon$ does not satisfy the Dirichlet boundary condition. We have  $\mathbf{v}_\varepsilon \vert _{\partial \mathcal{O}}=\varepsilon \bigl( \Lambda ^\varepsilon S_\varepsilon b(\mathbf{D})\widetilde{\mathbf{u}}_0 +\widetilde{\Lambda}^\varepsilon S_\varepsilon \widetilde{\mathbf{u}}_0\bigr)\vert _{\partial \mathcal{O}}$. We consider the
``discrepancy'' $\mathbf{w}_\varepsilon$, which is the solution of the problem
\begin{equation}
\label{w_eps problem}
B_\varepsilon \mathbf{w}_\varepsilon -\zeta Q_0^\varepsilon \mathbf{w}_\varepsilon =0\ \mbox{in}\ \mathcal{O};\quad
\mathbf{w}_\varepsilon \vert _{\partial \mathcal{O}}=\varepsilon \bigl(\Lambda ^\varepsilon S_\varepsilon b(\mathbf{D})\widetilde{\mathbf{u}}_0+\widetilde{\Lambda}^\varepsilon S_\varepsilon \widetilde{\mathbf{u}}_0\bigr)\vert _{\partial \mathcal{O}}.
\end{equation}
Here the equation is understood in the weak sense, as the following
identity for $\mathbf{w}_\varepsilon\in H^1(\mathcal{O};\mathbb{C}^n)$:
\begin{equation}
\label{w_eps int tozd}
\mathfrak{b}_{N,\varepsilon} [\mathbf{w}_\varepsilon ,\boldsymbol{\eta}]-\zeta (Q_0^\varepsilon \mathbf{w}_\varepsilon ,\boldsymbol{\eta})_{L_2(\mathcal{O})}=0,\quad \boldsymbol{\eta}\in H^1_0(\mathcal{O};\mathbb{C}^n).
\end{equation}
The discrepancy $\mathbf{w}_\varepsilon$ is often called the ``boundary layer correction term''. Allowing some freedom, along with
$\mathbf{w}_\varepsilon$, we shall use the notation $\mathbf{w}_\varepsilon(\cdot ;\zeta)$ for the solution of problem \eqref{w_eps problem}.
We introduce the operator taking $\mathbf{F}$ to $\mathbf{w}_\varepsilon$:
\begin{equation}
\label{2.46A}
\varepsilon W_D(\varepsilon ;\zeta) : L_2(\mathcal{O};\mathbb{C}^n)\ni\mathbf{F}\rightarrow \mathbf{w}_\varepsilon (\cdot ;\zeta)\in H^1(\mathcal{O};\mathbb{C}^n).
\end{equation}
Let us find more explicit expression for $W_D(\varepsilon ;\zeta)$. Clearly, the function
\begin{equation}
\label{**.2}
\mathbf{r}_\varepsilon (\mathbf{x};\zeta):=\mathbf{w}_\varepsilon (\mathbf{x};\zeta)-\varepsilon \bigl(K_D(\varepsilon ;\zeta)\mathbf{F}\bigr)(\mathbf{x})
\end{equation}
belongs to $H^1_0(\mathcal{O};\mathbb{C}^n)$ and satisfies the identity
\begin{equation}
\label{**.3}
\mathfrak{b}_{D,\varepsilon}[\mathbf{r}_\varepsilon,\boldsymbol{\eta}]-\zeta(Q_0^\varepsilon\mathbf{r}_\varepsilon ,\boldsymbol{\eta})_{L_2(\mathcal{O})}
=\varepsilon \mathcal{I}(\varepsilon ;\zeta)[\mathbf{F},\boldsymbol{\eta}],\quad\boldsymbol{\eta}\in H^1_0(\mathcal{O};\mathbb{C}^n),
\end{equation}
where
\begin{equation}
\label{**.4}
\mathcal{I}(\varepsilon;\zeta)[\mathbf{F},\boldsymbol{\eta}]:=-\mathfrak{b}_{N,\varepsilon}[K_D(\varepsilon ;\zeta)\mathbf{F},\boldsymbol{\eta}]
+\zeta (Q_0^\varepsilon K_D(\varepsilon;\zeta)\mathbf{F},\boldsymbol{\eta})_{L_2(\mathcal{O})}.
\end{equation}
By \eqref{b mathfrak eps <= c2},
\begin{equation}
\label{**.4a}
\begin{split}
|\mathcal{I}(\varepsilon;\zeta)[\mathbf{F},\boldsymbol{\eta}]| \le
{\mathfrak c}_2 \| K_D(\varepsilon ;\zeta)\mathbf{F}\|_{H^1(\mathcal{O})} \| \boldsymbol{\eta} \|_{H^1(\mathcal{O})}
+ |\zeta| \|Q_0\|_{L_\infty} \| K_D(\varepsilon ;\zeta)\mathbf{F}\|_{L_2(\mathcal{O})} \| \boldsymbol{\eta} \|_{L_2(\mathcal{O})},
\\
\mathbf{F} \in L_2(\mathcal{O};\mathbb{C}^n),\quad  \boldsymbol{\eta}\in H^1_0(\mathcal{O};\mathbb{C}^n).
\end{split}
\end{equation}
Hence, for $\mathbf{F}\in L_2(\mathcal{O};\mathbb{C}^n)$ fixed, relation \eqref{**.4} defines an antilinear continuous functional of $\boldsymbol{\eta}\in H^1_0(\mathcal{O};\mathbb{C}^n)$, which can be identified with an element from $H^{-1}(\mathcal{O};\mathbb{C}^n)$. This element
depends on $\mathbf{F}$ linearly, we denote it by $T(\varepsilon;\zeta)\mathbf{F}$. Thus,
\begin{equation}
\label{**.5}
\mathcal{I}(\varepsilon;\zeta)[\mathbf{F},\boldsymbol{\eta}]=(T(\varepsilon ;\zeta)\mathbf{F},\boldsymbol{\eta})_{L_2(\mathcal{O})},
\quad \mathbf{F} \in L_2(\mathcal{O};\mathbb{C}^n),\quad\boldsymbol{\eta}\in H^1_0(\mathcal{O};\mathbb{C}^n),
\end{equation}
where the right-hand side is understood as extension of the inner product in $L_2$ to pairs from $H^{-1}\times H^1_0$.
From \eqref{**.4a}, \eqref{**.5}, and the continuity of the operator $ K_D(\varepsilon ;\zeta): L_2(\mathcal{O};\mathbb{C}^n)
\to H^1(\mathcal{O};\mathbb{C}^n)$ it follows that the operator
$T(\varepsilon;\zeta) : L_2(\mathcal{O};\mathbb{C}^n)\rightarrow H^{-1}(\mathcal{O};\mathbb{C}^n)$ is continuous.

By \eqref{**.3} and \eqref{**.5}, we have
\begin{equation}
\label{**.6}
\mathbf{r}_\varepsilon =\varepsilon (B_{D,\varepsilon}-\zeta Q_0^\varepsilon)^{-1}T(\varepsilon ;\zeta)\mathbf{F},
\end{equation}
where the generalized resolvent is extended to a continuous operator acting from $H^{-1}(\mathcal{O};\mathbb{C}^n)$ to $H^1_0(\mathcal{O};\mathbb{C}^n)$. Now, by \eqref{**.2} and \eqref{**.6},
\begin{equation*}
\mathbf{w}_\varepsilon (\cdot;\zeta)=\varepsilon \bigl( K_D(\varepsilon ;\zeta)+(B_{D,\varepsilon}-\zeta Q_0^\varepsilon)^{-1}T(\varepsilon;\zeta)\bigr)\mathbf{F},
\end{equation*}
whence (see \eqref{2.46A})
\begin{equation}
\label{W_D(eps;zeta)}
W_D(\varepsilon;\zeta)=K_D(\varepsilon;\zeta)+(B_{D,\varepsilon}-\zeta Q_0^\varepsilon )^{-1}T(\varepsilon ;\zeta).
\end{equation}

The following theorem gives approximation for the solution $\mathbf{u}_\varepsilon$ in $H^1(\mathcal{O};\mathbb{C}^n)$ with error estimate of sharp order
$O(\varepsilon)$; in this approximation, the discrepancy $\mathbf{w}_\varepsilon$ is taken into account.

\begin{theorem}
\label{Theorem with Diriclet corrector}
Suppose that the assumptions of Theorem~\textnormal{\ref{Theorem Dirichlet H1}} are satisfied.
Let $\mathbf{w}_\varepsilon$ be the solution of problem~\eqref{w_eps problem}. Let $W_D(\varepsilon;\zeta)$ be the operator~\eqref{W_D(eps;zeta)}. Then for $\zeta \in \mathbb{C}\setminus \mathbb{R}_+$, $\vert \zeta \vert \geqslant 1$, and $0<\varepsilon\leqslant 1$ we have
\begin{equation}
\label{u_eps - V_eps +W-eps in H1}
\Vert \mathbf{u}_\varepsilon -\mathbf{v}_\varepsilon +\mathbf{w}_\varepsilon \Vert _{H^1(\mathcal{O})}\leqslant C_7 c(\phi)^4\varepsilon \Vert \mathbf{F}\Vert _{L_2(\mathcal{O})}.
\end{equation}
In operator terms,
\begin{equation}
\label{u_eps - V_eps +W-eps in H1 in operator terms}
\begin{split}
\Vert &(B_{D,\varepsilon}-\zeta Q_0^\varepsilon)^{-1}-(B_D^0-\zeta \overline{Q_0})^{-1}-\varepsilon K_D(\varepsilon ;\zeta)+\varepsilon W_D(\varepsilon;\zeta)\Vert _{L_2(\mathcal{O})\rightarrow H^1 (\mathcal{O})}
\leqslant C_7 c(\phi)^4\varepsilon .
\end{split}
\end{equation}
The constant $C_7$ depends only on the initial data~\eqref{problem data} and the domain~$\mathcal{O}$.
\end{theorem}

\section{Auxiliary statements}
\label{Section Lemmas}

\subsection{Estimates in the neighborhood of the boundary}

\begin{lemma}
\label{lemma ots int O_eps B_eps}
Suppose that Condition~\textnormal{\ref{condition varepsilon}} is satisfied.
Then for any $u\in H^1(\mathbb{R}^d)$ we have
\begin{equation*}
\int _{(\partial \mathcal{O})_\varepsilon }\vert u\vert ^2 \,d\mathbf{x}\leqslant \beta\varepsilon \Vert u\Vert _{H^1(\mathbb{R}^d)}\Vert u\Vert _{L_2(\mathbb{R}^d)},\quad 0<\varepsilon\leqslant\varepsilon _0.
\end{equation*}
The constant $\beta$ depends only on the domain $\mathcal{O}$.
\end{lemma}

\begin{lemma}
\label{Lemma 3.6 from Su15}
Suppose that Condition~\textnormal{\ref{condition varepsilon}} is satisfied.
Let $h(\mathbf{x})$ be a $\Gamma$-periodic function in~$\mathbb{R}^d$ such that $h\in L_2(\Omega)$.
Let $S_\varepsilon$ be the operator~\eqref{S_eps}. Denote~$\beta _* :=\beta (1+r_1)$, where $2r_1=\mathrm{diam}\,\Omega$.
Then for $0<\varepsilon\leqslant\varepsilon _1$ and any $\mathbf{u}\in H^1(\mathbb{R}^d;\mathbb{C}^k)$ we have
\begin{equation*}
\int _{(\partial\mathcal{O})_\varepsilon}\vert h^\varepsilon (\mathbf{x})\vert ^2 \vert (S_\varepsilon \mathbf{u})(\mathbf{x})\vert ^2\,d\mathbf{x}
\leqslant \beta _*\varepsilon\vert \Omega\vert ^{-1}\Vert h\Vert ^2_{L_2(\Omega)}\Vert \mathbf{u}\Vert _{H^1(\mathbb{R}^d)}\Vert \mathbf{u}\Vert _{L_2(\mathbb{R}^d)}.
\end{equation*}
\end{lemma}

Lemma \ref{Lemma 3.6 from Su15} is an analogue of Lemma~2.6 from~\cite{ZhPas}.
Lemmas \ref{lemma ots int O_eps B_eps} and \ref{Lemma 3.6 from Su15} were checked in~\cite[\S5]{PSu} under the condition $\partial \O \in C^1$,
but the proofs work also under Condition~\ref{condition varepsilon}.

\subsection{Properties of the matrix-valued functions $\Lambda$ and $\widetilde{\Lambda}$}

The following result was proved in \cite[Corollary~2.4]{PSu}.

\begin{lemma}
\label{Lemma Lambda in L infty}
Suppose that the $\Gamma$-periodic solution $\Lambda (\mathbf{x})$ of problem~\eqref{Lambda problem} is bounded: $\Lambda\in L_\infty$.
Then for any function $u\in H^1(\mathbb{R}^d)$ and $\varepsilon >0$ we have
\begin{equation*}
\int _{\mathbb{R}^d}\vert (\mathbf{D}\Lambda )^\varepsilon (\mathbf{x})\vert ^2\vert u(\mathbf{x})\vert ^2\,d\mathbf{x}
\leqslant \beta _1 \Vert u\Vert ^2 _{L_2(\mathbb{R}^d)}
+\beta _2 \varepsilon ^2\Vert \Lambda \Vert ^2_{L_\infty}\Vert \mathbf{D}u\Vert ^2 _{L_2(\mathbb{R}^d)}.
\end{equation*}
The constants $\beta _1$ and $\beta _2$ depend on $m$, $d$, $\alpha _0$, $\alpha _1$, $\Vert g\Vert _{L_\infty}$, and $\Vert g^{-1}\Vert _{L_\infty}$.
\end{lemma}

The following statement can be easily checked with the help of the H\"older inequality and the Sobolev embedding theorem; cf. \cite[Lemma~3.5]{MSu15}.

\begin{lemma}
\label{Lemma Lambda in Lp H1->L2}
\label{Lemma a embedding thm}
Let $h(\mathbf{x})$ be a $\Gamma$-periodic function in~$\mathbb{R}^d$ such that
\begin{equation}
\label{f condition lemma MSu}
h\in L_p(\Omega),\quad p=2\ \mbox{for}\ d=1,\quad p>2\ \mbox{for}\  d= 2,\quad p\geqslant d\ \mbox{for}\  d\geqslant 3.
\end{equation}
Then for $0<\varepsilon\leqslant 1$ the operator $[h^\varepsilon ]$ is a continuous mapping of $H^1(\mathbb{R}^d)$ to $L_2(\mathbb{R}^d)$, and
$$
\Vert [h^\varepsilon ]\Vert _{H^1(\mathbb{R}^d)\rightarrow L_2(\mathbb{R}^d)}\leqslant \Vert h\Vert _{L_p(\Omega)}C(\widehat{q},\Omega),
$$
where $C(\widehat{q},\Omega)$ is the norm of the embedding $H^1(\Omega)\hookrightarrow L_{\widehat{q}}(\Omega)$.
Here $\widehat{q}=\infty$ for $d=1$ and $\widehat{q}=2p(p-2)^{-1}$ for $d\geqslant 2$.
\end{lemma}

The following result was proved in \cite[Corollary~3.6]{MSu15}.

\begin{lemma}
\label{Lemma DLambda, Lambda in Lp}
Suppose that the $\Gamma$-periodic solution $\widetilde{\Lambda}(\mathbf{x})$ of problem \eqref{tildeLambda_problem}
satisfies condition~\eqref{f condition lemma MSu}.
Then for any $u\in H^2(\mathbb{R}^d)$ and $0<\varepsilon\leqslant 1$ we have
\begin{equation*}
\int _{\mathbb{R}^d}\vert (\mathbf{D}\widetilde{\Lambda})^\varepsilon(\mathbf{x})\vert ^2 \vert u(\mathbf{x})\vert ^2\,d\mathbf{x}
\leqslant
\widetilde{\beta}_1\Vert u\Vert ^2 _{H^1(\mathbb{R}^d)}
+\widetilde{\beta}_2\varepsilon ^2 \Vert \widetilde{\Lambda}\Vert _{L_p(\Omega)}^2 C(\widehat{q},\Omega)^2\Vert \mathbf{D}u\Vert ^2 _{H^1(\mathbb{R}^d)}.
\end{equation*}
Here $\widehat{q}$ is as in Lemma~\textnormal{\ref{Lemma Lambda in Lp H1->L2}}. The constants $\widetilde{\beta}_1$ and $\widetilde{\beta}_2$ depend only on $n$, $d$, $\alpha _0$, $\alpha _1$, $\rho$, $\Vert g\Vert _{L_\infty}$, $\Vert g^{-1}\Vert _{L_\infty}$, the norms $\Vert a_j\Vert _{L_\rho (\Omega)}$, $j=1,\dots,d$, and the parameters of the lattice~$\Gamma$.
\end{lemma}

\subsection{Lemma about $Q_0^\varepsilon - \overline{Q_0}$}

The proof of the following statement is quite similar to that of Lemma~3.7 from \cite{MSu15}. We omit the details.

\begin{lemma}
\label{Lemma Q0-overline Q0}
Let $Q_0(\mathbf{x})$ be a $\Gamma$-periodic $(n\times n)$-matrix-valued function such that $Q_0 \in L_\infty$.
Then the operator $[Q_0^\varepsilon -\overline{Q_0}]$ is a continuous mapping of
$H^1(\mathcal{O};\mathbb{C}^n)$ to $H^{-1}(\mathcal{O};\mathbb{C}^n)$, and we have
\begin{equation}
\label{lemma Q_eps -overline Q}
\Vert [Q_0^\varepsilon -\overline{Q_0}]\Vert _{H^1(\mathcal{O})\rightarrow H^{-1}(\mathcal{O})}\leqslant C_{Q_0}\varepsilon .
\end{equation}
The constant $C_{Q_0}$ depends on $d$, $\Vert Q_0\Vert _{L_\infty}$, and the parameters of the lattice $\Gamma$.
\end{lemma}

\section{Proof of Theorem~\ref{Theorem with Diriclet corrector}. Beginning of the  proofs of Theorems~\ref{Theorem Dirichlet L2} and~\ref{Theorem Dirichlet H1}}
\label{Section Proof Dirischlet corrector}

In this section, we prove Theorem~\ref{Theorem with Diriclet corrector} and reduce the proofs of
Theorems~\ref{Theorem Dirichlet L2} and~\ref{Theorem Dirichlet H1} to estimation of the correction term $\mathbf{w}_\varepsilon$.

\subsection{Associated problem in~$\mathbb{R}^d$}

By Lemma~\ref{Lemma hom problem estimates}, \eqref{PO}, and the inequality $\vert \zeta\vert \geqslant 1$, we have
\begin{align}
\label{tilde u_0 in L2}
\Vert \widetilde{\mathbf{u}}_0\Vert _{L_2(\mathbb{R}^d)} &\leqslant
k _1 c(\phi) \vert \zeta \vert ^{-1} \Vert \mathbf{F}\Vert _{L_2(\mathcal{O})};
\quad k_1 :=C_{\mathcal{O}}^{(0)} \Vert Q_0^{-1}\Vert _{L_\infty};
\\
\label{tilde u_0 in H1}
\Vert \widetilde{\mathbf{u}}_0\Vert _{H^1(\mathbb{R}^d)}&\leqslant
k_2 c(\phi)\vert \zeta \vert ^{-1/2}\Vert \mathbf{F}\Vert _{L_2(\mathcal{O})};
\quad k_2:= C_{\mathcal{O}}^{(1)}\left(\mathcal{C}_1+\Vert Q_0^{-1}\Vert _{L_\infty} \right);
\\
\label{tilde u_0 in H2}
\Vert \widetilde{\mathbf{u}}_0\Vert _{H^2(\mathbb{R}^d)}&\leqslant
k _3c(\phi)\Vert \mathbf{F}\Vert _{L_2(\mathcal{O})};\quad k_3:= C_{\mathcal{O}}^{(2)}\mathcal{C}_2.
\end{align}

We put
\begin{equation}
\label{tilde F =}
\widetilde{\mathbf{F}}:=(B^0-\zeta \overline{Q_0})\widetilde{\mathbf{u}}_0.
\end{equation}
Then $\widetilde{\mathbf{F}}\in L_2(\mathbb{R}^d;\mathbb{C}^n)$ and $\widetilde{\mathbf{F}}\vert _{\mathcal{O}}=\mathbf{F}$. Relations \eqref{L(xi)<=}, \eqref{tilde u_0 in L2}, and
\eqref{tilde u_0 in H2} imply that
\begin{align}
\label{tilde F<=}
\Vert \widetilde{\mathbf{F}}\Vert _{L_2(\mathbb{R}^d)}\leqslant C_L \Vert \widetilde{\mathbf{u}}_0\Vert _{H^2(\mathbb{R}^d)}+\vert \zeta \vert \vert \overline{Q_0}\vert \Vert \widetilde{\mathbf{u}}_0\Vert _{L_2(\mathbb{R}^d)}
\leqslant C_{\widetilde{F}}c(\phi)\Vert \mathbf{F}\Vert _{L_2(\mathcal{O})};
\quad C_{\widetilde{F}}:=k _3 C_L+k_1 \Vert Q_0\Vert _{L_\infty}.
\end{align}

Let $\widetilde{\mathbf{u}}_\varepsilon \in H^1(\mathbb{R}^d;\mathbb{C}^n)$ be the solution of the following
equation in $\mathbb{R}^d$:
\begin{equation}
\label{tilde u_eps problem}
B_\varepsilon \widetilde{\mathbf{u}}_\varepsilon -\zeta Q_0^\varepsilon \widetilde{\mathbf{u}}_\varepsilon =\widetilde{\mathbf{F}},
\end{equation}
i.~e., $\widetilde{\mathbf{u}}_\varepsilon =(B_\varepsilon -\zeta Q_0^\varepsilon )^{-1}\widetilde{\mathbf{F}}$.
Combining \eqref{tilde F =}--\eqref{tilde u_eps problem} and applying Theorems~\ref{Th L2 Rd} and~\ref{Theorem H1 Rd}, for $\zeta\in\mathbb{C}\setminus\mathbb{R}_+$, $\vert \zeta \vert \geqslant 1$, and
$0<\varepsilon \leqslant 1$ we obtain
\begin{align}
\label{tilde U eps -tilde u 0}
&\Vert \widetilde{\mathbf{u}}_\varepsilon -\widetilde{\mathbf{u}}_0\Vert _{L_2(\mathbb{R}^d)}\leqslant C_1 C_{\widetilde{F}}
 c(\phi)^3  \varepsilon \vert \zeta \vert ^{-1/2}\Vert \mathbf{F}\Vert _{L_2(\mathcal{O})},
\\
\label{tilde u_eps -v_eps <=}
&\Vert \widetilde{\mathbf{u}}_\varepsilon -\widetilde{\mathbf{v}}_\varepsilon \Vert _{L_2(\mathbb{R}^d)}\leqslant C_2C_{\widetilde{F}}c(\phi)^3\varepsilon \vert \zeta \vert ^{-1/2}\Vert \mathbf{F}\Vert _{L_2(\mathcal{O})},\\
\label{D tilde u_eps -v_eps <=}
&\Vert \mathbf{D}(\widetilde{\mathbf{u}}_\varepsilon -\widetilde{\mathbf{v}}_\varepsilon )\Vert _{L_2(\mathbb{R}^d)}\leqslant C_3C_{\widetilde{F}}c(\phi)^3\varepsilon\Vert \mathbf{F}\Vert _{L_2(\mathcal{O})}.
\end{align}
Now, \eqref{tilde u_eps -v_eps <=}, \eqref{D tilde u_eps -v_eps <=}, and the inequality $\vert \zeta \vert\geqslant 1$ imply that
\begin{equation}
\label{tilde u_eps -v_eps <= in H1}
\Vert \widetilde{\mathbf{u}}_\varepsilon -\widetilde{\mathbf{v}}_\varepsilon \Vert _{H^1(\mathbb{R}^d)}\leqslant \widetilde{C}_3 c(\phi)^3\varepsilon\Vert \mathbf{F}\Vert _{L_2(\mathcal{O})};
\quad \widetilde{C}_3 := (C_2+C_3)C_{\widetilde{F}}.
\end{equation}

\subsection{Proof of Theorem~\ref{Theorem with Diriclet corrector}}

Denote  $\mathbf{V}_\varepsilon :=\mathbf{u}_\varepsilon-\mathbf{v}_\varepsilon+\mathbf{w}_\varepsilon$. By \eqref{Dirichlet problem}, \eqref{w_eps problem}, and \eqref{w_eps int tozd}, the function $\mathbf{V}_\varepsilon\in H^1_0(\mathcal{O};\mathbb{C}^n)$ satisfies the identity
\begin{equation*}
\begin{split}
\mathfrak{b}_{D,\varepsilon}[\mathbf{V}_\varepsilon,\boldsymbol{\eta}]-\zeta (Q_0^\varepsilon\mathbf{V}_\varepsilon ,\boldsymbol{\eta})_{L_2(\mathcal{O})}
=(\mathbf{F},\boldsymbol{\eta})_{L_2(\mathcal{O})}-\mathfrak{b}_{N,\varepsilon}[\mathbf{v}_\varepsilon,\boldsymbol{\eta}]+\zeta (Q_0^\varepsilon\mathbf{v}_\varepsilon ,\boldsymbol{\eta})_{L_2(\mathcal{O})},
\quad \boldsymbol{\eta}\in H^1_0(\mathcal{O};\mathbb{C}^n).
\end{split}
\end{equation*}
We extend $\boldsymbol{\eta}$ by zero to $\mathbb{R}^d\setminus\mathcal{O}$, keeping the same notation.
Then $\boldsymbol{\eta}\in H^1(\mathbb{R}^d;\mathbb{C}^n)$. Recalling that $\widetilde{\mathbf{F}}$ is extension of  $\mathbf{F}$ and
$\widetilde{\mathbf{v}}_\varepsilon$ is extension of $\mathbf{v}_\varepsilon$, and
 using~\eqref{tilde u_eps problem}, we find
\begin{equation}
\label{tozd for V_eps var 2}
\mathfrak{b}_{D,\varepsilon}[\mathbf{V}_\varepsilon,\boldsymbol{\eta}]-\zeta (Q_0^\varepsilon\mathbf{V}_\varepsilon ,\boldsymbol{\eta})_{L_2(\mathcal{O})}=I_\varepsilon[\boldsymbol{\eta}],\quad \boldsymbol{\eta}\in H^1_0(\mathcal{O};\mathbb{C}^n),
\end{equation}
where the following notation is used:
\begin{equation}
\label{I_eps[eta]}
\begin{split}
I_\varepsilon [\boldsymbol{\eta}]:=\mathfrak{b}_\varepsilon[\widetilde{\mathbf{u}}_\varepsilon - \widetilde{\mathbf{v}}_\varepsilon,\boldsymbol{\eta}]-\zeta (Q_0^\varepsilon(\widetilde{\mathbf{u}}_\varepsilon - \widetilde{\mathbf{v}}_\varepsilon),\boldsymbol{\eta})_{L_2(\mathbb{R}^d)}, \quad \boldsymbol{\eta}\in H^1(\mathbb{R}^d;\mathbb{C}^n).
\end{split}
\end{equation}
Next, we estimate the functional \eqref{I_eps[eta]} with the help of \eqref{b_eps <=}, \eqref{tilde u_eps -v_eps <=}, and \eqref{tilde u_eps -v_eps <= in H1}:
\begin{equation}
\label{|Ieps|<=}
\begin{split}
\vert I_\varepsilon [\boldsymbol{\eta}]\vert &\leqslant C_{8}c(\phi)^3\varepsilon\Vert \mathbf{F}\Vert _{L_2(\mathcal{O})}\Vert \boldsymbol{\eta}\Vert _{H^1(\mathcal{O})}
+C_{9}c(\phi)^3\varepsilon\vert \zeta\vert ^{1/2}\Vert \mathbf{F}\Vert _{L_2(\mathcal{O})}\Vert (Q_0^\varepsilon)^{1/2}\boldsymbol{\eta}\Vert _{L_2(\mathcal{O})},
\end{split}
\end{equation}
where $C_{8}:=c_3 \widetilde{C}_3$ and $C_{9}:=\Vert Q_0\Vert ^{1/2}_{L_\infty}C_2C_{\widetilde{F}}$.

We substitute $\boldsymbol{\eta}=\mathbf{V}_\varepsilon$ in \eqref{tozd for V_eps var 2},
take the imaginary part, and apply \eqref{|Ieps|<=}. Then
\begin{equation}
\label{Im zeta (Q0eps V,V)<=}
\begin{split}
\vert \mathrm{Im}\,\zeta\vert (Q_0^\varepsilon\mathbf{V}_\varepsilon ,\mathbf{V}_\varepsilon )_{L_2(\mathcal{O})}
&=\vert \mathrm{Im}\,I_\varepsilon [\mathbf{V}_\varepsilon ]\vert
\leqslant C_{8}c(\phi)^3\varepsilon\Vert \mathbf{V}_\varepsilon\Vert _{H^1(\mathcal{O})}\Vert \mathbf{F}\Vert _{L_2(\mathcal{O})}
\\&
+C_{9}\vert \zeta\vert ^{1/2}c(\phi)^3\varepsilon\Vert \mathbf{F}\Vert _{L_2(\mathcal{O})}\Vert (Q_0^\varepsilon)^{1/2}\mathbf{V}_\varepsilon\Vert _{L_2(\mathcal{O})}.
\end{split}
\end{equation}
If $\mathrm{Re}\,\zeta\geqslant 0$ (and then $\mathrm{Im}\,\zeta \neq 0$), we deduce
\begin{equation}
\label{Q0Veps Re zeta >0}
(Q_0^\varepsilon \mathbf{V}_\varepsilon ,\mathbf{V}_\varepsilon )_{L_2(\mathcal{O})}
\leqslant 2 C_{8}c(\phi)^4\varepsilon\vert \zeta\vert ^{-1}\Vert \mathbf{V}_\varepsilon\Vert _{H^1(\mathcal{O})}\Vert \mathbf{F}\Vert _{L_2(\mathcal{O})}
+C_{9}^2c(\phi)^8 \varepsilon ^2 \vert \zeta\vert ^{-1} \Vert \mathbf{F}\Vert ^2_{L_2(\mathcal{O})},\ \mathrm{Re}\,\zeta\geqslant 0.
\end{equation}
If $\mathrm{Re}\,\zeta <0$, we take the real part in identity \eqref{tozd for V_eps var 2} with  $\boldsymbol{\eta}=\mathbf{V}_\varepsilon$. Note that $c(\phi)=1$ for such $\zeta$. Using \eqref{|Ieps|<=}, we obtain
\begin{equation}
\label{Re zeta (Q0eps V,V)<=}
\begin{split}
\vert \mathrm{Re}\,\zeta\vert (Q_0^\varepsilon\mathbf{V}_\varepsilon ,\mathbf{V}_\varepsilon )_{L_2(\mathcal{O})}
\leqslant C_{8}\varepsilon\Vert \mathbf{V}_\varepsilon\Vert _{H^1(\mathcal{O})}\Vert \mathbf{F}\Vert _{L_2(\mathcal{O})}
+C_{9} \varepsilon \vert \zeta\vert ^{1/2} \Vert \mathbf{F}\Vert _{L_2(\mathcal{O})}\Vert (Q_0^\varepsilon )^{1/2}\mathbf{V}_\varepsilon\Vert _{L_2(\mathcal{O})}.
\end{split}
\end{equation}
Summing up \eqref{Im zeta (Q0eps V,V)<=} and \eqref{Re zeta (Q0eps V,V)<=}, we
deduce the inequality
\begin{equation*}
(Q_0^\varepsilon\mathbf{V}_\varepsilon,\mathbf{V}_\varepsilon )_{L_2(\mathcal{O})}
\leqslant
4C_{8}\varepsilon\vert \zeta\vert ^{-1}\Vert \mathbf{V}_\varepsilon\Vert _{H^1(\mathcal{O})}\Vert \mathbf{F}\Vert _{L_2(\mathcal{O})}
+4C_{9}^2\varepsilon ^2\vert \zeta\vert ^{-1}\Vert \mathbf{F}\Vert ^2_{L_2(\mathcal{O})},\quad \mathrm{Re}\,\zeta <0.
\end{equation*}
Combining this with \eqref{Q0Veps Re zeta >0}, for all  $\zeta$ under consideration we obtain
\begin{equation}
\label{(Q0Eps V,V )<=}
(Q_0^\varepsilon\mathbf{V}_\varepsilon ,\mathbf{V}_\varepsilon )_{L_2(\mathcal{O})}
\leqslant 4C_{8}c(\phi)^4\varepsilon\vert \zeta\vert ^{-1}\Vert \mathbf{V}_\varepsilon\Vert _{H^1(\mathcal{O})}\Vert \mathbf{F}\Vert _{L_2(\mathcal{O})}
+4C_{9}^2 c(\phi)^8\varepsilon ^2\vert \zeta\vert ^{-1}\Vert \mathbf{F}\Vert ^2_{L_2(\mathcal{O})}.
\end{equation}
Now, \eqref{tozd for V_eps var 2} with $\boldsymbol{\eta}=\mathbf{V}_\varepsilon$, \eqref{|Ieps|<=}, and \eqref{(Q0Eps V,V )<=} imply that
\begin{equation*}
\begin{split}
\mathfrak{b}_{D,\varepsilon}[\mathbf{V}_\varepsilon ,\mathbf{V}_\varepsilon ]\leqslant
7C_{8}c(\phi)^4\varepsilon\Vert \mathbf{V}_\varepsilon\Vert _{H^1(\mathcal{O})}\Vert \mathbf{F}\Vert _{L_2(\mathcal{O})}+\frac{13}{2} C_{9}^2c(\phi)^8\varepsilon ^2\Vert \mathbf{F}\Vert ^2_{L_2(\mathcal{O})}.
\end{split}
\end{equation*}
Taking \eqref{H^1-norm <= BDeps^1/2} into account, we deduce
\begin{equation*}
\Vert \mathbf{V}_\varepsilon\Vert ^2_{H^1(\mathcal{O})}\leqslant C_7^2 c(\phi)^8\varepsilon ^2\Vert \mathbf{F}\Vert ^2_{L_2(\mathcal{O})};\quad C_7^2:=49c_4^4C_{8}^2+13c_4^2C_{9}^2,
\end{equation*}
which implies \eqref{u_eps - V_eps +W-eps in H1}.
\qed

Apart from estimate \eqref{u_eps - V_eps +W-eps in H1},
we also need to estimate the $L_2$-norm of $\mathbf{V}_\varepsilon$.

\begin{lemma}
Under the assumptions of Theorem~\textnormal{\ref{Theorem with Diriclet corrector}}, for  $\zeta\in\mathbb{C}\setminus\mathbb{R}_+$, $\vert \zeta\vert \geqslant 1$, and $0<\varepsilon\leqslant 1$
we have
  \begin{equation}
\label{u_eps - V_eps +W-eps in L_2}
\Vert \mathbf{u}_\varepsilon -\mathbf{v}_\varepsilon +\mathbf{w}_\varepsilon\Vert _{L_2(\mathcal{O})}\leqslant C_{10} c(\phi)^4\varepsilon\vert\zeta\vert ^{-1/2}\Vert \mathbf{F}\Vert _{L_2(\mathcal{O})}.
\end{equation}
The constant $C_{10}$ depends only on the initial data~\textnormal{\eqref{problem data}} and the domain $\mathcal{O}$.
\end{lemma}

\begin{proof}
By \eqref{u_eps - V_eps +W-eps in H1} and \eqref{(Q0Eps V,V )<=},
\begin{equation*}
(Q_0^\varepsilon\mathbf{V}_\varepsilon ,\mathbf{V}_\varepsilon)_{L_2(\mathcal{O})}
\leqslant 4 (C_7C_{8}+C_{9}^2)c(\phi)^8\varepsilon ^2\vert \zeta\vert ^{-1}\Vert \mathbf{F}\Vert ^2_{L_2(\mathcal{O})}.
\end{equation*}
This implies \eqref{u_eps - V_eps +W-eps in L_2} with the constant $C_{10}:=2\Vert Q_0^{-1}\Vert _{L_\infty}^{1/2}(C_7C_{8}+C_{9}^2)^{1/2}$.
\end{proof}

\subsection{Conclusions} 1) From \eqref{u_eps - V_eps +W-eps in H1} it follows that
\begin{equation}
\label{2.34 from draft}
\Vert \mathbf{u}_\varepsilon -\mathbf{v}_\varepsilon\Vert _{H^1(\mathcal{O})}\leqslant C_7 c(\phi)^4\varepsilon\Vert \mathbf{F}\Vert  _{L_2(\mathcal{O})}+\Vert \mathbf{w}_\varepsilon\Vert _{H^1(\mathcal{O})}.
\end{equation}
Hence, in order to prove \eqref{Th L2->H1 solutions},
it suffices to obtain an appropriate estimate for $\Vert \mathbf{w}_\varepsilon \Vert _{H^1(\mathcal{O})}$.

\noindent
2) By \eqref{u_eps - V_eps +W-eps in L_2},
\begin{equation}
\label{||u eps -u eps||L2<=}
\Vert \mathbf{u}_\varepsilon -\mathbf{u}_0\Vert _{L_2(\mathcal{O})}\leqslant C_{10} c(\phi)^4\varepsilon\vert \zeta\vert ^{-1/2}\Vert \mathbf{F}\Vert _{L_2(\mathcal{O})}+\Vert \mathbf{v}_\varepsilon -\mathbf{u}_0\Vert _{L_2(\mathcal{O})}+\Vert \mathbf{w}_\varepsilon\Vert _{L_2(\mathcal{O})}.
\end{equation}
We have
\begin{equation}
\label{||v eps -u eps||L2<=}
\Vert \mathbf{v}_\varepsilon -\mathbf{u}_0\Vert _{L_2(\mathcal{O})}\leqslant\varepsilon\Vert \Lambda ^\varepsilon S_\varepsilon b(\mathbf{D})\widetilde{\mathbf{u}}_0\Vert _{L_2(\mathbb{R}^d)}
+\varepsilon\Vert \widetilde{\Lambda}^\varepsilon S_\varepsilon\widetilde{\mathbf{u}}_0\Vert _{L_2(\mathbb{R}^d)}.
\end{equation}
From Proposition~\ref{Proposition f^eps S_eps}, \eqref{Lambda <=}, and \eqref{tilde Lambda<=}
it follows that
\begin{align}
\label{Lambda S_eps <=}
&\Vert [\Lambda ^\varepsilon]S_\varepsilon\Vert _{L_2(\mathbb{R}^d)\rightarrow L_2(\mathbb{R}^d)}\leqslant M_1,\\
\label{tilde Lambda S_eps <=}
&\Vert [\widetilde{\Lambda}^\varepsilon]S_\varepsilon\Vert _{L_2(\mathbb{R}^d)\rightarrow L_2(\mathbb{R}^d)}
\leqslant \widetilde{M}_1; \quad
\widetilde{M}_1:=\vert \Omega\vert ^{-1/2}(2r_0)^{-1}C_a n^{1/2}\alpha _0^{-1}\Vert g^{-1}\Vert _{L_\infty}.
\end{align}
Combining \eqref{<b^*b<}, \eqref{tilde u_0 in H1}, and \eqref{||v eps -u eps||L2<=}--\eqref{tilde Lambda S_eps <=}, we obtain
\begin{equation}
\label{||v eps - u 0||L2(O)}
\Vert \mathbf{v}_\varepsilon -\mathbf{u}_0\Vert _{L_2(\mathcal{O})}
\leqslant\varepsilon (M_1^2\alpha _1+\widetilde{M}_1^2)^{1/2}\Vert \widetilde{\mathbf{u}}_0\Vert  _{H^1(\mathbb{R}^d)}
\leqslant
\varepsilon (M_1^2\alpha _1+\widetilde{M}_1^2)^{1/2}k _2 c(\phi)\vert \zeta\vert ^{-1/2}\Vert \mathbf{F}\Vert _{L_2(\mathcal{O})}.
\end{equation}
Now, inequalities \eqref{||u eps -u eps||L2<=} and \eqref{||v eps - u 0||L2(O)} yield
\begin{equation}
\label{2.35 from draft}
\Vert \mathbf{u}_\varepsilon -\mathbf{u}_0\Vert _{L_2(\mathcal{O})}\leqslant C_{11} c(\phi)^4 \varepsilon \vert \zeta\vert ^{-1/2}\Vert \mathbf{F}\Vert _{L_2(\mathcal{O})}+\Vert \mathbf{w}_\varepsilon\Vert _{L_2(\mathcal{O})},
\end{equation}
where $C_{11}:=C_{10}+(M_1^2\alpha _1+\widetilde{M}_1^2)^{1/2}k _2$.
Thus, the proof of Theorem~\ref{Theorem Dirichlet L2} is reduced to appropriate  estimate for $\Vert \mathbf{w}_\varepsilon\Vert _{L_2(\mathcal{O})}$.

\section{The proof of Theorem~\ref{Theorem Dirichlet H1}}
\label{Section proof of L_2->H^1 theorem}

\subsection{Localization near the boundary}
Recall that
$(\partial\mathcal{O})_\varepsilon :=\lbrace \mathbf{x}\in \mathbb{R}^d :\mathrm{dist}\,\lbrace \mathbf{x};\partial \mathcal{O}\rbrace <\varepsilon \rbrace .$
Fix a smooth cut-off function $\theta _\varepsilon (\mathbf{x})$ in $\mathbb{R}^d$ such that
\begin{equation}
\label{theta_eps conditions}
\begin{split}
&\theta _\varepsilon \in C_0^\infty (\mathbb{R}^d),\quad \mathrm{supp}\,\theta _\varepsilon \subset (\partial \mathcal{O})_\varepsilon ,\quad 0\leqslant \theta _\varepsilon (\mathbf{x})\leqslant 1,\\
&\theta _\varepsilon (\mathbf{x})=1\ \mbox{for}\ \mathbf{x}\in \partial \mathcal{O};\quad \varepsilon\vert \nabla \theta _\varepsilon (\mathbf{x})\vert \leqslant \mu =\mathrm{Const}.
\end{split}
\end{equation}
The constant $\mu$ depends only on $d$ and the domain $\mathcal{O}$. Consider the following function in $\mathbb{R}^d$:
\begin{equation}
\label{phi eps}
\boldsymbol{\varphi }_\varepsilon (\mathbf{x}) :=\varepsilon \theta _\varepsilon (\mathbf{x})
\left(
\Lambda ^\varepsilon (\mathbf{x})(S_\varepsilon b(\mathbf{D})\widetilde{\mathbf{u}}_0)(\mathbf{x})
+\widetilde{\Lambda}^\varepsilon (\mathbf{x})(S_\varepsilon \widetilde{\mathbf{u}}_0)(\mathbf{x})
\right).
\end{equation}

\begin{lemma}\label{lemma w eps}
Suppose that $\mathbf{w}_\varepsilon$ is the solution of problem \eqref{w_eps problem}.
Suppose that $\boldsymbol{\varphi}_\varepsilon $ is given by \eqref{phi eps}.
Then for $\zeta \in \mathbb{C}\setminus \mathbb{R}_+$, $\vert \zeta \vert \geqslant 1$, and
$0<\varepsilon \leqslant \varepsilon _1$ we have
\begin{equation}
\label{3.3 from draft}
\Vert \mathbf{w}_\varepsilon \Vert _{H^1(\mathcal{O})}\leqslant c(\phi)\left(C_{12}\vert \zeta \vert ^{1/2}\Vert \boldsymbol{\varphi}_\varepsilon \Vert _{L_2(\mathcal{O})}+C_{13}\Vert \boldsymbol{\varphi}_\varepsilon \Vert _{H^1(\mathcal{O})}\right).
\end{equation}
The constants $C_{12}$ and $C_{13}$ depend only on the initial data \eqref{problem data} and the domain $\mathcal{O}$.
\end{lemma}

  \begin{proof}
We have $\mathbf{w}_\varepsilon \vert _{\partial\mathcal{O}}=\boldsymbol{\varphi}_\varepsilon \vert _{\partial\mathcal{O}}$. Therefore,  $\boldsymbol{\varrho}_\varepsilon :=\mathbf{w}_\varepsilon -\boldsymbol{\varphi}_\varepsilon\in H^1_0(\mathcal{O};\mathbb{C}^n)$.
By \eqref{w_eps int tozd},
\begin{equation}
\label{tozd b_d,eps rho eta}
\mathfrak{b}_{D,\varepsilon}[\boldsymbol{\varrho}_\varepsilon ,\boldsymbol{\eta}]-\zeta (Q_0^\varepsilon\boldsymbol{\varrho}_\varepsilon,\boldsymbol{\eta})_{L_2(\mathcal{O})}
=-\mathfrak{b}_{N,\varepsilon} [\boldsymbol{\varphi}_\varepsilon ,\boldsymbol{\eta}]+\zeta(Q_0^\varepsilon \boldsymbol{\varphi}_\varepsilon ,\boldsymbol{\eta})_{L_2(\mathcal{O})},
\quad \boldsymbol{\eta}\in H^1_0(\mathcal{O};\mathbb{C}^n).
\end{equation}
We substitute $\boldsymbol{\eta}=\boldsymbol{\varrho}_\varepsilon$ in~\eqref{tozd b_d,eps rho eta} and take the imaginary part.
Then, by~\eqref{b mathfrak eps <= c2},
\begin{equation}
\label{3.5 in draft}
\begin{split}
\vert \mathrm{Im}\,\zeta\vert (Q_0^\varepsilon \boldsymbol{\varrho}_\varepsilon ,\boldsymbol{\varrho}_\varepsilon)_{L_2(\mathcal{O})}
\leqslant \mathfrak{c}_2\Vert \boldsymbol{\varphi}_\varepsilon\Vert _{H^1(\mathcal{O})}\Vert \boldsymbol{\varrho}_\varepsilon \Vert _{H^1(\mathcal{O})}+\vert \zeta \vert \Vert Q_0\Vert ^{1/2}_{L_\infty}\Vert \boldsymbol{\varphi}_\varepsilon\Vert _{L_2(\mathcal{O})}\Vert (Q_0^\varepsilon )^{1/2}\boldsymbol{\varrho}_\varepsilon \Vert _{L_2(\mathcal{O})}.
\end{split}
\end{equation}
 If $\mathrm{Re}\,\zeta \geqslant 0$ (and then $\mathrm{Im}\,\zeta\neq 0$), we deduce
\begin{equation*}
(Q_0^\varepsilon \boldsymbol{\varrho}_\varepsilon ,\boldsymbol{\varrho}_\varepsilon)_{L_2(\mathcal{O})}\leqslant 2 \mathfrak{c}_2
 c(\phi) \vert \zeta \vert ^{-1} \Vert \boldsymbol{\varphi}_\varepsilon\Vert _{H^1(\mathcal{O})}\Vert \boldsymbol{\varrho}_\varepsilon\Vert _{H^1(\mathcal{O})}
+\Vert Q_0\Vert _{L_\infty} c(\phi )^2\Vert \boldsymbol{\varphi}_\varepsilon\Vert ^2 _{L_2(\mathcal{O})},\quad \mathrm{Re}\,\zeta\geqslant 0.
\end{equation*}
If $\mathrm{Re}\,\zeta <0$, we take the real part of the corresponding identity and obtain
\begin{equation}
\label{3.7 in draft}
\vert \mathrm{Re}\,\zeta\vert (Q_0^\varepsilon \boldsymbol{\varrho}_\varepsilon, \boldsymbol{\varrho}_\varepsilon )_{L_2(\mathcal{O})}
\leqslant \mathfrak{c}_2\Vert \boldsymbol{\varphi}_\varepsilon \Vert _{H^1(\mathcal{O})}\Vert \boldsymbol{\varrho}_\varepsilon\Vert _{H^1(\mathcal{O})}
+\vert \zeta \vert \Vert Q_0 \Vert ^{1/2}_{L_\infty}\Vert \boldsymbol{\varphi}_\varepsilon \Vert _{L_2(\mathcal{O})}\Vert (Q_0^\varepsilon )^{1/2}\boldsymbol{\varrho}_\varepsilon \Vert _{L_2(\mathcal{O})}.
\end{equation}
Summing up \eqref{3.5 in draft} and \eqref{3.7 in draft}, we deduce
\begin{equation*}
(Q_0^\varepsilon \boldsymbol{\varrho}_\varepsilon ,\boldsymbol{\varrho}_\varepsilon)_{L_2(\mathcal{O})}\leqslant 4 \mathfrak{c}_2 \vert \zeta\vert ^{-1}\Vert \boldsymbol{\varphi}_\varepsilon\Vert _{H^1(\mathcal{O})}\Vert \boldsymbol{\varrho}_\varepsilon \Vert _{H^1(\mathcal{O})}+4\Vert Q_0\Vert _{L_\infty}\Vert \boldsymbol{\varphi}_\varepsilon \Vert ^2 _{L_2(\mathcal{O})},\quad \mathrm{Re}\,\zeta <0.
\end{equation*}
Thus, for all $\zeta$ under consideration we have
\begin{equation}
\label{5.6a}
(Q_0^\varepsilon \boldsymbol{\varrho}_\varepsilon ,\boldsymbol{\varrho}_\varepsilon )_{L_2(\mathcal{O})}\leqslant 4 \mathfrak{c}_2 c(\phi) \vert \zeta \vert ^{-1} \Vert \boldsymbol{\varphi}_\varepsilon \Vert _{H^1(\mathcal{O})}\Vert \boldsymbol{\varrho}_\varepsilon \Vert _{H^1(\mathcal{O})}
+4\Vert Q_0\Vert _{L_\infty} c(\phi)^2\Vert \boldsymbol{\varphi}_\varepsilon \Vert ^2 _{L_2(\mathcal{O})}.
\end{equation}
From \eqref{tozd b_d,eps rho eta} with $\boldsymbol{\eta}=\boldsymbol{\varrho}_\varepsilon$, \eqref{5.6a}, and
\eqref{b mathfrak eps <= c2} it follows that
\begin{equation*}
\begin{split}
\mathfrak{b}_{D,\varepsilon }[\boldsymbol{\varrho}_\varepsilon ,\boldsymbol{\varrho}_\varepsilon ]
\leqslant 9 \mathfrak{c}_2 c(\phi)\Vert \boldsymbol{\varphi}_\varepsilon \Vert _{H^1(\mathcal{O})}\Vert \boldsymbol{\varrho}_\varepsilon\Vert _{H^1(\mathcal{O})}
+9 c(\phi)^2 \vert \zeta\vert \Vert Q_0\Vert _{L_\infty} \Vert \boldsymbol{\varphi}_\varepsilon \Vert ^2 _{L_2(\mathcal{O})}.
\end{split}
\end{equation*}
Together with \eqref{H^1-norm <= BDeps^1/2}, this implies
\begin{equation*}
\Vert \boldsymbol{\varrho}_\varepsilon \Vert _{H^1(\mathcal{O})}\leqslant 9 \mathfrak{c}_2 c_4^2 c(\phi)\Vert \boldsymbol{\varphi}_\varepsilon\Vert _{H^1(\mathcal{O})}
+3\sqrt{2}c_4 \Vert Q_0\Vert ^{1/2}_{L_\infty} c(\phi) \vert \zeta\vert ^{1/2} \Vert \boldsymbol{\varphi}_\varepsilon \Vert _{L_2(\mathcal{O})}.
\end{equation*}
Recalling that $\boldsymbol{\varrho}_\varepsilon =\mathbf{w}_\varepsilon -\boldsymbol{\varphi}_\varepsilon$,
we obtain \eqref{3.3 from draft} with the constants
$C_{13}:=9\mathfrak{c}_2 c_4^2+1$ and $C_{12}:=3\sqrt{2} c_4 \Vert Q_0\Vert ^{1/2}_{L_\infty}$.
\end{proof}

\subsection{Estimates for the function $\boldsymbol{\varphi}_\varepsilon $}
\begin{lemma}\label{lemma phi eps}
Let $\boldsymbol{\varphi}_\varepsilon $ be given by \eqref{phi eps}. Then for $\zeta \in \mathbb{C}\setminus \mathbb{R}_+$, $\vert \zeta \vert \geqslant 1$, and $0<\varepsilon \leqslant \varepsilon _1$ we have
\begin{align}
\label{3.11 from draft}
&\Vert \boldsymbol{\varphi}_\varepsilon \Vert _{L_2(\mathbb{R}^d)}\leqslant C_{14} c(\phi) \varepsilon \vert \zeta \vert ^{-1/2}\Vert \mathbf{F}\Vert _{L_2(\mathcal{O})},
\\
\label{3.12 from draft}
&\Vert \mathbf{D}\boldsymbol{\varphi}_\varepsilon \Vert _{L_2(\mathbb{R}^d)}\leqslant c(\phi)\left(C_{15} \varepsilon ^{1/2} \vert \zeta \vert ^{-1/4}
+C_{16}\varepsilon \right)\Vert \mathbf{F}\Vert _{L_2(\mathcal{O})}.
\end{align}
The constants $C_{14}$, $C_{15}$, and $C_{16}$ depend only on the initial data \eqref{problem data} and the domain $\mathcal{O}$.
\end{lemma}

\begin{proof}
First, we prove \eqref{3.11 from draft}.
From \eqref{<b^*b<}, \eqref{Lambda S_eps <=}, \eqref{tilde Lambda S_eps <=}, \eqref{theta_eps conditions}, and \eqref{phi eps} it follows that
\begin{equation*}
\begin{split}
\Vert \boldsymbol{\varphi}_\varepsilon \Vert _{L_2(\mathbb{R}^d)}
\leqslant \varepsilon M_1 \alpha _1 ^{1/2}\Vert \widetilde{\mathbf{u}}_0\Vert _{H^1(\mathbb{R}^d)}+\varepsilon \widetilde{M}_1\Vert \widetilde{\mathbf{u}}_0\Vert _{L_2(\mathbb{R}^d)}.
\end{split}
\end{equation*}
Combining this with \eqref{tilde u_0 in L2}, \eqref{tilde u_0 in H1}, and the inequality $|\zeta|\ge 1$, we obtain \eqref{3.11 from draft} with the constant $C_{14}:= M_1\alpha _1 ^{1/2}k _2+\widetilde{M}_1 k _1$.

To prove \eqref{3.12 from draft},  consider the derivatives:
\begin{equation*}
\begin{split}
\partial _j \boldsymbol{\varphi}_\varepsilon &= \varepsilon (\partial _j \theta _\varepsilon )(\Lambda ^\varepsilon S_\varepsilon b(\mathbf{D})\widetilde{\mathbf{u}}_0+\widetilde{\Lambda}^\varepsilon S_\varepsilon \widetilde{\mathbf{u}}_0)
+\theta _\varepsilon \left(
(\partial _j \Lambda )^\varepsilon S_\varepsilon b(\mathbf{D})\widetilde{\mathbf{u}}_0+(\partial _j \widetilde{\Lambda})^\varepsilon S_\varepsilon \widetilde{\mathbf{u}}_0
\right)
\\
&+\varepsilon \theta _\varepsilon (\Lambda ^\varepsilon S_\varepsilon b(\mathbf{D})\partial _j\widetilde{\mathbf{u}}_0+\widetilde{\Lambda}^\varepsilon S_\varepsilon \partial _j \widetilde{\mathbf{u}}_0).
\end{split}
\end{equation*}
Hence,
\begin{equation}
\label{3.13 from draft}
\begin{split}
\Vert \mathbf{D}\boldsymbol{\varphi}_\varepsilon  \Vert ^2 _{L_2(\mathbb{R}^d)}
&\leqslant 3\varepsilon ^2 \Vert (\nabla \theta _\varepsilon )(\Lambda ^\varepsilon S_\varepsilon b(\mathbf{D})\widetilde{\mathbf{u}}_0+\widetilde{\Lambda}^\varepsilon S_\varepsilon \widetilde{\mathbf{u}}_0)\Vert ^2 _{L_2(\mathbb{R}^d)}
\\
&+3\bigl\Vert \theta _\varepsilon
\bigl(
(\mathbf{D}\Lambda)^\varepsilon S_\varepsilon b(\mathbf{D})\widetilde{\mathbf{u}}_0+(\mathbf{D}\widetilde{\Lambda})^\varepsilon S_\varepsilon \widetilde{\mathbf{u}}_0
\bigr)\bigr\Vert ^2 _{L_2(\mathbb{R}^d)}
\\
&+3\varepsilon ^2 \sum _{j=1}^d \Vert \theta _\varepsilon (\Lambda ^\varepsilon S_\varepsilon b(\mathbf{D})\partial _j \widetilde{\mathbf{u}}_0+\widetilde{\Lambda}^\varepsilon S_\varepsilon \partial _j \widetilde{\mathbf{u}}_0)\Vert ^2 _{L_2(\mathbb{R}^d)}.
\end{split}
\end{equation}
Denote the consecutive terms in the right-hand side of~\eqref{3.13 from draft} by~$J_1(\varepsilon)$, $J_2(\varepsilon )$, and~$J_3(\varepsilon)$.
The term $J_1(\varepsilon)$ is estimated with the help of~\eqref{theta_eps conditions} and Lemma~\ref{Lemma 3.6 from Su15}:
\begin{equation}
\label{J_1 leqslant part 1}
\begin{split}
J_1(\varepsilon )
&\leqslant
6\mu ^2 \left(
\int _{(\partial\mathcal{O})_\varepsilon}\vert \Lambda ^\varepsilon S_\varepsilon b(\mathbf{D})\widetilde{\mathbf{u}}_0\vert ^2\,d\mathbf{x}
+\int _{(\partial\mathcal{O})_\varepsilon }\vert \widetilde{\Lambda}^\varepsilon S_\varepsilon \widetilde{\mathbf{u}}_0\vert ^2\,d\mathbf{x}
\right)
\\
&\leqslant 6\mu ^2
\beta _*\varepsilon \vert \Omega\vert ^{-1}\Vert \Lambda \Vert ^2 _{L_2(\Omega)}\Vert b(\mathbf{D})\widetilde{\mathbf{u}}_0\Vert _{H^1(\mathbb{R}^d)}\Vert b(\mathbf{D})\widetilde{\mathbf{u}}_0\Vert _{L_2(\mathbb{R}^d)}
\\
&+ 6\mu ^2 \beta _*\varepsilon\vert \Omega\vert ^{-1}\Vert \widetilde{\Lambda}\Vert ^2 _{L_2(\Omega)}\Vert \widetilde{\mathbf{u}}_0\Vert _{H^1(\mathbb{R}^d)}\Vert \widetilde{\mathbf{u}}_0\Vert _{L_2(\mathbb{R}^d)}
,
\quad 0<\varepsilon\leqslant\varepsilon _1.
\end{split}
\end{equation}
According to \eqref{tilde Lambda<=} and \eqref{tilde Lambda S_eps <=},
$\vert \Omega\vert ^{-1/2}\Vert \widetilde{\Lambda}\Vert _{L_2(\Omega )}\leqslant \widetilde{M_1}$. Combining this with \eqref{<b^*b<}, \eqref{Lambda <=}, and \eqref{J_1 leqslant part 1}, we obtain
\begin{equation*}
\begin{split}
J_1(\varepsilon)\leqslant 6\mu ^2 \beta _* \varepsilon \left( M_1^2 \alpha _1 \Vert \widetilde{\mathbf{u}}_0\Vert _{H^2(\mathbb{R}^d)}\Vert \widetilde{\mathbf{u}}_0\Vert _{H^1(\mathbb{R}^d)} + \widetilde{M}_1^2\Vert \widetilde{\mathbf{u}}_0\Vert _{H^1(\mathbb{R}^d)}\Vert \widetilde{\mathbf{u}}_0\Vert _{L_2(\mathbb{R}^d)}\right),
\quad 0<\varepsilon\leqslant \varepsilon _1.
\end{split}
\end{equation*}
Together with \eqref{tilde u_0 in L2}--\eqref{tilde u_0 in H2} and the inequality $\vert \zeta\vert \geqslant 1$, this implies
\begin{equation}
\label{J1<=}
J_1(\varepsilon)\leqslant \kappa _1 c(\phi)^2 \varepsilon \vert \zeta\vert ^{-1/2} \Vert \mathbf{F}\Vert ^2 _{L_2(\mathcal{O})},\quad 0<\varepsilon\leqslant\varepsilon _1; \quad \kappa _1:=6\mu ^2 \beta _* k_2(M_1^2\alpha _1 k_3+\widetilde{M}_1^2k_1).
\end{equation}

From \eqref{D tilde Lambda} it follows that $\vert\Omega\vert ^{-1/2}\Vert \mathbf{D}\widetilde{\Lambda}\Vert _{L_2(\Omega)}\leqslant\widetilde{M}_2$,
where
\begin{equation}
\label{tilde M2 =}
\widetilde{M}_2:=\vert \Omega\vert ^{-1/2}C_an^{1/2}\alpha _0^{-1}\Vert g^{-1}\Vert _{L_\infty}.
\end{equation}
The term $J_2(\varepsilon )$ is estimated similarly to $J_1(\varepsilon )$
with the help of Lemma~\ref{Lemma 3.6 from Su15} and relations \eqref{<b^*b<}, \eqref{DLambda<=},  \eqref{tilde u_0 in L2}--\eqref{tilde u_0 in H2}, and \eqref{theta_eps conditions}. We arrive at
\begin{equation}
\label{J2<=}
J_2(\varepsilon)\leqslant \kappa _2c(\phi)^2 \varepsilon \vert \zeta\vert ^{-1/2} \Vert \mathbf{F}\Vert ^2 _{L_2(\mathcal{O})},\quad 0<\varepsilon\leqslant\varepsilon _1;\quad \kappa _2 := 6\beta _* k_2(M_2^2k_3\alpha _1 +\widetilde{M}_2^2k_1).
\end{equation}

Finally, the term $J_3(\varepsilon)$ is estimated by using \eqref{<b^*b<}, \eqref{Lambda S_eps <=}, \eqref{tilde Lambda S_eps <=}, and \eqref{theta_eps conditions}:
\begin{equation*}
\begin{split}
J_3(\varepsilon) \leqslant
6\varepsilon ^2\left(
M_1^2\alpha _1 \Vert \widetilde{\mathbf{u}}_0\Vert ^2 _{H^2(\mathbb{R}^d)}+\widetilde{M}_1^2\Vert \widetilde{\mathbf{u}}_0\Vert ^2 _{H^1(\mathbb{R}^d)}
\right).
\end{split}
\end{equation*}
Together with \eqref{tilde u_0 in H1}, \eqref{tilde u_0 in H2}, and the inequality $\vert \zeta\vert \geqslant 1$, this yields
\begin{equation}
\label{J3<=}
J_3(\varepsilon)\leqslant \kappa _3 c(\phi)^2 \varepsilon ^2 \Vert \mathbf{F}\Vert ^2 _{L_2(\mathcal{O})},\quad 0<\varepsilon\leqslant 1;
\quad \kappa _3 :=6 M_1^2 \alpha _1 k_3^2+6\widetilde{M}_1^2k_2^2.
\end{equation}
Now, relations \eqref{3.13 from draft}, \eqref{J1<=}, \eqref{J2<=}, and \eqref{J3<=} imply \eqref{3.12 from draft}
with the constants $C_{15} :=(\kappa _1+\kappa _2)^{1/2}$ and $C_{16} :=\kappa _3^{1/2}$.
\end{proof}

\subsection{Completion of the proof of Theorem~\ref{Theorem Dirichlet H1}}

From Lemmas~\ref{lemma w eps} and~\ref{lemma phi eps} it follows that
\begin{equation*}
\Vert \mathbf{w}_\varepsilon\Vert _{H^1(\mathcal{O})}
\leqslant c (\phi)^2
\left(C_{13}C_{15} \varepsilon ^{1/2}  \vert \zeta\vert ^{-1/4} +(C_{12}C_{14}+C_{13}C_{14}+C_{13}C_{16})\varepsilon \right)
\Vert \mathbf{F}\Vert _{L_2(\mathcal{O})}
\end{equation*}
for $\zeta\in\mathbb{C}\setminus\mathbb{R}_+$, $\vert \zeta\vert \geqslant 1$, and $0<\varepsilon\leqslant\varepsilon _1$.
Together with~\eqref{2.34 from draft}, this implies~\eqref{Th L2->H1 solutions} with the constants $C_5:=C_{13}C_{15}$ and $C_6:=C_7+C_{12}C_{14}+C_{13}C_{14}+C_{13}C_{16}$.

It remains to check~\eqref{Th fluxes}. By \eqref{b_l <=} and \eqref{Th L2->H1 solutions},
\begin{equation}
\label{3.19 from draft}
\Vert \mathbf{p}_\varepsilon -g^\varepsilon b(\mathbf{D})\mathbf{v}_\varepsilon \Vert _{L_2(\mathcal{O})}
\leqslant \Vert g\Vert _{L_\infty}(d\alpha _1 )^{1/2}(C_5c(\phi)^2\varepsilon ^{1/2}\vert\zeta\vert ^{-1/4}+C_6 c(\phi)^4\varepsilon )\Vert \mathbf{F}\Vert _{L_2(\mathcal{O})}.
\end{equation}
We have
\begin{equation}
\label{3.20 from draft}
\begin{split}
g^\varepsilon b(\mathbf{D})\mathbf{v}_\varepsilon
&=
g^\varepsilon b(\mathbf{D})\mathbf{u}_0
+g^\varepsilon (b(\mathbf{D})\Lambda )^\varepsilon S_\varepsilon b(\mathbf{D})\widetilde{\mathbf{u}}_0
+g^\varepsilon (b(\mathbf{D})\widetilde{\Lambda})^\varepsilon S_\varepsilon\widetilde{\mathbf{u}}_0\\
&+\varepsilon\sum _{l=1}^d g^\varepsilon b_l (\Lambda ^\varepsilon S_\varepsilon b(\mathbf{D})D_l\widetilde{\mathbf{u}}_0+\widetilde{\Lambda}^\varepsilon S_\varepsilon D_l\widetilde{\mathbf{u}}_0).
\end{split}
\end{equation}
The fourth term in the right-hand side of~\eqref{3.20 from draft} is estimated with the help of~\eqref{b_l <=}, \eqref{Lambda S_eps <=}, and~\eqref{tilde Lambda S_eps <=}:
\begin{equation}
\label{5.18a}
\begin{split}
\Bigl\Vert &\varepsilon\sum _{l=1}^d g^\varepsilon b_l (\Lambda ^\varepsilon S_\varepsilon b(\mathbf{D})D_l\widetilde{\mathbf{u}}_0+\widetilde{\Lambda}^\varepsilon S_\varepsilon D_l\widetilde{\mathbf{u}}_0)\Bigr\Vert _{L_2(\mathcal{O})}
\\
&\leqslant
\varepsilon \Vert g\Vert _{L_\infty}\alpha_1 ^{1/2}
\left(
M_1\sum _{l=1}^d \Vert b(\mathbf{D})D_l\widetilde{\mathbf{u}}_0\Vert _{L_2(\mathbb{R}^d)}+\widetilde{M}_1\sum _{l=1}^d\Vert D_l\widetilde{\mathbf{u}}_0\Vert _{L_2(\mathbb{R}^d)}
\right).
\end{split}
\end{equation}
Combining this with~\eqref{<b^*b<}, \eqref{tilde u_0 in H1}, \eqref{tilde u_0 in H2}, and the condition $\vert\zeta\vert\geqslant 1$,
we deduce
\begin{equation}
\label{4-th chlen in tozd for fluxes}
\begin{split}
\Bigl\Vert \varepsilon\sum _{l=1}^d g^\varepsilon b_l (\Lambda ^\varepsilon S_\varepsilon b(\mathbf{D})D_l\widetilde{\mathbf{u}}_0+\widetilde{\Lambda}^\varepsilon S_\varepsilon D_l\widetilde{\mathbf{u}}_0)\Bigr\Vert _{L_2(\mathcal{O})}
\leqslant
C_{17}c(\phi)\varepsilon\Vert \mathbf{F}\Vert _{L_2(\mathcal{O})},
\end{split}
\end{equation}
where
$C_{17} :=\Vert g\Vert _{L_\infty} (d\alpha _1)^{1/2}\bigl(M_1\alpha _1^{1/2} k_3+\widetilde{M}_1k_2\bigr).$

Next, by Proposition~\ref{Proposition S__eps - I}, we have
\begin{equation}
\label{5.19a}
\begin{split}
\Vert g^\varepsilon b(\mathbf{D})\widetilde{\mathbf{u}}_0-g^\varepsilon S_\varepsilon b(\mathbf{D})\widetilde{\mathbf{u}}_0\Vert _{L_2(\mathbb{R}^d)}
\leqslant
\varepsilon r_1\Vert g\Vert _{L_\infty}\Vert \mathbf{D}b(\mathbf{D})\widetilde{\mathbf{u}}_0\Vert _{L_2(\mathbb{R}^d)}.
\end{split}
\end{equation}
Together with \eqref{<b^*b<} and \eqref{tilde u_0 in H2}, this implies
\begin{equation}
\label{3.22 from draft}
\Vert g^\varepsilon b(\mathbf{D})\widetilde{\mathbf{u}}_0-g^\varepsilon S_\varepsilon b(\mathbf{D})\widetilde{\mathbf{u}}_0\Vert _{L_2(\mathbb{R}^d)}
\leqslant
C_{18} c(\phi) \varepsilon \Vert \mathbf{F}\Vert _{L_2(\mathcal{O})},
\end{equation}
where
$C_{18} :=r_1\alpha_1^{1/2}k_3\Vert g\Vert _{L_\infty}$. Now, relations \eqref{tilde g}, \eqref{3.19 from draft}, \eqref{3.20 from draft}, \eqref{4-th chlen in tozd for fluxes}, and \eqref{3.22 from draft} imply \eqref{Th fluxes} with the constants $\widetilde{C}_5:=(d\alpha_1)^{1/2}\Vert g\Vert _{L_\infty}C_5$ and $\widetilde{C}_6:=(d\alpha_1)^{1/2}\Vert g\Vert _{L_\infty}C_6+C_{17}+C_{18}$. \qed

\section{The proof of Theorem~\ref{Theorem Dirichlet L2}}
\label{Section proof of L_2 rightarrow L_2 theorem}
\subsection{Estimate for the discrepancy $\mathbf{w}_\varepsilon$ in $L_2$}

\begin{lemma}
Suppose that $\mathbf{w}_\varepsilon$ is the solution of problem \eqref{w_eps problem}.
Suppose that the number $\varepsilon _1$ is subject to Condition~\textnormal{\ref{condition varepsilon}}.
Then for $\zeta\in \mathbb{C}\setminus\mathbb{R}_+$, $\vert \zeta\vert\geqslant 1$, and
$0<\varepsilon\leqslant\varepsilon _1$ we have
\begin{equation}
\label{3.23 from draft}
\Vert \mathbf{w}_\varepsilon \Vert _{L_2(\mathcal{O})}\leqslant c(\phi)^5(C_{19}\varepsilon \vert \zeta\vert ^{-1/2}+C_{20}\varepsilon ^2)\Vert \mathbf{F}\Vert _{L_2(\mathcal{O})}.
\end{equation}
The constants $C_{19}$ and $C_{20}$ depend only on the initial data \eqref{problem data} and the domain $\mathcal{O}$.
\end{lemma}

\begin{proof}
Recall that $\boldsymbol{\varrho}_\varepsilon =\mathbf{w}_\varepsilon -\boldsymbol{\varphi}_\varepsilon$
satisfies~\eqref{tozd b_d,eps rho eta}. We substitute $\boldsymbol{\eta} = \boldsymbol{\eta}_\varepsilon =(B_{D,\varepsilon}-\zeta ^*Q_0^\varepsilon)^{-1}\boldsymbol{\Phi}$ with $\boldsymbol{\Phi}\in L_2(\mathcal{O};\mathbb{C}^n)$ into this identity.
Then the left-hand side of \eqref{tozd b_d,eps rho eta} can be written as
\begin{equation*}
\mathfrak{b}_{D,\varepsilon}[\boldsymbol{\varrho}_\varepsilon,\boldsymbol{\eta}_\varepsilon]-\zeta(Q_0^\varepsilon \boldsymbol{\varrho}_\varepsilon ,\boldsymbol{\eta}_\varepsilon )_{L_2(\mathcal{O})} =(\boldsymbol{\varrho}_\varepsilon,\boldsymbol{\Phi})_{L_2(\mathcal{O})}.
\end{equation*}
Hence,
\begin{equation}
\label{3.24 from draft}
(\mathbf{w}_\varepsilon-\boldsymbol{\varphi}_\varepsilon ,\boldsymbol{\Phi})_{L_2(\mathcal{O})}=-\mathfrak{b}_{N,\varepsilon} [\boldsymbol{\varphi}_\varepsilon ,\boldsymbol{\eta}_\varepsilon ]+\zeta (Q_0^\varepsilon \boldsymbol{\varphi}_\varepsilon ,\boldsymbol{\eta}_\varepsilon )_{L_2(\mathcal{O})}.
\end{equation}

To approximate $\boldsymbol{\eta}_\varepsilon$ in $H^1(\mathcal{O};\mathbb{C}^n)$, we apply the already proved Theorem~\ref{Theorem Dirichlet H1}. Denote $\boldsymbol{\eta}_0=(B_D^0-\zeta ^*\overline{Q_0})^{-1}\boldsymbol{\Phi}$ and $\widetilde{\boldsymbol{\eta}}_0=P_\mathcal{O}\boldsymbol{\eta}_0$. The first order approximation of $\boldsymbol{\eta}_\varepsilon$
is given by $\boldsymbol{\eta}_0+\varepsilon\Lambda^\varepsilon S_\varepsilon b(\mathbf{D})\widetilde{\boldsymbol{\eta}}_0+\varepsilon\widetilde{\Lambda}^\varepsilon S_\varepsilon\widetilde{\boldsymbol{\eta}}_0$.
We rewrite \eqref{3.24 from draft} in the form
\begin{equation}
\label{3.25 from draft}
\begin{split}
(\mathbf{w}_\varepsilon-\boldsymbol{\varphi}_\varepsilon ,\boldsymbol{\Phi})_{L_2(\mathcal{O})}&=-\mathfrak{b}_{N,\varepsilon} [\boldsymbol{\varphi}_\varepsilon ,\boldsymbol{\eta}_\varepsilon -\boldsymbol{\eta} _0 -\varepsilon\Lambda^\varepsilon S_\varepsilon b(\mathbf{D})\widetilde{\boldsymbol{\eta}}_0-\varepsilon\widetilde{\Lambda}^\varepsilon S_\varepsilon\widetilde{\boldsymbol{\eta}}_0]
-\mathfrak{b}_{N,\varepsilon} [\boldsymbol{\varphi }_\varepsilon ,\boldsymbol{\eta}_0]
\\
&-\mathfrak{b}_{N,\varepsilon}[\boldsymbol{\varphi }_\varepsilon ,\varepsilon\Lambda^\varepsilon S_\varepsilon b(\mathbf{D})\widetilde{\boldsymbol{\eta}}_0+\varepsilon\widetilde{\Lambda}^\varepsilon S_\varepsilon\widetilde{\boldsymbol{\eta}}_0]
+\zeta (Q_0^\varepsilon \boldsymbol{\varphi}_\varepsilon ,\boldsymbol{\eta}_\varepsilon )_{L_2(\mathcal{O})}.
\end{split}
\end{equation}
Denote the consecutive terms in the right-hand side of this identity by $\mathcal{I}_j(\varepsilon)$, $j=1,2,3,4$.

Lemma \ref{Lemma estimates u_D,eps} and Lemma \ref{lemma phi eps} imply the following estimate for
the term $\mathcal{I}_4(\varepsilon)$:
\begin{equation}
\label{3.26 from draft}
\vert \mathcal{I}_4(\varepsilon)\vert \leqslant C_{21}c(\phi)^2 \varepsilon \vert \zeta\vert ^{-1/2} \Vert\mathbf{F}\Vert _{L_2(\mathcal{O})}\Vert \boldsymbol{\Phi}\Vert _{L_2(\mathcal{O})};\quad C_{21}:=C_{14}\Vert Q_0\Vert _{L_\infty}\Vert Q_0^{-1}\Vert _{L_\infty}.
\end{equation}

To estimate $\mathcal{I}_1(\varepsilon)$, we apply \eqref{b mathfrak eps <= c2},
Theorem~\ref{Theorem Dirichlet H1}, and Lemma~\ref{lemma phi eps}:
\begin{equation*}
\begin{split}
\vert \mathcal{I}_1(\varepsilon )\vert
&\leqslant \mathfrak{c}_2 c(\phi)\left( C_{15} \varepsilon ^{1/2} \vert \zeta\vert ^{-1/4}+(C_{14}+C_{16})\varepsilon\right)
\Vert \mathbf{F}\Vert _{L_2(\mathcal{O})}
\\
&\times
\left(C_5 c(\phi)^2\varepsilon ^{1/2}\vert \zeta\vert ^{-1/4}+C_6 c(\phi)^4\varepsilon\right)
 \Vert \boldsymbol{\Phi}\Vert _{L_2(\mathcal{O})}.
\end{split}
\end{equation*}
Hence,
\begin{equation}
\label{3.27 from draft}
\vert \mathcal{I}_1(\varepsilon)\vert
\leqslant
c(\phi)^5\left({\gamma}_1 \varepsilon \vert\zeta\vert ^{-1/2} +{\gamma}_2\varepsilon ^2\right)\Vert \mathbf{F}\Vert _{L_2(\mathcal{O})}\Vert \boldsymbol{\Phi}\Vert _{L_2(\mathcal{O})},
\end{equation}
where ${\gamma}_1:=\mathfrak{c}_2(C_5(C_{14}+C_{15}+C_{16})+C_6C_{15})$,  ${\gamma}_2:=\mathfrak{c}_2\left(C_5(C_{14}+C_{16})+C_6(C_{14}+C_{15}+C_{16})\right)$.

Next, we have
\begin{equation}
\label{*.1}
\mathcal{I}_2(\varepsilon)=-\mathfrak{b}_{N,\varepsilon}[\boldsymbol{\varphi}_\varepsilon ,\boldsymbol{\eta}_0]
=-\mathfrak{b}_{N,\varepsilon} [\boldsymbol{\varphi}_\varepsilon ,S_\varepsilon \widetilde{\boldsymbol{\eta}}_0]
-\mathfrak{b}_{N,\varepsilon} [\boldsymbol{\varphi}_\varepsilon , \boldsymbol{\eta}_0-S_\varepsilon\widetilde{\boldsymbol{\eta}}_0 ].
\end{equation}
From Proposition~\ref{Proposition S__eps - I} and estimate \eqref{tilde u_0 in H2} for $\widetilde{\boldsymbol{\eta}}_0$
it follows that
\begin{equation*}
\Vert \boldsymbol{\eta}_0-S_\varepsilon \widetilde{\boldsymbol{\eta}}_0\Vert _{H^1(\mathcal{O})}
\leqslant \Vert \widetilde{\boldsymbol{\eta}}_0 - S_\varepsilon \widetilde{\boldsymbol{\eta}}_0\Vert _{H^1(\mathbb{R}^d)}
\leqslant
\varepsilon r_1\Vert \widetilde{\boldsymbol{\eta}}_0\Vert _{H^2(\mathbb{R}^d)}
\leqslant c(\phi) \varepsilon r_1 k_3 \Vert \boldsymbol{\Phi}\Vert _{L_2(\mathcal{O})}.
\end{equation*}
Combining this with \eqref{b mathfrak eps <= c2} and Lemma~\ref{lemma phi eps}, we obtain
\begin{equation}
\label{*.1a}
\vert \mathfrak{b}_{N,\varepsilon} [\boldsymbol{\varphi}_\varepsilon , \boldsymbol{\eta}_0-S_\varepsilon\widetilde{\boldsymbol{\eta}}_0 ]\vert
\leqslant
c(\phi)^2\left(\gamma _3 \varepsilon \vert \zeta\vert ^{-1/2} +\gamma _4\varepsilon ^2\right)
\Vert \mathbf{F}\Vert _{L_2(\mathcal{O})}\Vert \boldsymbol{\Phi}\Vert _{L_2(\mathcal{O})},
\end{equation}
where $\gamma _3 :=\mathfrak{c}_2 r_1k_3C_{15}$ and $\gamma _4 :=\mathfrak{c}_2 r_1 k_3(C_{14}+C_{15}+C_{16})$.

Let us estimate the first term in the right-hand side of~\eqref{*.1}. According to~\eqref{b mathfrak eps =},
\begin{equation}
\label{I2(eps),=sum I2j}
\begin{split}
\vert \mathfrak{b}_{N,\varepsilon} [\boldsymbol{\varphi}_\varepsilon ,S_\varepsilon \widetilde{\boldsymbol{\eta}}_0]\vert
&
\leqslant
\left\vert \int _\mathcal{O}\langle g^\varepsilon b(\mathbf{D})\boldsymbol{\varphi}_\varepsilon ,b(\mathbf{D})S_\varepsilon\widetilde{\boldsymbol{\eta}}_0\rangle\,d\mathbf{x}\right\vert
\\
&+\sum _{j=1}^d \int _\mathcal{O}\left(\vert \langle a_j^\varepsilon D_j\boldsymbol{\varphi}_\varepsilon ,S_\varepsilon\widetilde{\boldsymbol{\eta}}_0\rangle\vert +\vert\langle (a_j^\varepsilon)^*\boldsymbol{\varphi}_\varepsilon ,D_jS_\varepsilon\widetilde{\boldsymbol{\eta}}_0\rangle\vert \right)\,d\mathbf{x}
\\
&+\left\vert \int _\mathcal{O}\langle Q^\varepsilon \boldsymbol{\varphi}_\varepsilon ,S_\varepsilon\widetilde{\boldsymbol{\eta}}_0\rangle\,d\mathbf{x}\right\vert
+\lambda \left\vert \int _{\mathcal{O}}\langle Q_0^\varepsilon \boldsymbol{\varphi}_\varepsilon ,S_\varepsilon\widetilde{\boldsymbol{\eta}}_0\rangle\,d\mathbf{x}\right\vert
\\
&=:\sum _{k=1}^4 \mathcal{I}_2^{(k)}(\varepsilon).
\end{split}
\end{equation}
Since $\boldsymbol{\varphi}_\varepsilon$ is supported in $(\partial\mathcal{O})_\varepsilon$, all integrals in
\eqref{I2(eps),=sum I2j} are taken over $(\partial\mathcal{O})_\varepsilon \cap \mathcal{O}$. The term $\mathcal{I}_2^{(1)}(\varepsilon)$ is estimated with the help of Lemma~\ref{lemma ots int O_eps B_eps}, \eqref{S_eps <= 1}, and \eqref{<b^*b<}:
\begin{equation*}
\begin{split}
\mathcal{I}_2^{(1)}(\varepsilon )
&\leqslant
\Vert g\Vert _{L_\infty} \alpha _1^{1/2}\Vert \mathbf{D}\boldsymbol{\varphi}_\varepsilon \Vert _{L_2(\mathbb{R}^d)}
\biggl(\int _{(\partial\mathcal{O})_\varepsilon}\vert b(\mathbf{D})S_\varepsilon \widetilde{\boldsymbol{\eta}}_0\vert ^2 \,d\mathbf{x}\biggr) ^{1/2}
\\
&\leqslant
\Vert g\Vert _{L_\infty} \alpha _1^{1/2}\Vert \mathbf{D}\boldsymbol{\varphi}_\varepsilon \Vert _{L_2(\mathbb{R}^d)}
(\beta \varepsilon )^{1/2}\left( \Vert b(\mathbf{D}) \widetilde{\boldsymbol{\eta}}_0\Vert _{H^1(\mathbb{R}^d)}\Vert b(\mathbf{D})\widetilde{\boldsymbol{\eta}}_0\Vert _{L_2(\mathbb{R}^d)}\right)^{1/2}.
\end{split}
\end{equation*}
Applying \eqref{<b^*b<}, \eqref{tilde u_0 in H1} and \eqref{tilde u_0 in H2} for $\widetilde{\boldsymbol{\eta}}_0$, and \eqref{3.12 from draft},
we see that
\begin{equation}
\label{I_2^1<=}
\mathcal{I}_2^{(1)}(\varepsilon )
\leqslant
c(\phi)^2 (\gamma _5 \varepsilon \vert \zeta\vert ^{-1/2}+\gamma _6\varepsilon ^2)\Vert \mathbf{F}\Vert _{L_2(\mathcal{O})}\Vert \boldsymbol{\Phi}\Vert _{L_2(\mathcal{O})},
\end{equation}
where $\gamma _5 :=\beta^{1/2}\Vert g\Vert _{L_\infty} \alpha _1 (k_2k_3)^{1/2}(C_{15}+C_{16})$ and
$\gamma _6 :=\beta^{1/2}\Vert g\Vert _{L_\infty}\alpha _1(k_2k_3)^{1/2}C_{16}$.

The term $\mathcal{I}_2^{(2)}(\varepsilon)$ satisfies
\begin{equation}
\label{*.2}
\mathcal{I}_2^{(2)}(\varepsilon)
\leqslant
\sum _{j=1}^d \Vert D_j\boldsymbol{\varphi}_\varepsilon\Vert _{L_2(\mathbb{R}^d)}
\biggl(\int _{(\partial\mathcal{O})_\varepsilon }\vert (a_j^\varepsilon )^*S_\varepsilon \widetilde{\boldsymbol{\eta}}_0\vert ^2\,d\mathbf{x}\biggr) ^{1/2}
+\sum _{j=1}^d \Vert \boldsymbol{\varphi}_\varepsilon \Vert _{L_2(\mathbb{R}^d)}\Vert a_j^\varepsilon S_\varepsilon D_j \widetilde{\boldsymbol{\eta}}_0\Vert _{L_2(\mathbb{R}^d)}.
\end{equation}
By Lemma~\ref{Lemma 3.6 from Su15}, we have
\begin{equation*}
\int _{(\partial\mathcal{O})_\varepsilon }\vert (a_j^\varepsilon )^*S_\varepsilon \widetilde{\boldsymbol{\eta}}_0\vert ^2\,d\mathbf{x}
\leqslant
\beta _*\varepsilon \vert \Omega \vert ^{-1}
\Vert a_j\Vert ^2 _{L_2(\Omega)}
\Vert \widetilde{\boldsymbol{\eta}}_0\Vert _{H^1(\mathbb{R}^d)}\Vert \widetilde{\boldsymbol{\eta}}_0\Vert _{L_2(\mathbb{R}^d)}.
\end{equation*}
Combining this with \eqref{tilde u_0 in L2}, \eqref{tilde u_0 in H1} for $\widetilde{\boldsymbol{\eta}}_0$ and \eqref{3.12 from draft},
we obtain the following estimate for the first summand in the right-hand side of~\eqref{*.2}:
\begin{equation}
\label{*.2a}
\sum _{j=1}^d \Vert D_j\boldsymbol{\varphi}_\varepsilon\Vert _{L_2(\mathbb{R}^d)}
\biggl(\int _{(\partial\mathcal{O})_\varepsilon }\vert (a_j^\varepsilon )^*S_\varepsilon \widetilde{\boldsymbol{\eta}}_0\vert ^2\,d\mathbf{x}\biggr) ^{1/2}
\leqslant
\gamma _7 c(\phi)^2 \varepsilon\vert \zeta\vert ^{-1/2}\Vert \mathbf{F}\Vert _{L_2(\mathcal{O})}\Vert \boldsymbol{\Phi}\Vert _{L_2(\mathcal{O})},
\end{equation}
where $\gamma _7 := C_a \left( \beta _*\vert \Omega \vert ^{-1} k_1k_2\right)^{1/2}(C_{15}+C_{16})$.
The second summand in the right-hand side of~\eqref{*.2} is estimated by Proposition~\ref{Proposition f^eps S_eps},
\eqref{tilde u_0 in H1} for $\widetilde{\boldsymbol{\eta}}_0$, and \eqref{3.11 from draft}:
\begin{equation*}
\begin{split}
\sum _{j=1}^d \Vert \boldsymbol{\varphi}_\varepsilon \Vert _{L_2(\mathbb{R}^d)}\Vert a_j^\varepsilon S_\varepsilon D_j \widetilde{\boldsymbol{\eta}}_0\Vert _{L_2(\mathbb{R}^d)}
&\leqslant
\vert \Omega\vert ^{-1/2}\sum _{j=1}^d \Vert \boldsymbol{\varphi}_\varepsilon \Vert _{L_2(\mathbb{R}^d)}\Vert a_j\Vert _{L_2(\Omega)}\Vert D_j\widetilde{\boldsymbol{\eta}}_0\Vert _{L_2(\mathbb{R}^d)}
\\
&\leqslant
\gamma _8 c(\phi)^2 \varepsilon \vert \zeta\vert ^{-1/2}\Vert \mathbf{F}\Vert _{L_2(\mathcal{O})}\Vert \boldsymbol{\Phi}\Vert _{L_2(\mathcal{O})},
\end{split}
\end{equation*}
where $\gamma _8 := \vert \Omega \vert ^{-1/2} C_a C_{14} k_2$.
Together with \eqref{*.2} and \eqref{*.2a}, this implies
\begin{equation}
\label{I_2^2<=}
\mathcal{I}_2^{(2)}(\varepsilon)
\leqslant
(\gamma _7+\gamma _8) c(\phi)^2 \varepsilon \vert \zeta \vert ^{-1/2}\Vert \mathbf{F}\Vert _{L_2(\mathcal{O})}\Vert \boldsymbol{\Phi}\Vert _{L_2(\mathcal{O})}.
\end{equation}

We proceed to estimation of the term $\mathcal{I}_2^{(3)}(\varepsilon)$:
\begin{equation}
\label{*.star}
\mathcal{I}_2^{(3)}(\varepsilon)
\leqslant
\Vert \vert Q^\varepsilon \vert ^{1/2}\boldsymbol{\varphi}_\varepsilon\Vert _{L_2(\mathbb{R}^d)}
\biggl( \int _{(\partial \mathcal{O})_\varepsilon }\vert Q^\varepsilon \vert \vert S_\varepsilon \widetilde{\boldsymbol{\eta}}_0\vert ^2\,d\mathbf{x}\biggr) ^{1/2}.
\end{equation}
The first factor in the right-hand side of \eqref{*.star} is estimated by Lemma~\ref{Lemma a embedding thm} and condition~\eqref{Q condition}:
\begin{equation}
\label{*.star star}
\Vert \vert Q^\varepsilon \vert ^{1/2} \boldsymbol{\varphi}_\varepsilon \Vert _{L_2(\mathbb{R}^d)}
\leqslant
C(\check{q},\Omega)\Vert Q\Vert ^{1/2}_{L_s(\Omega)}\Vert \boldsymbol{\varphi}_\varepsilon \Vert _{H^1(\mathbb{R}^d)},
\end{equation}
where $\check{q}=\infty$ for $d=1$, $\check{q}=2s/(s-1)$ for $d\geqslant 2$.
The second factor in the right-hand side of~\eqref{*.star} is estimated with the help of Lemma~\ref{Lemma 3.6 from Su15}:
\begin{equation}
\label{*.star star star}
\int _{(\partial \mathcal{O})_\varepsilon }\vert Q^\varepsilon \vert \vert S_\varepsilon \widetilde{\boldsymbol{\eta}}_0\vert ^2\,d\mathbf{x}
\leqslant
\beta _*\varepsilon\vert \Omega\vert ^{-1}\Vert Q\Vert _{L_1(\Omega)}\Vert \widetilde{\boldsymbol{\eta}}_0\Vert _{H^1(\mathbb{R}^d)}\Vert \widetilde{\boldsymbol{\eta}}_0\Vert _{L_2(\mathbb{R}^d)}.
\end{equation}
Combining \eqref{tilde u_0 in L2} and \eqref{tilde u_0 in H1} for $\widetilde{\boldsymbol{\eta}}_0$,
\eqref{*.star}--\eqref{*.star star star}, and using Lemma~\ref{lemma phi eps}, we find
\begin{equation}
\label{I_2^3<=}
\mathcal{I}_2^{(3)}(\varepsilon)
\leqslant
\gamma _{9} c(\phi)^2 \varepsilon\vert \zeta\vert ^{-1/2}\Vert \mathbf{F}\Vert _{L_2(\mathcal{O})}\Vert \boldsymbol{\Phi}\Vert _{L_2(\mathcal{O})},
\end{equation}
where
$\gamma_{9}:=C(\check{q},\Omega)\Vert Q\Vert _{L_s(\Omega)}^{1/2} \Vert Q\Vert _{L_1(\Omega)}^{1/2}
\left(\beta _*\vert \Omega \vert ^{-1}k_1k_2\right)^{1/2}(C_{14}+C_{15}+C_{16})$.

Relations \eqref{S_eps <= 1}, \eqref{tilde u_0 in L2} for $\widetilde{\boldsymbol{\eta}}_0$, and \eqref{3.11 from draft}
imply the following estimate for the term $\mathcal{I}_2^{(4)}(\varepsilon)$:
\begin{equation}
\label{I_2^4<=}
\mathcal{I}_2^{(4)}(\varepsilon) \leqslant
\lambda \Vert Q_0\Vert _{L_\infty}\Vert \boldsymbol{\varphi}_\varepsilon \Vert _{L_2(\mathbb{R}^d)}
\Vert S_\varepsilon \widetilde{\boldsymbol{\eta}}_0\Vert _{L_2(\mathbb{R}^d)}
\leqslant
\gamma _{10}c(\phi)^2\varepsilon\vert \zeta\vert ^{-1/2}\Vert \mathbf{F}\Vert _{L_2(\mathcal{O})}\Vert \boldsymbol{\Phi}\Vert _{L_2(\mathcal{O})},
\end{equation}
where $\gamma _{10}:=\lambda \Vert Q_0\Vert _{L_\infty}C_{14}k_1$.

Thus, combining \eqref{*.1}--\eqref{I_2^1<=}, \eqref{I_2^2<=}, \eqref{I_2^3<=}, and \eqref{I_2^4<=}, we obtain
\begin{equation}
\label{I_2(eps)<=}
\mathcal{I}_2(\varepsilon) \leqslant
c(\phi)^2 \left( \widehat{\gamma} \varepsilon\vert \zeta\vert ^{-1/2} +\widetilde{\gamma} \varepsilon ^2\right)
\Vert \mathbf{F}\Vert _{L_2(\mathcal{O})}\Vert \boldsymbol{\Phi}\Vert _{L_2(\mathcal{O})},
\end{equation}
where $\widehat{\gamma}:=\gamma _3 +\gamma _5 +\gamma _7 +\gamma _8 +\gamma _{9}+\gamma _{10}$ and
$\widetilde{\gamma}:= \gamma _4 +\gamma _6$.

It remains to estimate $\mathcal{I}_3(\varepsilon)$:
\begin{equation}
\label{3.40 from draft}
\begin{split}
\vert \mathcal{I}_3(\varepsilon )\vert
&=\vert \mathfrak{b}_{N,\varepsilon} [\boldsymbol{\varphi}_\varepsilon ,\varepsilon \Lambda ^\varepsilon S_\varepsilon b(\mathbf{D})\widetilde{\boldsymbol{\eta}}_0+\varepsilon\widetilde{\Lambda}^\varepsilon
S_\varepsilon \widetilde{\boldsymbol{\eta}}_0]\vert
\\
&\leqslant
\left\vert \left(g^\varepsilon b(\mathbf{D})\boldsymbol{\varphi}_\varepsilon ,
(b(\mathbf{D})\Lambda)^\varepsilon S_\varepsilon b(\mathbf{D})\widetilde{\boldsymbol{\eta}}_0\right)_{L_2(\mathcal{O})}\right\vert
\\
&+
\left\vert \bigl(g^\varepsilon b(\mathbf{D})\boldsymbol{\varphi}_\varepsilon ,
(b(\mathbf{D})\widetilde{\Lambda})^\varepsilon S_\varepsilon \widetilde{\boldsymbol{\eta}}_0\bigr)_{L_2(\mathcal{O})}\right\vert
\\
&+
\biggl\vert \biggl(g^\varepsilon b(\mathbf{D})\boldsymbol{\varphi}_\varepsilon ,\varepsilon \sum _{l=1}^d b_l \Lambda ^\varepsilon S_\varepsilon b(\mathbf{D})D_l\widetilde{\boldsymbol{\eta}}_0\biggr) _{L_2(\mathcal{O})}\biggr\vert
\\
&+
\biggl\vert \biggl(g^\varepsilon b(\mathbf{D})\boldsymbol{\varphi}_\varepsilon ,\varepsilon \sum _{l=1}^d b_l \widetilde{\Lambda} ^\varepsilon S_\varepsilon D_l\widetilde{\boldsymbol{\eta}}_0\biggr) _{L_2(\mathcal{O})}\biggr\vert
\\
&+
\sum _{j=1}^d\left\vert \left( a_j^\varepsilon D_j\boldsymbol{\varphi}_\varepsilon ,
\varepsilon \Lambda ^\varepsilon S_\varepsilon b(\mathbf{D})\widetilde{\boldsymbol{\eta}}_0+\varepsilon\widetilde{\Lambda}^\varepsilon
S_\varepsilon \widetilde{\boldsymbol{\eta}}_0\right)_{L_2(\mathcal{O})}\right\vert
\\
&+
\sum _{j=1}^d\left\vert \left( (a_j^\varepsilon)^*\boldsymbol{\varphi}_\varepsilon,
(D_j\Lambda)^\varepsilon S_\varepsilon b(\mathbf{D})\widetilde{\boldsymbol{\eta}}_0
+(D_j\widetilde{\Lambda})^\varepsilon S_\varepsilon \widetilde{\boldsymbol{\eta}}_0\right)_{L_2(\mathcal{O})}\right\vert
\\
&+
\sum _{j=1}^d \left\vert\left( (a_j^\varepsilon )^*\boldsymbol{\varphi}_\varepsilon ,
\varepsilon\Lambda ^\varepsilon S_\varepsilon b(\mathbf{D})D_j\widetilde{\boldsymbol{\eta}}_0
+\varepsilon\widetilde{\Lambda}^\varepsilon S_\varepsilon D_j\widetilde{\boldsymbol{\eta}}_0\right) _{L_2(\mathcal{O})}\right\vert
\\
&+
\left\vert \left(Q^\varepsilon\boldsymbol{\varphi}_\varepsilon ,
\varepsilon \Lambda ^\varepsilon S_\varepsilon b(\mathbf{D})\widetilde{\boldsymbol{\eta}}_0
+\varepsilon\widetilde{\Lambda}^\varepsilon S_\varepsilon \widetilde{\boldsymbol{\eta}}_0\right) _{L_2(\mathcal{O})}\right\vert
\\
&+
\lambda\left\vert \left(Q_0^\varepsilon \boldsymbol{\varphi}_\varepsilon,
\varepsilon\Lambda ^\varepsilon S_\varepsilon b(\mathbf{D})\widetilde{\boldsymbol{\eta}}_0
+\varepsilon \widetilde{\Lambda}^\varepsilon S_\varepsilon \widetilde{\boldsymbol{\eta}}_0\right)_{L_2(\mathcal{O})}\right\vert .
\end{split}
\end{equation}
The consecutive terms in the right-hand side of \eqref{3.40 from draft} are denoted by $\mathcal{I}_3^{(j)}(\varepsilon)$, $j=1,\dots,9$.

Using \eqref{<b^*b<} and Lemma~\ref{Lemma 3.6 from Su15}, and taking into account that $\boldsymbol{\varphi}_\varepsilon$
is supported in $(\partial\mathcal{O})_\varepsilon$, we estimate the first term:
\begin{equation*}
\begin{split}
\mathcal{I}_3^{(1)}(\varepsilon)
&\leqslant \Vert g\Vert _{L_\infty}\alpha _1^{1/2}\Vert \mathbf{D}\boldsymbol{\varphi}_\varepsilon \Vert _{L_2(\mathbb{R}^d)}
\biggl(
\int _{(\partial\mathcal{O})_\varepsilon}\vert (b(\mathbf{D})\Lambda )^\varepsilon S_\varepsilon b(\mathbf{D})\widetilde{\boldsymbol{\eta}}_0\vert ^2\,d\mathbf{x}\biggr) ^{1/2}
\\
&\leqslant
\Vert g\Vert _{L_\infty}\alpha _1^{1/2}\Vert \mathbf{D}\boldsymbol{\varphi}_\varepsilon \Vert _{L_2(\mathbb{R}^d)}
(\beta _*\varepsilon )^{1/2}\vert \Omega\vert ^{-1/2}\Vert b(\mathbf{D})\Lambda\Vert _{L_2(\Omega)}\Vert b(\mathbf{D})\widetilde{\boldsymbol{\eta}}_0\Vert ^{1/2}_{H^1(\mathbb{R}^d)}\Vert b(\mathbf{D})\widetilde{\boldsymbol{\eta}}_0\Vert ^{1/2}_{L_2(\mathbb{R}^d)}.
\end{split}
\end{equation*}
Now we apply Lemma~\ref{lemma phi eps} and estimates \eqref{tilde u_0 in H1}, \eqref{tilde u_0 in H2} for $\widetilde{\boldsymbol{\eta}}_0$.
Taking \eqref{<b^*b<} and \eqref{b(D)Lambda<=} into account, we arrive at
\begin{equation}
\label{3.41 from draft}
\mathcal{I}_3^{(1)}(\varepsilon)
\leqslant c(\phi)^2({\gamma}_{11}\varepsilon\vert \zeta\vert ^{-1/2}+{\gamma}_{12}\varepsilon ^2)\Vert \mathbf{F}\Vert _{L_2(\mathcal{O})}\Vert \boldsymbol{\Phi}\Vert _{L_2(\mathcal{O})}.
\end{equation}
Here
${\gamma}_{11}:=\Vert g\Vert _{L_\infty}^{3/2}\Vert g^{-1}\Vert _{L_\infty}^{1/2} \alpha _1 (m \beta _* k_2k_3)^{1/2}(C_{15}+C_{16})$ and
${\gamma}_{12}:=\Vert g\Vert _{L_\infty}^{3/2}\Vert g^{-1}\Vert _{L_\infty}^{1/2} \alpha _1 (m\beta _* k_2k_3)^{1/2}C_{16}$.
In a similar way, using \eqref{b(D) tilde Lambda <=}, we obtain
\begin{equation}
\label{3.42 from draft}
\mathcal{I}_3^{(2)}(\varepsilon)
\leqslant {\gamma}_{13} c(\phi)^2 \varepsilon\vert\zeta\vert ^{-1/2}\Vert \mathbf{F}\Vert _{L_2(\mathcal{O})}\Vert \boldsymbol{\Phi}\Vert _{L_2(\mathcal{O})},
\end{equation}
where
${\gamma}_{13}:=\Vert g\Vert _{L_\infty}\Vert g^{-1}\Vert _{L_\infty}(\alpha _1 \beta _* n k_1 k_2)^{1/2}\vert\Omega\vert ^{-1/2}\alpha _0 ^{-1/2}C_a(C_{15}+C_{16})$.

To estimate $\mathcal{I}_3^{(3)}(\varepsilon)$, we apply \eqref{<b^*b<}, \eqref{b_l <=}, and \eqref{Lambda S_eps <=}:
\begin{equation*}
\mathcal{I}^{(3)}_3(\varepsilon)
\leqslant \varepsilon \Vert g\Vert _{L_\infty}\alpha _1 ^{3/2}d^{1/2}  M_1\Vert
\mathbf{D}\boldsymbol{\varphi}_\varepsilon\Vert _{L_2(\mathbb{R}^d)}\Vert \widetilde{\boldsymbol{\eta}}_0\Vert _{H^2(\mathbb{R}^d)}.
\end{equation*}
Together with \eqref{tilde u_0 in H2} for $\widetilde{\boldsymbol{\eta}}_0$ and Lemma \ref{lemma phi eps}, this implies
\begin{equation}
\label{3.43 from draft}
\mathcal{I}_3^{(3)}(\varepsilon)\leqslant
c(\phi)^2 \left({\gamma}_{14}\varepsilon \vert \zeta\vert ^{-1/2}+{\gamma}_{15}\varepsilon ^2\right)\Vert \mathbf{F}\Vert _{L_2(\mathcal{O})}\Vert\boldsymbol{\Phi}\Vert _{L_2(\mathcal{O})},
\end{equation}
where ${\gamma}_{14}:=\Vert g\Vert _{L_\infty}\alpha _1 ^{3/2} d^{1/2} M_1k_3C_{15}$ and ${\gamma}_{15}:=\Vert g\Vert _{L_\infty}\alpha _1 ^{3/2} d^{1/2} M_1 k_3(C_{15}+C_{16})$.

In a similar way, using \eqref{tilde Lambda S_eps <=}, we obtain
\begin{equation}
\label{3.44 from draft}
\mathcal{I}_3^{(4)}(\varepsilon)
\leqslant
{\gamma}_{16} c(\phi)^2 \varepsilon\vert\zeta\vert^{-1/2}\Vert \mathbf{F}\Vert _{L_2(\mathcal{O})}\Vert \boldsymbol{\Phi}\Vert _{L_2(\mathcal{O})},
\end{equation}
where
${\gamma}_{16}:=\Vert g\Vert _{L_\infty}d^{1/2}\alpha _1 \widetilde{M}_1 k_2(C_{15}+C_{16})$.

Now, we estimate the term $\mathcal{I}_3^{(5)}(\varepsilon)$:
\begin{equation}
\label{*.3}
\mathcal{I}_3^{(5)}(\varepsilon )
\leqslant \varepsilon \sum _{j=1}^d \Vert D_j \boldsymbol{\varphi}_\varepsilon \Vert _{L_2(\mathbb{R}^d)}
\left(\Vert (a_j^\varepsilon )^* \Lambda ^\varepsilon S_\varepsilon b(\mathbf{D})\widetilde{\boldsymbol{\eta}}_0\Vert _{L_2(\mathbb{R}^d)}
+ \Vert (a_j^\varepsilon)^*\widetilde{\Lambda}^\varepsilon S_\varepsilon \widetilde{\boldsymbol{\eta}}_0\Vert _{L_2(\mathbb{R}^d)}\right).
\end{equation}
By Proposition~\ref{Proposition f^eps S_eps},
\begin{equation}
\label{*.3aa}
\Vert (a_j^\varepsilon )^* \Lambda ^\varepsilon S_\varepsilon b(\mathbf{D})\widetilde{\boldsymbol{\eta}}_0\Vert _{L_2(\mathbb{R}^d)}
\leqslant
\vert \Omega\vert ^{-1/2}\Vert a_j^*\Lambda\Vert _{L_2(\Omega)}\Vert b(\mathbf{D})\widetilde{\boldsymbol{\eta}}_0\Vert _{L_2(\mathbb{R}^d)}.
\end{equation}
From the H\"older inequality and the Sobolev embedding theorem it follows that
\begin{equation}
\label{*.3a}
\Vert a_j^*\Lambda \Vert _{L_2(\Omega)}
\leqslant C(q,\Omega)\Vert a_j\Vert _{L_\rho (\Omega)}\Vert \Lambda\Vert _{H^1(\Omega)},
\end{equation}
where $q=\infty$ for $d=1$ and $q=2\rho /(\rho -2)$ for $d\geqslant 2$.
Similarly,
\begin{equation}
\label{*.3aaa}
\begin{split}
\Vert (a_j^\varepsilon)^*\widetilde{\Lambda}^\varepsilon S_\varepsilon \widetilde{\boldsymbol{\eta}}_0\Vert _{L_2(\mathbb{R}^d)}
\leqslant
\vert \Omega \vert ^{-1/2}C(q,\Omega)\Vert a_j\Vert _{L_\rho (\Omega)}\Vert \widetilde{\Lambda}\Vert _{H^1(\Omega)}\Vert \widetilde{\boldsymbol{\eta}}_0\Vert _{L_2(\mathbb{R}^d)}.
\end{split}
\end{equation}
From \eqref{*.3}--\eqref{*.3aaa} it follows that
\begin{equation}
\label{*.3b}
\mathcal{I}_3^{(5)}(\varepsilon)\leqslant \varepsilon \widehat{C}_a C(q,\Omega) \vert \Omega \vert ^{-1/2}
\Vert \mathbf{D} \boldsymbol{\varphi}_\varepsilon \Vert _{L_2(\mathbb{R}^d)}
\left(
\Vert \Lambda \Vert _{H^1(\Omega)}\Vert b(\mathbf{D})\widetilde{\boldsymbol{\eta}}_0\Vert _{L_2(\mathbb{R}^d)}
+\Vert \widetilde{\Lambda}\Vert _{H^1(\Omega)}\Vert \widetilde{\boldsymbol{\eta}}_0\Vert _{L_2(\mathbb{R}^d)}
\right).
\end{equation}
By \eqref{Lambda <=} and \eqref{DLambda<=},
\begin{equation}
\label{Lambda in H1 <=}
\vert \Omega \vert ^{-1/2}\Vert \Lambda \Vert _{H^1(\Omega)}\leqslant M_1+M_2.
\end{equation}
According to \eqref{tilde Lambda<=}, \eqref{D tilde Lambda}, \eqref{tilde Lambda S_eps <=}, and \eqref{tilde M2 =},
\begin{equation}
\label{tilde Lambda in H1 <=}
\vert \Omega \vert ^{-1/2}\Vert \widetilde{\Lambda} \Vert _{H^1(\Omega)}\leqslant \widetilde{M}_1+\widetilde{M}_2.
\end{equation}
Relations \eqref{<b^*b<}, \eqref{3.12 from draft}, \eqref{*.3b}--\eqref{tilde Lambda in H1 <=}, and inequalities
\eqref{tilde u_0 in L2}, \eqref{tilde u_0 in H1} for $\widetilde{\boldsymbol{\eta}}_0$ imply that
\begin{equation}
\label{I_3^5<=}
\mathcal{I}_3^{(5)}(\varepsilon)
\leqslant
\gamma _{17} c(\phi)^2 \varepsilon\vert \zeta\vert ^{-1/2}
\Vert \mathbf{F}\Vert _{L_2(\mathcal{O})}\Vert \boldsymbol{\Phi}\Vert _{L_2(\mathcal{O})}.
\end{equation}
Here
$\gamma_{17}:=\widehat{C}_aC(q,\Omega)(C_{15}+C_{16})
\left((M_1+M_2)\alpha _1^{1/2}k_2+(\widetilde{M}_1+\widetilde{M}_2)k_1\right)$.

We proceed to estimation of $\mathcal{I}_3^{(6)}(\varepsilon)$:
\begin{equation}
\label{*.4}
\begin{split}
\mathcal{I}_3^{(6)}(\varepsilon)
&\leqslant
\sum _{j=1}^d \Vert (a_j^\varepsilon )^*\boldsymbol{\varphi}_\varepsilon \Vert _{L_2(\mathbb{R}^d)}
\biggl(
\int _{(\partial\mathcal{O})_\varepsilon }\vert (D_j \Lambda )^\varepsilon S_\varepsilon b(\mathbf{D})\widetilde{\boldsymbol{\eta}}_0\vert ^2\,d\mathbf{x}\biggr) ^{1/2}
\\
&+\sum _{j=1}^d \Vert (a_j^\varepsilon )^*\boldsymbol{\varphi}_\varepsilon \Vert _{L_2(\mathbb{R}^d)}
\biggl(\int _{(\partial\mathcal{O})_\varepsilon }\vert (D_j \widetilde{\Lambda} )^\varepsilon S_\varepsilon \widetilde{\boldsymbol{\eta}}_0\vert ^2\,d\mathbf{x}\biggr) ^{1/2}.
\end{split}
\end{equation}
From Lemma~\ref{Lemma a embedding thm} it follows that
\begin{equation}
\label{*.5}
\Vert (a_j^\varepsilon )^*\boldsymbol{\varphi}_\varepsilon\Vert _{L_2(\mathbb{R}^d)}
\leqslant C(q,\Omega)\Vert a_j\Vert _{L_\rho (\Omega )}\Vert \boldsymbol{\varphi}_\varepsilon\Vert _{H^1(\mathbb{R}^d)},
\end{equation}
where $q=\infty$ for $d=1$, $q=2\rho/(\rho -2)$ for $d\geqslant 2$.
By \eqref{*.4}, \eqref{*.5}, and Lemma \ref{Lemma 3.6 from Su15}, we have
\begin{equation*}
\begin{split}
\mathcal{I}_3^{(6)}(\varepsilon )
&\leqslant
C(q,\Omega)\widehat{C}_a
\left(\beta _*\vert\Omega \vert ^{-1}\varepsilon\right)^{1/2}
\Vert \boldsymbol{\varphi}_\varepsilon\Vert _{H^1(\mathbb{R}^d)}
\\
&\times
\left(
\Vert \mathbf{D}\Lambda \Vert _{L_2(\Omega)}\Vert b(\mathbf{D})\widetilde{\boldsymbol{\eta}}_0\Vert _{H^1(\mathbb{R}^d)}^{1/2}\Vert b(\mathbf{D})\widetilde{\boldsymbol{\eta}}_0\Vert _{L_2(\mathbb{R}^d)}^{1/2}
+\Vert \mathbf{D}\widetilde{\Lambda}\Vert _{L_2(\Omega)}\Vert \widetilde{\boldsymbol{\eta}}_0\Vert _{H^1(\mathbb{R}^d)}^{1/2}\Vert \widetilde{\boldsymbol{\eta}}_0\Vert _{L_2(\mathbb{R}^d)}^{1/2}
\right).
\end{split}
\end{equation*}
Combining this with \eqref{<b^*b<}, \eqref{DLambda<=}, \eqref{D tilde Lambda}, \eqref{tilde M2 =}, inequalities
\eqref{tilde u_0 in L2}--\eqref{tilde u_0 in H2} for $\widetilde{\boldsymbol{\eta}}_0$, and Lemma~\ref{lemma phi eps}, we obtain
\begin{equation}
\label{I_3^6<=}
\mathcal{I}_3^{(6)}(\varepsilon)
\leqslant
c(\phi)^2(\gamma _{18}\varepsilon \vert \zeta \vert ^{-1/2}+\gamma _{19}\varepsilon ^2)\Vert \mathbf{F}\Vert _{L_2(\mathcal{O})}\Vert \boldsymbol{\Phi}\Vert _{L_2(\mathcal{O})},
\end{equation}
where
$\gamma _{18}:=(C_{14}+C_{15}+C_{16})C(q,\Omega)\widehat{C}_a (\beta _* k_2)^{1/2}
\bigl(M_2 (\alpha _1 k_3)^{1/2} +\widetilde{M}_2k_1^{1/2}\bigr)$ and
$\gamma _{19} :=C_{16}C(q,\Omega)\widehat{C}_a M_2 (\beta _* \alpha _1 k_2 k_3 )^{1/2}$.

The term $\mathcal{I}_3^{(7)}(\varepsilon)$ is estimated with the help of  \eqref{Lambda S_eps <=}, \eqref{tilde Lambda S_eps <=}, and  \eqref{*.5}:
\begin{equation*}
\begin{split}
\mathcal{I}_3^{(7)}(\varepsilon)
\leqslant
\varepsilon C(q,\Omega)\sum _{j=1}^d \Vert a_j\Vert _{L_\rho (\Omega)}\Vert \boldsymbol{\varphi}_\varepsilon \Vert _{H^1(\mathbb{R}^d)}
\left(
M_1\Vert b(\mathbf{D})D_j\widetilde{\boldsymbol{\eta}}_0\Vert _{L_2(\mathbb{R}^d)}
+\widetilde{M}_1\Vert D_j\widetilde{\boldsymbol{\eta}}_0\Vert _{L_2(\mathbb{R}^d)}\right).
\end{split}
\end{equation*}
Now, applying Lemma~\ref{lemma phi eps}, \eqref{<b^*b<}, and inequalities \eqref{tilde u_0 in H1}, \eqref{tilde u_0 in H2} for  $\widetilde{\boldsymbol{\eta}}_0$,  we arrive at
\begin{equation}
\label{I_3^7<=}
\mathcal{I}_3^{(7)}(\varepsilon)
\leqslant
c(\phi)^2(\gamma _{20}\varepsilon \vert \zeta\vert ^{-1/2}+\gamma _{21}\varepsilon ^2)\Vert \mathbf{F}\Vert _{L_2(\mathcal{O})}\Vert \boldsymbol{\Phi}\Vert _{L_2(\mathcal{O})},
\end{equation}
where
$\gamma _{20}:= \widehat{C}_aC(q,\Omega)\bigl(C_{15}M_1\alpha_1^{1/2}k_3+(C_{14}+C_{15}+C_{16})\widetilde{M}_1k_2\bigr)$ and
$\gamma _{21}:=\widehat{C}_aC(q,\Omega)M_1\alpha _1 ^{1/2}k_3(C_{14}+C_{15}+C_{16})$.
Let us estimate the term $\mathcal{I}_3^{(8)}(\varepsilon)$:
\begin{equation}
\label{*.7}
\mathcal{I}_3^{(8)}(\varepsilon)
\leqslant
\varepsilon \Vert \vert Q^\varepsilon \vert ^{1/2}\boldsymbol{\varphi}_\varepsilon\Vert _{L_2(\mathbb{R}^d)}
\left(
\Vert \vert Q^\varepsilon \vert ^{1/2}\Lambda ^\varepsilon S_\varepsilon b(\mathbf{D})\widetilde{\boldsymbol{\eta}}_0\Vert _{L_2(\mathbb{R}^d)}
+
\Vert \vert Q^\varepsilon \vert ^{1/2}\widetilde{\Lambda} ^\varepsilon S_\varepsilon \widetilde{\boldsymbol{\eta}}_0\Vert _{L_2(\mathbb{R}^d)}
\right).
\end{equation}
By Proposition~\ref{Proposition f^eps S_eps} and \eqref{<b^*b<}, we have
\begin{equation}
\label{*.7a}
\Vert \vert Q^\varepsilon \vert ^{1/2}\Lambda ^\varepsilon S_\varepsilon b(\mathbf{D})\widetilde{\boldsymbol{\eta}}_0\Vert _{L_2(\mathbb{R}^d)}
\leqslant
\alpha _1^{1/2}\vert \Omega\vert ^{-1/2}\Vert \vert Q\vert ^{1/2}\Lambda\Vert _{L_2(\Omega)}\Vert \widetilde{\boldsymbol{\eta}}_0\Vert _{H^1(\mathbb{R}^d)}.
\end{equation}
From the H\"older inequality and the Sobolev embedding theorem it follows that
\begin{equation}
\label{*.7b}
\Vert \vert Q\vert ^{1/2}\Lambda\Vert _{L_2(\Omega)}
\leqslant C(\check{q},\Omega)\Vert Q\Vert ^{1/2}_{L_s(\Omega)}\Vert \Lambda\Vert _{H^1(\Omega)},
\end{equation}
where $\check{q}=\infty$ for $d=1$ and $\check{q}=2s/(s-1)$ for $d\ge 2$. Similarly,
\begin{equation}
\label{*.7c}
\Vert \vert Q^\varepsilon \vert ^{1/2}\widetilde{\Lambda} ^\varepsilon S_\varepsilon \widetilde{\boldsymbol{\eta}}_0\Vert _{L_2(\mathbb{R}^d)}
\leqslant \vert\Omega\vert ^{-1/2} C(\check{q},\Omega)\Vert Q\Vert ^{1/2}_{L_s(\Omega)}\Vert \widetilde{\Lambda}\Vert _{H^1(\Omega)}\Vert \widetilde{\boldsymbol{\eta}}_0\Vert _{L_2(\mathbb{R}^d)}.
\end{equation}
Relations \eqref{*.star star}, \eqref{Lambda in H1 <=}, \eqref{tilde Lambda in H1 <=}, and  \eqref{*.7}--\eqref{*.7c}
imply  that
\begin{equation*}
\mathcal{I}_3^{(8)}(\varepsilon)
\leqslant
\varepsilon C(\check{q},\Omega)^2\Vert Q\Vert _{L_s(\Omega)}\Vert \boldsymbol{\varphi}_\varepsilon \Vert _{H^1(\mathbb{R}^d)}
\left(
\alpha _1^{1/2}(M_1+M_2)\Vert \widetilde{\boldsymbol{\eta}}_0\Vert _{H^1(\mathbb{R}^d)}+(\widetilde{M}_1+\widetilde{M}_2)\Vert \widetilde{\boldsymbol{\eta}}_0\Vert _{L_2(\mathbb{R}^d)}\right).
\end{equation*}
Together with estimates \eqref{tilde u_0 in L2}, \eqref{tilde u_0 in H1} for $\widetilde{\boldsymbol{\eta}}_0$ and Lemma~\ref{lemma phi eps},
this yields
\begin{equation}
\label{I_3^8<=}
\mathcal{I}_3^{(8)}(\varepsilon)
\leqslant
\gamma _{22} c(\phi)^2 \varepsilon\vert \zeta\vert ^{-1/2}\Vert \mathbf{F}\Vert _{L_2(\mathcal{O})}\Vert \boldsymbol{\Phi}\Vert _{L_2(\mathcal{O})},
\end{equation}
where
$\gamma _{22}: = C(\check{q},\Omega)^2\Vert Q\Vert _{L_s(\Omega)}(C_{14}+C_{15}+C_{16})
\bigl((M_1+M_2)\alpha _1^{1/2}k_2+(\widetilde{M}_1+\widetilde{M}_2)k_1\bigr)$.

The term $\mathcal{I}_3^{(9)}(\varepsilon)$ is estimated by using \eqref{<b^*b<}, \eqref{Lambda S_eps <=}, \eqref{tilde Lambda S_eps <=}, inequalities \eqref{tilde u_0 in L2}, \eqref{tilde u_0 in H1} for $\widetilde{\boldsymbol{\eta}}_0$, and Lemma~\ref{lemma phi eps}.
We arrive at
\begin{equation}
\label{3.49 from draft}
\mathcal{I}_3^{(9)}(\varepsilon)
\leqslant
{\gamma}_{23} c(\phi)^2 \varepsilon\vert\zeta\vert^{-1/2}\Vert \mathbf{F}\Vert _{L_2(\mathcal{O})}\Vert \boldsymbol{\Phi}\Vert _{L_2(\mathcal{O})},
\end{equation}
where ${\gamma}_{23}:=\lambda \Vert Q_0\Vert _{L_\infty}C_{14}\left(M_1\alpha _1^{1/2}k_2+\widetilde{M}_1k_1\right)$.

Finally, relations \eqref{3.40 from draft}--\eqref{3.44 from draft}, \eqref{I_3^5<=}, \eqref{I_3^6<=}, \eqref{I_3^7<=}, \eqref{I_3^8<=},
and \eqref{3.49 from draft} imply
\begin{equation}
\label{3.50 from draft}
|\mathcal{I}_3(\varepsilon)| \leqslant c(\phi)^2
\left(\widehat{\gamma}' \varepsilon \vert \zeta\vert ^{-1/2} +\widetilde{\gamma}' \varepsilon ^2\right)
\Vert \mathbf{F}\Vert _{L_2(\mathcal{O})}\Vert \boldsymbol{\Phi}\Vert _{L_2(\mathcal{O})},
\end{equation}
where $\widehat{\gamma}':= \gamma_{11}+\gamma_{13}+\gamma_{14}+\gamma_{16}+\gamma_{17}+\gamma_{18}+\gamma _{20}+\gamma _{22}+\gamma _{23}$ and $\widetilde{\gamma}' := \gamma _{12}+\gamma _{15}+\gamma _{19}+\gamma _{21}$.

Thus, we have estimated all terms in the right-hand side of \eqref{3.25 from draft}. From \eqref{3.25 from draft}--\eqref{3.27 from draft}, \eqref{I_2(eps)<=}, and \eqref{3.50 from draft} it follows that
\begin{equation*}
\left\vert (\mathbf{w}_\varepsilon -\boldsymbol{\varphi}_\varepsilon ,\boldsymbol{\Phi})_{L_2(\mathcal{O})}\right\vert
\leqslant c(\phi)^5(\gamma _*\varepsilon\vert\zeta\vert ^{-1/2}+\gamma _{**}\varepsilon ^2 )\Vert \mathbf{F}\Vert _{L_2(\mathcal{O})}\Vert \boldsymbol{\Phi}\Vert _{L_2(\mathcal{O})}, \quad \boldsymbol{\Phi} \in L_2(\mathcal{O};\mathbb{C}^n).
\end{equation*}
Here $\gamma _*:=C_{21} + \gamma _1+\widehat{\gamma} + \widehat{\gamma}'$ and
 $\gamma_{**}:=\gamma _2+\widetilde{\gamma} + \widetilde{\gamma}'$.
  Hence,
  \begin{equation*}
  \Vert \mathbf{w}_\varepsilon -\boldsymbol{\varphi}_\varepsilon\Vert _{L_2(\mathcal{O})}
  \leqslant
  c(\phi)^5(\gamma _*\varepsilon\vert\zeta\vert ^{-1/2}+\gamma _{**}\varepsilon ^2 )\Vert \mathbf{F}\Vert _{L_2(\mathcal{O})}.
  \end{equation*}
Together with \eqref{3.11 from draft}, this yields \eqref{3.23 from draft} with the constants
$C_{19}:=\gamma _*+C_{14}$ and $C_{20}:=\gamma _{**}$.
\end{proof}

\subsection{Completion of the proof of Theorem~\ref{Theorem Dirichlet L2}}

From \eqref{2.35 from draft} and \eqref{3.23 from draft} it follows that
\begin{equation}
\label{raznost res with eps ^2}
\Vert \mathbf{u}_\varepsilon -\mathbf{u}_0\Vert _{L_2(\mathcal{O})}
\leqslant C_{22}c(\phi)^5(\varepsilon\vert\zeta\vert ^{-1/2}+\varepsilon ^2)\Vert \mathbf{F}\Vert _{L_2(\mathcal{O})},
\end{equation}
where $C_{22}:=\max\lbrace C_{11}+C_{19};C_{20}\rbrace$.
In order to deduce \eqref{Th L2}, we also need the following rough estimate:
\begin{equation}
\label{raznost res grubo}
\Vert (B_{D,\varepsilon}-\zeta Q_0^\varepsilon)^{-1}-(B_D^0-\zeta\overline{Q_0})^{-1}\Vert _{L_2(\mathcal{O})\rightarrow L_2(\mathcal{O})}
\leqslant 2\Vert Q_0 ^{-1}\Vert _{L_\infty} c(\phi)\vert \zeta\vert ^{-1}
\end{equation}
for any $\zeta\in\mathbb{C}\setminus\mathbb{R}_+$ and $0<\varepsilon\leqslant 1$, which follows from
\eqref{2.10a} and \eqref{lemma hom probl 1}.
For $\vert\zeta\vert \leqslant \varepsilon ^{-2}$ we use \eqref{raznost res with eps ^2} and note that
$\varepsilon ^2\leqslant \varepsilon \vert \zeta\vert ^{-1/2}$. For $\vert \zeta \vert >\varepsilon ^{-2}$ we apply
\eqref{raznost res grubo} and take into account that $\vert \zeta\vert ^{-1}<\varepsilon \vert \zeta\vert ^{-1/2}$.
This implies \eqref{Th L2} with $C_4:=2\max\lbrace \Vert Q_0^{-1}\Vert _{L_\infty} ;C_{22}\rbrace $.
\qed

\section{Special cases}
\label{Chapter Special cases}

\subsection{Removal of the smoothing operator $S_\varepsilon$ in the corrector}

It turns out that the smoothing operator $S_\varepsilon$ in the corrector can be removed under some additional assumptions on the matrix-valued
functions $\Lambda(\mathbf{x})$ and $\widetilde{\Lambda}(\mathbf{x})$.

\begin{condition}
\label{Condition Lambda in L infty}
Suppose that the $\Gamma$-periodic solution $\Lambda (\mathbf{x})$ of problem~\eqref{Lambda problem} is bounded, i.~e., $\Lambda\in L_\infty (\mathbb{R}^d)$.
\end{condition}

\begin{condition}
\label{Condition tilde Lambda in Lp}
Suppose that the $\Gamma$-periodic solution $\widetilde{\Lambda}(\mathbf{x})$ of problem \eqref{tildeLambda_problem} is such that
$\widetilde{\Lambda}\in L_p(\Omega)$, where $p=2$ for $d=1$, $p>2$ for $d=2$, and $p=d$ for $d\geqslant 3$.
\end{condition}

Some cases where Conditions~\ref{Condition Lambda in L infty} and~\ref{Condition tilde Lambda in Lp}
are fulfilled were distinguished in \cite[Lemma~8.7]{BSu06} and \cite[Proposition~8.11]{SuAA}, respectively.

\begin{proposition}[\cite{BSu06}]
\label{Proposition Lambda in L infty <=}
Suppose that at least one of the following assumptions is satisfied:
$1^\circ )$ $d\leqslant 2${\rm ;}
$2^\circ )$ $d\geqslant 1$, and the operator $A_\varepsilon$ is of the form $A_\varepsilon =\mathbf{D}^* g^\varepsilon (\mathbf{x})\mathbf{D}$, where the matrix $g(\mathbf{x})$ has real entries{\rm ;}
$3^\circ )$ the dimension $d$ is arbitrary, and $g^0=\underline{g}$, i.~e., relations \eqref{underline-g} are valid.
Then Condition~\textnormal{\ref{Condition Lambda in L infty}} is fulfilled.
\end{proposition}

\begin{proposition}[\cite{SuAA}]
 Suppose that at least one of the following assumptions is satisfied:
$1^\circ )$ $d\leqslant 4${\rm ;}
$2^\circ )$ the dimension $d$ is arbitrary, and the operator $A_\varepsilon$ is of the form $A_\varepsilon =\mathbf{D}^* g^\varepsilon (\mathbf{x})\mathbf{D}$, where the matrix $g(\mathbf{x})$ has real entries.
Then  Condition~\textnormal{\ref{Condition tilde Lambda in Lp}} is fulfilled.
\end{proposition}

\begin{remark}
\label{Remark scalar problem}
If $A_\varepsilon =\mathbf{D}^* g^\varepsilon (\mathbf{x})\mathbf{D}$, where $g(\mathbf{x})$ is symmetric matrix with real entries, from   \textnormal{\cite[Chapter {\rm III}, Theorem~{\rm 13.1}]{LaU}} it follows that $\Lambda \in L_\infty$ and
$\widetilde{\Lambda} \in L_\infty$. So, Conditions \textnormal{\ref{Condition Lambda in L infty}} and \textnormal{\ref{Condition tilde Lambda in Lp}} are fulfilled. Moreover, the norm $\Vert\Lambda\Vert _{L_\infty}$ does not exceed a constant depending  on $d$, $\Vert g\Vert _{L_\infty}$,
$\Vert g^{-1}\Vert _{L_\infty}$, and $\Omega$, while the norm $\Vert \widetilde{\Lambda}\Vert _{L_\infty}$ is controlled in terms of $d$, $\rho$, $\Vert g\Vert _{L_\infty}$, $\Vert g^{-1}\Vert _{L_\infty}$, $\Vert a_j\Vert _{L_\rho (\Omega)}$, $j=1,\dots ,d$, and $\Omega$.
\end{remark}

In this subsection, our goal is to prove the following theorem.

\begin{theorem}
\label{Theorem no S-eps}
Suppose that the assumptions of Theorem~\textnormal{\ref{Theorem Dirichlet H1}} are satisfied.
Suppose also that Conditions~\textnormal{\ref{Condition Lambda in L infty}} and~\textnormal{\ref{Condition tilde Lambda in Lp}} hold.
Denote
\begin{align}
\label{K_D^0}
K_D^0(\varepsilon ;\zeta) :=& (\Lambda ^\varepsilon  b(\mathbf{D})+ \widetilde{\Lambda}^\varepsilon  )(B_D^0-\zeta \overline{Q_0})^{-1},
\\
\label{G3(eps;zeta)}
G^0_D(\varepsilon ;\zeta ):=& \widetilde{g}^\varepsilon  b(\mathbf{D})(B_D^0-\zeta\overline{Q_0})^{-1}
+g^\varepsilon \bigl(b(\mathbf{D})\widetilde{\Lambda}\bigr)^\varepsilon (B_D^0-\zeta\overline{Q_0})^{-1}.
\end{align}
Then for $\zeta \in \mathbb{C}\setminus \mathbb{R}_+$, $\vert \zeta\vert\geqslant 1$, and $0<\varepsilon\leqslant\varepsilon_1$ we have
\begin{align}
\label{Th no S-eps 3}
\begin{split}
\Vert  &(B_{D,\varepsilon}-\zeta Q_0^\varepsilon )^{-1} - (B_D^0-\zeta \overline{Q_0})^{-1}
- \eps K_D^0(\varepsilon ;\zeta) \Vert _{L_2(\mathcal{O})\rightarrow H^1(\mathcal{O})}
\leqslant C_5 c(\phi)^2\varepsilon ^{1/2}\vert \zeta \vert ^{-1/4}+C_{23}c(\phi)^4\varepsilon ,
\end{split}
\\
\label{Th no S-eps fluxes 3}
\begin{split}
\Vert& g^\varepsilon b(\mathbf{D})(B_{D,\varepsilon}-\zeta Q_0^\varepsilon)^{-1}- G^0_D(\varepsilon ;\zeta)\Vert _{L_2(\mathcal{O})\rightarrow L_2(\mathcal{O})}
\leqslant \widetilde{C}_5 c(\phi)^2\varepsilon ^{1/2}\vert \zeta\vert ^{-1/4}+\widetilde{C}_{23}c(\phi)^4\varepsilon .
\end{split}
\end{align}
The constants $C_5$ and $\widetilde{C}_{5}$ are as in Theorem~\textnormal{\ref{Theorem Dirichlet H1}}.
The constants $C_{23}$ and $\widetilde{C}_{23}$ depend only on the initial data \eqref{problem data}, the domain~$\mathcal{O}$,
and also on $p$ and the norms $\Vert \Lambda\Vert _{L_\infty}$, $\Vert \widetilde{\Lambda}\Vert _{L_p(\Omega)}$.
\end{theorem}

The continuity of the operators \eqref{K_D^0} and \eqref{G3(eps;zeta)}
under the assumptions of Theorem~\ref{Theorem no S-eps} follows from Lemmas \ref{Lemma Lambda in L infty}, \ref{Lemma Lambda in Lp H1->L2}, and \ref{Lemma DLambda, Lambda in Lp}.

To prove Theorem~\ref{Theorem no S-eps}, we need the following lemmas.
Their proofs are similar to the proofs of Lemmas~8.7 and~8.8 from~\cite{MSu15}.

\begin{lemma}
\label{Lemma Lambda (S-I)}
Suppose that Condition~\textnormal{\ref{Condition Lambda in L infty}} is satisfied.
Let $S_\varepsilon$ be the Steklov smoothing operator given by~\eqref{S_eps}. Then for~$0<\varepsilon\leqslant 1$ we have
\begin{equation}
\label{lemma Lambda (S-I)}
\Vert [\Lambda ^\varepsilon ]b(\mathbf{D})(S_\varepsilon -I)\Vert _{H^2(\mathbb{R}^d)\rightarrow H^1(\mathbb{R}^d)}\leqslant\mathfrak{C}_\Lambda .
\end{equation}
The constant~$\mathfrak{C}_\Lambda$ depends only on $m$, $d$, $\alpha _0$, $\alpha _1$, $\Vert g\Vert _{L_\infty}$, $\Vert g^{-1}\Vert _{L_\infty}$, the parameters of the lattice~$\Gamma$, and the norm~$\Vert \Lambda\Vert _{L_\infty}$.
\end{lemma}

\begin{proof}
Let $\boldsymbol{\Phi}\in H^2(\mathbb{R}^d;\mathbb{C}^n)$. By \eqref{S_eps <= 1}, \eqref{<b^*b<}, and Condition~\ref{Condition Lambda in L infty},
\begin{equation}
\label{1 lemma Lambda (S-I)}
\Vert \Lambda ^\varepsilon b(\mathbf{D})(S_\varepsilon -I)\boldsymbol{\Phi}\Vert _{L_2(\mathbb{R}^d)}
\leqslant 2\alpha _1^{1/2}\Vert \Lambda \Vert _{L_\infty}\Vert \mathbf{D}\boldsymbol{\Phi}\Vert _{L_2(\mathbb{R}^d)}.
\end{equation}
Clearly,
\begin{equation*}
\begin{split}
\Vert \mathbf{D}\left(\Lambda ^\varepsilon b(\mathbf{D})(S_\varepsilon -I)\boldsymbol{\Phi}\right)\Vert ^2 _{L_2}
\leqslant
2\varepsilon ^{-2}\Vert (\mathbf{D}\Lambda )^\varepsilon (S_\varepsilon -I)b(\mathbf{D})\boldsymbol{\Phi}\Vert ^2 _{L_2}
+2\Vert \Lambda \Vert ^2_{L_\infty}\Vert (S_\varepsilon -I)b(\mathbf{D})\mathbf{D}\boldsymbol{\Phi}\Vert ^2 _{L_2}.
\end{split}
\end{equation*}
By Lemma~\ref{Lemma Lambda in L infty}, this yields
\begin{equation*}
\begin{split}
\Vert \mathbf{D}\left(\Lambda ^\varepsilon b(\mathbf{D})(S_\varepsilon -I)\boldsymbol{\Phi}\right)\Vert ^2 _{L_2(\mathbb{R}^d)}
&\leqslant
2\beta _1 \varepsilon ^{-2}\Vert (S_\varepsilon -I)b(\mathbf{D})\boldsymbol{\Phi}\Vert ^2 _{L_2(\mathbb{R}^d)}
\\
&+ 2\Vert \Lambda \Vert ^2_{L_\infty}(\beta _2+1)\Vert (S_\varepsilon -I)b(\mathbf{D})\mathbf{D}\boldsymbol{\Phi}\Vert ^2_{L_2(\mathbb{R}^d)}.
\end{split}
\end{equation*}
Applying \eqref{S_eps <= 1}, \eqref{<b^*b<}, and Proposition~\ref{Proposition S__eps - I}, we find
\begin{equation}
\label{2 lemma Lambda (S-I)}
\begin{split}
\Vert \mathbf{D}\left(\Lambda ^\varepsilon b(\mathbf{D})(S_\varepsilon -I)\boldsymbol{\Phi}\right)\Vert ^2 _{L_2(\mathbb{R}^d)}
&\leqslant \alpha _1 \left(2\beta _1 r_1^2+8\Vert \Lambda \Vert ^2_{L_\infty}(\beta _2+1)\right)\Vert \mathbf{D}^2\boldsymbol{\Phi}\Vert _{L_2(\mathbb{R}^d)}^2.
\end{split}
\end{equation}
Finally, relations~\eqref{1 lemma Lambda (S-I)} and~\eqref{2 lemma Lambda (S-I)} imply~\eqref{lemma Lambda (S-I)} with
$\mathfrak{C}_\Lambda ^2 :=\alpha _1\left(2\beta _1 r_1^2 +8\Vert \Lambda \Vert ^2 _{L_\infty}(\beta _2+1)\right)$.
\end{proof}

\begin{lemma}
\label{Lemma tilde Lambda(S-I)}
Suppose that Condition~\textnormal{\ref{Condition tilde Lambda in Lp}} is satisfied. Let $S_\varepsilon$ be the Steklov smoothing operator
given by~\eqref{S_eps}. Then for $0<\varepsilon\leqslant 1$ we have
\begin{equation}
\label{7.10a}
\Vert [\widetilde{\Lambda}^\varepsilon ](S_\varepsilon -I)\Vert _{H^2(\mathbb{R}^d)\rightarrow H^1(\mathbb{R}^d)}\leqslant \mathfrak{C}_{\widetilde{\Lambda}}.
\end{equation}
The constant $\mathfrak{C}_{\widetilde{\Lambda}}$ depends only on $n$, $d$, $\alpha _0$, $\alpha _1$, $\rho$, $\Vert g\Vert _{L_\infty}$, $\Vert g^{-1}\Vert _{L_\infty}$, îthe norms $\Vert a_j\Vert _{L_\rho (\Omega)}$, $j=1,\dots,d$, and also on $p$, the norm
$\Vert \widetilde{\Lambda}\Vert _{L_p(\Omega)}$, and the parameters of the lattice $\Gamma$.
\end{lemma}

\begin{proof}
Let $\boldsymbol{\Phi}\in H^2(\mathbb{R}^d;\mathbb{C}^n)$.
From \eqref{S_eps <= 1}, Lemma~\ref{Lemma Lambda in Lp H1->L2}, and Condition~\ref{Condition tilde Lambda in Lp}
it follows that
\begin{equation}
\label{1 lemma tilde Lambda (S-I)}
\Vert \widetilde{\Lambda}^\varepsilon (S_\varepsilon -I)\boldsymbol{\Phi}\Vert _{L_2(\mathbb{R}^d)}
\leqslant 2 C(\widehat{q},\Omega)\Vert \widetilde{\Lambda}\Vert _{L_p(\Omega)}\Vert \boldsymbol{\Phi}\Vert _{H^1(\mathbb{R}^d)}.
\end{equation}
Consider the derivatives:
$\partial _j \bigl(\widetilde{\Lambda}^\varepsilon (S_\varepsilon -I)\boldsymbol{\Phi}\bigr)
=\varepsilon ^{-1}(\partial _j \widetilde{\Lambda})^\varepsilon (S_\varepsilon -I)\boldsymbol{\Phi}
+\widetilde{\Lambda}^\varepsilon (S_\varepsilon -I)\partial _j \boldsymbol{\Phi}$.
Together with Lemmas~\ref{Lemma Lambda in Lp H1->L2} and \ref{Lemma DLambda, Lambda in Lp}, this yields
\begin{equation*}
\begin{split}
\Vert \mathbf{D}\bigl(\widetilde{\Lambda}^\varepsilon (S_\varepsilon -I)\boldsymbol{\Phi}\bigr)\Vert ^2 _{L_2(\mathbb{R}^d)}
&\leqslant 2\widetilde{\beta}_1\varepsilon ^{-2}\Vert (S_\varepsilon -I)\boldsymbol{\Phi} \Vert ^2 _{H^1(\mathbb{R}^d)}
\\
&+2(\widetilde{\beta}_2+1)\Vert \widetilde{\Lambda}\Vert ^2_{L_p(\Omega)}C(\widehat{q},\Omega)^2\Vert \mathbf{D}(S_\varepsilon -I)\boldsymbol{\Phi}\Vert ^2_{H^1(\mathbb{R}^d)}.
\end{split}
\end{equation*}
Combining this with \eqref{S_eps <= 1} and Proposition~\ref{Proposition S__eps - I}, we obtain
\begin{equation}
\label{2 lemma tilde Lambda (S-I)}
\Vert \mathbf{D}\bigl(\widetilde{\Lambda}^\varepsilon (S_\varepsilon -I)\boldsymbol{\Phi}\bigr)\Vert ^2 _{L_2(\mathbb{R}^d)}
\leqslant
\bigl(
2\widetilde{\beta}_1r_1^2+8(\widetilde{\beta}_2+1)\Vert \widetilde{\Lambda}\Vert ^2 _{L_p(\Omega)}C(\widehat{q},\Omega)^2\bigr)
\Vert \mathbf{D}\boldsymbol{\Phi}\Vert ^2 _{H^1(\mathbb{R}^d)}.
\end{equation}
Now, \eqref{1 lemma tilde Lambda (S-I)} and \eqref{2 lemma tilde Lambda (S-I)} imply \eqref{7.10a} with
$\mathfrak{C}_{\widetilde{\Lambda}}^2 :=2\widetilde{\beta}_1r_1^2+(8 \widetilde{\beta}_2 + 12)
C(\widehat{q},\Omega)^2\Vert \widetilde{\Lambda}\Vert ^2_{L_p(\Omega)}$.
\end{proof}

\subsection{Proof of Theorem~\ref{Theorem no S-eps}}

Under Condition~\ref{Condition Lambda in L infty},  by
\eqref{lemma hom probl 3}, \eqref{PO}, and Lemma~\ref{Lemma Lambda (S-I)}, we have
\begin{equation}
\label{d-vo Th 7/6 1}
\begin{split}
\varepsilon  \Vert [\Lambda ^\varepsilon ] b(\mathbf{D}) (S_\varepsilon -I) P_\mathcal{O}(B_D^0-\zeta\overline{Q_0})^{-1}\Vert _{L_2(\mathcal{O})\rightarrow H^1(\mathcal{O})}
\leqslant  \mathfrak{C}_\Lambda C_\mathcal{O}^{(2)}\mathcal{C}_2 c(\phi) \varepsilon.
\end{split}
\end{equation}
Similarly, under Condition~\ref{Condition tilde Lambda in Lp}, from \eqref{lemma hom probl 3}, \eqref{PO}, and Lemma~\ref{Lemma tilde Lambda(S-I)}
it follows that
\begin{equation}
\label{d-vo Th 7/6 2}
\varepsilon\Vert [\widetilde{\Lambda}^\varepsilon](S_\varepsilon -I)P_\mathcal{O}(B_D^0-\zeta\overline{Q_0})^{-1}\Vert _{L_2(\mathcal{O})\rightarrow H^1(\mathcal{O})}
\leqslant \mathfrak{C}_{\widetilde{\Lambda}}C_\mathcal{O}^{(2)}\mathcal{C}_2c(\phi) \varepsilon.
\end{equation}
Relations~\eqref{Th L2->H1}, \eqref{d-vo Th 7/6 1}, and \eqref{d-vo Th 7/6 2} imply estimate~\eqref{Th no S-eps 3} with $C_{23}:=C_6+(\mathfrak{C}_\Lambda+\mathfrak{C}_{\widetilde{\Lambda}})C_\mathcal{O}^{(2)}\mathcal{C}_2$.

Let us check~\eqref{Th no S-eps fluxes 3}. By analogy with~\eqref{3.19 from draft},  from \eqref{Th no S-eps 3} it follows that
\begin{equation}
\label{d-vo Th 7/6 3}
\begin{split}
\Vert &g^\varepsilon b(\mathbf{D})(B_{D,\varepsilon}-\zeta Q_0^\varepsilon )^{-1}-g^\varepsilon b(\mathbf{D})
\bigl(I+\varepsilon (\Lambda ^\varepsilon b(\mathbf{D})+ \widetilde{\Lambda}^\varepsilon )
\bigr)
(B_D^0-\zeta\overline{Q_0})^{-1}\Vert _{L_2(\mathcal{O}) \rightarrow L_2(\mathcal{O})}
\\
&\leqslant (d\alpha _1)^{1/2}\Vert g\Vert _{L_\infty}
\bigl(
C_5 c(\phi )^2\varepsilon ^{1/2}\vert \zeta\vert ^{-1/4}+C_{23}c(\phi)^4\varepsilon
\bigr).
\end{split}
\end{equation}
Next, by analogy with \eqref{3.20 from draft},
\begin{equation}
\label{d-vo Th 7/6 4}
\begin{split}
\varepsilon & g^\varepsilon b(\mathbf{D})
\bigl(\Lambda ^\varepsilon b(\mathbf{D})+\widetilde{\Lambda}^\varepsilon \bigr)(B_D^0-\zeta\overline{Q_0})^{-1}
\\
&=
g^\varepsilon \bigl(b(\mathbf{D})\Lambda\bigr)^\varepsilon b(\mathbf{D})(B_D^0-\zeta \overline{Q_0})^{-1}
+
g^\varepsilon \bigl( b(\mathbf{D})\widetilde{\Lambda}\bigr)^\varepsilon (B_D^0-\zeta\overline{Q_0})^{-1}
\\
&+
\varepsilon\sum _{l=1}^d g^\varepsilon b_l \bigl(\Lambda ^\varepsilon b(\mathbf{D})D_l+\widetilde{\Lambda}^\varepsilon D_l\bigr)
(B_D^0-\zeta\overline{Q_0})^{-1}.
\end{split}
\end{equation}
Using \eqref{b_l <=}, \eqref{lemma hom probl 3}, and
Condition~\ref{Condition Lambda in L infty}, we obtain
\begin{equation}
\label{d-vo Th 7/6 5}
\begin{split}
\varepsilon & \sum _{l=1}^d \Vert g^\varepsilon b_l \Lambda ^\varepsilon b(\mathbf{D})D_l(B_D^0-\zeta\overline{Q_0})^{-1}\Vert _{L_2(\mathcal{O})\rightarrow L_2(\mathcal{O})}
\\
&\leqslant
\varepsilon\alpha _1 d\Vert g\Vert _{L_\infty}\Vert \Lambda\Vert _{L_\infty}\Vert \mathbf{D}^2(B_D^0-\zeta\overline{Q_0})^{-1}\Vert _{L_2(\mathcal{O})\rightarrow L_2(\mathcal{O})}
\leqslant
\varepsilon\alpha _1 d \Vert g\Vert _{L_\infty}\Vert \Lambda\Vert _{L_\infty}\mathcal{C}_2 c(\phi).
\end{split}
\end{equation}
Next, from \eqref{b_l <=}, \eqref{lemma hom probl 3}, \eqref{PO},  Lemma~\ref{Lemma Lambda in Lp H1->L2}, and Condition~\ref{Condition tilde Lambda in Lp} it follows that
\begin{equation}
\begin{split}
\label{d-vo Th 7/6 5a}
\varepsilon&\sum _{l=1}^d \Vert g^\varepsilon b_l \widetilde{\Lambda}^\varepsilon  D_l (B_D^0-\zeta\overline{Q_0})^{-1}\Vert _{L_2(\mathcal{O})\rightarrow L_2(\mathcal{O})}
\\
&\leqslant
\varepsilon (d\alpha _1)^{1/2}\Vert g\Vert _{L_\infty}\Vert \widetilde{\Lambda}^\varepsilon \mathbf{D}(B_D^0-\zeta\overline{Q_0})^{-1}\Vert _{L_2(\mathcal{O})\rightarrow L_2(\mathcal{O})}\\
&\leqslant \varepsilon (d\alpha _1)^{1/2}\Vert g\Vert _{L_\infty}\Vert \widetilde{\Lambda}^\varepsilon P_\mathcal{O}\mathbf{D}(B_D^0-\zeta\overline{Q_0})^{-1}\Vert _{L_2(\mathcal{O})\rightarrow L_2(\mathbb{R}^d)}
\\
&\leqslant \varepsilon (d\alpha _1)^{1/2}\Vert g\Vert _{L_\infty}\Vert \widetilde{\Lambda}\Vert _{L_p(\Omega)}C(\widehat{q},\Omega)C_\mathcal{O}^{(1)}\Vert \mathbf{D}(B_D^0-\zeta\overline{Q_0})^{-1}\Vert _{L_2(\mathcal{O})\rightarrow H^1(\mathcal{O})}\\
&\leqslant \varepsilon (d\alpha _1)^{1/2}\Vert g\Vert _{L_\infty}\Vert \widetilde{\Lambda}\Vert _{L_p(\Omega)}C(\widehat{q},\Omega)C_\mathcal{O}^{(1)}\mathcal{C}_2 c(\phi).
\end{split}
\end{equation}
Together with \eqref{d-vo Th 7/6 5}, this shows that the third term in the right-hand side of~\eqref{d-vo Th 7/6 4} does not exceed
$\widehat{C}_{23}c(\phi) \varepsilon$, where
$\widehat{C}_{23}:=(d\alpha _1)^{1/2}\mathcal{C}_2\Vert g\Vert _{L_\infty}\bigl( (d\alpha _1)^{1/2}\Vert \Lambda \Vert _{L_\infty}+C(\widehat{q},\Omega)C_{\mathcal{O}}^{(1)}\Vert \widetilde{\Lambda}\Vert _{L_p(\Omega)} \bigr)$.
Combining this with \eqref{d-vo Th 7/6 3} and \eqref{d-vo Th 7/6 4}, we arrive at estimate \eqref{Th no S-eps fluxes 3} with the constant $\widetilde{C}_{23}:=(d\alpha_1)^{1/2}\Vert g\Vert _{L_\infty}C_{23}+\widehat{C}_{23}$. \qed

\begin{remark}
If only Condition~\textnormal{\ref{Condition Lambda in L infty} (}respectively, Condition~\textnormal{\ref{Condition tilde Lambda in Lp})} is satisfied, then
the smoothing operator $S_\varepsilon$ can be removed only in the term of the corrector containing $\Lambda^\varepsilon$
\textnormal{(}respectively, $\widetilde{\Lambda}^\varepsilon$\textnormal{)}.
\end{remark}

\subsection{The case where the corrector is equal to zero}

Suppose that $g^0=\overline{g}$, i.~e., relations \eqref{overline-g} are satisfied.
Then the $\Gamma$-periodic solution of problem~\eqref{Lambda problem} is equal to zero: $\Lambda (\mathbf{x})=0$.
In addition, suppose that
\begin{equation}
\label{sum Dj aj =0}
\sum _{j=1}^d D_j a_j(\mathbf{x})^* =0.
\end{equation}
Then the $\Gamma$-periodic solution of problem~\eqref{tildeLambda_problem} is also equal to zero: $\widetilde{\Lambda}(\mathbf{x})=0$.
Hence, ${\mathbf v}_\eps = {\mathbf u}_0$ (see \eqref{v_eps}, \eqref{v_eps=}).
The solution of problem \eqref{w_eps problem} is equal to zero: ${\mathbf w}_\varepsilon =0$.
Theorem~\ref{Theorem with Diriclet corrector} implies the following result.

\begin{proposition}
\label{Proposition K=0}
Suppose that the assumptions of Theorem~\textnormal{\ref{Theorem Dirichlet L2}} are satisfied.
Suppose that relations~\eqref{overline-g} and~\eqref{sum Dj aj =0} hold.
Then for $\zeta\in\mathbb{C}\setminus\mathbb{R}_+$, $\vert\zeta\vert\geqslant 1$, and $0<\varepsilon\leqslant 1$ we have
\begin{equation*}
\Vert (B_{D,\varepsilon}-\zeta Q_0^\varepsilon )^{-1}-(B_D^0-\zeta \overline{Q_0})^{-1}\Vert _{L_2(\mathcal{O})\rightarrow H^1(\mathcal{O})}
\leqslant C_7 c(\phi)^4\varepsilon.
\end{equation*}
\end{proposition}

\subsection{Special case} Suppose that $g^0=\underline{g}$, i.~e., relations~\eqref{underline-g} are satisfied.
Then, by Proposition~\ref{Proposition Lambda in L infty <=}($3^\circ$), Condition~\ref{Condition Lambda in L infty} is fulfilled.
Herewith, by \cite[Remark~3.5]{BSu05}, the matrix-valued function \eqref{tilde g} is constant and coincides with $g^0$, i.~e., $\widetilde{g}(\mathbf{x})=g^0=\underline{g}$. Hence, $\widetilde{g}^\varepsilon b(\mathbf{D})(B_D^0-\zeta\overline{Q_0})^{-1}=g^0b(\mathbf{D})(B_D^0-\zeta\overline{Q_0})^{-1}$.
In addition, suppose that relation~\eqref{sum Dj aj =0} holds. Then $\widetilde{\Lambda}(\mathbf{x})=0$, and Theorem~\ref{Theorem no S-eps}
implies the following result.

\begin{proposition}
\label{Proposition tilde Lambda =0 in Sec. 7}
Suppose that the assumptions of Theorem~\textnormal{\ref{Theorem Dirichlet L2}} are satisfied.
Suppose that relations~\eqref{underline-g} and~\eqref{sum Dj aj =0} hold.
Then for $\zeta\in\mathbb{C}\setminus\mathbb{R}_+$, $\vert\zeta\vert\geqslant 1$, and $0<\varepsilon\leqslant\varepsilon _1$ we have
\begin{equation*}
\Vert g^\varepsilon b(\mathbf{D})(B_{D,\varepsilon}-\zeta Q_0^\varepsilon )^{-1}-g^0b(\mathbf{D})(B_D^0-\zeta\overline{Q_0})^{-1}\Vert _{L_2(\mathcal{O})\rightarrow L_2(\mathcal{O})}
\leqslant \widetilde{C}_5 c(\phi)^2\varepsilon ^{1/2}\vert\zeta\vert ^{-1/4}+\widetilde{C}_{23}c(\phi)^4\varepsilon .
\end{equation*}
\end{proposition}

\section{Estimates in a strictly interior subdomain}
\label{Chapter Interior subdomain}

\subsection{General case}
Using Theorem~\ref{Theorem Dirichlet L2} and the results for homogenization problem in~$\mathbb{R}^d$,
it is possible to improve error estimates in $H^1(\mathcal{O}')$ for any strictly interior subdomain $\mathcal{O}'$ of $\mathcal{O}$.

The following result is proved by the same method as Theorem~7.1 from \cite{Su15}.

\begin{theorem}
\label{Theorem O'}
Suppose that the assumptions of Theorem~\textnormal{\ref{Theorem Dirichlet H1}} are satisfied.
Let $\mathcal{O}'$ be a strictly interior subdomain of the domain~$\mathcal{O}$. Denote $\delta :=\mathrm{dist}\,\lbrace \mathcal{O}';\partial \mathcal{O}\rbrace$. Then for $\zeta\in\mathbb{C}\setminus\mathbb{R}_+$, $\vert\zeta\vert\geqslant 1$, and $0<\varepsilon\leqslant\varepsilon _1$ we have
\begin{align}
\label{Th O'}
\Vert &\mathbf{u}_\varepsilon -\mathbf{v}_\varepsilon \Vert _{H^1(\mathcal{O}')}
\leqslant c(\phi)^6\varepsilon(C_{24}'\vert\zeta\vert ^{-1/2}\delta ^{-1}+C_{24}'')\Vert \mathbf{F}\Vert _{L_2(\mathcal{O})},
\\
\label{Th O' fluxes}
\Vert &\mathbf{p}_\varepsilon -\widetilde{g}^\varepsilon S_\varepsilon b(\mathbf{D})\widetilde{\mathbf{u}}_0- g^\varepsilon \bigl(b(\mathbf{D})\widetilde{\Lambda}\bigr)^\varepsilon S_\varepsilon \widetilde{\mathbf{u}}_0\Vert _{L_2(\mathcal{O}')}
\leqslant c(\phi)^6\varepsilon \bigl(\widetilde{C}_{24}'\vert\zeta\vert ^{-1/2}\delta ^{-1}+\widetilde{C}_{24}''\bigr) \Vert \mathbf{F}\Vert _{L_2(\mathcal{O})}.
\end{align}
The constants $C_{24}'$, $C_{24}''$, $\widetilde{C}_{24}'$, and $\widetilde{C}_{24}''$ depend only on the initial data~\eqref{problem data} and the domain~$\mathcal{O}$.
\end{theorem}

\begin{proof}
We fix a smooth cut-off function $\chi (\mathbf{x})$ such that
\begin{equation}
\label{hi properties}
\chi\in C_0^\infty (\mathcal{O});\quad 0\leqslant \chi (\mathbf{x})\leqslant 1;\quad \chi (\mathbf{x})=1 \ \mbox{for}\ \mathbf{x}\in \mathcal{O}';\quad \vert \nabla \chi (\mathbf{x})\vert \leqslant \kappa \delta ^{-1}.
\end{equation}
The constant $\kappa$ depends only on the dimension~$d$ and the domain~$\mathcal{O}$.
Let $\mathbf{u}_\varepsilon$ be the solution of problem~\eqref{Dirichlet problem}, and let $\widetilde{\mathbf{u}}_\varepsilon $ be the solution  of equation~\eqref{tilde u_eps problem}. Then
\begin{equation}
\label{tozd u eps - tilde u eps with eta}
\mathfrak{b}_{N,\varepsilon} [\mathbf{u}_\varepsilon -\widetilde{\mathbf{u}}_\varepsilon ,\boldsymbol{\eta}]-\zeta (Q_0^\varepsilon (\mathbf{u}_\varepsilon -\widetilde{\mathbf{u}}_\varepsilon ),\boldsymbol{\eta})_{L_2(\mathcal{O})}=0,\quad \boldsymbol{\eta}\in H^1_0(\mathcal{O};\mathbb{C}^n).
\end{equation}
We substitute $\boldsymbol{\eta}=\chi ^2 (\mathbf{u}_\varepsilon -\widetilde{\mathbf{u}}_\varepsilon )$ in \eqref{tozd u eps - tilde u eps with eta} and denote
\begin{equation}
\label{U(eps)=}
\mathfrak{U}(\varepsilon ):=\mathfrak{b}_{N,\varepsilon} [\chi (\mathbf{u}_\varepsilon -\widetilde{\mathbf{u}}_\varepsilon ),\chi (\mathbf{u}_\varepsilon -\widetilde{\mathbf{u}}_\varepsilon )]=\mathfrak{b}_{D,\varepsilon}[\chi (\mathbf{u}_\varepsilon -\widetilde{\mathbf{u}}_\varepsilon ),\chi (\mathbf{u}_\varepsilon -\widetilde{\mathbf{u}}_\varepsilon )].
\end{equation}
The corresponding identity can be written as
\begin{equation}
\label{tozd u eps - tilde u eps}
\begin{split}
\mathfrak{U}(\varepsilon)&-\zeta \left( Q_0^\varepsilon \chi (\mathbf{u}_\varepsilon -\widetilde{\mathbf{u}}_\varepsilon),\chi (\mathbf{u}_\varepsilon -\widetilde{\mathbf{u}}_\varepsilon)\right)_{L_2(\mathcal{O})}
\\
&=
2i\mathrm{Im}\,\left(g^\varepsilon \mathbf{z}_\varepsilon , b(\mathbf{D})\chi (\mathbf{u}_\varepsilon -\widetilde{\mathbf{u}}_\varepsilon)\right)_{L_2(\mathcal{O})}
+(g^\varepsilon\mathbf{z}_\varepsilon ,\mathbf{z}_\varepsilon )_{L_2(\mathcal{O})}
\\
&+2i\mathrm{Im}\,\sum _{j=1}^d\left( (D_j\chi)(\mathbf{u}_\varepsilon -\widetilde{\mathbf{u}}_\varepsilon ), (a_j^\varepsilon)^*\chi (\mathbf{u}_\varepsilon -\widetilde{\mathbf{u}}_\varepsilon)\right)_{L_2(\mathcal{O})}
,
\end{split}
\end{equation}
where $\mathbf{z}_\varepsilon :=\sum _{l=1}^d b_l (D_l \chi)(\mathbf{u}_\varepsilon -\widetilde{\mathbf{u}}_\varepsilon)$.
Denote the consecutive summands in the right-hand side of \eqref{tozd u eps - tilde u eps} as $i\mathfrak{I}_1(\varepsilon)$, $\mathfrak{I}_2(\varepsilon)$, and $i\mathfrak{I}_3(\varepsilon)$.
Let us estimate these terms. We can extend the function $\chi (\mathbf{u}_\varepsilon -\widetilde{\mathbf{u}}_\varepsilon )$
by zero to $\mathbb{R}^d \setminus \mathcal{O}$ and apply estimates in $\mathbb{R}^d$.
By \eqref{b_eps >=} and \eqref{U(eps)=},
\begin{equation}
\label{frak I 1<=}
\begin{split}
\vert \mathfrak{I}_1(\varepsilon )\vert &\leqslant 2\Vert g\Vert _{L_\infty}^{1/2}
\Vert \mathbf{z}_\varepsilon\Vert _{L_2(\mathcal{O})}
\Vert (g^\varepsilon)^{1/2} b(\mathbf{D})\chi (\mathbf{u}_\varepsilon -\widetilde{\mathbf{u}}_\varepsilon )\Vert  _{L_2(\mathcal{O})}
\leqslant 4\Vert g\Vert _{L_\infty}^{1/2} \Vert \mathbf{z}_\varepsilon\Vert _{L_2(\mathcal{O})} \mathfrak{U}(\varepsilon)^{1/2}.
\end{split}
\end{equation}
Obviously, $\mathfrak{I}_2(\varepsilon)\leqslant \Vert g\Vert _{L_\infty}\Vert \mathbf{z}_\varepsilon\Vert ^2 _{L_2(\mathcal{O})}$.
The norm of $\mathbf{z}_\varepsilon$ is estimated with the help of \eqref{b_l <=} and \eqref{hi properties}:
\begin{equation}
\label{z eps <=}
\Vert \mathbf{z}_\varepsilon\Vert _{L_2(\mathcal{O})}
\leqslant (d\alpha _1)^{1/2}\kappa \delta ^{-1}\Vert \mathbf{u}_\varepsilon -\widetilde{\mathbf{u}}_\varepsilon\Vert _{L_2(\mathcal{O})}.
\end{equation}
Hence,
\begin{equation}
\label{**8.10A}
\mathfrak{I}_2(\varepsilon)
\leqslant \gamma _{24}\delta^{-2}\Vert \mathbf{u}_\varepsilon -\widetilde{\mathbf{u}}_\varepsilon\Vert ^2_{L_2(\mathcal{O})};\quad\gamma _{24} :=d\alpha _1\kappa^2\Vert g\Vert _{L_\infty}.
\end{equation}
The term $\mathfrak{I}_3(\varepsilon)$ is estimated by Lemma~\ref{Lemma a embedding thm}, \eqref{H^1-norm <= BDeps^1/2}, \eqref{hi properties}, and \eqref{U(eps)=}:
\begin{equation}
\label{I_3 star <=}
\begin{split}
\vert \mathfrak{I}_3(\varepsilon)\vert & \leqslant 2\Vert (\mathbf{D}\chi )(\mathbf{u}_\varepsilon -\widetilde{\mathbf{u}}_\varepsilon )\Vert _{L_2(\mathcal{O})} \biggl( \sum _{j=1}^d \Vert (a_j^\varepsilon )^*\chi (\mathbf{u}_\varepsilon -\widetilde{\mathbf{u}}_\varepsilon )\Vert ^2 _{L_2(\mathcal{O})}\biggr)^{1/2}
\\
&\leqslant\gamma _{25}\delta ^{-1}\Vert \mathbf{u}_\varepsilon -\widetilde{\mathbf{u}}_\varepsilon \Vert _{L_2(\mathcal{O})}\mathfrak{U}(\varepsilon)^{1/2};
\quad \gamma _{25}:=2c_4 C(q,\Omega)\widehat{C}_a\kappa ,
\end{split}
\end{equation}
where $q=\infty$ for $d=1$ and $q=2\rho(\rho -2)^{-1}$ for $d\geqslant 2$.

Take the imaginary part in \eqref{tozd u eps - tilde u eps}. Then
\begin{equation*}
\mathrm{Im}\,\zeta\Vert (Q_0^\varepsilon)^{1/2}\chi (\mathbf{u}_\varepsilon -\widetilde{\mathbf{u}}_\varepsilon )\Vert ^2 _{L_2(\mathcal{O})}=-\mathfrak{I}_1(\varepsilon)-\mathfrak{I}_3(\varepsilon).
\end{equation*}
Therefore, relations \eqref{frak I 1<=}, \eqref{z eps <=}, and \eqref{I_3 star <=} imply that
\begin{equation}
\label{Im zeta || ||}
\vert \mathrm{Im}\,\zeta\vert \Vert (Q_0^\varepsilon)^{1/2}\chi (\mathbf{u}_\varepsilon -\widetilde{\mathbf{u}}_\varepsilon )\Vert ^2_{L_2(\mathcal{O})}
\leqslant
\gamma _{26}\delta ^{-1}\Vert \mathbf{u}_\varepsilon -\widetilde{\mathbf{u}}_\varepsilon\Vert _{L_2(\mathcal{O})}\mathfrak{U}(\varepsilon)^{1/2},
\end{equation}
where $\gamma_{26}: = 4 \Vert g\Vert _{L_\infty}^{1/2} (d\alpha _1)^{1/2} \kappa +\gamma_{25}$.
If $\mathrm{Re}\,\zeta\geqslant 0$ (and then $\mathrm{Im}\,\zeta\neq 0$), this yields
\begin{equation}
\label{ || || Re zeta >=0}
\begin{split}
 \Vert (Q_0^\varepsilon)^{1/2}\chi (\mathbf{u}_\varepsilon -\widetilde{\mathbf{u}}_\varepsilon )\Vert ^2_{L_2(\mathcal{O})}
\leqslant \gamma_{26} c(\phi)\vert \zeta\vert^{-1} \delta ^{-1}\Vert \mathbf{u}_\varepsilon -\widetilde{\mathbf{u}}_\varepsilon\Vert _{L_2(\mathcal{O})}\mathfrak{U}(\varepsilon)^{1/2}, \quad \mathrm{Re}\,\zeta\geqslant 0.
\end{split}
\end{equation}
If $\mathrm{Re}\,\zeta <0$, taking the real part in \eqref{tozd u eps - tilde u eps} and using~\eqref{**8.10A}, we have
\begin{equation}
\label{Re || ||}
\begin{split}
\vert\mathrm{Re}\,\zeta\vert\Vert (Q_0^\varepsilon)^{1/2}\chi (\mathbf{u}_\varepsilon -\widetilde{\mathbf{u}}_\varepsilon)\Vert ^2 _{L_2(\mathcal{O})}\leqslant
\mathfrak{I}_2(\varepsilon)
\leqslant \gamma _{24}\delta ^{-2}\Vert \mathbf{u}_\varepsilon -\widetilde{\mathbf{u}}_\varepsilon\Vert ^2_{L_2(\mathcal{O})},
\quad  \mathrm{Re}\,\zeta <0.
\end{split}
\end{equation}
Summimg up \eqref{Im zeta || ||} and \eqref{Re || ||}, we obtain
\begin{equation}
\label{|| || Re <0}
\begin{split}
\vert \zeta\vert \Vert (Q_0^\varepsilon)^{1/2}\chi (\mathbf{u}_\varepsilon -\widetilde{\mathbf{u}}_\varepsilon )\Vert ^2 _{L_2(\mathcal{O})}
\leqslant
\gamma_{26}\delta ^{-1}\Vert \mathbf{u}_\varepsilon -\widetilde{\mathbf{u}}_\varepsilon\Vert _{L_2(\mathcal{O})}\mathfrak{U}(\varepsilon)^{1/2}
+\gamma _{24} \delta ^{-2}\Vert \mathbf{u}_\varepsilon -\widetilde{\mathbf{u}}_\varepsilon\Vert ^2_{L_2(\mathcal{O})}
\end{split}
\end{equation}
for $\mathrm{Re}\,\zeta <0$.
As a result, \eqref{ || || Re zeta >=0} and \eqref{|| || Re <0} imply that
\begin{equation}
\label{|| || all zeta}
\begin{split}
\vert \zeta\vert \Vert (Q_0^\varepsilon )^{1/2}\chi (\mathbf{u}_\varepsilon -\widetilde{\mathbf{u}}_\varepsilon)\Vert ^2 _{L_2(\mathcal{O})}
&\leqslant \gamma_{26} c(\phi)\delta ^{-1}\Vert \mathbf{u}_\varepsilon -\widetilde{\mathbf{u}}_\varepsilon\Vert _{L_2(\mathcal{O})}\mathfrak{U}(\varepsilon)^{1/2}
+\gamma _{24}\delta ^{-2}\Vert \mathbf{u}_\varepsilon -\widetilde{\mathbf{u}}_\varepsilon\Vert ^2_{L_2(\mathcal{O})}
\end{split}
\end{equation}
for all $\zeta$ under consideration.
Taking the real part in \eqref{tozd u eps - tilde u eps} and using \eqref{**8.10A} and \eqref{|| || all zeta}, we obtain
\begin{equation*}
\begin{split}
\mathfrak{U}(\varepsilon)
\leqslant
\gamma_{26} c(\phi) \delta ^{-1}\Vert \mathbf{u}_\varepsilon -\widetilde{\mathbf{u}}_\varepsilon\Vert _{L_2(\mathcal{O})}\mathfrak{U}(\varepsilon)^{1/2} +2\gamma_{24}\delta ^{-2}\Vert \mathbf{u}_\varepsilon -\widetilde{\mathbf{u}}_\varepsilon \Vert ^2_{L_2(\mathcal{O})}.
\end{split}
\end{equation*}
Hence,
$\mathfrak{U}(\varepsilon)\leqslant \gamma _{27}^2  c(\phi)^2 \delta ^{-2}\Vert \mathbf{u}_\varepsilon -\widetilde{\mathbf{u}}_\varepsilon\Vert ^2_{L_2(\mathcal{O})}$, where $\gamma _{27}^2 :=\gamma _{26}^2+4\gamma _{24}$.
By \eqref{b_D,eps ots} and \eqref{U(eps)=}, we deduce
\begin{equation}
\label{D chi u-eps -}
\Vert \mathbf{D}\chi (\mathbf{u}_\varepsilon -\widetilde{\mathbf{u}}_\varepsilon)\Vert _{L_2(\mathcal{O})}
\leqslant \gamma _{27} c_*^{-1/2} c(\phi) \delta ^{-1}\Vert \mathbf{u}_\varepsilon -\widetilde{\mathbf{u}}_\varepsilon \Vert _{L_2(\mathcal{O})}.
\end{equation}

Estimates \eqref{Th L2 solutions} and \eqref{tilde U eps -tilde u 0} imply that
\begin{equation}
\label{u-eps -tilde u eps L2}
\Vert \mathbf{u}_\varepsilon -\widetilde{\mathbf{u}}_\varepsilon \Vert _{L_2(\mathcal{O})}
\leqslant \gamma _{28}c(\phi)^5\varepsilon\vert \zeta\vert ^{-1/2}\Vert \mathbf{F}\Vert _{L_2(\mathcal{O})},\quad 0<\varepsilon\leqslant \varepsilon _1,
\end{equation}
where $\gamma _{28}:=C_4+C_1C_{\widetilde{F}}$.
From \eqref{D chi u-eps -} and \eqref{u-eps -tilde u eps L2} it follows that
\begin{equation}
\label{8.18a}
\Vert\mathbf{D} \chi (\mathbf{u}_\varepsilon -\widetilde{\mathbf{u}}_\varepsilon)\Vert _{L_2(\mathcal{O})}
\leqslant \gamma _{29} c(\phi)^6 \varepsilon\vert \zeta\vert ^{-1/2} \delta ^{-1}\Vert \mathbf{F}\Vert _{L_2(\mathcal{O})}.
\end{equation}
Here $\gamma _{29} :=c_*^{-1/2}\gamma _{27}\gamma _{28}$. By~\eqref{u-eps -tilde u eps L2} and \eqref{8.18a},
\begin{equation}
\label{u-eps - tilde u-eps H1}
\Vert \mathbf{u}_\varepsilon -\widetilde{\mathbf{u}}_\varepsilon \Vert _{H^1(\mathcal{O}')}
\leqslant c(\phi)^6\varepsilon \vert \zeta\vert ^{-1/2}(\gamma _{29}\delta ^{-1}+\gamma _{28})\Vert \mathbf{F}\Vert _{L_2(\mathcal{O})}.
\end{equation}
Combining \eqref{v_eps=} and \eqref{tilde u_eps -v_eps <= in H1}, we find
\begin{equation}
\label{tilde u-eps - v-eps}
\Vert \widetilde{\mathbf{u}}_\varepsilon -\mathbf{v}_\varepsilon \Vert _{H^1(\mathcal{O}')}
\leqslant \Vert \widetilde{\mathbf{u}}_\varepsilon -\widetilde{\mathbf{v}}_\varepsilon \Vert _{H^1(\mathbb{R}^d)}
\leqslant \widetilde{C}_3 c(\phi)^3\varepsilon\Vert \mathbf{F}\Vert _{L_2(\mathcal{O})}.
\end{equation}
Now, relations \eqref{u-eps - tilde u-eps H1} and \eqref{tilde u-eps - v-eps} imply \eqref{Th O'} with
$C_{24}':=\gamma_{29}$ and $C_{24}'':=\gamma _{28}+\widetilde{C}_3$.

Let us prove \eqref{Th O' fluxes}. By \eqref{b_l <=} and \eqref{Th O'},
\begin{equation*}
\Vert \mathbf{p}_\varepsilon -g^\varepsilon b(\mathbf{D})\mathbf{v}_\varepsilon \Vert _{L_2(\mathcal{O}')}
\leqslant (d\alpha _1)^{1/2}\Vert g\Vert _{L_\infty} c(\phi)^6\varepsilon
(C_{24}'\vert \zeta\vert ^{-1/2}\delta ^{-1}+C_{24}'')\Vert \mathbf{F}\Vert _{L_2(\mathcal{O})}.
\end{equation*}
Combining this with \eqref{3.20 from draft}, \eqref{4-th chlen in tozd for fluxes}, and \eqref{3.22 from draft}, we deduce
 estimate \eqref{Th O' fluxes} with the constants
$\widetilde{C}_{24}' :=(d\alpha _1)^{1/2}\Vert g\Vert _{L_\infty}C_{24}'$ and
$\widetilde{C}_{24}'' :=(d\alpha _1)^{1/2}\Vert g\Vert _{L_\infty}C_{24}''+C_{17}+C_{18}$.
\end{proof}

\subsection{Removal of the smoothing operator in the corrector}
Provided that the matrix-valued functions $\Lambda (\mathbf{x})$ and $\widetilde{\Lambda}(\mathbf{x})$ are subject to Conditions~\ref{Condition Lambda in L infty} and \ref{Condition tilde Lambda in Lp}, respectively, the smoothing operator~$S_\varepsilon$ in the corrector can be removed.

\begin{theorem}
\label{Theorem O' no S-eps}
Suppose that the assumptions of Theorem~\textnormal{\ref{Theorem O'}} are satisfied.
Suppose also that  Conditions \textnormal{\ref{Condition Lambda in L infty}} and \textnormal{\ref{Condition tilde Lambda in Lp}} hold. Let
 $K_D^0(\varepsilon ;\zeta)$ and $G^0_D(\varepsilon ;\zeta)$ be given by~\eqref{K_D^0} and \eqref{G3(eps;zeta)}, respectively. Then for $\zeta\in\mathbb{C}\setminus\mathbb{R}_+$, $\vert \zeta\vert \geqslant 1$, and  $0<\varepsilon\leqslant\varepsilon _1$ we have
\begin{align}
\label{Th 8/2 3}
\begin{split}
\Vert &(B_{D,\varepsilon}-\zeta Q_0^\varepsilon )^{-1}- (B_D^0-\zeta \overline{Q_0})^{-1} -
\varepsilon  K_D^0(\varepsilon ;\zeta) \Vert _{L_2(\mathcal{O})\rightarrow H^1(\mathcal{O}')}
\leqslant c(\phi)^6\varepsilon (C_{24}'\vert \zeta\vert ^{-1/2}\delta ^{-1}+C_{25}),
\end{split}
\\
\label{Th 8/2 3 fluxes}
\begin{split}
\Vert &g^\varepsilon b(\mathbf{D})(B_{D,\varepsilon}-\zeta Q_0^\varepsilon )^{-1}-G^0_D(\varepsilon ;\zeta )\Vert _{L_2(\mathcal{O})\rightarrow L_2(\mathcal{O}')}
\leqslant c(\phi)^6\varepsilon (\widetilde{C}_{24}'\vert \zeta\vert ^{-1/2}\delta ^{-1}+\widetilde{C}_{25}).
\end{split}
\end{align}
The constants $C_{24}'$ and $\widetilde{C}_{24}'$ are as in Theorem~\textnormal{\ref{Theorem O'}}.
The constants $C_{25}$ and $\widetilde{C}_{25}$ depend only on the initial data \eqref{problem data},
the domain~$\mathcal{O}$, and also on $p$ and the norms $\Vert \Lambda\Vert _{L_\infty}$, $\Vert \widetilde{\Lambda}\Vert _{L_p(\Omega)}$.
\end{theorem}

\begin{proof}
Inequality \eqref{Th 8/2 3}
with $C_{25}:=C_{24}''+(\mathfrak{C}_\Lambda+\mathfrak{C}_{\widetilde{\Lambda}}) C_\mathcal{O}^{(2)}\mathcal{C}_2$
 is a consequence of \eqref{d-vo Th 7/6 1}, \eqref{d-vo Th 7/6 2}, and \eqref{Th O'}.

Let us check~\eqref{Th 8/2 3 fluxes}. Similarly to \eqref{3.19 from draft}, from \eqref{Th 8/2 3} it follows that
\begin{equation*}
\begin{split}
\Vert &g^\varepsilon b(\mathbf{D})(B_{D,\varepsilon}-\zeta Q_0^\varepsilon )^{-1}-g^\varepsilon b(\mathbf{D})(I+\varepsilon\Lambda ^\varepsilon b(\mathbf{D})+\varepsilon\widetilde{\Lambda}^\varepsilon)(B_D^0-\zeta\overline{Q_0})^{-1}\Vert _{L_2(\mathcal{O})\rightarrow L_2(\mathcal{O}')}
\\
&\leqslant
(d\alpha _1)^{1/2}\Vert g\Vert _{L_\infty} c(\phi)^6\varepsilon (C_{24}'\vert \zeta\vert ^{-1/2}\delta ^{-1}+C_{25}).
\end{split}
\end{equation*}
Together with \eqref{d-vo Th 7/6 4}--\eqref{d-vo Th 7/6 5a}, this yields estimate \eqref{Th 8/2 3 fluxes} with
   $\widetilde{C}_{25}: =(d\alpha _1)^{1/2}\Vert g\Vert _{L_\infty}C_{25} +\widehat{C}_{23}$.
\end{proof}

\section{``Another''\, approximation of the generalized resolvent}
\label{Section Another approximation}

In Theorems of Sections~\ref{Chapter 2 Dirichlet}, \ref{Chapter Special cases}, and~\ref{Chapter Interior subdomain}, it was assumed that
$\zeta\in\mathbb{C}\setminus\mathbb{R}_+$ and $\vert \zeta\vert \geqslant 1$.
In the present section, we obtain the results valid in a larger domain of the spectral parameter.

\goodbreak
\subsection{General case}

\begin{condition}\label{cond_eps_flat}
Let $0< \varepsilon_\flat \le 1$.  Let $c_\flat\geqslant 0$ be a common lower bound of the operators
$\widetilde{B}_{D,\varepsilon}=(f^\varepsilon )^*B_{D,\varepsilon}f^\varepsilon $ for any $0< \varepsilon \le \varepsilon_\flat$
and $\widetilde{B}_D^0=f_0B_D^0f_0$.
\end{condition}

\begin{theorem}
\label{Theorem Dr appr}
Suppose that $\mathcal{O}\subset\mathbb{R}^d$ is a bounded domain of class $C^{1,1}$.
Suppose that the number $\varepsilon_1$ is subject to Condition~\textnormal{\ref{condition varepsilon}}.
Let $0< \varepsilon_\flat \le \varepsilon_1$. Suppose that $c_\flat\geqslant 0$ is subject to Condition~\textnormal{\ref{cond_eps_flat}}.
Let $\zeta\in\mathbb{C}\setminus [c_\flat,\infty)$. Denote $\psi =\mathrm{arg}\,(\zeta -c_\flat)$, $0<\psi <2\pi$, and
\begin{equation}
\label{rho(zeta)}
\varrho _ \flat (\zeta) :=\begin{cases}
c(\psi)^2\vert \zeta -c_\flat\vert ^{-2}, &\vert \zeta -c_\flat\vert <1,\\
c(\psi)^2, &\vert \zeta -c_\flat\vert \geqslant 1.
\end{cases}
\end{equation}
Here $c(\psi)$ is defined by \eqref{c(phi)}.
Let $\mathbf{u}_\varepsilon$ be the solution of problem~\eqref{Dirichlet problem}, and let $\mathbf{u}_0$ be the solution of
problem~\eqref{Dirichlet homogenized problem}.
Let $K_D(\varepsilon ;\zeta)$ be given by \eqref{K_D(eps,zeta)}. Let $\mathbf{v}_\varepsilon$ be defined by~\eqref{v_eps}, \eqref{v_eps=}.
Then for $0<\varepsilon \leqslant \varepsilon _\flat$ we have
\begin{align}
\label{Th dr appr 1 solutions}
\Vert \mathbf{u}_\varepsilon -\mathbf{u}_0\Vert _{L_2(\mathcal{O})}
&\leqslant
C_{26} \varepsilon \varrho _ \flat (\zeta ) \Vert \mathbf{F}\Vert _{L_2(\mathcal{O})},
\\
\label{Th dr appr 2 solutions}
\Vert \mathbf{u}_\varepsilon-\mathbf{v}_\varepsilon\Vert _{H^1(\mathcal{O})}
&\leqslant
C_{27} \bigl(\varepsilon ^{1/2}\varrho _\flat (\zeta)^{1/2}+ \varepsilon
\vert 1+\zeta \vert ^{1/2} \varrho_\flat(\zeta) \bigr) \Vert \mathbf{F}\Vert _{L_2(\mathcal{O})}.
\end{align}
In operator terms,
\begin{align}
\label{Th dr appr 1}
\Vert &(B_{D,\varepsilon}-\zeta Q_0^\varepsilon )^{-1}-(B_D^0-\zeta \overline{Q_0})^{-1}\Vert _{L_2(\mathcal{O})\rightarrow L_2(\mathcal{O})}\leqslant C_{26} \varepsilon \varrho _ \flat (\zeta ),
\\
\label{Th dr appr 2}
\begin{split}
\Vert &(B_{D,\varepsilon}-\zeta Q_0^\varepsilon )^{-1}-(B_D^0-\zeta \overline{Q_0})^{-1}
-\varepsilon K_D(\varepsilon ;\zeta )\Vert _{L_2(\mathcal{O})\rightarrow H^1(\mathcal{O})}
\\
&\leqslant
C_{27} \bigl(\varepsilon ^{1/2}\varrho _\flat (\zeta)^{1/2}+ \varepsilon
\vert 1+\zeta \vert ^{1/2} \varrho_\flat(\zeta) \bigr).
\end{split}
\end{align}
Let $\widetilde{g}(\mathbf{x})$ be defined by~\eqref{tilde g}.
For $0<\varepsilon\leqslant\varepsilon _\flat$ the flux $\mathbf{p}_\varepsilon =g^\varepsilon b(\mathbf{D})\mathbf{u}_\varepsilon$ satisfies
\begin{equation}
\label{Th dr appr fluxes}
\begin{split}
\Vert &\mathbf{p}_\varepsilon -\widetilde{g}^\varepsilon S_\varepsilon b(\mathbf{D})\widetilde{\mathbf{u}}_0-g^\varepsilon \bigl(b(\mathbf{D})\widetilde{\Lambda}\bigr)^\varepsilon S_\varepsilon\widetilde{\mathbf{u}}_0
\Vert _{L_2(\mathcal{O})}
\leqslant \widetilde{C}_{27} \bigl(\varepsilon ^{1/2}\varrho _\flat (\zeta)^{1/2}+ \varepsilon
\vert 1+\zeta \vert ^{1/2} \varrho_\flat(\zeta) \bigr) \Vert \mathbf{F}\Vert _{L_2(\mathcal{O})}.
\end{split}
\end{equation}
The constants $C_{26}$, $C_{27}$, and $\widetilde{C}_{27}$ depend only on the initial data~\eqref{problem data} and the domain $\mathcal{O}$.
\end{theorem}

\begin{remark}
\textnormal{1)} Expression $c(\psi)^2\vert \zeta - c_\flat \vert ^{-2}$ in \eqref{rho(zeta)}
is inverse to the square of the distance from~$\zeta$ to $[c_\flat ,\infty)$. \textnormal{2)}
 By~\eqref{b D,eps >= H1-norm}, \eqref{Q_0=}, \eqref{b_D^0 ots 2}, and \eqref{f_0<=}, for any $\varepsilon_\flat\in (0,1]$
one can take $c_\flat$ equal to $4^{-1}\alpha _0\Vert g^{-1}\Vert ^{-1}_{L_\infty}\Vert Q_0\Vert _{L_\infty}^{-1} (\mathrm{diam}\,\mathcal{O})^{-2}$.
\textnormal{3)} Let $\lambda_1^0$ be the first eigenvalue of the operator $B_D^0$, and let $\nu >0$ be arbitrarily small number.
By Theorem \textnormal{\ref{Theorem Dirichlet L2} (}with $Q_0=I${\textnormal )}, the resolvent of $B_{D,\varepsilon}$ converges to the resolvent of $B_D^0$
in the $L_2$-operator norm. Hence, if $\varepsilon_\flat$ is sufficiently small, the number $\lambda_1^0 - \nu$
is a lower bound of the operator $B_{D,\eps}$ for any $0< \varepsilon \le \varepsilon_\flat$.
Then one can take $c_\flat = \| Q_0\|_{L_\infty}^{-1}(\lambda_1^0 - \nu)$.
 \textnormal{4)} It is easy to give the upper bound for $c_\flat$. By \eqref{b_D,eps ots} and
 \eqref{tilde b D,eps=}, we have $c_\flat \le c_3 \| Q_0^{-1}\|_{L_\infty} \mu_1^0$, where $\mu_1^0$ is the first eigenvalue of the operator
 $-\Delta +I$ with the Dirichlet condition on $\partial {\mathcal O}$. So, $c_\flat$ is controlled in terms of the data~\eqref{problem data}
 and the domain $\mathcal{O}$.
\end{remark}

\begin{remark}
Estimates \eqref{Th dr appr 1 solutions}--\eqref{Th dr appr fluxes} are useful for bounded values of $|\zeta|$ and small
$\varepsilon \varrho _\flat (\zeta)$. In this case, the value $\varepsilon ^{1/2}\varrho _\flat (\zeta)^{1/2}+ \varepsilon
\vert 1+\zeta \vert ^{1/2} \varrho_\flat(\zeta)$ is majorated by $C\varepsilon ^{1/2}\varrho _\flat (\zeta)^{1/2}$.
For large  $\vert\zeta\vert$  {\rm (}and $\phi$  separated from~$0$ and $2\pi${\rm )}  application of Theorems \textnormal{\ref{Theorem Dirichlet L2}} and \textnormal{\ref{Theorem Dirichlet H1}} is preferable.
    \end{remark}

We start with the following two lemmas.

\begin{lemma}
\label{Lemma Dr ots solutions}
Under Condition~\textnormal{\ref{cond_eps_flat}}, for $0< \varepsilon \leqslant \varepsilon_\flat$
and $\zeta\in\mathbb{C}\setminus [c_\flat,\infty)$ we have
\begin{align}
\label{****.1}
&\Vert(B_{D,\varepsilon }-\zeta Q_0^\varepsilon )^{-1}\Vert _{L_2(\mathcal{O})\rightarrow L_2(\mathcal{O})}
\leqslant \|f\|^2_{L_\infty} c(\psi) |\zeta - c_\flat|^{-1},
\\
\label{****.2}
&\Vert(B_{D,\varepsilon }-\zeta Q_0^\varepsilon )^{-1}\Vert _{L_2(\mathcal{O})\rightarrow H^1(\mathcal{O})}
\leqslant {\mathcal C}_3 (1+|\zeta|)^{-1/2} \varrho_\flat(\zeta)^{1/2},
\\
\label{****.3}
&\Vert(B_{D}^0-\zeta \overline{Q_0} )^{-1}\Vert _{L_2(\mathcal{O})\rightarrow L_2(\mathcal{O})}
\leqslant \|f\|^2_{L_\infty} c(\psi) |\zeta - c_\flat|^{-1},
\\
\label{****.4}
&\Vert(B_{D}^0 -\zeta \overline{Q_0})^{-1}\Vert _{L_2(\mathcal{O})\rightarrow H^1(\mathcal{O})}
\leqslant {\mathcal C}_3 (1+|\zeta|)^{-1/2} \varrho_\flat(\zeta)^{1/2},
\\
\label{****.5}
&\Vert(B_{D}^0 -\zeta \overline{Q_0})^{-1}\Vert _{L_2(\mathcal{O})\rightarrow H^2(\mathcal{O})}
\leqslant {\mathcal C}_4 \varrho_\flat(\zeta)^{1/2}.
\end{align}
Here
${\mathcal C}_3 := c_4 \|f\|_{L_\infty} (c_\flat +1)^{1/2}(c_\flat +2)^{1/2}$ and
${\mathcal C}_4 := \widehat{c} (c_\flat +2) \|f\|_{L_\infty}\|f^{-1}\|_{L_\infty}$.
\end{lemma}

\begin{proof}
Under our assumptions, the spectrum of the operator $\widetilde{B}_{D,\varepsilon}$ is contained in $[c_\flat,\infty)$. Hence,
$
\Vert ( \widetilde{B}_{D,\varepsilon} - \zeta I)^{-1} \Vert _{L_2(\mathcal{O})\rightarrow L_2(\mathcal{O})}\leqslant
c(\psi) |\zeta - c_\flat|^{-1}.
$
Together with \eqref{B deps and tilde B D,eps tozd resolvent}, this implies \eqref{****.1}.

Next, from \eqref{tilde b D,eps=} and \eqref{B deps and tilde B D,eps tozd resolvent} it follows that
\begin{equation*}
\begin{split}
\Vert B_{D,\varepsilon }^{1/2}(B_{D,\varepsilon}-\zeta Q_0^\varepsilon )^{-1}\Vert_{L_2 \rightarrow L_2}
=\Vert \widetilde{B}_{D,\varepsilon}^{1/2}(\widetilde{B}_{D,\varepsilon}-\zeta I)^{-1}(f^\varepsilon )^*\Vert_{L_2 \rightarrow L_2}
\leqslant \Vert f\Vert _{L_\infty}\sup \limits _{x\geqslant c_\flat}\frac{x^{1/2}}{\vert x-\zeta \vert}.
\end{split}
\end{equation*}
A calculation shows that
\begin{equation*}
\sup \limits _{x\geqslant c_\flat}\frac{x}{\vert x-\zeta \vert ^2}\leqslant
\begin{cases}
(c_\flat +1)c(\psi)^2\vert \zeta -c_\flat\vert ^{-2}, &\vert \zeta -c_\flat\vert <1,\\
(c_\flat +1)c(\psi)^2\vert \zeta -c_\flat\vert ^{-1}, &\vert \zeta -c_\flat\vert \geqslant 1.
\end{cases}
\end{equation*}
Note that  $\vert \zeta \vert +1 \leqslant 2+c_\flat$ for $\vert \zeta -c_\flat\vert <1$
and $(\vert \zeta \vert+1) \vert \zeta -c_\flat\vert ^{-1}\leqslant 2+c_\flat$ for $\vert \zeta -c_\flat\vert \geqslant 1$.
Therefore,
\begin{equation*}
(\vert \zeta \vert +1)^{1/2}\Vert B_{D,\varepsilon }^{1/2}(B_{D,\varepsilon}-\zeta Q_0^\varepsilon )^{-1}\Vert _{L_2(\mathcal{O})\rightarrow L_2(\mathcal{O})}
\leqslant \Vert f\Vert _{L_\infty}(c_\flat +1)^{1/2}(c_\flat +2)^{1/2}\varrho _ \flat (\zeta )^{1/2}.
\end{equation*}
Together with \eqref{H^1-norm <= BDeps^1/2} this implies \eqref{****.2}.

Estimates \eqref{****.3} and \eqref{****.4} are proved similarly to \eqref{****.1} and \eqref{****.2}, respectively,
with the help of~\eqref{H^1-norm <= BD0^1/2}, \eqref{f_0<=} and \eqref{B_D0 and tilde B_D0 resolvents}.

It remains to check \eqref{****.5}.
By \eqref{BD0 ^-1 L2->H2}--\eqref{B_D0 and tilde B_D0 resolvents},
\begin{equation}
\label{ex 9.38}
\begin{split}
\Vert &(B_D^0-\zeta\overline{Q_0})^{-1}\Vert _{L_2(\mathcal{O})\rightarrow H^2(\mathcal{O})}
\leqslant \Vert (B_D^0)^{-1}\Vert _{L_2(\mathcal{O})\rightarrow H^2(\mathcal{O})}\Vert B_D^0(B_D^0-\zeta \overline{Q_0})^{-1}\Vert _{L_2(\mathcal{O})\rightarrow L_2(\mathcal{O})}
\\
&\leqslant
\widehat{c}\Vert f\Vert  _{L_\infty}\Vert f^{-1}\Vert _{L_\infty}\sup\limits _{x\geqslant c_\flat}x\vert x-\zeta\vert ^{-1}
\leqslant
\widehat{c}\Vert f\Vert  _{L_\infty}\Vert f^{-1}\Vert _{L_\infty}\sup\limits _{x\geqslant c_\flat}(x+1)\vert x-\zeta\vert ^{-1}.
\end{split}
\end{equation}
A calculation shows that
\begin{equation}
\label{sup (x+1)/|x-zeta|}
\sup \limits _{x\geqslant c_\flat}\frac{(x+1)^2}{\vert x-\zeta \vert ^2}\leqslant(c_\flat+2)^2\varrho _ \flat (\zeta),\quad \zeta \in \mathbb{C}\setminus [c_\flat,\infty) .
\end{equation}
Relations \eqref{ex 9.38} and \eqref{sup (x+1)/|x-zeta|}  imply \eqref{****.5}.
\end{proof}

\begin{lemma}
\label{Lemma Dr ots corrector}
Under Condition~\textnormal{\ref{cond_eps_flat}}, for $0< \varepsilon \leqslant \varepsilon_\flat$
and $\zeta\in\mathbb{C}\setminus [c_\flat,\infty)$ we have
\begin{align}
\label{Dr ots K L2}
&\Vert K_D(\varepsilon;\zeta)\Vert _{L_2(\mathcal{O})\rightarrow L_2(\mathcal{O})}
\leqslant \mathcal{C}_5(1+\vert\zeta\vert )^{-1/2}\varrho _\flat (\zeta )^{1/2},
\\
&\varepsilon\Vert K_D(\varepsilon;\zeta)\Vert _{L_2(\mathcal{O})\rightarrow H^1(\mathcal{O})}
\leqslant \mathcal{C}_6\left(\varepsilon +(1+\vert\zeta\vert )^{-1/2}\right)\varrho _\flat (\zeta )^{1/2}.
\label{Dr ots K H1}
\end{align}
The constants $\mathcal{C}_5$ and $\mathcal{C}_6$ depend only on the initial data \eqref{problem data}.
\end{lemma}

\begin{proof}
Combining \eqref{<b^*b<}, \eqref{PO}, \eqref{K_D(eps,zeta)}, \eqref{Lambda S_eps <=}, and \eqref{tilde Lambda S_eps <=}, we obtain
\begin{equation*}
\begin{split}
\Vert K_D(\varepsilon ;\zeta )\Vert _{L_2(\mathcal{O})\rightarrow L_2(\mathcal{O})}
&\leqslant \alpha _1^{1/2}M_1 C_{\mathcal{O}}^{(1)}\Vert (B_D^0-\zeta \overline{Q_0})^{-1}\Vert _{L_2(\mathcal{O})\rightarrow H^1(\mathcal{O})}
\\
&+\widetilde{M}_1 C_{\mathcal{O}}^{(0)}\Vert (B_D^0-\zeta \overline{Q_0})^{-1}\Vert _{L_2(\mathcal{O})\rightarrow L_2(\mathcal{O})}.
\end{split}
\end{equation*}
Together with \eqref{****.4}, this yields estimate \eqref{Dr ots K L2} with  $\mathcal{C}_5 :=\mathcal{C}_3\bigl(\alpha _1^{1/2}M_1C_\mathcal{O}^{(1)}+\widetilde{M}_1C_\mathcal{O}^{(0)}\bigr)$. Next,
\begin{equation*}
\begin{split}
\varepsilon\Vert \mathbf{D}K_D(\varepsilon;\zeta)\Vert _{L_2(\mathcal{O})\rightarrow L_2(\mathcal{O})}
&\leqslant
\bigl\Vert \bigl( (\mathbf{D}\Lambda)^\varepsilon b(\mathbf{D})+(\mathbf{D}\widetilde{\Lambda})^\varepsilon \bigr) S_\varepsilon P_\mathcal{O}(B_D^0-\zeta\overline{Q_0})^{-1}\bigr\Vert _{L_2(\mathcal{O})\rightarrow L_2(\mathcal{O})}
\\
&+ \varepsilon \bigl\Vert \bigl(\Lambda ^\varepsilon b(\mathbf{D})+\widetilde{\Lambda}^\varepsilon \bigr)S_\varepsilon\mathbf{D}P_\mathcal{O}(B_D^0-\zeta\overline{Q_0})^{-1}\Vert _{L_2(\mathcal{O})\rightarrow L_2(\mathcal{O})}.
\end{split}
\end{equation*}
By Proposition \ref{Proposition f^eps S_eps}, \eqref{<b^*b<}, \eqref{DLambda<=}, \eqref{D tilde Lambda}, \eqref{PO}, \eqref{Lambda S_eps <=},  \eqref{tilde Lambda S_eps <=},  and \eqref{tilde M2 =}, this implies
\begin{equation*}
\begin{split}
\varepsilon \Vert \mathbf{D}K_D(\varepsilon;\zeta)\Vert _{L_2(\mathcal{O})\rightarrow L_2(\mathcal{O})}
&\leqslant
(M_2\alpha _1^{1/2} + \widetilde{M}_2 + \varepsilon\widetilde{M}_1) C_\mathcal{O}^{(1)} \Vert (B_D^0-\zeta\overline{Q_0})^{-1}\Vert _{L_2(\mathcal{O})\rightarrow H^1(\mathcal{O})}
\\
&+\varepsilon M_1\alpha _1^{1/2}C_\mathcal{O}^{(2)}\Vert (B_D^0-\zeta\overline{Q_0})^{-1}\Vert _{L_2(\mathcal{O})\rightarrow H^2(\mathcal{O})}.
\end{split}
\end{equation*}
So, by Lemma~\ref{Lemma Dr ots solutions},
\begin{equation}
\label{H}
\varepsilon \Vert \mathbf{D}K_D(\varepsilon;\zeta)\Vert _{L_2(\mathcal{O})\rightarrow L_2(\mathcal{O})}
\leqslant\widehat{\mathcal{C}}_6(1+\vert\zeta\vert )^{-1/2}\varrho _\flat (\zeta)^{1/2}+\widetilde{\mathcal{C}}_6\varepsilon\varrho _\flat (\zeta)^{1/2}.
\end{equation}
Here $\widehat{\mathcal{C}}_6:=\mathcal{C}_3 C_\mathcal{O}^{(1)} (M_2\alpha _1^{1/2}+\widetilde{M}_2 + \widetilde{M}_1)$, $\widetilde{\mathcal{C}}_6:=M_1\alpha _1^{1/2}C_\mathcal{O}^{(2)}\mathcal{C}_4$.

Combining \eqref{Dr ots K L2} and \eqref{H}, we arrive at estimate \eqref{Dr ots K H1} with $\mathcal{C}_6:=\max\lbrace \mathcal{C}_5+\widehat{\mathcal{C}}_6;\widetilde{\mathcal{C}}_6\rbrace $.
\end{proof}

\subsection{Proof of Theorem \ref{Theorem Dr appr}}
 Let $0< \varepsilon \leqslant \varepsilon_\flat \le \eps_1$  and $\zeta\in\mathbb{C}\setminus [c_\flat,\infty)$.
First, we prove~\eqref{Th dr appr 1}. From \eqref{Th L2} with $\zeta =-1$ it follows that
\begin{equation}
\label{Th L2 in -1}
\Vert (B_{D,\varepsilon }+ Q_0^\varepsilon )^{-1}-(B_D^0+ \overline{Q_0})^{-1}\Vert _{L_2(\mathcal{O})\rightarrow L_2(\mathcal{O})}\leqslant C_4\varepsilon .
\end{equation}
We have
\begin{equation}
\label{tozd gen res}
\begin{split}
(&B_{D,\varepsilon }-\zeta Q_0^\varepsilon )^{-1}-(B_D^0-\zeta \overline{Q_0})^{-1}\\
&= (B_{D,\varepsilon}-\zeta Q_0^\varepsilon )^{-1}(B_{D,\varepsilon}+Q_0^\varepsilon)
\left((B_{D,\varepsilon}+Q_0^\varepsilon )^{-1}-(B_D^0+\overline{Q_0})^{-1}\right)
(B_D^0+\overline{Q_0})(B_D^0-\zeta \overline{Q_0})^{-1}
\\
&+(1+\zeta)(B_{D,\varepsilon}-\zeta Q_0^\varepsilon )^{-1}(Q_0^\varepsilon -\overline{Q_0})(B_D^0-\zeta \overline{Q_0})^{-1}.
\end{split}
\end{equation}
Denote the consecutive terms in the right-hand side of~\eqref{tozd gen res} by $\mathcal{T}_1(\varepsilon ;\zeta)$ and $\mathcal{T}_2(\varepsilon ;\zeta)$. By \eqref{B deps and tilde B D,eps tozd resolvent},
\begin{equation}
\label{tozd lev}
\begin{split}
&\Vert (B_{D,\varepsilon}-\zeta Q_0^\varepsilon )^{-1}(B_{D,\varepsilon}+Q_0^\varepsilon)\Vert _{L_2(\mathcal{O})\rightarrow L_2(\mathcal{O})}
\\
&\leqslant \Vert f\Vert _{L_\infty}\Vert f^{-1}\Vert _{L_\infty}\Vert (\widetilde{B}_{D,\varepsilon}-\zeta I)^{-1}(\widetilde{B}_{D,\varepsilon}+I)\Vert _{L_2(\mathcal{O})\rightarrow L_2(\mathcal{O})}
\leqslant \Vert f\Vert _{L_\infty}\Vert f^{-1}\Vert _{L_\infty}  \sup\limits _{x\geqslant c_\flat}\frac{(x+1)}{\vert x-\zeta \vert}.
\end{split}
\end{equation}
Similarly to \eqref{tozd lev}, taking \eqref{f_0<=} into account, we obtain
\begin{equation}
\label{tozd prav}
\Vert (B_D^0+\overline{Q_0})(B_D^0-\zeta \overline{Q_0})^{-1}\Vert _{L_2(\mathcal{O})\rightarrow L_2(\mathcal{O})}
\leqslant \Vert f\Vert _{L_\infty}\Vert f^{-1}\Vert _{L_\infty}\sup\limits _{x\geqslant c_\flat}\frac{(x+1)}{\vert x-\zeta \vert}.
\end{equation}
Now, relations \eqref{sup (x+1)/|x-zeta|}, \eqref{Th L2 in -1}, \eqref{tozd lev}, and \eqref{tozd prav} imply that
\begin{equation}
\label{9.11a}
\Vert \mathcal{T}_1(\varepsilon ;\zeta)\Vert _{L_2(\mathcal{O})\rightarrow L_2(\mathcal{O})}
\leqslant \gamma_{30} \varepsilon \varrho _\flat (\zeta);
\quad \gamma_{30} := C_4\Vert f\Vert ^2_{L_\infty}\Vert f^{-1}\Vert ^2_{L_\infty}(c_\flat +2)^2.
\end{equation}

The second term in the right-hand side of~\eqref{tozd gen res} satisfies
\begin{equation}
\label{9.12}
\begin{split}
\Vert \mathcal{T}_2(\varepsilon ;\zeta )\Vert _{L_2(\mathcal{O})\rightarrow L_2(\mathcal{O})}
&\leqslant \vert 1+\zeta \vert \Vert (B_{D,\varepsilon}-\zeta Q_0^\varepsilon )^{-1}\Vert _{H^{-1}(\mathcal{O})\rightarrow L_2(\mathcal{O})}
\\
&\times\Vert [Q_0^\varepsilon -\overline{Q_0}]\Vert _{H^1(\mathcal{O})\rightarrow H^{-1}(\mathcal{O})}
\Vert (B_D^0-\zeta \overline{Q_0})^{-1}\Vert _{L_2(\mathcal{O})\rightarrow H^1(\mathcal{O})}.
\end{split}
\end{equation}
Note that the range of the operator~$(B_{D,\varepsilon}-\zeta ^*Q_0^\varepsilon )^{-1}$ lies in $H^1_0(\mathcal{O};\mathbb{C}^n)$.
Then, by duality, from~\eqref{****.2} we obtain
\begin{equation}
\label{dvoistvennost}
\Vert (B_{D,\varepsilon}-\zeta Q_0^\varepsilon )^{-1}\Vert _{H^{-1}(\mathcal{O})\rightarrow L_2(\mathcal{O})}
=\Vert (B_{D,\varepsilon}-\zeta ^*Q_0^\varepsilon )^{-1}\Vert _{L_2(\mathcal{O})\rightarrow H^1(\mathcal{O})}
\leqslant {\mathcal C}_3 (1+|\zeta|)^{-1/2} \varrho_\flat(\zeta)^{1/2}.
\end{equation}

Now, from \eqref{lemma Q_eps -overline Q}, \eqref{****.4}, \eqref{9.12}, and \eqref{dvoistvennost} it follows that
\begin{equation}
\label{9.20a}
\Vert \mathcal{T}_2(\varepsilon ;\zeta)\Vert _{L_2(\mathcal{O})\rightarrow L_2(\mathcal{O})}
\leqslant \gamma_{31}\varepsilon \varrho _\flat (\zeta); \quad
\gamma_{31} := C_{Q_0} {\mathcal C}^2_3.
\end{equation}
As a result, relations \eqref{tozd gen res}, \eqref{9.11a}, and \eqref{9.20a} yield~\eqref{Th dr appr 1} with the constant
$C_{26} :=\gamma_{30} + \gamma_{31}$.

Let us prove \eqref{Th dr appr 2}. By inequality~\eqref{Th L2->H1} with $\zeta =-1$,
\begin{equation}
\label{Th L2 H1 in -1}
\begin{split}
\Vert &(B_{D,\varepsilon}+Q_0^\varepsilon )^{-1}-(B_D^0+\overline{Q_0})^{-1}-\varepsilon K_D(\varepsilon ;-1)\Vert _{L_2(\mathcal{O})\rightarrow H^1(\mathcal{O})} \leqslant (C_5+C_6)\varepsilon ^{1/2}.
\end{split}
\end{equation}
 We have
\begin{equation}
\label{A sec. 9}
\begin{split}
&(B_{D,\varepsilon}-\zeta Q_0^\varepsilon )^{-1}-(B_D^0-\zeta \overline{Q_0})^{-1}-\varepsilon K_D(\varepsilon ;\zeta)
\\
&= \left((B_{D,\varepsilon}+Q_0^\varepsilon )^{-1}-(B_D^0+\overline{Q_0})^{-1}-\varepsilon K_D(\varepsilon;-1)\right)
(B_D^0 +\overline{Q_0})(B_D^0-\zeta\overline{Q_0})^{-1}
\\
&+  (\zeta +1)(B_{D,\varepsilon}-\zeta Q_0^\varepsilon )^{-1}Q_0^\varepsilon
\left( (B_{D,\varepsilon}+Q_0^\varepsilon )^{-1}-(B_D^0+\overline{Q_0})^{-1} \right)
(B_D^0 +\overline{Q_0})(B_D^0-\zeta\overline{Q_0})^{-1}
\\
&+ (1+\zeta)(B_{D,\varepsilon}-\zeta Q_0^\varepsilon )^{-1}(Q_0^\varepsilon -\overline{Q_0})
(B_D^0-\zeta\overline{Q_0})^{-1}.
\end{split}
\end{equation}
Denote the consecutive summands in the right-hand side of \eqref{A sec. 9} by~$\mathcal{L}_1(\varepsilon ;\zeta )$, $\mathcal{L}_2(\varepsilon ;\zeta )$, and $\mathcal{L}_3(\varepsilon ;\zeta )$. (Note that $\mathcal{L}_3(\varepsilon ;\zeta )$ coincides with $\mathcal{T}_2(\varepsilon ;\zeta )$.)
We have
\begin{equation}
\label{9.29a new}
\begin{split}
\Vert {\mathcal{L}}_1(\varepsilon;\zeta)\Vert _{L_2(\mathcal{O})\rightarrow H^1(\mathcal{O})}
&\leqslant
\Vert (B_{D,\varepsilon}+Q_0^\varepsilon )^{-1}-(B_D^0+\overline{Q_0})^{-1}-\varepsilon K_D(\varepsilon;-1)\Vert _{L_2(\mathcal{O})\rightarrow H^1(\mathcal{O})}
\\
&\times
\Vert (B_D^0 +\overline{Q_0})(B_D^0-\zeta\overline{Q_0})^{-1}\Vert _{L_2(\mathcal{O})\rightarrow L_2(\mathcal{O})}.
\end{split}
\end{equation}
Together with \eqref{sup (x+1)/|x-zeta|},  \eqref{tozd prav}, and \eqref{Th L2 H1 in -1}, this yields
\begin{equation}
\label{B}
\Vert {\mathcal{L}}_1(\varepsilon;\zeta)\Vert _{L_2(\mathcal{O})\rightarrow H^1(\mathcal{O})}
\leqslant \gamma _{32}  \varepsilon ^{1/2}\varrho _\flat (\zeta)^{1/2};\quad \gamma_{32} := (C_5+C_6) (c_\flat +2) \Vert f\Vert _{L_\infty}\Vert f^{-1}\Vert _{L_\infty}.
\end{equation}

Now, consider the second term in the right-hand side of \eqref{A sec. 9}.
We have
\begin{equation}
\label{***.33}
\begin{split}
&\| \mathcal{L}_2(\varepsilon ;\zeta ) \|_{L_2(\mathcal{O})\rightarrow H^1(\mathcal{O})}
\leqslant
\vert \zeta +1\vert  \|  (B_{D,\varepsilon} - \zeta Q_0^\varepsilon)^{-1} \|_{L_2(\mathcal{O})\rightarrow H^1(\mathcal{O})}
\Vert Q_0 \Vert _{L_\infty}
\\
&\times \Vert (B_{D,\varepsilon}+Q_0^\varepsilon )^{-1}-(B_D^0+\overline{Q_0})^{-1}
\Vert _{L_2(\mathcal{O})\rightarrow L_2(\mathcal{O})}
\Vert (B_D^0 +\overline{Q_0})(B_D^0-\zeta\overline{Q_0})^{-1}\Vert _{L_2(\mathcal{O})\rightarrow L_2(\mathcal{O})}.
\end{split}
\end{equation}
Combining this with \eqref{****.2}, \eqref{sup (x+1)/|x-zeta|},  \eqref{Th L2 in -1}, and \eqref{tozd prav}, we obtain
\begin{equation}
\label{***.3}
\Vert {\mathcal{L}}_2(\varepsilon;\zeta)\Vert _{L_2(\mathcal{O})\rightarrow H^1(\mathcal{O})}
\leqslant \gamma _{33} \vert 1+\zeta\vert ^{1/2} \varepsilon \varrho _\flat (\zeta);
\quad  \gamma_{33} := \mathcal{C}_3 C_4  (c_\flat +2) \Vert f\Vert _{L_\infty} \Vert f^{-1}\Vert ^3_{L_\infty}.
\end{equation}

It remains to estimate the third term in the right-hand side of \eqref{A sec. 9}.
By~\eqref{lemma Q_eps -overline Q},
\begin{equation}
\label{9.31a new}
\|  \mathcal{L}_3(\varepsilon ;\zeta ) \|_{L_2(\mathcal{O})\rightarrow H^1(\mathcal{O})}
\leqslant \varepsilon \vert 1+\zeta \vert  C_{Q_0} \Vert (B_{D,\varepsilon}-\zeta Q_0^\varepsilon )^{-1}\Vert _{H^{-1}(\mathcal{O})\rightarrow H^1(\mathcal{O})}\Vert (B_D^0-\zeta\overline{Q_0})^{-1}\Vert _{L_2(\mathcal{O})\rightarrow H^1(\mathcal{O})}.
\end{equation}
Taking \eqref{H^1-norm <= BDeps^1/2} and \eqref{****.4} into account, we see that
\begin{equation}
\label{***.5}
\| \mathcal{L}_3(\varepsilon ;\zeta ) \|_{L_2(\mathcal{O})\rightarrow H^1(\mathcal{O})}
\leqslant c_4 C_{Q_0}\mathcal{C}_3\varepsilon \vert 1+\zeta \vert ^{1/2} \varrho_\flat(\zeta )^{1/2} \Vert B_{D,\varepsilon }^{1/2}(B_{D,\varepsilon}-\zeta Q_0^\varepsilon )^{-1}\Vert _{H^{-1}(\mathcal{O})\rightarrow L_2(\mathcal{O})}.
\end{equation}
By duality, using \eqref{tilde b D,eps=} and \eqref{B deps and tilde B D,eps tozd resolvent}, we obtain
\begin{equation}
\label{dv}
\begin{split}
\Vert B_{D,\varepsilon}^{1/2}(B_{D,\varepsilon}-\zeta Q_0^\varepsilon )^{-1}\Vert _{H^{-1}(\mathcal{O})\rightarrow L_2(\mathcal{O})}
&=\Vert \widetilde{B}_{D,\varepsilon }^{1/2}(\widetilde{B}_{D,\varepsilon}-\zeta I)^{-1} (f^\varepsilon)^* \Vert _{H^{-1}(\mathcal{O})\rightarrow L_2(\mathcal{O})}
\\
&=\Vert f^\varepsilon \widetilde{B}_{D,\varepsilon }^{1/2}(\widetilde{B}_{D,\varepsilon}-\zeta ^* I)^{-1}\Vert _{L_2(\mathcal{O})\rightarrow H^1(\mathcal{O})}.
\end{split}
\end{equation}
Since the range of the operator $f^\varepsilon \widetilde{B}_{D,\varepsilon }^{1/2}(\widetilde{B}_{D,\varepsilon}-\zeta ^* I)^{-1}$
 lies in $H^1_0(\mathcal{O};\mathbb{C}^n)$, from \eqref{H^1-norm <= BDeps^1/2} and \eqref{tilde b D,eps=} it follows that
\begin{equation}
\label{dv-a}
\begin{split}
\Vert  f^\varepsilon \widetilde{B}_{D,\varepsilon}^{1/2}(\widetilde{B}_{D,\varepsilon}-\zeta ^* I)^{-1}\Vert _{L_2(\mathcal{O})\rightarrow H^1(\mathcal{O})}
&\leqslant c_4 \Vert {B}_{D,\varepsilon}^{1/2} f^\varepsilon \widetilde{B}_{D,\varepsilon}^{1/2}
(\widetilde{B}_{D,\varepsilon}-\zeta ^* I)^{-1}\Vert _{L_2(\mathcal{O})\rightarrow L_2(\mathcal{O})}
\\
&= c_4 \Vert  \widetilde{B}_{D,\varepsilon}(\widetilde{B}_{D,\varepsilon}-\zeta ^* I)^{-1}\Vert _{L_2(\mathcal{O})\rightarrow L_2(\mathcal{O})}.
\end{split}
\end{equation}
Together with \eqref{sup (x+1)/|x-zeta|} and \eqref{dv}, this yields
\begin{equation}
\label{dv2}
\Vert B_{D,\varepsilon}^{1/2}(B_{D,\varepsilon}-\zeta Q_0^\varepsilon )^{-1}\Vert _{H^{-1}(\mathcal{O})\rightarrow L_2(\mathcal{O})}
\leqslant c_4 (c_\flat +2)\varrho_\flat(\zeta )^{1/2}.
\end{equation}
Combining \eqref{***.5} and \eqref{dv2}, we find
\begin{equation}
\label{L3<=}
\|  \mathcal{L}_3(\varepsilon ;\zeta ) \|_{L_2(\mathcal{O})\rightarrow H^1(\mathcal{O})}
\leqslant \gamma_{34} \varepsilon \vert 1+\zeta \vert ^{1/2} \varrho _ \flat (\zeta );
\quad  \gamma_{34}:=  c_4^2 (c_\flat +2)C_{Q_0}\mathcal{C}_3.
\end{equation}

As a result, relations \eqref{A sec. 9}, \eqref{B}, \eqref{***.3},  and \eqref{L3<=} imply that
\begin{equation}
\label{dr alppr so srezkoy}
\begin{split}
\Vert &(B_{D,\varepsilon}-\zeta Q_0^\varepsilon )^{-1}-(B_D^0-\zeta \overline{Q_0})^{-1}-\varepsilon K_D(\varepsilon ;\zeta )\Vert _{L_2(\mathcal{O})\rightarrow H^1(\mathcal{O})}\\
&\leqslant \gamma _{32}\varepsilon ^{1/2}\varrho _\flat (\zeta)^{1/2}+ (\gamma_{33}+\gamma_{34})
\varepsilon  \vert 1+\zeta \vert ^{1/2} \varrho_\flat(\zeta).
\end{split}
\end{equation}
This yields \eqref{Th dr appr 2} with the constant $C_{27} = \max\{\gamma_{32};\gamma_{33}+\gamma_{34}\}$.

It remains to check~\eqref{Th dr appr fluxes}.
From \eqref{b_l <=} and \eqref{Th dr appr 2 solutions} it follows that
\begin{equation}
\label{st.st}
\begin{split}
\Vert &\mathbf{p}_\varepsilon -g^\varepsilon b(\mathbf{D})\mathbf{v}_\varepsilon\Vert _{L_2(\mathcal{O})}
\leqslant (d\alpha _1)^{1/2}\Vert g\Vert _{L_\infty} C_{27}
\bigl( \varepsilon ^{1/2}\varrho _ \flat (\zeta )^{1/2} + \varepsilon  \vert 1+\zeta \vert ^{1/2} \varrho_\flat(\zeta) \bigr)
\Vert \mathbf{F}\Vert _{L_2(\mathcal{O})}.
\end{split}
\end{equation}
Next,  taking \eqref{<b^*b<} into account, by analogy with \eqref{3.20 from draft}, \eqref{5.18a}, and \eqref{5.19a}, we obtain
\begin{equation}
\label{st}
\begin{split}
\Vert  g^\varepsilon b(\mathbf{D})\mathbf{v}_\varepsilon -\widetilde{g}^\varepsilon S_\varepsilon b(\mathbf{D})\widetilde{\mathbf{u}}_0-g^\varepsilon \bigl(b(\mathbf{D})\widetilde{\Lambda}\bigr)^\varepsilon S_\varepsilon\widetilde{\mathbf{u}}_0\Vert _{L_2(\mathcal{O})}
\leqslant \gamma_{35}\varepsilon \Vert \widetilde{\mathbf{u}}_0\Vert _{H^2(\mathbb{R}^d)}.
\end{split}
\end{equation}
Here $\gamma _{35} :=\Vert g\Vert _{L_\infty}\alpha _1^{1/2}\bigl( M_1(\alpha _1 d)^{1/2}+\widetilde{M}_1 d^{1/2}
+r_1\bigr)$.
From \eqref{PO} and \eqref{****.5} it follows that
\begin{equation}
\label{9.48}
\Vert \widetilde{\mathbf{u}}_0\Vert _{H^2(\mathbb{R}^d)}
\leqslant \gamma_{36}\rho _\flat (\zeta)^{1/2}\Vert \mathbf{F}\Vert _{L_2(\mathcal{O})};
\quad\gamma _{36}:=C_\mathcal{O}^{(2)} {\mathcal C}_4.
\end{equation}
Combining this with \eqref{st.st} and \eqref{st}, we arrive at estimate~\eqref{Th dr appr fluxes} with the constant
$\widetilde{C}_{27} :=(d\alpha _1)^{1/2}\Vert g\Vert _{L_\infty}C_{27}+\gamma _{35} \gamma _{36}$.
\qed

\begin{corollary}
Under the assumptions of Theorem~\emph{\ref{Theorem Dr appr}}, for $0<\varepsilon \leqslant \varepsilon _\flat$
and $\zeta\in\mathbb{C}\setminus [c_\flat,\infty)$  we have
\begin{align}
\label{cor1}
&\Vert \mathbf{u}_\varepsilon-\mathbf{v}_\varepsilon\Vert _{H^1(\mathcal{O})}
\leqslant
C_{28}\varepsilon ^{1/2}\varrho _ \flat (\zeta )^{3/4}\Vert \mathbf{F}\Vert _{L_2(\mathcal{O})},
\\
\label{cor2}
&\Vert \mathbf{p}_\varepsilon -\widetilde{g}^\varepsilon S_\varepsilon b(\mathbf{D})\widetilde{\mathbf{u}}_0-g^\varepsilon \bigl(b(\mathbf{D})\widetilde{\Lambda}\bigr)^\varepsilon S_\varepsilon\widetilde{\mathbf{u}}_0
\Vert _{L_2(\mathcal{O})}
\leqslant
\widetilde{C}_{28}\varepsilon ^{1/2}\varrho _ \flat (\zeta )^{3/4}\Vert \mathbf{F}\Vert _{L_2(\mathcal{O})}.
\end{align}
The constants $C_{28}$ and $\widetilde{C}_{28}$ depend only on the initial data~\eqref{problem data} and the domain $\mathcal{O}$.
\end{corollary}

\begin{proof}
 Relations \eqref{****.2}, \eqref{****.4}, and  \eqref{Dr ots K H1} yield the following rough estimate:
\begin{equation}
\label{dr alppr corr grubo}
\begin{split}
\Vert &(B_{D,\varepsilon}-\zeta Q_0^\varepsilon )^{-1}-(B_D^0-\zeta \overline{Q_0})^{-1}-\varepsilon K_D(\varepsilon ;\zeta )\Vert _{L_2(\mathcal{O})\rightarrow H^1(\mathcal{O})}\\
&\leqslant \gamma _{37}(\varepsilon +  (1+ |\zeta|) ^{-1/2} )\varrho_\flat(\zeta)^{1/2};\quad
\gamma _{37}:= 2\mathcal{C}_3+\mathcal{C}_6.
\end{split}
\end{equation}
For $\vert 1+\zeta\vert ^{1/2}\varrho _\flat (\zeta)^{1/4}\leqslant\varepsilon ^{-1/2}$ we use \eqref{dr alppr so srezkoy} and note that
$\varepsilon\vert 1+\zeta\vert ^{1/2}\varrho _\flat (\zeta)\leqslant  \varepsilon^{1/2}\varrho _\flat (\zeta)^{3/4}.$
For $\vert 1+\zeta\vert ^{1/2}\varrho _\flat (\zeta)^{1/4}>\varepsilon ^{-1/2}$ we apply \eqref{dr alppr corr grubo} and take the inequality
$(1+ |\zeta|) ^{-1/2}\varrho _\flat (\zeta)^{1/2} <  \varepsilon^{1/2}\varrho _\flat (\zeta)^{3/4}$ into account.
 This yields \eqref{cor1} with $C_{28}:=\max
\lbrace \gamma _{32}+\gamma _{33}+\gamma _{34};2\gamma _{37}\rbrace$.

Relations \eqref{st}, \eqref{9.48},  and \eqref{cor1} imply \eqref{cor2} with
$\widetilde{C}_{28} :=(d\alpha _1)^{1/2}\Vert g\Vert _{L_\infty}C_{28}+\gamma _{35} \gamma _{36}$.
\end{proof}

\subsection{Removal of $S_\varepsilon$}

\begin{theorem}
\label{Theorem Dr appr no S_eps}
Suppose that the assumptions of Theorem~\textnormal{\ref{Theorem Dr appr}} are satisfied.
Suppose also that Conditions~\textnormal{\ref{Condition Lambda in L infty}} and~\textnormal{\ref{Condition tilde Lambda in Lp}}
hold. Let $K_D^0(\varepsilon ;\zeta)$ and $G^0_D(\varepsilon ;\zeta)$ be defined by~\eqref{K_D^0} and \eqref{G3(eps;zeta)}, respectively.
Then for $0<\varepsilon\leqslant \varepsilon _\flat$ and $\zeta\in\mathbb{C}\setminus[c_\flat ,\infty)$ we have
\begin{align}
\label{Th dr appr no S_eps 5}
\begin{split}
\Vert &(B_{D,\varepsilon}-\zeta Q_0^\varepsilon )^{-1}- (B_D^0-\zeta\overline{Q_0})^{-1} - \varepsilon K_D^0(\varepsilon; \zeta) \Vert _{L_2(\mathcal{O})\rightarrow H^1(\mathcal{O})}
\\
&\leqslant C_{29} \bigl(\varepsilon ^{1/2}\varrho _\flat (\zeta)^{1/2}+ \varepsilon
\vert 1+\zeta \vert ^{1/2} \varrho_\flat(\zeta) \bigr),
\end{split}
\\
\label{Th dr appr no S_eps 6}
\begin{split}
\Vert &g^\varepsilon b(\mathbf{D})(B_{D,\varepsilon}-\zeta Q_0^\varepsilon )^{-1}-G^0_D(\varepsilon;\zeta)\Vert _{L_2(\mathcal{O})\rightarrow L_2(\mathcal{O})}
\leqslant
\widetilde{C}_{29} \bigl(\varepsilon ^{1/2}\varrho _\flat (\zeta)^{1/2}+ \varepsilon
\vert 1+\zeta \vert ^{1/2} \varrho_\flat(\zeta) \bigr).
\end{split}
\end{align}
The constants $C_{29}$ and $\widetilde{C}_{29}$ depend only on the initial data~\eqref{problem data}, the domain~$\mathcal{O}$,
and also on $p$ and the norms $\Vert \Lambda\Vert _{L_\infty}$, $\Vert \widetilde{\Lambda}\Vert _{L_p(\Omega)}$.
\end{theorem}

\begin{proof}
Applying Lemmas~\ref{Lemma Lambda (S-I)} and~\ref{Lemma tilde Lambda(S-I)} together with \eqref{Th dr appr 2} and \eqref{9.48}, we obtain
\eqref{Th dr appr no S_eps 5} with   $C_{29} :=C_{27}+(\mathfrak{C}_\Lambda +\mathfrak{C}_{\widetilde{\Lambda}})\gamma _{36}$.

  Let us check~\eqref{Th dr appr no S_eps 6}. By \eqref{b_l <=} and \eqref{Th dr appr no S_eps 5},
\begin{equation}
\label{ex 9.45}
\begin{split}
\Vert & g^\varepsilon b(\mathbf{D})(B_{D,\varepsilon}-\zeta Q_0^\varepsilon )^{-1}- g^\varepsilon b(\mathbf{D})(I+\varepsilon \Lambda ^\varepsilon b(\mathbf{D})+\varepsilon \widetilde{\Lambda}^\varepsilon)(B_D^0-\zeta\overline{Q_0})^{-1}\Vert _{L_2({\mathcal O})\rightarrow L_2({\mathcal O})}
\\
&\leqslant (d\alpha _1)^{1/2}\Vert g\Vert _{L_\infty}C_{29}
\bigl(\varepsilon ^{1/2}\varrho _\flat (\zeta)^{1/2}+ \varepsilon
\vert 1+\zeta \vert ^{1/2} \varrho_\flat(\zeta) \bigr).
\end{split}
\end{equation}
Relation~\eqref{d-vo Th 7/6 4} remains true.
By analogy with \eqref{d-vo Th 7/6 5} and \eqref{d-vo Th 7/6 5a}, using \eqref{****.5}, we obtain
\begin{equation}
\label{9-star}
\begin{split}
\varepsilon&\Bigl\Vert \sum _{l=1}^d g^\varepsilon b_l\bigl(\Lambda^\varepsilon b(\mathbf{D})D_l+\widetilde{\Lambda}^\varepsilon D_l\bigr)(B^0_D-\zeta \overline{Q_0})^{-1}\Bigr\Vert _{L_2(\mathcal{O})\rightarrow L_2(\mathcal{O})}
\\
&\leqslant \gamma_{38} \varepsilon \Vert (B_D^0-\zeta \overline{Q_0})^{-1}\Vert _{L_2(\mathcal{O})\rightarrow H^2(\mathcal{O})}
\leqslant \gamma_{38} {\mathcal C}_4 \varepsilon \varrho _\flat (\zeta )^{1/2},
\end{split}
\end{equation}
where
$\gamma_{38} :=\Vert g\Vert _{L_\infty}\bigl(\alpha_1 d \Vert \Lambda\Vert _{L_\infty}+(\alpha _1 d)^{1/2}\Vert \widetilde{\Lambda}\Vert _{L_p(\Omega)}C(\widehat{q},\Omega)C_\mathcal{O}^{(1)}\bigr)$.
Now, relations \eqref{d-vo Th 7/6 4},  \eqref{ex 9.45}, and \eqref{9-star} imply \eqref{Th dr appr no S_eps 6} with
$\widetilde{C}_{29} :=(d\alpha _1)^{1/2}\Vert g\Vert _{L_\infty}C_{29}+\gamma _{38} {\mathcal C}_4$.
\end{proof}

\begin{remark}
If only Condition~\textnormal{\ref{Condition Lambda in L infty} (}respectively, Condition~\textnormal{\ref{Condition tilde Lambda in Lp})} is satisfied, then
the smoothing operator $S_\varepsilon$ can be removed only in the term of the corrector containing $\Lambda^\varepsilon$
\textnormal{(}respectively, $\widetilde{\Lambda}^\varepsilon$\textnormal{)}.
\end{remark}

\subsection{Approximation with the boundary layer correction term}

Now, using Theorem~\ref{Theorem with Diriclet corrector},
 we obtain ``another'' approximation with the boundary layer correction term.

\begin{theorem}
Suppose that $\mathcal{O}\subset\mathbb{R}^d$ is a bounded domain of class $C^{1,1}$.
Let $0< \eps_\flat \le 1$. Suppose that $c_\flat \ge 0$ is subject to Condition~\textnormal{\ref{cond_eps_flat}}.
Let $0< \eps \le \eps_\flat$ and $\zeta\in\mathbb{C}\setminus [c_\flat ,\infty)$.
Let $\mathbf{u}_\varepsilon$ be the solution of problem~\eqref{Dirichlet problem},
and let $\mathbf{v}_\varepsilon$ be defined by~\eqref{v_eps}, \eqref{v_eps=}.
Let $\mathbf{w}_\varepsilon$ be the solution of problem~\eqref{w_eps problem}.
Suppose that $K_D(\varepsilon ;\zeta)$ and $W_D(\varepsilon;\zeta)$ are given by \eqref{K_D(eps,zeta)} and~\eqref{W_D(eps;zeta)}, respectively.
We have
\begin{equation*}
\Vert \mathbf{u}_\varepsilon -\mathbf{v}_\varepsilon +\mathbf{w}_\varepsilon \Vert _{H^1(\mathcal{O})}
\leqslant (C_{30}+C_{31}\vert 1+\zeta \vert ^{1/2})\varepsilon \varrho _\flat (\zeta)\Vert \mathbf{F}\Vert _{L_2(\mathcal{O})}.
\end{equation*}
In operator terms,
\begin{equation}
\label{dr appr with W_D(eps;zeta)}
\begin{split}
\Vert &(B_{D,\varepsilon}-\zeta Q_0^\varepsilon )^{-1}-(B_D^0-\zeta\overline{Q_0})^{-1}-\varepsilon K_D(\varepsilon ;\zeta)+\varepsilon W_D(\varepsilon ;\zeta)\Vert _{L_2(\mathcal{O})\rightarrow H^1 (\mathcal{O})}
\\
&\leqslant (C_{30}+C_{31}\vert 1+ \zeta \vert ^{1/2} )\varepsilon \varrho _\flat (\zeta).
\end{split}
\end{equation}
The constants~$C_{30}$ and $C_{31}$ depend only on the initial data~\eqref{problem data} and the domain~$\mathcal{O}$. If the matrix-valued function $Q_0(\mathbf{x})$ is constant, then $C_{31}=0$.
\end{theorem}

\begin{remark}
Taking $\varepsilon_\flat=1$ and $c_\flat =0$, for $\vert\zeta\vert\geqslant 1$ we have $\varrho _\flat (\zeta )=c(\phi)^2$. So, if $Q_0(\mathbf{x})$ is constant, then $C_{31}=0$ and estimate \eqref{dr appr with W_D(eps;zeta)} improves inequality \eqref{u_eps - V_eps +W-eps in H1 in operator terms} with respect to $\phi$.
\end{remark}

\begin{proof}
Using estimate~\eqref{u_eps - V_eps +W-eps in H1 in operator terms}
with $\zeta = -1$ and taking~\eqref{W_D(eps;zeta)} into account, we obtain
\begin{equation}
\label{th discrepancy in -1}
\begin{split}
\Vert (B_{D,\varepsilon}+Q_0^\varepsilon)^{-1}-(B_D^0+\overline{Q_0})^{-1}+\varepsilon (B_{D,\varepsilon}+Q_0^\varepsilon)^{-1} T(\varepsilon ;-1)\Vert _{L_2(\mathcal{O})\rightarrow H^1 (\mathcal{O})}
\leqslant C_7\varepsilon.
\end{split}
\end{equation}

Next, from the definition of~$T(\varepsilon;\zeta)$ (see~\eqref{**.4}, \eqref{**.5}) it is clear that
$$
T(\varepsilon ;-1)(B_D^0+\overline{Q_0})(B_D^0-\zeta\overline{Q_0})^{-1}=T(\varepsilon;\zeta).
$$
Combining this identity and \eqref{W_D(eps;zeta)}, it is easy to check that
\begin{equation}
\label{9.52a}
\begin{split}
(&B_{D,\varepsilon}-\zeta Q_0^\varepsilon )^{-1}-(B_D^0-\zeta \overline{Q_0})^{-1}-\varepsilon K_D(\varepsilon;\zeta)+\varepsilon W_D(\varepsilon ;\zeta)
\\
&=(B_{D,\varepsilon}-\zeta Q_0^\varepsilon )^{-1}-(B_D^0-\zeta \overline{Q_0})^{-1}+\varepsilon (B_{D,\varepsilon}-\zeta Q_0^\varepsilon )^{-1}T(\varepsilon;\zeta)
\\
&=(B_{D,\varepsilon}-\zeta Q_0^\varepsilon)^{-1} (B_{D,\varepsilon}+ Q_0^\varepsilon)
\left(
(B_{D,\varepsilon}+Q_0^\varepsilon)^{-1}-(B_D^0+\overline{Q_0})^{-1}+\varepsilon (B_{D,\varepsilon}+Q_0^\varepsilon)^{-1}T(\varepsilon ;-1)
\right)
\\
&\times
(B_D^0+\overline{Q_0})(B_D^0-\zeta \overline{Q_0})^{-1}
+(\zeta +1)(B_{D,\varepsilon}-\zeta Q_0^\varepsilon )^{-1}(Q^\varepsilon _0 -\overline{Q_0})(B_D^0-\zeta \overline{Q_0})^{-1}.
\end{split}
\end{equation}
Denote the first summand on the right by $\mathrm{J}(\varepsilon;\zeta)$.
Note that the second term is ${\mathcal L}_3(\varepsilon;\zeta)$; cf.~\eqref{A sec. 9}. Obviously, if $Q_0^\varepsilon(\mathbf{x})=\overline{Q_0}$, then $\mathcal{L}_3(\varepsilon;\zeta)=0$.

 From  \eqref{b_D,eps ots}, \eqref{tilde b D,eps=}, \eqref{B deps and tilde B D,eps tozd resolvent}, and \eqref{sup (x+1)/|x-zeta|} it follows that
\begin{equation*}
\begin{split}
\Vert& B_{D,\varepsilon}^{1/2}(B_{D,\varepsilon}-\zeta Q_0^\varepsilon)^{-1} (B_{D,\varepsilon}+ Q_0^\varepsilon)
\boldsymbol{\Phi}\Vert _{L_2(\mathcal{O})}
=\Vert \widetilde{B}_{D,\varepsilon}^{1/2} (\widetilde{B}_{D,\varepsilon}-\zeta I)^{-1} (\widetilde{B}_{D,\varepsilon}+ I)
(f^\varepsilon )^{-1}\boldsymbol{\Phi}\Vert _{L_2(\mathcal{O})}
\\
&=\Vert (\widetilde{B}_{D,\varepsilon}-\zeta I)^{-1} (\widetilde{B}_{D,\varepsilon}+ I)
\widetilde{B}_{D,\varepsilon}^{1/2}  (f^\varepsilon )^{-1}\boldsymbol{\Phi}\Vert _{L_2(\mathcal{O})}
\\
&\leqslant \Vert (\widetilde{B}_{D,\varepsilon}-\zeta I)^{-1}   (\widetilde{B}_{D,\varepsilon}+ I) \Vert _{L_2(\mathcal{O})\rightarrow L_2(\mathcal{O})}\Vert B_{D,\varepsilon }^{1/2}\boldsymbol{\Phi}\Vert _{L_2(\mathcal{O})}
\leqslant c_3^{1/2}(c_\flat +2)\varrho _\flat (\zeta)^{1/2} \|\boldsymbol{\Phi}\|_{H^1(\mathcal{O})}
\end{split}
\end{equation*}
for any function $\boldsymbol{\Phi}\in H^1_0(\mathcal{O};\mathbb{C}^n)$. Hence, by \eqref{H^1-norm <= BDeps^1/2} and \eqref{th discrepancy in -1},
\begin{equation*}
\begin{split}
\Vert& \mathrm{J}(\varepsilon ;\zeta)\Vert _{L_2(\mathcal{O})\rightarrow H^1(\mathcal{O})}
\leqslant c_4 \Vert B_{D,\varepsilon}^{1/2} \mathrm{J}(\varepsilon ;\zeta)\Vert _{L_2(\mathcal{O})\rightarrow L_2(\mathcal{O})}
\\
&\leqslant c_4 c_3^{1/2}(c_\flat +2)\varrho _\flat (\zeta)^{1/2} C_7 \varepsilon   \Vert ({B}_{D}^0 + \overline{Q_0}) ({B}_{D}^0 -\zeta \overline{Q_0})^{-1} \Vert _{L_2(\mathcal{O})\rightarrow L_2(\mathcal{O})}.
\end{split}
\end{equation*}
Together with \eqref{sup (x+1)/|x-zeta|} and \eqref{tozd prav}, this yields
\begin{equation}
\label{J final estimate}
\begin{split}
\Vert \mathrm{J}(\varepsilon ;\zeta)\Vert _{L_2(\mathcal{O})\rightarrow H^1(\mathcal{O})}
\leqslant C_{30}\varepsilon\varrho _\flat (\zeta);\quad C_{30}:=c_4 c_3^{1/2} C_7 (c_\flat +2)^2 \|f\|_{L_\infty} \|f^{-1}\|_{L_\infty}.
\end{split}
\end{equation}
Finally, \eqref{L3<=}, \eqref{9.52a}, and \eqref{J final estimate}
imply the required estimate \eqref{dr appr with W_D(eps;zeta)} with $C_{31}:=\gamma _{34}$.
\end{proof}

\subsection{Special cases}
The following statements can be checked similarly to Propositions~\ref{Proposition K=0} and~\ref{Proposition tilde Lambda =0 in Sec. 7}.

\begin{proposition}
Suppose that $0< \eps_\flat \le 1$ and $c_\flat$ is subject to Condition~\textnormal{\ref{cond_eps_flat}}.
Suppose that relations~\eqref{overline-g} and~\eqref{sum Dj aj =0} hold.
Then for $0<\varepsilon\leqslant \varepsilon _\flat$ and $\zeta\in\mathbb{C}\setminus[c_\flat,\infty)$  we have
\begin{equation*}
\begin{split}
\Vert &(B_{D,\varepsilon}-\zeta Q_0^\varepsilon )^{-1}-(B_D^0-\zeta\overline{Q_0})^{-1}\Vert _{L_2(\mathcal{O})\rightarrow H^1(\mathcal{O})}
\leqslant (C_{30} +C_{31} \vert 1+\zeta\vert ^{1/2}) \varepsilon  \varrho _\flat (\zeta).
\end{split}
\end{equation*}
\end{proposition}

\begin{proposition}
\label{Proposition tilde Lambda =0}
Suppose that the assumptions of Theorem~\textnormal{\ref{Theorem Dr appr}} are satisfied.
Suppose that relations~\eqref{underline-g} and~\eqref{sum Dj aj =0} hold.
Then for $0<\varepsilon\leqslant \varepsilon _\flat$ and $\zeta\in\mathbb{C}\setminus [c_\flat ,\infty)$ we have
\begin{equation*}
\begin{split}
&\Vert g^\varepsilon b(\mathbf{D})(B_{D,\varepsilon}-\zeta Q_0^\varepsilon )^{-1}-g^0b(\mathbf{D})(B_D^0-\zeta\overline{Q_0})^{-1}\Vert _{L_2(\mathcal{O})\rightarrow L_2(\mathcal{O})}
\\
&\leqslant \widetilde{C}_{29}\bigl( \varepsilon ^{1/2} \varrho _\flat (\zeta)^{1/2} + \varepsilon |1+ \zeta|^{1/2} \varrho_\flat (\zeta) \bigr).
\end{split}
\end{equation*}
\end{proposition}

\subsection{Estimates in a strictly interior subdomain}

\begin{theorem}
\label{Theorem Dr appr strictly interior subdomain}
Suppose that the assumptions of Theorem~\textnormal{\ref{Theorem Dr appr}} are satisfied.
Let $\mathcal{O}'$ be a strictly interior subdomain of the domain $\mathcal{O}$. Let $\delta :=\mathrm{dist}\,\lbrace\mathcal{O}';\partial\mathcal{O}\rbrace$.  Then for $0<\varepsilon\leqslant \varepsilon _\flat$ and $\zeta \in {\mathbb C} \setminus [c_\flat, \infty)$  we have
\begin{align}
\label{Th Dr appr strictly interior subdomain solutions}
\begin{split}
\Vert &(B_{D,\varepsilon}-\zeta Q_0^\varepsilon )^{-1}-(B_D^0-\zeta\overline{Q_0})^{-1} - \varepsilon K_D(\varepsilon;\zeta)
\Vert _{L_2(\mathcal{O})\rightarrow H^1(\mathcal{O}')}
\\
&\leqslant  \varepsilon  \bigl(C_{32}' \delta^{-1} \varrho _\flat (\zeta)^{1/2} + C_{32}'' |1+\zeta|^{1/2} \varrho _\flat (\zeta)\bigr),
\end{split}
\\
\label{Th Dr appr strictly interior subdomain fluxes}
\begin{split}
\Vert& g^\varepsilon b(\mathbf{D})(B_{D,\varepsilon}-\zeta Q_0^\varepsilon )^{-1}-
G_D(\varepsilon;\zeta) \Vert _{L_2(\mathcal{O})\rightarrow L_2(\mathcal{O}')}
\\
&
\leqslant  \varepsilon  \bigl( \widetilde{C}_{32}' \delta^{-1} \varrho _\flat (\zeta)^{1/2} + \widetilde{C}_{32}'' |1+\zeta|^{1/2}
\varrho _\flat (\zeta)\bigr).
\end{split}
\end{align}
The constants $C_{32}'$, $C_{32}''$, $\widetilde{C}_{32}'$, and $\widetilde{C}_{32}''$
depend only on the initial data \eqref{problem data} and the domain $\mathcal{O}$.
\end{theorem}

\begin{proof}
Estimate \eqref{Th O'} at the point $\zeta = -1$ means that
\begin{align}
\label{99.54}
\Vert &(B_{D,\varepsilon} + Q_0^\varepsilon )^{-1}-(B_D^0 + \overline{Q_0})^{-1} - \varepsilon K_D(\varepsilon;-1)
\Vert _{L_2(\mathcal{O})\rightarrow H^1(\mathcal{O}')}
\leqslant
\varepsilon  \bigl(C_{24}' \delta^{-1}  + C_{24}''\bigr)
\end{align}
for $0<\varepsilon\leqslant\varepsilon _1$. We apply identity \eqref{A sec. 9}. The first term $\mathcal{L}_1(\varepsilon ;\zeta )$ satisfies
\begin{equation*}
\begin{split}
\Vert  \mathcal{L}_1(\varepsilon ;\zeta ) \Vert _{L_2(\mathcal{O})\rightarrow H^1(\mathcal{O}')}
&\leqslant
\Vert (B_{D,\varepsilon} + Q_0^\varepsilon )^{-1}-(B_D^0 + \overline{Q_0})^{-1} - \varepsilon K_D(\varepsilon;-1)
\Vert _{L_2(\mathcal{O})\rightarrow H^1(\mathcal{O}')}
\\
&\times \Vert (B_{D}^0  + \overline{Q_0}  ) (B_D^0 - \zeta \overline{Q_0})^{-1} \Vert _{L_2(\mathcal{O})\rightarrow L_2(\mathcal{O})}.
\end{split}
\end{equation*}
Combining this with \eqref{sup (x+1)/|x-zeta|},  \eqref{tozd prav}, and \eqref{99.54}, we obtain
\begin{equation}
\label{99.55}
\Vert  \mathcal{L}_1(\varepsilon ;\zeta ) \Vert _{L_2(\mathcal{O})\rightarrow H^1(\mathcal{O}')}
\leqslant
\bigl( \gamma_{38} \delta^{-1} + \gamma_{39} \bigr) \varepsilon \varrho _\flat (\zeta)^{1/2}, \quad 0<\varepsilon\leqslant\varepsilon _1,
\end{equation}
where $\gamma_{38} := C_{24}' \|f\|_{L_\infty} \|f^{-1} \|_{L_\infty} (c_\flat +2)$ and
$\gamma_{39} := C_{24}'' \|f\|_{L_\infty} \|f^{-1} \|_{L_\infty} (c_\flat +2)$.

As a result, relations \eqref{A sec. 9}, \eqref{***.3}, \eqref{L3<=}, and \eqref{99.55} imply
\eqref{Th Dr appr strictly interior subdomain solutions} with $C_{32}' = \gamma_{38}$ and $C_{32}'' = \gamma_{33} + \gamma_{34} + \gamma_{39}$.
Estimate \eqref{Th Dr appr strictly interior subdomain fluxes} is deduced from \eqref{Th Dr appr strictly interior subdomain solutions}
by analogy with \eqref{3.19 from draft}--\eqref{3.22 from draft}.
Instead of \eqref{tilde u_0 in H2}, we use \eqref{9.48}.
\end{proof}

\subsection{Removal of $S_\varepsilon$ in approximations in a strictly interior subdomain}

\begin{theorem}
\label{Theorem Dr appr strictly interior subdomain no S_eps}
Suppose that the assumptions of Theorem~\textnormal{\ref{Theorem Dr appr strictly interior subdomain}} hold.
 Suppose also that Conditions~\textnormal{\ref{Condition Lambda in L infty}} and~\textnormal{\ref{Condition tilde Lambda in Lp}} are satisfied.
 Let $K^0_D(\varepsilon;\zeta)$ and $G^0_D(\varepsilon;\zeta)$ be given by~\eqref{K_D^0} and \eqref{G3(eps;zeta)}, respectively.
Then for $0<\varepsilon\leqslant \varepsilon _\flat$ and $\zeta\in\mathbb{C}\setminus [c_\flat ,\infty)$ we have
\begin{align}
\label{Th Dr appr st int no S-eps sec. 3-1}
\begin{split}
\Vert &(B_{D,\varepsilon}-\zeta Q_0^\varepsilon )^{-1}-(B_D^0-\zeta \overline{Q_0})^{-1} - \varepsilon K^0_D(\varepsilon;\zeta) \Vert _{L_2(\mathcal{O})\rightarrow H^1(\mathcal{O}')}
\\
&\leqslant \varepsilon \bigl( C_{32}'\delta ^{-1} \varrho _\flat (\zeta )^{1/2} + C_{33} |1+\zeta|^{1/2} \varrho _\flat (\zeta ) \bigr),
 \end{split}
\\
\label{Th Dr appr st int no S-eps sec. 3-2}
\begin{split}
\Vert &g^\varepsilon b(\mathbf{D})(B_{D,\varepsilon}-\zeta Q_0^\varepsilon )^{-1}-G^0_D(\varepsilon ;\zeta )\Vert _{L_2(\mathcal{O})\rightarrow L_2(\mathcal{O}')}
\\
&\leqslant \varepsilon \bigl( \widetilde{C}_{32}'\delta ^{-1} \varrho _\flat (\zeta )^{1/2}
+ \widetilde{C}_{33} |1+\zeta|^{1/2} \varrho _\flat (\zeta ) \bigr).
\end{split}
\end{align}
Here the constants $C_{32}'$ and $\widetilde{C}_{32}'$ are as in~\eqref{Th Dr appr strictly interior subdomain solutions}, \eqref{Th Dr appr strictly interior subdomain fluxes}.
The constants $C_{33}$ and $\widetilde{C}_{33}$ depend on the initial data~\eqref{problem data}, the domain~$\mathcal{O}$, and also on
$p$ and the norms $\Vert \Lambda\Vert _{L_\infty}$, $\Vert \widetilde{\Lambda}\Vert _{L_p(\Omega)}$.
\end{theorem}

\begin{proof}
Combining Lemma~\ref{Lemma Lambda (S-I)}, Lemma~\ref{Lemma tilde Lambda(S-I)}, and relations \eqref{K_D(eps,zeta)}, \eqref{9.48}, \eqref{Th Dr appr strictly interior subdomain solutions}, we arrive at estimate~\eqref{Th Dr appr st int no S-eps sec. 3-1} with $C_{33}: =C_{32}''+(\mathfrak{C}_\Lambda +\mathfrak{C}_{\widetilde{\Lambda}})\gamma  _{36}.$

Inequality \eqref{Th Dr appr st int no S-eps sec. 3-2} is deduced from \eqref{Th Dr appr st int no S-eps sec. 3-1}.
By analogy with \eqref{ex 9.45}, using \eqref{9-star}, we obtain estimate \eqref{Th Dr appr st int no S-eps sec. 3-2} with  $\widetilde{C}_{33}: =(d\alpha _1)^{1/2}\Vert g\Vert _{L_\infty}C_{33}+\gamma _{38} {\mathcal C}_4$.
\end{proof}

\section{More results\label{more}}

In the present section, for $\mathrm{Re}\,\zeta >0$, we show that
the estimates of Theorems \ref{Theorem Dirichlet L2}, \ref{Theorem Dirichlet H1}, and \ref{Theorem O'} can be "improved"; this concerns the behavior of the right-hand sides with respect to $\phi=\mathrm{arg}\,\zeta$.
However, an extra "bad term" (with respect to $|\zeta|$) appears; in the case where $Q_0(\mathbf{x})$ is constant, this term is equal to zero, and we obtain the "real improvement".
The method is based on the identities for generalized resolvents from Section~\ref{Section Another approximation}.
Due to these identities, we transfer the already proven estimates from the left half-plane to the symmetric point of the right one.

Also, we obtain some new versions of estimates for the fluxes.

\subsection{Estimates in the $(L_2\rightarrow L_2) $- and $(L_2\rightarrow H^1)$-operator norms}

\begin{theorem}
\label{Theorem improvement L2, H1}
Under the assumptions of Theorem \textnormal{\ref{Theorem Dirichlet H1}}, for $0<\varepsilon\leqslant\varepsilon_1$, $\zeta\in\mathbb{C}\setminus\mathbb{R}_+$, $\vert\zeta\vert\geqslant 1$, and $\mathrm{Re}\,\zeta \geqslant 0$, we have
\begin{align}
\label{Th L2 improvement}
\Vert & (B_{D,\varepsilon}-\zeta Q_0^\varepsilon)^{-1}-(B_D^0-\zeta \overline{Q_0})^{-1}\Vert _{L_2(\mathcal{O})\rightarrow L_2(\mathcal{O})}
\leqslant C_{34} c(\phi)^2\varepsilon\vert \zeta\vert ^{-1/2}+C_{35} c(\phi)^2\varepsilon ,
\\
\label{Th H1 improvement}
\begin{split}
\Vert &(B_{D,\varepsilon}-\zeta Q_0^\varepsilon)^{-1}-(B_D^0-\zeta \overline{Q_0})^{-1}-\varepsilon K_D(\varepsilon ;\zeta)\Vert _{L_2(\mathcal{O})\rightarrow H^1(\mathcal{O})}\\
&\leqslant C_{36}\left(c(\phi )\varepsilon ^{1/2}\vert \zeta\vert ^{-1/4}+c(\phi)^2\varepsilon\right)
+C_{37}(\mathrm{Re}\,\zeta )^{1/2} c(\phi)^2\varepsilon ,
\end{split}
\\
\label{Th fluxes improvement}
\begin{split}
\Vert & g^\varepsilon b(\mathbf{D})(B_{D,\varepsilon}-\zeta Q_0^\varepsilon )^{-1}-G_D(\varepsilon ;\zeta)\Vert _{L_2(\mathcal{O})\rightarrow L_2(\mathcal{O})}\\
&\leqslant \widetilde{C}_{36}\left(c(\phi )\varepsilon ^{1/2}\vert \zeta\vert ^{-1/4}+c(\phi)^2\varepsilon\right)
+\widetilde{C}_{37}(\mathrm{Re}\,\zeta )^{1/2} c(\phi)^2\varepsilon .
\end{split}
\end{align}
The constants $C_{34}$, $C_{35}$, $C_{36}$, $C_{37}$, $\widetilde{C}_{36}$, and $\widetilde{C}_{37}$ depend only on the initial data \eqref{problem data} and the domain $\mathcal{O}$. If the matrix-valued function $Q_0(\mathbf{x})$ is constant, then $C_{35}=C_{37}=\widetilde{C}_{37}=0$.
\end{theorem}

\begin{proof}
Let $\zeta = \mathrm{Re}\,\zeta + i \mathrm{Im}\,\zeta$, $\mathrm{Re}\,\zeta \ge 0$, $\mathrm{Im}\,\zeta \ne 0$.
Let $\widehat{\zeta}=-\mathrm{Re}\,\zeta +i \mathrm{Im}\,\zeta$. Then
$|\widehat{\zeta}| = |{\zeta}|$ and $c(\widehat{\phi})=1$, where $\widehat{\phi}=\mathrm{arg}\,\widehat{\zeta}$. According to \eqref{Th L2},
\begin{equation}
\label{ogin}
\Vert (B_{D,\varepsilon}-\widehat{\zeta}Q_0^\varepsilon)^{-1}
-(B_D^0-\widehat{\zeta}\overline{Q_0})^{-1}
\Vert _{L_2(\mathcal{O})\rightarrow L_2(\mathcal{O})}
\leqslant C_4\varepsilon \vert {\zeta}\vert ^{-1/2}.
\end{equation}

Similarly to \eqref{tozd gen res}, we have
\begin{equation}
\label{tri}
\begin{split}
&(B_{D,\varepsilon}-\zeta Q_0^\varepsilon )^{-1}-(B_D^0-\zeta\overline{Q_0})^{-1}
\\
&=(B_{D,\varepsilon}-\zeta Q_0^\varepsilon )^{-1}(B_{D,\varepsilon}-\widehat{\zeta}Q_0^\varepsilon)
\left(
(B_{D,\varepsilon}-\widehat{\zeta}Q_0^\varepsilon)^{-1}
-(B_D^0 -\widehat{\zeta}\overline{Q_0})^{-1}
\right)
(B_D^0-\widehat{\zeta}\overline{Q_0})(B_D^0-\zeta \overline{Q_0})^{-1}
\\
&+(\zeta -\widehat{\zeta})(B_{D,\varepsilon}-\zeta Q_0^\varepsilon )^{-1}(Q_0^\varepsilon -\overline{Q_0})(B_D^0-\zeta\overline{Q_0})^{-1}.
\end{split}
\end{equation}
Denote the consecutive terms in the right-hand side of \eqref{tri} by $\mathcal{J}_1(\varepsilon ;\zeta)$ and $\mathcal{J}_2(\varepsilon ;\zeta)$.
By \eqref{ogin} and the analogs of \eqref{tozd lev} and \eqref{tozd prav},
\begin{equation}
\label{chetyre}
\Vert \mathcal{J}_1(\varepsilon ;\zeta)\Vert _{L_2(\mathcal{O})\rightarrow L_2(\mathcal{O})}
\leqslant C_4\varepsilon\vert\zeta\vert ^{-1/2}\Vert f\Vert ^2_{L_\infty}\Vert f^{-1}\Vert ^2_{L_\infty}\sup _{x\geqslant 0}\frac{\vert x-\widehat{\zeta}\vert ^2}{\vert x-\zeta\vert ^2}.
\end{equation}
The computation shows that
\begin{equation}
\label{chtyre.a}
\begin{split}
\sup _{x \geqslant 0}\frac{\vert x -\widehat{\zeta}\vert }{\vert x -\zeta\vert }
\leqslant 2c(\phi).
\end{split}
\end{equation}
Estimates \eqref{chetyre} and \eqref{chtyre.a} imply the inequality
\begin{equation}
\label{piat'}
\Vert \mathcal{J}_1(\varepsilon ;\zeta)\Vert _{L_2(\mathcal{O})\rightarrow L_2(\mathcal{O})}
\leqslant C_{34} c(\phi)^2\varepsilon\vert \zeta\vert ^{-1/2};\quad C_{34}:=4C_4\Vert f\Vert ^2_{L_\infty}\Vert f^{-1}\Vert ^2_{L_\infty}.
\end{equation}
Since $\zeta-\widehat{\zeta}=2\mathrm{Re}\,\zeta$, similarly to \eqref{9.12} and \eqref{dvoistvennost}, taking
\eqref{lemma Q_eps -overline Q} into account, we obtain
\begin{equation*}
\Vert \mathcal{J}_2(\varepsilon ;\zeta)\Vert _{L_2(\mathcal{O})\rightarrow L_2(\mathcal{O})}
\leqslant
2(\mathrm{Re}\,\zeta) C_{Q_0}\varepsilon\Vert (B_{D,\varepsilon}-\zeta ^* Q_0^\varepsilon)^{-1}\Vert _{L_2(\mathcal{O})\rightarrow H^1(\mathcal{O})}\Vert (B_D^0-\zeta\overline{Q_0})^{-1}\Vert _{L_2(\mathcal{O})\rightarrow H^1(\mathcal{O})}.
\end{equation*}
Together with Lemmas \ref{Lemma estimates u_D,eps} and \ref{Lemma hom problem estimates}, this yields
\begin{equation}
\label{shest'}
\begin{split}
\Vert \mathcal{J}_2(\varepsilon ;\zeta)\Vert _{L_2(\mathcal{O})\rightarrow L_2(\mathcal{O})}
&\leqslant C_{35} c(\phi)^2\varepsilon ;\quad C_{35} := 2C_{Q_0}(\mathcal{C}_1+\Vert Q_0^{-1}\Vert _{L_\infty})^2 .
\end{split}
\end{equation}
Combining \eqref{tri}, \eqref{piat'} and \eqref{shest'}, we arrive at estimate \eqref{Th L2 improvement}.

Now we proceed to the proof of estimate \eqref{Th H1 improvement}.
We apply \eqref{Th L2->H1} at the point $\widehat{\zeta}$:
\begin{equation}
\label{A}
\Vert (B_{D,\varepsilon}- \widehat{\zeta} Q_0^\varepsilon )^{-1}-(B_D^0-\widehat{\zeta}\overline{Q_0})^{-1}-\varepsilon K_D(\varepsilon ;\widehat{\zeta})\Vert _{L_2(\mathcal{O})\rightarrow H^1(\mathcal{O})}
\leqslant
C_5\varepsilon ^{1/2}\vert {\zeta}\vert ^{-1/4}+C_6\varepsilon .
\end{equation}

Similarly to \eqref{A sec. 9}, we have
\begin{equation}
\label{sem'}
\begin{split}
(&B_{D,\varepsilon}-\zeta Q_0^\varepsilon )^{-1}-(B_D^0-\zeta\overline{Q_0})^{-1}-\varepsilon K_D(\varepsilon ;\zeta)
\\
&=\bigl((B_{D,\varepsilon}-\widehat{\zeta} Q_0^\varepsilon )^{-1}-(B_D^0-\widehat{\zeta}\overline{Q_0})^{-1}-\varepsilon  K_D(\varepsilon ;\widehat{\zeta})
\bigr)(B_D^0 -\widehat{\zeta}\overline{Q_0})(B_D^0-\zeta\overline{Q_0})^{-1}
\\
&+(\zeta -\widehat{\zeta})(B_{D,\varepsilon}-\zeta Q_0^\varepsilon )^{-1} Q_0^\varepsilon
\bigl((B_{D,\varepsilon}-\widehat{\zeta} Q_0^\varepsilon )^{-1}-(B_D^0-\widehat{\zeta}\overline{Q_0})^{-1}
\bigr) (B_D^0 -\widehat{\zeta}\overline{Q_0})(B_D^0-\zeta\overline{Q_0})^{-1}
\\
&+(\zeta -\widehat{\zeta})(B_{D,\varepsilon}-\zeta Q_0^\varepsilon )^{-1}(Q_0^\varepsilon -\overline{Q_0})(B_D^0-\zeta \overline{Q_0})^{-1}.
\end{split}
\end{equation}
Denote the consecutive summands in the right-hand side of \eqref{sem'} by $\mathfrak{L} _1(\varepsilon ;\zeta)$, $\mathfrak{L} _2(\varepsilon ;\zeta)$, and $\mathfrak{L} _3(\varepsilon ;\zeta)$.
(Note that $\mathfrak{L} _3(\varepsilon ;\zeta)$ coincides with $\mathcal{J}_2(\varepsilon ;\zeta)$.)
Similarly to \eqref{tozd prav}, by \eqref{chtyre.a},
\begin{equation}
\label{10.13a}
\Vert (B_D^0-\widehat{\zeta}\overline{Q_0})(B_D^0-\zeta\overline{Q_0})^{-1}\Vert _{L_2(\mathcal{O})\rightarrow L_2(\mathcal{O})}\leqslant 2\Vert f\Vert _{L_\infty}\Vert f^{-1}\Vert _{L_\infty}c(\phi).
\end{equation}
So, by analogy with \eqref{9.29a new}, taking \eqref{A} and \eqref{10.13a} into account, we have
\begin{equation}
\label{L1 start}
\Vert  \mathfrak{L}_1(\varepsilon;\zeta )\Vert _{L_2(\mathcal{O})\rightarrow H^1(\mathcal{O})}
\leqslant \gamma _{40} c(\phi)\varepsilon ^{1/2}\vert {\zeta}\vert ^{-1/4}+\gamma_{41}c(\phi)\varepsilon,
\end{equation}
where $\gamma _{40} :=2 C_5\Vert f\Vert _{L_\infty}\Vert f^{-1}\Vert _{L_\infty}$ and
$\gamma _{41} :=2 C_6\Vert f\Vert _{L_\infty}\Vert f^{-1}\Vert _{L_\infty}$.

Similarly to \eqref{***.33}, using Lemma \ref{Lemma estimates u_D,eps} and relations \eqref{ogin}, \eqref{10.13a}, we obtain
\begin{equation}
\label{desiat'}
\begin{split}
\Vert \mathfrak{L} _2(\varepsilon ;\zeta)\Vert _{L_2(\mathcal{O})\rightarrow H^1(\mathcal{O})}
\leqslant \gamma_{42} c(\phi)^2 \varepsilon;
\quad \gamma_{42} := 4 C_4 ( {\mathcal C}_1 + \|Q_0^{-1}\|_{L_\infty})\Vert f\Vert _{L_\infty}\Vert f^{-1}\Vert^3 _{L_\infty}.
\end{split}
\end{equation}

The term $\mathfrak{L}_3(\varepsilon ;\zeta)$ is estimated by using Lemma \ref{Lemma hom problem estimates} and
\eqref{lemma Q_eps -overline Q} (cf. \eqref{9.31a new}--\eqref{L3<=}):
\begin{equation}
\label{piatnadtsat'}
\begin{split}
\Vert &\mathfrak{L}_3(\varepsilon ;\zeta)\Vert _{L_2 \rightarrow H^1}
\leqslant
2(\mathrm{Re}\,\zeta) C_{Q_0}\varepsilon\Vert (B_{D,\varepsilon}-\zeta Q_0^\varepsilon )^{-1}\Vert_{H^{-1} \rightarrow H^1}
\Vert (B_D^0-\zeta \overline{Q_0})^{-1}\Vert _{L_2 \rightarrow H^1}
\\
&\leqslant
2\varepsilon (\mathrm{Re}\,\zeta)  C_{Q_0}(\mathcal{C}_1+\Vert Q_0^{-1}\Vert _{L_\infty})c(\phi)\vert \zeta\vert ^{-1/2}
c_4^2\sup _{x \geqslant 0}\frac{x}{\vert x-\zeta ^*\vert}
\leqslant C_{37}c(\phi)^2\varepsilon (\mathrm{Re}\,\zeta )^{1/2}.
\end{split}
\end{equation}
Here $C_{37} := 2c_4^2C_{Q_0}(\mathcal{C}_1+\Vert Q_0^{-1}\Vert_{L_\infty})$.

As a result, relations  \eqref{sem'} and \eqref{L1 start}--\eqref{piatnadtsat'} imply estimate \eqref{Th H1 improvement} with the constant $C_{36}:=\max\lbrace \gamma_{40}; \gamma_{41}+\gamma _{42}\rbrace$.
Estimate \eqref{Th fluxes improvement} is deduced from \eqref{3.20 from draft}, \eqref{4-th chlen in tozd for fluxes},
  \eqref{3.22 from draft}, and \eqref{Th H1 improvement}.
\end{proof}

\subsection{Removal of $S_\varepsilon$}

\begin{theorem}
\label{Theorem L2->H1 improvement special cases}
Suppose that the assumptions of Theorem \textnormal{\ref{Theorem Dirichlet H1}} are satisfied.
 Suppose also that Conditions \textnormal{\ref{Condition Lambda in L infty}} and \textnormal{\ref{Condition tilde Lambda in Lp}} hold. Let $K^0_D(\varepsilon;\zeta)$ and $G^0_D(\varepsilon;\zeta)$ be defined by \eqref{K_D^0} and \eqref{G3(eps;zeta)},
 respectively. Then for $0<\varepsilon\leqslant\varepsilon _1$ and $\zeta\in\mathbb{C}\setminus\mathbb{R}_+$, $\vert\zeta\vert\geqslant 1$, $\mathrm{Re}\,\zeta\geqslant 0$, we have
\begin{align}
\label{10.20a}
\begin{split}
\Vert &(B_{D,\varepsilon}-\zeta Q_0^\varepsilon )^{-1}-(B_D^0-\zeta\overline{Q_0})^{-1}
- \varepsilon K^0_D(\varepsilon;\zeta) \Vert _{L_2(\mathcal{O})\rightarrow H^1(\mathcal{O})}
\\
&\leqslant
C_{38}\left( c(\phi)\varepsilon ^{1/2}\vert\zeta\vert ^{-1/4}+ c(\phi)^2\varepsilon \right)
+C_{37}(\mathrm{Re}\,\zeta )^{1/2} c(\phi)^2\varepsilon,
\end{split}
\\
\label{10.20b}
\begin{split}
\Vert & g^\varepsilon b(\mathbf{D})(B_{D,\varepsilon}-\zeta Q_0^\varepsilon )^{-1}-G^0_D(\varepsilon;\zeta)\Vert _{L_2(\mathcal{O})\rightarrow L_2(\mathcal{O})}
\\
&\leqslant
\widetilde{C}_{38}\left( c(\phi)\varepsilon ^{1/2}\vert\zeta\vert ^{-1/4}+ c(\phi)^2\varepsilon\right)
+\widetilde{C}_{37}(\mathrm{Re}\,\zeta )^{1/2} c(\phi)^2\varepsilon .
\end{split}
\end{align}
The constants $C_{37}$ and $\widetilde{C}_{37}$ are the same as in Theorem \textnormal{\ref{Theorem improvement L2, H1}}.
The constants $C_{38}$ and $\widetilde{C}_{38}$ depend only on the initial data \eqref{problem data}, the domain $\mathcal{O}$, on $p$ and the norms $\Vert\Lambda\Vert _{L_\infty}$ and $\Vert\widetilde{\Lambda}\Vert _{L_p(\Omega)}$.
\end{theorem}

\begin{proof}
The proof is similar to that of Theorem \ref{Theorem no S-eps}. To obtain \eqref{10.20a},
we use estimates \eqref{d-vo Th 7/6 1}, \eqref{d-vo Th 7/6 2}, and \eqref{Th H1 improvement}. By analogy with \eqref{d-vo Th 7/6 3}--\eqref{d-vo Th 7/6 5a}, estimate \eqref{10.20b} is deduced from \eqref{10.20a}.
 We omit the details.
\end{proof}

\subsection{Special case}
Similarly to Proposition~~\ref{Proposition tilde Lambda =0 in Sec. 7},
the following statement can be deduced from Theorem \ref{Theorem L2->H1 improvement special cases}.

\begin{proposition}
\label{Proposition tilde Lambda =0 improvement}
Suppose that the assumptions of Theorem~\textnormal{\ref{Theorem Dirichlet L2}} are satisfied.
Assume that relations~\eqref{underline-g} and~\eqref{sum Dj aj =0} hold.
Then for $0<\varepsilon\leqslant \varepsilon_1$, $\zeta\in\mathbb{C}\setminus\mathbb{R}_+$, $\vert\zeta\vert\geqslant 1$, $\mathrm{Re}\,\zeta\geqslant 0$, we have
\begin{equation*}
\begin{split}
\Vert &g^\varepsilon b(\mathbf{D})(B_{D,\varepsilon}-\zeta Q_0^\varepsilon )^{-1}-g^0b(\mathbf{D})(B_D^0-\zeta\overline{Q_0})^{-1}\Vert _{L_2(\mathcal{O})\rightarrow L_2(\mathcal{O})}
\\
&\leqslant
\widetilde{C}_{38}\left(c(\phi)\varepsilon^{1/2}\vert\zeta\vert ^{-1/4}+ c(\phi)^2\varepsilon\right)
+\widetilde{C}_{37}(\mathrm{Re}\,\zeta )^{1/2}c(\phi)^2\varepsilon .
\end{split}
\end{equation*}
\end{proposition}

\subsection{Estimates in a strictly interior subdomain}

\begin{theorem}
\label{Theorem O' improvement}
Under the assumptions of Theorem \textnormal{\ref{Theorem O'}}, for $0<\varepsilon\leqslant\varepsilon _1$ and $\zeta\in\mathbb{C}\setminus\mathbb{R}_+$, $\vert\zeta\vert\geqslant 1$, $\mathrm{Re}\,\zeta\geqslant 0$, we have
\begin{align}
\label{Th O' improvement}
\begin{split}
\Vert &(B_{D,\varepsilon}-\zeta Q_0^\varepsilon )^{-1}-(B_D^0-\zeta\overline{Q_0})^{-1}
- \varepsilon K^0_D(\varepsilon;\zeta) \Vert _{L_2(\mathcal{O})\rightarrow H^1(\mathcal{O}')}
\\
&\leqslant
C_{39} \varepsilon  \left( \delta^{-1} c(\phi)   \vert\zeta\vert ^{-1/2}+ c(\phi)^2 \right)
+C_{37}(\mathrm{Re}\,\zeta )^{1/2} c(\phi)^2\varepsilon,
\end{split}
\\
\label{Th O' fluxes improvement}
\begin{split}
\Vert & g^\varepsilon b(\mathbf{D})(B_{D,\varepsilon}-\zeta Q_0^\varepsilon )^{-1}-G^0_D(\varepsilon;\zeta)\Vert _{L_2(\mathcal{O})\rightarrow L_2(\mathcal{O}')}
\\
&\leqslant
\widetilde{C}_{39} \varepsilon  \left( \delta^{-1} c(\phi)   \vert\zeta\vert ^{-1/2}+ c(\phi)^2 \right)
+\widetilde{C}_{37} (\mathrm{Re}\,\zeta )^{1/2} c(\phi)^2\varepsilon .
\end{split}
\end{align}
Here the constants $C_{37}$ and $\widetilde{C}_{37}$ are the same as in Theorem \textnormal{\ref{Theorem improvement L2, H1}}.
The constants $C_{39}$ and $\widetilde{C}_{39}$ depend only on the initial data \eqref{problem data} and the domain $\mathcal{O}$.
\end{theorem}

\begin{proof}
Let ${\zeta}=\mathrm{Re}\,\zeta +i \mathrm{Im}\,\zeta$, $\mathrm{Re}\,\zeta \ge 0$, $\mathrm{Im}\,\zeta \ne 0$,
and $|{\zeta}| \ge 1$. Let $\widehat{\zeta}=-\mathrm{Re}\,\zeta +i \mathrm{Im}\,\zeta$.
Estimate \eqref{Th O'} at the point $\widehat{\zeta}$ means that
\begin{equation}
\label{dvadsat' dva}
\begin{split}
\bigl\Vert
(B_{D,\varepsilon}-\widehat{\zeta}Q_0^\varepsilon)^{-1}-(B_D^0-\widehat{\zeta}\overline{Q_0})^{-1}-\varepsilon K_D(\varepsilon;\widehat{\zeta})
\bigr\Vert _{L_2(\mathcal{O}) \to H^1(\mathcal{O}')}
 \leqslant \varepsilon \bigl( C_{24}' |\zeta|^{-1/2} \delta^{-1} + C_{24}'' \bigr)
\end{split}
\end{equation}
for $0< \varepsilon \leqslant \varepsilon_1$. Next, we use identity \eqref {sem'}. By  \eqref{10.13a} and \eqref{dvadsat' dva},
\begin{equation}
\label{L1}
\Vert  \mathfrak{L}_1(\varepsilon;\zeta )\Vert _{L_2(\mathcal{O})\rightarrow H^1(\mathcal{O}')}
\leqslant 2 \varepsilon \bigl( C_{24}' |\zeta|^{-1/2} \delta^{-1} + C_{24}'' \bigr) \| f \|_{L_\infty} \| f^{-1} \|_{L_\infty} c(\phi).
\end{equation}
Relations \eqref{desiat'}, \eqref{piatnadtsat'}, and \eqref{L1} imply \eqref{Th O' improvement} with
$C_{39} := \max \{ 2 C_{24}' \|f\|_{L_\infty}  \|f^{-1}\|_{L_\infty}; 2 C_{24}'' \|f\|_{L_\infty}  \|f^{-1}\|_{L_\infty} + \gamma_{42} \}$.

Approximation \eqref{Th O' fluxes improvement} for the flux follows from \eqref{3.20 from draft}, \eqref{4-th chlen in tozd for fluxes}, \eqref{3.22 from draft}, and \eqref{Th O' improvement}.
\end{proof}

The following result is deduced from Theorem \ref{Theorem O' improvement} similarly to the proof of Theorem~\ref{Theorem O' no S-eps}.

\begin{theorem}
\label{Theorem O' improvement no S-eps}
Suppose that the assumptions of Theorem \textnormal{\ref{Theorem O'}} are satisfied.
 Suppose also that Conditions \textnormal{\ref{Condition Lambda in L infty}} and \textnormal{\ref{Condition tilde Lambda in Lp}} hold. Let $K^0_D(\varepsilon ;\zeta)$ and $G^0_D(\varepsilon ;\zeta)$ be given by~\eqref{K_D^0} and \eqref{G3(eps;zeta)}, respectively. Then for $\zeta\in\mathbb{C}\setminus\mathbb{R}_+$, $\vert \zeta\vert \geqslant 1$, $\mathrm{Re}\,\zeta\geqslant 0$, and $0<\varepsilon\leqslant\varepsilon _1$ we have
\begin{align*}
\begin{split}
\Vert &(B_{D,\varepsilon}-\zeta Q_0^\varepsilon )^{-1}- (B_D^0-\zeta \overline{Q_0})^{-1}
- \varepsilon K^0_D(\varepsilon ;\zeta) \Vert _{L_2(\mathcal{O})\rightarrow H^1(\mathcal{O}')}
\\
&\leqslant C_{40}c(\phi)^2\varepsilon (\vert\zeta\vert ^{-1/2}\delta ^{-1}+1)
+C_{37}c(\phi)^2\varepsilon (\mathrm{Re}\,\zeta )^{1/2},
\end{split}
\\
\begin{split}
\Vert &g^\varepsilon b(\mathbf{D})(B_{D,\varepsilon}-\zeta Q_0^\varepsilon )^{-1}-G^0_D(\varepsilon ;\zeta )\Vert _{L_2(\mathcal{O})\rightarrow L_2(\mathcal{O}')}
\\
&\leqslant \widetilde{C}_{40}c(\phi)^2\varepsilon (\vert\zeta\vert ^{-1/2}\delta ^{-1}+1)
+\widetilde{C}_{37}c(\phi)^2\varepsilon (\mathrm{Re}\,\zeta )^{1/2}.
\end{split}
\end{align*}
The constants $C_{37}$ and $\widetilde{C}_{37}$ are the same as in Theorem~\textnormal{\ref{Theorem improvement L2, H1}}.
The constants $C_{40}$ and $\widetilde{C}_{40}$ depend only on the initial data \eqref{problem data},
the domain~$\mathcal{O}$, and also on $p$ and the norms $\Vert \Lambda\Vert _{L_\infty}$, $\Vert \widetilde{\Lambda}\Vert _{L_p(\Omega)}$.
\end{theorem}

\subsection{Approximation for the flux}
We obtain some new versions of estimates for the flux.

\begin{proposition}
\label{Propositiom sec. 10 eps 1/2}
Under the assumptions of Theorem \textnormal{\ref{Theorem Dirichlet H1}}, for $\zeta\in\mathbb{C}\setminus\mathbb{R}_+$, $\vert\zeta\vert\geqslant 1$, $\mathrm{Re}\,\zeta \geqslant 0$,
and $0<\varepsilon\leqslant\varepsilon _1$ we have
\begin{align}
\label{Prop sec 10 2}
\begin{split}
\Vert & g^\varepsilon b(\mathbf{D})(B_{D,\varepsilon}-\zeta Q_0^\varepsilon )^{-1}- G_D(\varepsilon ;\zeta)\Vert _{L_2(\mathcal{O})\rightarrow L_2(\mathcal{O})}
\leqslant {C}_{41}c(\phi )^{3/2}\varepsilon ^{1/2}\vert\zeta\vert ^{-1/4}+ {C}_{42} c(\phi )^{3/2}\varepsilon ^{1/2}.
\end{split}
\end{align}
The constants ${C}_{41}$ and ${C}_{42}$ depend only on the initial data \eqref{problem data} and the domain $\mathcal{O}$. If the matrix-valued function $Q_0(\mathbf{x})$ is constant, then ${C}_{42}=0$.
\end{proposition}

\begin{proof}
We start with a rough estimate for the left-hand side of \eqref{Prop sec 10 2}.
By \eqref{b_l <=} and \eqref{2.10b},
\begin{equation}
\label{zero}
\begin{split}
\Vert g^\varepsilon b(\mathbf{D})(B_{D,\varepsilon}-\zeta Q_0^\varepsilon )^{-1}\Vert _{L_2(\mathcal{O})\rightarrow L_2(\mathcal{O})}
\leqslant (d\alpha _1)^{1/2}\Vert g\Vert _{L_\infty} \mathcal{C}_1 c(\phi)\vert\zeta\vert ^{-1/2}.
\end{split}
\end{equation}

Next, using \eqref{<b^*b<}, \eqref{PO}, and Proposition \ref{Proposition f^eps S_eps}, we estimate the operator~$G_D(\varepsilon;\zeta)$:
\begin{equation}
\label{10.38}
\begin{split}
\Vert G_D(\varepsilon;\zeta)\Vert _{L_2(\mathcal{O})\rightarrow L_2(\mathcal{O})}
&\leqslant \vert\Omega\vert ^{-1/2}\Vert \widetilde{g}\Vert _{L_2(\Omega)}
\alpha _1^{1/2}C_\mathcal{O}^{(1)}\Vert (B_D^0-\zeta\overline{Q_0})^{-1}\Vert _{L_2(\mathcal{O})\rightarrow H^1(\mathcal{O})}
\\
&+\Vert g\Vert _{L_\infty}
\vert \Omega\vert ^{-1/2}\Vert b(\mathbf{D})\widetilde{\Lambda}\Vert _{L_2(\Omega)}
C_\mathcal{O}^{(0)}\Vert (B_D^0-\zeta\overline{Q_0})^{-1}\Vert _{L_2(\mathcal{O})\rightarrow L_2(\mathcal{O})}.
\end{split}
\end{equation}
By \eqref{tilde g} and \eqref{b(D)Lambda<=}, we have
$\vert\Omega\vert ^{-1/2}\Vert \widetilde{g}\Vert _{L_2(\Omega)}\leqslant \gamma _{43}:=
\Vert g\Vert _{L_\infty}\bigl( m^{1/2}\Vert g\Vert ^{1/2}_{L_\infty}\Vert g^{-1}\Vert ^{1/2}_{L_\infty}+1\bigr)$.
Together with Lemma \ref{Lemma hom problem estimates}, \eqref{b(D) tilde Lambda <=}, and \eqref{10.38},
this implies
\begin{align*}
&\Vert G_D(\varepsilon;\zeta)\Vert _{L_2(\mathcal{O})\rightarrow L_2(\mathcal{O})}
\leqslant \gamma _{44} c(\phi)\vert \zeta\vert ^{-1/2},\quad \zeta\in\mathbb{C}\setminus\mathbb{R}_+,\quad |\zeta|\ge 1, \quad 0<\varepsilon\leqslant 1;\\
\begin{split}
&\gamma _{44} :=\alpha _1^{1/2}C_\mathcal{O}^{(1)}\gamma _{43}\left(\Vert Q_0^{-1}\Vert _{L_\infty}+\mathcal{C}_1\right)
+ |\Omega|^{-1/2} C_\mathcal{O}^{(0)}\Vert Q_0^{-1}\Vert _{L_\infty}C_a n^{1/2}\alpha _0^{-1/2}\Vert g \Vert _{L_\infty} \Vert g^{-1}\Vert _{L_\infty}.
\end{split}
\end{align*}
Combining this with \eqref{zero}, we obtain
\begin{equation}
\label{grubo for fluxes}
\Vert g^\varepsilon b(\mathbf{D})(B_{D,\varepsilon}-\zeta Q_0^\varepsilon )^{-1}-G_D(\varepsilon;\zeta)\Vert _{L_2(\mathcal{O})\rightarrow L_2(\mathcal{O})}
\leqslant\gamma _{45} c(\phi)\vert\zeta\vert ^{-1/2}
\end{equation}
for $0<\varepsilon\leqslant 1$, $\zeta\in\mathbb{C}\setminus\mathbb{R}_+$, $|\zeta|\ge 1$. Here $\gamma_{45}:=(d\alpha _1)^{1/2}\Vert g\Vert _{L_\infty}\mathcal{C}_1+\gamma _{44}$.

By \eqref{Th fluxes improvement} and \eqref{grubo for fluxes}, we have
\begin{equation*}
\begin{split}
\Vert &g^\varepsilon b(\mathbf{D})(B_{D,\varepsilon}-\zeta Q_0^\varepsilon )^{-1}-G_D(\varepsilon;\zeta)\Vert _{L_2(\mathcal{O})\rightarrow L_2(\mathcal{O})}
\\
&\leqslant
\min
\lbrace \gamma _{45}c(\phi)\vert\zeta\vert ^{-1/2};\widetilde{C}_{36}(c(\phi)\varepsilon ^{1/2}\vert\zeta\vert ^{-1/4}+c(\phi)^2\varepsilon )
+\widetilde{C}_{37}(\mathrm{Re}\,\zeta )^{1/2}c(\phi)^2\varepsilon\rbrace
\\
&\leqslant
{C}_{41}c(\phi)^{3/2}\varepsilon ^{1/2}\vert\zeta\vert ^{-1/4}+ {C}_{42}c(\phi)^{3/2}\varepsilon ^{1/2},
\end{split}
\end{equation*}
where ${C}_{41}:=\widetilde{C}_{36}+(\gamma _{45}\widetilde{C}_{36})^{1/2}$ and ${C}_{42}:=(\gamma _{45} \widetilde{C}_{37})^{1/2}$. By Theorem \ref{Theorem improvement L2, H1}, if the matrix-valued function $Q_0(\mathbf{x})$ is constant, then $\widetilde{C}_{37}=0$, whence ${C}_{42}=0$.
\end{proof}

\begin{proposition}
\label{Proposition 10/11 flux}
Under the assumptions of Theorem \textnormal{\ref{Theorem Dirichlet H1}}, for $\zeta\in\mathbb{C}\setminus\mathbb{R}_+$, $\vert\zeta\vert\geqslant 1$, and $0<\varepsilon\leqslant\varepsilon _1$ we have
\begin{equation}
\label{fluxues new estimate}
\Vert g^\varepsilon b(\mathbf{D})(B_{D,\varepsilon}-\zeta Q_0^\varepsilon )^{-1}-G_D(\varepsilon;\zeta)\Vert _{L_2(\mathcal{O})\rightarrow L_2(\mathcal{O})}
\leqslant C_{43} c(\phi)^{5/2}\varepsilon ^{1/2}\vert\zeta\vert ^{-1/4}.
\end{equation}
The constant $C_{43}$ depends only on the initial data \eqref{problem data} and the domain $\mathcal{O}$.
\end{proposition}

\begin{proof}
Estimate \eqref{fluxues new estimate} can be checked similarly to \eqref{Prop sec 10 2} by using \eqref{2.41a} and \eqref{grubo for fluxes}. The constant $C_{43}$ is given by $C_{43}:=\widetilde{C}_5+(\gamma _{45}\widetilde{C}_6)^{1/2}$.
\end{proof}

\section{Applications of the general results}
\label{Section Examples}

\subsection{The scalar elliptic operator}
\label{sec.scalar case}
Let $n=1$, $m=d$, $b(\mathbf{D})=\mathbf{D}$, and let $g(\mathbf{x})$ be a $\Gamma$-periodic symmetric
$(d\times d)$-matrix-valued function \textit{with real entries}. Suppose that $g(\mathbf{x})>0$ and \hbox{$g,g^{-1}\in L_\infty(\mathbb{R}^d)$}. Obviously, condition \eqref{<b^*b<} holds with $\alpha _0=\alpha _1 =1$. We have $b(\mathbf{D})^*g^\varepsilon (\mathbf{x}) b(\mathbf{D})=-\mathrm{div}\,g^\varepsilon (\mathbf{x})\nabla$.
Next, let $\mathbf{A}(\mathbf{x})=\mathrm{col}\lbrace A_1(\mathbf{x}),\dots ,A_d(\mathbf{x})\rbrace$, where $A_j(\mathbf{x})$, $j=1,\dots,d$, are $\Gamma$-periodic real-valued functions such that
\begin{equation}
\label{A_j in L rho}
A_j\in L_\rho (\Omega),\quad\rho=2\ \mbox{for}\ d=1,\quad \rho >d\ \mbox{for}\ d\geqslant 2;\quad j=1,\dots ,d.
\end{equation}
Suppose that $v(\mathbf{x})$ and $\mathcal{V}(\mathbf{x})$ are real-valued $\Gamma$-periodic functions such that
\begin{equation}
\label{v,V condition}
v,\mathcal{V}\in L_s(\Omega),\quad s=1\ \mbox{for}\ d=1,\quad s>d/2\ \mbox{for}\ d\geqslant 2; \quad\int_\Omega v(\mathbf{x})\,d\mathbf{x}=0.
\end{equation}

In $L_2(\mathcal{O})$, we consider the operator $\mathfrak{B}_{D,\varepsilon}$ given formally by the differential expression
\begin{equation}
\label{mathfrak B_D,eps}
\mathfrak{B}_{D,\varepsilon}=(\mathbf{D}-\mathbf{A}^\varepsilon(\mathbf{x}))^*g^\varepsilon (\mathbf{x})(\mathbf{D}-\mathbf{A}^\varepsilon (\mathbf{x}))+\varepsilon ^{-1}v^\varepsilon (\mathbf{x})+\mathcal{V}^\varepsilon (\mathbf{x})
\end{equation}
with the Dirichlet condition on~$\partial\mathcal{O}$.
The precise definition of the operator~$\mathfrak{B}_{D,\varepsilon}$ is given in terms of the corresponding quadratic form.
The operator~\eqref{mathfrak B_D,eps} can be treated as the periodic Schr\"odinger operator with the metric~$g^\varepsilon$,
the magnetic potential $\mathbf{A}^\varepsilon$, and the electric potential~$\varepsilon ^{-1}v^\varepsilon +\mathcal{V}^\varepsilon$
containing the singular term $\varepsilon ^{-1}v^\varepsilon$.
It is easily seen (cf. \cite[Subsection~13.1]{SuAA}) that the operator~\eqref{mathfrak B_D,eps}
can be represented as
\begin{equation}
\label{mathfrak B_D,eps in other words}
\mathfrak{B}_{D,\varepsilon}=\mathbf{D}^*g^\varepsilon (\mathbf{x})\mathbf{D}+\sum _{j=1}^d\left(a_j^\varepsilon (\mathbf{x})D_j+D_j(a_j^\varepsilon (\mathbf{x}))^*\right) +Q^\varepsilon (\mathbf{x}).
\end{equation}
Here the real-valued function $Q(\mathbf{x})$ is given by
$Q(\mathbf{x})=\mathcal{V}(\mathbf{x})+\langle g(\mathbf{x})\mathbf{A}(\mathbf{x}),\mathbf{A}(\mathbf{x})\rangle$.
The complex-valued functions $a_j(\mathbf{x})$ are given by
$a_j(\mathbf{x})=-\eta _j(\mathbf{x})+i\xi _j(\mathbf{x})$, $j=1,\dots, d$,
where $\eta _j(\mathbf{x})$ are the components of the vector-valued function
$\boldsymbol{\eta}(\mathbf{x})=g(\mathbf{x})\mathbf{A}(\mathbf{x})$, and $\xi_j(\mathbf{x})$ are defined in terms of
  the $\Gamma$-periodic solution $\Phi (\mathbf{x})$ of the problem $\Delta \Phi(\mathbf{x})=v(\mathbf{x})$, $\int _\Omega \Phi(\mathbf{x})\,d\mathbf{x}=0$, by the relations $\xi _j (\mathbf{x})=-\partial _j \Phi (\mathbf{x})$.
  We have
$v(\mathbf{x})=-\sum _{j=1}^d\partial _j \xi _j(\mathbf{x})$.
It is easy to check that the functions $a_j$ satisfy condition~\eqref{a_j cond} with suitable  $\rho '$ depending on $\rho$ ans $s$, and the norms $\Vert a_j\Vert _{L_{\rho '}(\Omega)}$ are controlled in terms of $\Vert g\Vert _{L_\infty}$, $\Vert \mathbf{A}\Vert _{L_\rho (\Omega)}$, $\Vert v\Vert _{L_s(\Omega)}$, and the parameters of the lattice~$\Gamma$. (See \cite[Subsection~13.1]{SuAA}.)
The function~$Q$ satisfies condition~\eqref{Q condition} with suitable $s'=\min \lbrace s;\rho/2\rbrace$.

Suppose that $Q_0(\mathbf{x})$ is a positive definite and bounded $\Gamma$-periodic function.
We consider the positive definite operator
$\mathcal{B}_{D,\varepsilon}:=\mathfrak{B}_{D,\varepsilon}+\lambda Q_0^\varepsilon$.
Here we choose the constant $\lambda$ in accordance with condition~\eqref{lambda =} for the operator
with the coefficients $g$, $a_j$, $j=1,\dots,d$, $Q$, and $Q_0$ defined above.
We are interested in the behavior of the operator~$(\mathcal{B}_{D,\varepsilon}-\zeta Q_0^\varepsilon)^{-1}$, where $\zeta \in\mathbb{C}\setminus\mathbb{R}_+$. In the case under consideration, the initial data \eqref{problem data}
reduces to the following set
\begin{equation}
\label{problem data for the scalar operator}
\begin{split}
&d,\;\rho\;,s;\ \Vert g\Vert _{L_\infty},\ \Vert g^{-1}\Vert _{L_\infty},\ \Vert \mathbf{A}\Vert_{L_\rho (\Omega)}, \ \Vert v\Vert _{L_s(\Omega)},\ \Vert \mathcal{V}\Vert _{L_s(\Omega)},\\
&\Vert Q_0\Vert _{L_\infty},\  \Vert Q_0^{-1}\Vert _{L_\infty};\  \mbox{the parameters of the lattice}\ \Gamma.
\end{split}
\end{equation}

Let us describe the effective operator. The $\Gamma$-periodic solution of problem~\eqref{Lambda problem} is the row
$\Lambda (\mathbf{x})=i\Psi (\mathbf{x})$, $\Psi (\mathbf{x})=(\psi _1(\mathbf{x}),\dots ,\psi _d(\mathbf{x}))$,
where $\psi _j \in\widetilde{H}^1(\Omega)$ is the solution of the problem
$\mathrm{div}\,g(\mathbf{x})(\nabla \psi _j (\mathbf{x})+\mathbf{e}_j)=0$, $\int _\Omega \psi_j(\mathbf{x})\,d\mathbf{x}=0$.
Here $\mathbf{e}_j$, $j=1,\dots,d$, is the standard orthonormal basis in~$\mathbb{R}^d$. Clearly, the functions $\psi _j(\mathbf{x})$ are real-valued, while the entries of the row $\Lambda(\mathbf{x})$ are purely imaginary.
According to~\eqref{tilde g}, the columns of the $(d\times d)$-matrix-valued function $\widetilde{g}(\mathbf{x})$ are given by
$g(\mathbf{x})(\nabla \psi _j (\mathbf{x})+\mathbf{e}_j)$, $j=1,\dots,d$. The effective matrix is defined by the general rule~\eqref{g^0}: $g^0=\vert \Omega\vert ^{-1}\int _\Omega\widetilde{g}(\mathbf{x})\,d\mathbf{x}$. Clearly, the matrices $\widetilde{g}(\mathbf{x})$ and $g^0$ have real entries.

The periodic solution of problem~\eqref{tildeLambda_problem}
can be represented as $\widetilde{\Lambda}(\mathbf{x})=\widetilde{\Lambda}_1(\mathbf{x})+i\widetilde{\Lambda}_2(\mathbf{x})$, where
 the real-valued $\Gamma$-periodic functions~$\widetilde{\Lambda}_1(\mathbf{x})$ and~$\widetilde{\Lambda}_2(\mathbf{x})$
 are the solutions of the problems:
\begin{align*}
&-\div g(\mathbf{x})\nabla \widetilde{\Lambda}_1(\mathbf{x})+v(\mathbf{x})=0,\quad\int _\Omega\widetilde{\Lambda}_1(\mathbf{x})\,d\mathbf{x}=0;
\\
&-\div g(\mathbf{x})\nabla \widetilde{\Lambda}_2(\mathbf{x})+\div g(\mathbf{x})\mathbf{A}(\mathbf{x})=0,\quad\int_\Omega\widetilde{\Lambda}_2(\mathbf{x})\,d\mathbf{x}=0.
\end{align*}
The column $V$ (see \eqref{V}) can be written as~$V=V_1+iV_2$, where $V_1$ and $V_2$ are
defined by $V_1= \overline{\langle g\nabla\widetilde{\Lambda}_2, \nabla\Psi \rangle}$ and $V_2= - \overline{\langle g \nabla \widetilde{\Lambda}_1, \nabla\Psi \rangle}$.
Clearly, $V_1$ and $V_2$ have real entries.
According to~\eqref{W}, the constant~$W$ is given by
$W= \overline{ \langle g \nabla \widetilde{\Lambda}_1,\nabla\widetilde{\Lambda}_1\rangle} +
\overline{\langle g \nabla\widetilde{\Lambda}_2, \nabla\widetilde{\Lambda}_2\rangle}$.
The effective operator for~$\mathcal{B}_{D,\varepsilon}$ is defined by
\begin{equation*}
\mathcal{B}_D^0u=-\div g^0\nabla u +2i\langle\nabla u, V_1+\overline{\boldsymbol{\eta}}\rangle +(-W+\overline{Q}+\lambda \overline{Q_0})u,
\quad u\in H^2(\mathcal{O})\cap H^1_0(\mathcal{O}).
\end{equation*}
This operator can be represented as
$\mathcal{B}_D^0=(\mathbf{D}-\mathbf{A}^0)^*g^0(\mathbf{D}-\mathbf{A}^0)+\mathcal{V}^0+\lambda \overline{Q_0}$,
where
$\mathbf{A}^0 :=(g^0)^{-1}(V_1+\overline{g\mathbf{A}})$ and $\mathcal{V}^0 :=\overline{\mathcal{V}}+\overline{\langle g\mathbf{A},\mathbf{A}\rangle}-\langle g^0\mathbf{A}^0,\mathbf{A}^0\rangle -W$.

According to Remark~\ref{Remark scalar problem}, in the case under consideration, Conditions~\ref{Condition Lambda in L infty} and~\ref{Condition tilde Lambda in Lp} are satisfied, and the norms~$\Vert \Lambda\Vert _{L_\infty}$, $\Vert \widetilde{\Lambda}\Vert _{L_\infty}$ are controlled in terms of the initial data~\eqref{problem data for the scalar operator}.
Therefore, it is possible to use the simpler corrector \eqref{K_D^0}:
\begin{equation}
\label{K_D for scalar case}
\mathcal{K}_D^0(\varepsilon ;\zeta):=\bigl(\Lambda ^\varepsilon \mathbf{D}+\widetilde{\Lambda}^\varepsilon \bigr)(\mathcal{B}_D^0-\zeta\overline{Q_0})^{-1}
=\bigl( \Psi ^\varepsilon \nabla + \widetilde{\Lambda}^\varepsilon \bigr)(\mathcal{B}_D^0-\zeta\overline{Q_0})^{-1}.
\end{equation}
The operator \eqref{G3(eps;zeta)} can be written as $G^0_D(\varepsilon;\zeta)=-i\mathcal{G}^0_D (\varepsilon;\zeta)$, where
\begin{equation}
\label{mathcal G_3}
\mathcal{G}^0_D(\varepsilon ;\zeta ):=\widetilde{g}^\varepsilon \nabla(\mathcal{B}_D^0-\zeta\overline{Q_0})^{-1}
+g^\varepsilon (\nabla\widetilde{\Lambda})^\varepsilon (\mathcal{B}_D^0-\zeta \overline{Q_0})^{-1}.
\end{equation}
Applying Theorems~\ref{Theorem Dirichlet L2} and~\ref{Theorem no S-eps}, we deduce the following result.

\begin{proposition}
\label{Proposition 1 for scalar case}
Suppose that the assumptions of Subsection~\textnormal{\ref{sec.scalar case}} are satisfied.
Let $\zeta\in\mathbb{C}\setminus \mathbb{R}_+$, $\zeta =\vert \zeta \vert e^{i\phi}$, $0<\phi<2\pi$, and $\vert \zeta\vert \geqslant 1$.
Suppose that $\varepsilon _1$ is subject to Condition~\textnormal{\ref{condition varepsilon}}. Then for $0<\varepsilon\leqslant\varepsilon _1$ we have
\begin{align}
\label{Pr. 10.1_1}
\Vert &(\mathcal{B}_{D,\varepsilon}-\zeta Q_0^\varepsilon)^{-1}-(\mathcal{B}_D^0-\zeta \overline{Q_0} )^{-1}\Vert _{L_2(\mathcal{O})\rightarrow L_2(\mathcal{O})}\leqslant C_4 c(\phi)^5\varepsilon\vert \zeta\vert ^{-1/2},
\\
\label{Pr. 10.1_2}
\begin{split}
\Vert &(\mathcal{B}_{D,\varepsilon }-\zeta Q_0^\varepsilon )^{-1}-(\mathcal{B}_D^0-\zeta \overline{Q_0})^{-1}-\varepsilon \mathcal{K}_D^0(\varepsilon ;\zeta)\Vert _{L_2(\mathcal{O})\rightarrow H^1(\mathcal{O})}
\leqslant
C_5 c(\phi)^2\varepsilon ^{1/2}\vert \zeta\vert ^{-1/4}+C_{23} c(\phi)^4\varepsilon ,
\end{split}
\\
\label{Pr. 10.1_3}
\begin{split}
\Vert & g^\varepsilon \nabla (\mathcal{B}_{D,\varepsilon }-\zeta Q_0^\varepsilon )^{-1}- \mathcal{G}^0_D(\varepsilon ;\zeta)\Vert _{L_2(\mathcal{O})\rightarrow L_2(\mathcal{O})}
\leqslant
\widetilde{C}_5 c(\phi)^2\varepsilon ^{1/2}\vert \zeta\vert ^{-1/4}+\widetilde{C}_{23} c(\phi)^4\varepsilon .
\end{split}
\end{align}
Here $c(\phi)$ is given by~\eqref{c(phi)}. The constants $C_4$, $C_5$, $C_{23}$, $\widetilde{C}_5$, and $\widetilde{C}_{23}$
depend only on the initial data~\eqref{problem data for the scalar operator} and the domain~$\mathcal{O}$.
\end{proposition}

The results of Section~\ref{more} also can be applied to the operator $\mathcal{B}_{D,\varepsilon}$.

``Another''\, approximation for~$(\mathcal{B}_{D,\varepsilon}-\zeta Q_0^\varepsilon )^{-1}$
follows from Theorems~\ref{Theorem Dr appr} and~\ref{Theorem Dr appr no S_eps}.

\begin{proposition}
\label{Proposition anither appr scalar case}
Suppose that the assumptions of Subsection~\textnormal{\ref{sec.scalar case}} are satisfied.
Denote $f(\mathbf{x}):=Q_0(\mathbf{x})^{-1/2}$ and $f_0:=(\overline{Q_0})^{-1/2}$.
Let $\widetilde{\mathcal{B}}_{D,\varepsilon}:=f^\varepsilon\mathcal{B}_{D,\varepsilon}f^\varepsilon$ and $\widetilde{\mathcal{B}}_D^0:=f_0\mathcal{B}_D^0f_0$.
Suppose that~$\varepsilon_1$ is subject to Condition~\textnormal{\ref{condition varepsilon}}.
Let $0< \eps_\flat \le \eps_1$. Suppose that $c_\flat \geqslant 0$ is a common lower bound of the operators $\widetilde{\mathcal{B}}_{D,\varepsilon}$ for any $0< \eps \le \eps_\flat$ and $\widetilde{\mathcal{B}}_D^0$. Let $\varrho _\flat (\zeta )$ be given by \eqref{rho(zeta)}. Then for $0<\varepsilon\leqslant \varepsilon _\flat$
and $\zeta \in \mathbb{C}\setminus [c_\flat ,\infty)$ we have
\begin{align*}
\Vert &(\mathcal{B}_{D,\varepsilon}-\zeta Q_0^\varepsilon)^{-1}-(\mathcal{B}_D^0-\zeta \overline{Q_0} )^{-1}\Vert _{L_2(\mathcal{O})\rightarrow L_2(\mathcal{O})}
\leqslant C_{26} \varepsilon \varrho _\flat (\zeta),
\\
\begin{split}
\Vert&
(\mathcal{B}_{D,\varepsilon }-\zeta Q_0^\varepsilon )^{-1}-(\mathcal{B}_D^0-\zeta \overline{Q_0})^{-1}-\varepsilon \mathcal{K}_D^0(\varepsilon ;\zeta)\Vert _{L_2(\mathcal{O})\rightarrow H^1(\mathcal{O})}
\\
&\leqslant C_{29} \bigl( \varepsilon^{1/2} \varrho_\flat (\zeta)^{1/2} + \varepsilon |1+ \zeta|^{1/2} \varrho_\flat (\zeta) \bigr),
\end{split}
\\
\begin{split}
\Vert & g^\varepsilon \nabla (\mathcal{B}_{D,\varepsilon }-\zeta Q_0^\varepsilon )^{-1}- \mathcal{G}^0_D(\varepsilon ;\zeta)\Vert _{L_2(\mathcal{O})\rightarrow L_2(\mathcal{O})}
\leqslant \widetilde{C}_{29}
\bigl( \varepsilon^{1/2} \varrho_\flat (\zeta)^{1/2} + \varepsilon |1+ \zeta|^{1/2} \varrho_\flat (\zeta) \bigr).
\end{split}
\end{align*}
The constants $C_{26}$, $C_{29}$, and $\widetilde{C}_{29}$ depend only on the initial data \eqref{problem data for the scalar operator} and the domain~$\mathcal{O}$.
\end{proposition}

\subsection{The periodic Schr\"odinger operator}
\label{sec. Periodic Schred. op.}

Suppose that $\check{g}(\mathbf{x})$ is a $\Gamma$-periodic symmetric $(d\times d)$-matrix-valued function in $\mathbb{R}^d$
with real entries and such that $\check{g}(\mathbf{x})>0$; $\check{g},\check{g}^{-1}\in L_\infty (\mathbb{R}^d)$.
Suppose that $\check{v}(\mathbf{x})$ is a real-valued $\Gamma$-periodic function such that
$\check{v}\in L_s(\Omega)$, $s=1$ for $d=1$, $s>d/2$ for $d\geqslant 2$.
By $\check{\mathcal{A}}$ we denote the operator in $L_2(\mathbb{R}^d)$ corresponding to the quadratic form
$\int _{\mathbb{R}^d}\left(\langle\check{g}(\mathbf{x})\mathbf{D}u,\mathbf{D}u\rangle+\check{v}(\mathbf{x})\vert u\vert ^2\right)\,d\mathbf{x}$, $u\in H^1(\mathbb{R}^d)$.
Adding an appropriate constant to the potential~$\check{v}(\mathbf{x})$,
\textit{we may assume that the bottom of the spectrum of $\check{\mathcal{A}}$ is the point $\lambda_0=0$}.
Under this condition, the operator~$\check{\mathcal{A}}$ admits a factorization (see \cite[Chapter~6, Subsection~1.1]{BSu}).

Now, in $L_2(\mathcal{O})$, we consider the operator $\check{\mathcal{A}}_D=\mathbf{D}^*\check{g}(\mathbf{x})\mathbf{D}+\check{v}(\mathbf{x})$
with the Dirichlet boundary condition.
The precise definition of the operator~$\check{\mathcal{A}}_D$ is given in terms of the quadratic form
\begin{equation}
\label{check a}
\check{\mathfrak{a}}[u,u]=\int _\mathcal{O}\left(\langle\check{g}(\mathbf{x})\mathbf{D}u,\mathbf{D}u\rangle+\check{v}(\mathbf{x})\vert u\vert ^2\right)\,d\mathbf{x},\quad u\in H^1_0(\mathcal{O}).
\end{equation}
The operator $\check{\mathcal{A}}_D$ inherits a factorization of $\check{\mathcal{A}}$.
To describe this factorization, we consider the equation
\begin{equation}
\label{eq omega}
\mathbf{D}^*\check{g}(\mathbf{x})\mathbf{D}\omega (\mathbf{x})+\check{v}(\mathbf{x})\omega(\mathbf{x})=0.
\end{equation}
This equation has a $\Gamma$-periodic solution~$\omega\in \widetilde{H}^1(\Omega)$ defined up to a constant factor.
This factor can be fixed so that $\omega (\mathbf{x})>0$ and
$\int _\Omega \omega ^2(\mathbf{x})\,d\mathbf{x}=\vert \Omega \vert$.
Moreover, this solution is positive definite and bounded: $0<\omega _0\leqslant \omega (\mathbf{x})\leqslant\omega_1<\infty$.
The norms $\Vert \omega\Vert _{L_\infty}$ and $\Vert \omega ^{-1}\Vert _{L_\infty}$ are controlled in terms of $\Vert \check{g}\Vert _{L_\infty}$, $\Vert \check{g}^{-1}\Vert _{L_\infty}$, and $\Vert \check{v}\Vert _{L_s(\Omega)}$. Note that $\omega$ and $\omega ^{-1}$ are multipliers in $H^1_0(\mathcal{O})$.
Substituting $u=\omega z$, $z\in H^1_0(\mathcal{O})$, and taking \eqref{eq omega} into account,
we represent the form~\eqref{check a} as
$\check{\mathfrak{a}}[u,u]=\int _\mathcal{O} \omega (\mathbf{x})^2\langle\check{g}(\mathbf{x})\mathbf{D}z,\mathbf{D}z\rangle\,d\mathbf{x}$.
Hence, the operator $\check{\mathcal{A}}_D$ can be written in a factorized form as follows:
\begin{equation}
\label{check A_D}
\check{\mathcal{A}}_D=\omega ^{-1}\mathbf{D}^* g\mathbf{D}\omega ^{-1},\quad g=\omega^2\check{g}.
\end{equation}

Now we consider the operator
\begin{equation}
\label{check A_D,eps}
\check{\mathcal{A}}_{D,\varepsilon}=(\omega ^\varepsilon) ^{-1}\mathbf{D}^* g^\varepsilon\mathbf{D}(\omega ^\varepsilon) ^{-1}
\end{equation}
with rapidly oscillating coefficients.
In the initial terms, the operator \eqref{check A_D,eps} can be written as
\begin{equation}
\label{check A_D,eps in initial terms}
\check{\mathcal{A}}_{D,\varepsilon}=\mathbf{D}^*\check{g}^\varepsilon\mathbf{D}+\varepsilon ^{-2}\check{v}^\varepsilon.
\end{equation}
It can be interpreted as the Schr\"odinger operator with the rapidly oscillating metric
$\check{g}^\varepsilon$ and the strongly singular potential~$\varepsilon ^{-2}\check{v}^\varepsilon$.

Next, let $\mathbf{A}=\mathrm{col}\,\lbrace A_1(\mathbf{x}),\dots ,A_d(\mathbf{x})\rbrace $, where $A_j(\mathbf{x})$ are $\Gamma$-periodic real-valued functions satisfying~\eqref{A_j in L rho}.
Let $\widehat{v}(\mathbf{x})$ and $\check{\mathcal{V}}(\mathbf{x})$ be $\Gamma$-periodic real-valued functions such that
\begin{equation}
\label{hat v, check V}
\widehat{v},\check{\mathcal{V}}\in L_s(\Omega), \quad s=1\ \mbox{for}\ d=1,\quad s>d/2\ \mbox{for}\ d\geqslant 2;\quad
\int _\Omega \widehat{v}(\mathbf{x})\omega ^2 (\mathbf{x})\,d\mathbf{x}=0.
\end{equation}
In $L_2(\mathcal{O})$, consider the operator~$\check{\mathfrak{B}}_{D,\varepsilon}$ given formally by the differential expression
\begin{equation}
\label{check mathfrak B_D,eps}
\check{\mathfrak{B}}_{D,\varepsilon}
=(\mathbf{D}-\mathbf{A}^\varepsilon )^* \check{g}^\varepsilon (\mathbf{D}-\mathbf{A}^\varepsilon)+\varepsilon ^{-2}\check{v}^\varepsilon +\varepsilon ^{-1}\widehat{v}^\varepsilon +\check{\mathcal{V}}^\varepsilon
\end{equation}
with the Dirichlet boundary condition. The precise definition is given in terms of the corresponding quadratic form.
The operator $\check{\mathfrak{B}}_{D,\varepsilon}$ can be treated as the Schr\"odinger operator with the metric~$\check{g}^\varepsilon$, the magnetic potential~$\mathbf{A}^\varepsilon$, and the electric potential~$\varepsilon ^{-2}\check{v}^\varepsilon +\varepsilon ^{-1}\widehat{v}^\varepsilon +\check{\mathcal{V}}^\varepsilon$ containing the singular summands $\varepsilon ^{-2}\check{v}^\varepsilon$ and $\varepsilon ^{-1}\widehat{v}^\varepsilon $.
We put
\begin{equation}
\label{v=, V=}
v(\mathbf{x}):=\widehat{v}(\mathbf{x})\omega ^2(\mathbf{x}),\quad \mathcal{V}(\mathbf{x}):=\check{\mathcal{V}}(\mathbf{x})\omega ^2(\mathbf{x}).
\end{equation}
Using \eqref{check A_D,eps} and \eqref{check A_D,eps in initial terms}, we see that $\check{\mathfrak{B}}_{D,\varepsilon}=(\omega ^\varepsilon )^{-1}\mathfrak{B}_{D,\varepsilon}(\omega ^\varepsilon )^{-1}$, where the operator $\mathfrak{B}_{D,\varepsilon}$
is given by the expression~\eqref{mathfrak B_D,eps} with $g$ defined in \eqref{check A_D}, and $v$, $\mathcal{V}$ defined by \eqref{v=, V=}. Taking \eqref{hat v, check V} into account and using the properties of the function $\omega$,
 we see that the coefficients $v$ and $\mathcal{V}$ satisfy conditions~\eqref{v,V condition}.
 Then the operator~$\mathfrak{B}_{D,\varepsilon}$ can be represented in the form~\eqref{mathfrak B_D,eps in other words},
 where $a_j$, $j=1,\dots,d$, and $Q$ are defined in terms of $g$, $\mathbf{A}$, $v$, and $\mathcal{V}$
 as in Subsection~\ref{sec.scalar case}.

Let $\check{Q}_0(\mathbf{x})$ be a $\Gamma$-periodic positive definite and bounded real-valued function.
Next, we choose the constant $\lambda$ according to condition~\eqref{lambda =}
for the operator whose coefficients $g$, $a_j$, $j=1,\dots ,d$, and $Q$ are the same as the coefficients of
$\mathfrak{B}_{D,\varepsilon}$, and the coefficient $Q_0$ is given by $Q_0(\mathbf{x}):=\check{Q}_0(\mathbf{x})\omega ^2(\mathbf{x})$.
Then the operators $\check{\mathcal{B}}_{D,\varepsilon}:=\check{\mathfrak{B}}_{D,\varepsilon}+\lambda \check{Q}_0^\varepsilon$ and  $\mathcal{B}_{D,\varepsilon}:=\mathfrak{B}_{D,\varepsilon}+\lambda Q_0^\varepsilon$ satisfy the following relation: $\check{\mathcal{B}}_{D,\varepsilon}=(\omega ^\varepsilon )^{-1}\mathcal{B}_{D,\varepsilon}(\omega ^\varepsilon )^{-1}$. Obviously,
\begin{equation}
\label{check Resolvent}
(\check{\mathcal{B}}_{D,\varepsilon}-\zeta\check{Q}_0^\varepsilon)^{-1}=\omega ^\varepsilon (\mathcal{B}_{D,\varepsilon}-\zeta Q_0^\varepsilon )^{-1}\omega ^\varepsilon .
\end{equation}
Now the initial data reduces to the following set
\begin{equation}
\label{problem data for Schredinger}
\begin{split}
&d,\;\rho,\; s;\ \Vert \check{g}\Vert _{L_\infty},\  \Vert \check{g}^{-1}\Vert _{L_\infty},\ \Vert \mathbf{A}\Vert _{L_\rho (\Omega)},\  \Vert \check{v}\Vert _{L_s(\Omega)}, \ \Vert \widehat{v}\Vert _{L_s(\Omega)},\ \Vert \check{\mathcal{V}}\Vert _{L_s(\Omega)},
\\
&\Vert \check{Q}_0\Vert _{L_\infty}, \ \Vert \check{Q}_0^{-1}\Vert _{L_\infty};\  \mbox{the parameters of the lattice}\ \Gamma .
\end{split}
\end{equation}
Applying \eqref{check Resolvent} and Propositions~\ref{Proposition 1 for scalar case} and~\ref{Proposition anither appr scalar case},
 we obtain the following result.

\begin{proposition}
\label{Proposition Schredinger}
Suppose that the assumptions of
Subsection~\textnormal{\ref{sec. Periodic Schred. op.}} are satisfied.
Let $\mathcal{B}_D^0$ be the effective operator for the operator $\mathcal{B}_{D,\varepsilon}$.
Let $\mathcal{K}_D^0(\varepsilon ;\zeta)$ and $\mathcal{G}^0_D(\varepsilon;\zeta)$ be the operators~\eqref{K_D for scalar case} and~\eqref{mathcal G_3} for the operator~$\mathcal{B}_{D,\varepsilon}$.  Suppose that $\varepsilon _1$ is subject to Condition~\textnormal{\ref{condition varepsilon}}.

\noindent
$1^\circ$.
Let $\zeta\in\mathbb{C}\setminus\mathbb{R}_+$, $\zeta =\vert \zeta\vert e^{i\phi}$, $0<\phi <2\pi$, and $\vert \zeta \vert \geqslant 1$.
Then for $0<\varepsilon\leqslant\varepsilon _1$ we have
\begin{align}
\label{Pr. 10.3_1}
\Vert &(\check{\mathcal{B}}_{D,\varepsilon}-\zeta\check{Q}_0^\varepsilon )^{-1}-\omega ^\varepsilon (\mathcal{B}_D^0-\zeta\overline{Q_0})^{-1}\omega ^\varepsilon \Vert _{L_2(\mathcal{O})\rightarrow L_2(\mathcal{O})}
\leqslant C_4\Vert \omega \Vert ^2_{L_\infty}c(\phi)^5\varepsilon\vert \zeta\vert ^{-1/2},
\\
\label{Pr. 10.3_2}
\begin{split}
\Vert &(\omega ^\varepsilon )^{-1}(\check{\mathcal{B}}_{D,\varepsilon}-\zeta \check{Q}_0^\varepsilon )^{-1}-(\mathcal{B}_D^0-\zeta\overline{Q_0})^{-1}\omega ^\varepsilon -\varepsilon \mathcal{K}_D^0(\varepsilon ;\zeta )\omega ^\varepsilon \Vert _{L_2(\mathcal{O})\rightarrow H^1(\mathcal{O})}
\\
&\leqslant C_5\Vert \omega \Vert _{L_\infty}c(\phi)^2\varepsilon ^{1/2}\vert \zeta\vert ^{-1/4}
+C_{23}\Vert \omega\Vert _{L_\infty} c(\phi)^4 \varepsilon ,
\end{split}
\\
\label{Pr. 10.3_3}
\begin{split}
\Vert & g^\varepsilon \nabla (\omega ^\varepsilon )^{-1}(\check{\mathcal{B}}_{D,\varepsilon}-\zeta \check{Q}_0^\varepsilon )^{-1}
-\mathcal{G}^0_D(\varepsilon;\zeta)\omega ^\varepsilon \Vert _{L_2(\mathcal{O})\rightarrow L_2(\mathcal{O})}
\\
&\leqslant \widetilde{C}_5\Vert \omega\Vert _{L_\infty} c(\phi)^2\varepsilon ^{1/2}\vert \zeta\vert ^{-1/4}
+\widetilde{C}_{23}\Vert \omega\Vert _{L_\infty} c(\phi)^4\varepsilon .
\end{split}
\end{align}
Here $c(\phi)$ is given by \eqref{c(phi)}.

\noindent
$2^\circ$.  Denote $f(\mathbf{x}):=Q_0(\mathbf{x})^{-1/2}$ and $f_0:=(\overline{Q_0})^{-1/2}$. 
Let $\widetilde{\mathcal{B}}_{D,\varepsilon}:=f^\varepsilon \mathcal{B}_{D,\varepsilon}f^\varepsilon$ and $\widetilde{\mathcal{B}}_{D}^0:= f_0\mathcal{B}_D^0 f_0$.
Let $0< \eps_\flat \le \eps_1$.
Suppose that $c_\flat \geqslant 0$ is a common lower bound of the operators~$\widetilde{\mathcal{B}}_{D,\varepsilon}$ for any $0< \eps \le \eps_\flat$ 
and~$\widetilde{\mathcal{B}}_{D}^0$. Then
for $0<\varepsilon\leqslant\varepsilon _\flat$ and $\zeta\in\mathbb{C}\setminus [c_\flat ,\infty)$ we have
\begin{align*}
\Vert &(\check{\mathcal{B}}_{D,\varepsilon}-\zeta\check{Q}_0^\varepsilon )^{-1}-\omega ^\varepsilon (\mathcal{B}_D^0-\zeta\overline{Q_0})^{-1}\omega ^\varepsilon \Vert _{L_2(\mathcal{O})\rightarrow L_2(\mathcal{O})}
\leqslant C_{26}\Vert \omega\Vert ^2_{L_\infty} \varepsilon \varrho _\flat (\zeta) ,
\\
\begin{split}
\Vert &(\omega ^\varepsilon )^{-1}(\check{\mathcal{B}}_{D,\varepsilon}-\zeta\check{Q}_0^\varepsilon )^{-1}-(\mathcal{B}_D^0-\zeta\overline{Q_0})^{-1}\omega ^\varepsilon -\varepsilon\mathcal{K}_D^0(\varepsilon;\zeta)\omega ^\varepsilon \Vert _{L_2(\mathcal{O})\rightarrow H^1(\mathcal{O})}
\\
&\leqslant C_{29} \Vert \omega\Vert _{L_\infty} \bigl( \varepsilon^{1/2} \varrho_\flat (\zeta)^{1/2} + \varepsilon |1+ \zeta|^{1/2} \varrho_\flat (\zeta) \bigr),
\end{split}
\\
\begin{split}
\Vert & g^\varepsilon\nabla (\omega ^\varepsilon )^{-1}(\check{\mathcal{B}}_{D,\varepsilon}-\zeta\check{Q}_0^\varepsilon )^{-1}
-\mathcal{G}_D^0(\varepsilon;\zeta)\omega ^\varepsilon \Vert _{L_2(\mathcal{O})\rightarrow L_2(\mathcal{O})}
\\
&\leqslant \widetilde{C}_{29} \Vert \omega\Vert _{L_\infty} \bigl( \varepsilon^{1/2} \varrho_\flat (\zeta)^{1/2} + \varepsilon |1+ \zeta|^{1/2} \varrho_\flat (\zeta) \bigr).
\end{split}
\end{align*}
Here $\varrho _\flat (\zeta)$ is given by \eqref{rho(zeta)}.

The constants $C_4$, $C_5$, $C_{23}$,~$C_{26}$, $C_{29}$, $\widetilde{C}_5$, $\widetilde{C}_{23}$, $\widetilde{C}_{29}$, and
 $\Vert \omega\Vert _{L_\infty}$ depend only on the initial data \eqref{problem data for Schredinger} and the domain~$\mathcal{O}$.
\end{proposition}

\begin{proof}
Multiplying the operators under the norm sign in \eqref{Pr. 10.1_1} by $\omega ^\varepsilon$ from both sides and
using~\eqref{check Resolvent}, we arrive at~\eqref{Pr. 10.3_1}.

From \eqref{check Resolvent} it follows that $(\omega ^\varepsilon )^{-1}(\check{\mathcal{B}}_{D,\varepsilon}-\zeta\check{Q}_0^\varepsilon )^{-1}=(\mathcal{B}_{D,\varepsilon}-\zeta Q_0^\varepsilon )^{-1}\omega ^\varepsilon$.
Multiplying the operators under the norm sign in~\eqref{Pr. 10.1_2} by~$\omega ^\varepsilon$ from the right,
we obtain~\eqref{Pr. 10.3_2}. Similarly, \eqref{Pr. 10.1_3} implies \eqref{Pr. 10.3_3}.

The results of assertion~$2^\circ$ are deduced from Proposition~\textnormal{\ref{Proposition anither appr scalar case}}
in a similar way.
\end{proof}

\begin{remark}
Proposition~\textnormal{\ref{Proposition Schredinger}} demonstrates that for the operators~\eqref{mathfrak B_D,eps} and \eqref{check mathfrak B_D,eps} the nature of the results is different.
Because of the presence of the strongly singular potential~$\varepsilon ^{-2}\check{v}^\varepsilon$, the generalized resolvent~$(\check{\mathcal{B}}_{D,\varepsilon}-\zeta\check{Q}_0^\varepsilon)^{-1}$ has no limit in the $L_2(\mathcal{O})$-operator norm.
It is approximated by the operator $(\mathcal{B}_D^0-\zeta\overline{Q_0})^{-1}$ sandwiched between
 the rapidly oscillating factors $\omega ^\varepsilon$.
\end{remark}

\end{document}